\documentclass[10pt]{article} 
\usepackage{amsmath,amssymb,amsfonts,amsthm,amscd,graphicx,psfrag,epsfig}
\usepackage{color}
\definecolor{blue}{rgb}{0,0,0.7}
\definecolor{red}{rgb}{0.75, 0, 0}
\usepackage[titletoc,toc]{appendix}

\newtheorem{theorem}{Theorem}[section]

\newtheorem{theorem-definition}[theorem]{Theorem-Definition}
\newtheorem{theorem-construction}[theorem]{Theorem-Construction}
\newtheorem{lemma-definition}[theorem]{Lemma--Definition}
\newtheorem{lemma-construction}[theorem]{Lemma--Construction}
\newtheorem{lemma}[theorem]{Lemma}
\newtheorem{proposition}[theorem]{Proposition}
\newtheorem{corollary}[theorem]{Corollary}
\newtheorem{conjecture}[theorem]{Conjecture}
\newtheorem{definition}[theorem]{Definition}

\newcommand{\old}[1]{}
\newcommand{\lr}{{\rm L}}

\newcommand{\Z}{{\mathbb Z}}
\newcommand{\R}{{\mathbb R}}
\newcommand{\Q}{{\mathbb Q}}
\newcommand{\C}{{\mathbb C}}

\newcommand{\A}{{\rm A}}

\newcommand{\B}{{\rm B}}
\newcommand{\G}{{\rm G}}
\newcommand{\U}{{\rm U}}

\newcommand{\wt}{\widetilde}
\newcommand{\val}{{\rm val}}

\newcommand{\lms}{\longmapsto}
\newcommand{\lra}{\longrightarrow}
\newcommand{\hra}{\hookrightarrow}
\newcommand{\ra}{\rightarrow}
\newcommand{\be}{\begin{equation}}
\newcommand{\ee}{\end{equation}}
\newcommand{\bt}{\begin{theorem}}
\newcommand{\et}{\end{theorem}}
\newcommand{\bd}{\begin{definition}}
\newcommand{\ed}{\end{definition}}
\newcommand{\bp}{\begin{proposition}}
\newcommand{\ep}{\end{proposition}}
\newcommand{\blc}{\begin{lemma-construction}}
\newcommand{\elc}{\end{lemma-construction}}

\newcommand{\bl}{\begin{lemma}}
\newcommand{\el}{\end{lemma}}
\newcommand{\bc}{\begin{corollary}}
\newcommand{\ec}{\end{corollary}}
\newcommand{\bcon}{\begin{conjecture}}
\newcommand{\econ}{\end{conjecture}}
\newcommand{\la}{\label}

\begin{document}
\date{}

\title{Geometry of canonical bases and mirror symmetry}
\author{Alexander Goncharov, Linhui Shen}

\maketitle

\begin{abstract}\begin{footnotesize}
A {\it decorated surface}  $S$ is an oriented surface 
with boundary 
and a finite, possibly empty,  set  
of {\it special} 
points on the boundary,  considered modulo isotopy. 
Let ${\rm G}$ be a split reductive group over $\Q$.

A pair $({\rm G}, S)$ gives rise to a moduli space ${\cal A}_{{\rm G}, S}$, closely related to the moduli space of 
${\rm G}$-local systems on $S$. It is equipped 
with a positive structure \cite{FG1}.
So 
a set ${\cal A}_{{\rm G},  S}(\Z^t)$ of its integral tropical points  is defined. 
We introduce a rational positive function ${\cal W}$ on the space ${\cal A}_{{\rm G},   S}$, called the 
{\it potential}. 
Its tropicalisation is a  function 
${\cal W}^t: {\cal A}_{{\rm G}, S}(\Z^t) \to \Z$. 
 The condition ${\cal W}^t\geq 0$ defines a subset of {\it positive integral tropical points} 
${\cal A}^+_{{\rm G}, S}(\Z^t)$. 
For ${\rm G=SL}_2$, we recover the set of positive integral ${\cal A}$-laminations on $  S$  
 from \cite{FG1}.

We prove
 that when $S$ is a disc with $n$ special points on the boundary, 
the set 
${\cal A}^+_{{\rm G}, S}(\Z^t)$ parametrises top dimensional components of the 
 fibers of the convolution maps. Therefore, 
via the geometric Satake correspondence \cite{L4}, \cite{G}, \cite{MV}, \cite{BD} they 
provide a canonical basis in the  tensor product  invariants 
of irreducible modules of the Langlands dual group ${\rm G}^L$:
\be \la{tmima}
(V_{\lambda_1}\otimes\ldots\otimes V_{\lambda_n})^{{\rm G}^L}.
\ee 
When ${\rm G=GL}_m$, $n=3$, there is a special coordinate system on 
${\cal A}_{{\rm G},  S}$ \cite{FG1}. 
We show that it identifies the set ${\cal A}^+_{{\rm GL_m}, S}(\Z^t)$ 
with Knutson-Tao's hives \cite{KT}. 
Our  result generalises a theorem of 
 Kamnitzer \cite{K1}, who  used hives to parametrise top 
components of convolution varieties for ${\rm G=GL}_m$, $n=3$.   
For ${\rm G=GL}_m$, $n>3$, we prove Kamnitzer's conjecture \cite{K1}. 
Our parametrisation  is naturally cyclic invariant. 
We show that  for any ${\rm G}$ and $n=3$ it agrees with Berenstein-Zelevinsky's  
 parametrisation \cite{BZ}, whose cyclic invariance is obscure. 

\vskip 2mm

 We define more general positive spaces with potentials $({\cal A}, {\cal W})$,
 parametrising mixed configurations of flags. Using them, 
we define a generalization of Mirkovi\'{c}-Vilonen cycles \cite{MV},   
 and a canonical basis in $V_{\lambda_1}\otimes\ldots\otimes V_{\lambda_n}$, 
generalizing the Mirkovi\'{c}-Vilonen  basis in  $V_{\lambda}$. 
Our construction comes naturally with a parametrisation of the generalised MV cycles. 
For the classical MV cycles it is equivalent to the one discovered by Kamnitzer \cite{K}. 
\vskip 2mm

We prove that  the set ${\cal A}^+_{{\rm G}, S}(\Z^t)$   parametrises 
top dimensional components of a new moduli space,  {\it surface affine Grasmannian}, generalising 
the fibers of the convolution maps. 
These components are usually infinite dimensional. 
We define their dimension being
an element of a $\Z$-torsor, rather then an integer. 
We define a new moduli space ${\rm Loc}_{G^L, S}$, which reduces to the moduli spaces of 
$G^L$-local systems on $S$ if $S$ has no special points. 
The set ${\cal A}^+_{{\rm G}, S}(\Z^t)$ parametrises a basis in the linear space of regular functions on 
${\rm Loc}_{G^L, S}$.

\vskip 2mm

We suggest that the potential ${\cal W}$ itself, not only its tropicalization, is important -- 
it should be viewed as the potential for a Landau-Ginzburg 
model on ${\cal A}_{{\rm G},   S}$. 
We conjecture that the pair $({\cal A}_{{\rm G}, S}, {\cal W})$ is 
the mirror dual to ${\rm Loc}_{G^L, S}$. 
In a special case, 
we recover Givental's description of the quantum cohomology connection 
for flag varieties and its generalisation \cite{GLO2}. \cite{R2}.  
We formulate equivariant homological mirror symmetry conjectures 
parallel to our parametrisations of canonical bases. 

\end{footnotesize}

\end{abstract}

\tableofcontents

\section{Introduction}
\la{sec1}

\subsection{Geometry of canonical bases in representation theory}
\la{sec1.1}

\subsubsection{Configurations of flags and parametrization of canonical bases}
\la{sec1.1.1}

Let ${\rm G}$ be a split semisimple simply-connected  algebraic group over $\Q$. 
There are several basic vector spaces studied in representation theory of the Langlands dual
 group ${\rm G}^L$:

\begin{enumerate}
\item The weight $\lambda$ component $U({\cal N}^L)^{(\lambda)}$ in 
the universal enveloping algebra $U({\cal N}^L)$ of the maximal nilpotent Lie subalgebra in 
the Lie algebra of ${\rm G}^L$. 

\item The weight  
$\mu$ subspace $V^{(\mu)}_\lambda$ in the highest weight $\lambda$ representation $V_\lambda$ of ${\rm G}^L$.

\item The tensor product invariants $
(V_{\lambda_1}\otimes ... \otimes V_{\lambda_n})^{{\rm G}^L}.$

\item The weight  
$\mu$ subspaces in the tensor products $V_{\lambda_1}\otimes ... \otimes V_{\lambda_n}$. 

\end{enumerate}
Calculation of
the dimensions of these spaces, in the cases 1)-3), 
is a fascinating classical problem, which led to Weyl's character formula and 
Kostant's partition function. 

\vskip 3mm 

The first examples of special bases 
in finite dimensional representations are Gelfand-Tsetlin's bases \cite{GT1}, \cite{GT2}. Other examples of special bases were given by De Concini-Kazhdan 
\cite{DCK}.

The {\it canonical bases} in the spaces above were constructed by Lusztig \cite{L1}, \cite{L3}, \cite{L5}. 
Another approach to canonical bases  was developed  
 by Kashiwara \cite{Ka}. 
Canonical bases in representations of  ${\rm GL}_3, {\rm Sp}_4$ were defined 
by Gelfand-Zelevinsky-Retakh 
 \cite{GZ}, \cite{RZ}.

Closely related, 
but in general different bases were considered by 
Nakajima \cite{N1}, \cite{N2}, Malkin \cite{Ma}, Mirkovi\'{c}-Vilonen \cite{MV}, 
and extensively studied afterwords. 
Abusing terminology, we also call 
them  canonical bases. 

It was discovered by Lusztig \cite{L}  that, in the cases 1)-2), 
the sets parametrising canonical  bases 
in representations of the group ${\rm G}$ are 
intimately related to the Langlands dual group ${\rm G}^L$. 

Kashiwara discovered in the cases 1)-2) an additional {\it crystal structure} on these sets, 
and Joseph proved a rigidity theorem \cite{J} asserting that, equipped with the crystal structure, 
the sets of parameters 
are uniquely determined.  

\vskip 3mm
One of the results of this paper is
 a uniform geometric construction of the sets parametrizing all of these canonical bases, which 
leads to a natural uniform construction of  canonical 
bases parametrized by these sets in the cases 2)-4).  
In particular, we get a new canonical bases in the case 4), generalizing 
the Mirkovi\'{c}-Vilonen (MV)  basis in $V_\lambda$.  
To explain our set-up let us recall some basic notions.

\vskip 3mm
A {\it positive space} ${\cal Y}$ is a space, 
which could be a stack whose generic part is a variety, equipped with a {\it positive atlas}. 
The latter is a collection of 
rational coordinate systems 
with subtraction free transition functions between 
any pair of the coordinate systems. 
Therefore the set ${\cal Y}(\Z^t)$ of the {\it  integral tropical points} of ${\cal Y}$ is well defined. 
We review all this in Section \ref{sec2.1.2}.

 Let $({\cal Y}, {\cal W})$ be a {\it positive pair} given by a positive space ${\cal Y}$ equipped with a 
positive rational function ${\cal W}$. 
Then one can tropicalize the function ${\cal W}$, getting a  $\Z$-valued function
$$
{\cal W}^t: {\cal Y}(\Z^t)\lra \Z.
$$
Therefore a positive pair $({\cal Y}, {\cal W})$ determines a set of {\it positive integral tropical points}:
\be \la{postropintp}
{\cal Y}_{\cal W}^{+}(\Z^t):=\{l\in {\cal Y}(\Z^t)~|~{\cal W}^t(l)\geq 0\}.
\ee
We usually omit ${\cal W}$ in the notation and denote the set by ${\cal Y}^{+}(\Z^t)$.  

To introduce the positive pairs $({\cal Y}, {\cal W})$ which play the basic role 
in this paper, we need to review some basic facts about flags and decorated flags in ${\rm G}$.

\paragraph{Decorated flags and associated characters.} Below 
${\rm G}$ is a split reductive group over $\Q$. 
Recall that the {\it flag variety}  ${\cal B}$ parametrizes Borel subgroups in {\rm G}. 
Given a Borel subgroup {\rm B}, one has an isomorphism ${\cal B} = {\rm G/B}$. 

Let ${\rm G'}$ be the adjoint group of ${\rm G}$. 
The group ${\rm G}'$ acts by conjugation on pairs $({\rm U}, \chi)$, 
where $\chi: {\rm U} \to {\Bbb A}^1$ is an additive
 character of a maximal unipotent subgroup ${\rm U}$ in ${\rm G}'$. 
The subgroup ${\rm U}$ stabilizes each pair $({\rm U}, \chi)$. 
A character $\chi$ is {\it non-degenerate}  
if ${\rm U}$ is the stabilizer of  $({\rm U}, \chi)$.  
  The {\it principal affine space}\footnote{Inspite of the name, it is not an affine variety.}   
${\cal A}_{\rm G'}$ parametrizes pairs $({\rm U}, \chi)$ where 
$\chi$ is a non-degenerate additive character of a maximal unipotent group ${\rm U}$. 
Therefore   
there is an isomorphism 
$$
i_{\chi}: {\cal A}_{\rm G'}\stackrel{\sim}{\lra} {\rm G'/U}.
$$ 
This isomorphism is not canonical: the coset $[{\rm U}] \in {\rm G'/U}$
does not determine a point of ${\cal A}_{\rm G'}$. To specify a point one needs 
to choose a non-degenerate character $\chi$. 
One can determine uniquely the character by using  a {\it pinning}, 
 see Sections 
\ref{sec2.1.1}-\ref{sec4.1}. 
So 
writing ${\cal A}_{\rm G'} = {\rm G'}/{\rm U}$ we abuse notation, 
keeping in mind a choice of the character ${\chi}$, or a pinning. 

Having said this,  one defines 
the principal affine space ${\cal A}_{\rm G}$ for the group ${\rm G}$ by 
$
{\cal A}_{\rm G}:= {\rm G/U}. 
$ 
We often write ${\cal A}$ instead of ${\cal A}_{\G}$.
The points of ${\cal A}$  are called {\it decorated flags} in {\rm G}.
The group $\G$ acts on ${\cal A}$ from the left.
For each $\A\in {\cal A}$,  let $\U_\A$ be its stabilizer. It is a maximal unipotent subgroup of $\G$.
There is a canonical projection
\be \la{5.29.12.111}
\pi: {\cal A} \lra {\cal B}, ~~~~\mbox{\rm $\pi (\A):=$ the normalizer of $\U_{\A}$}.
\ee

The projection ${\rm G}\to {\rm G'}$ gives rise to a map 
$p: {\cal A}_{\rm G} \lra {\cal A}_{\rm G'}$ whose fibers are torsors over the center of ${\rm G}$.  
Let $p(\A) = ({\rm U}_\A, \chi_\A)$. Here ${\rm U}_\A$ is a maximal unipotent subgroup of ${\rm G}'$. 
It is identified with 
a similar
subgroup of ${\rm G}$, also denoted by ${\rm U}_\A$.  
So a decorated flag $\A$ in ${\rm G}$ 
 provides a non-degenerate 
character of the maximal unipotent subgroup ${\rm U}_\A$ in ${\rm G}$:
\be \la{11.20.11.10a}
\chi_{\rm A}: {\rm U_\A} \lra {\Bbb A}^1.
\ee
Clearly,  if $u\in \U_{\A}$, then $gug^{-1}\in \U_{g\cdot \A}$, and 
\be \la{obvious}
\chi_{\A}(u)=\chi_{g\cdot \A}(gug^{-1}).
\ee

\paragraph{Example.} A flag for ${\rm SL}_m$ is a nested collection of subspaces 
in an $m$-dimensional vector space $V_m$ equipped with a volume form $\omega \in {\rm det}V_m^*$: 
$$
F_\bullet = F_0 \subset F_1 \subset \ldots \subset F_{m-1} \subset F_m, ~~~~ {\rm dim}F_i=i. 
$$
A decorated flag for ${\rm SL}_m$ is a flag $F_\bullet$ with a choice of non-zero vectors  
$f_i \subset  F_{i}/F_{i-1}$ for each $i=1, \ldots, m-1$, called {\it decorations}. 
For example, ${\cal A}_{SL_2}$ parametrises non-zero vectors in 
a symplectic space $(V_2, \omega)$. The subgroup preserving a 
vector $f \in V_2 -\{0\}$ is given by transformations 
$u_{f}(a): v \lms v+a\omega(v, f) v$. Its character $\chi_f$ is given by 
$\chi_f(u_{f}(a))=a$.

\vskip 2mm
Our  basic geometric objects are the following three types of configuration spaces: 
\be \la{mixedconf}
{\rm Conf}_n({\cal A})= {\rm G} \backslash {\cal A}^n, ~~~~
{\rm Conf}({\cal A}^n, {\cal B}):= {\rm G} \backslash ({\cal A}^n\times {\cal B}), ~~~~
{\rm Conf}({\cal B}, {\cal A}^n, {\cal B}):= {\rm G} \backslash ({\cal B}\times {\cal A}^n\times  {\cal B}).
\ee

\paragraph{The potential ${\cal W}$.} A key observation is that there is a natural rational function
$$
\chi^o: {\rm Conf}({\cal B}, {\cal A}, {\cal B}) = {\rm G} \backslash ({\cal B} 
\times {\cal A}\times {\cal B}) \lra {\Bbb A}^1. 
$$
Let us explain its definition. A pair of Borel subgroups 
$\{{\rm B}_1, {\rm B}_2\}$ is  {\it generic} if 
${\rm B}_1 \cap {\rm B}_2$ is a Cartan subgroup in ${\rm G}$. A pair $\{{\rm A}_1, {\rm B}_2\}\in 
{\cal A} \times {\cal B}$ is generic if 
the pair $(\pi({\rm A}_1), {\rm B}_2)$ is generic. Generic pairs $\{{\rm A}_1, {\rm B}_2\}$ form a principal homogeneous 
${\rm G}$-space. 
Thus, given a triple $\{ {\rm B}_{1}, {\rm A}_2,{\rm B}_{3}\} \in {\cal B} \times {\cal A} \times {\cal B}$ such that 
$\{{\rm A}_2, {\rm B}_3\}$ and $\{{\rm A}_2, {\rm B}_1\}$ are generic, 
there is a unique $u\in {\rm U}_{{\rm A}_2}$  such that
\be
\la{7.20.9.8}
\{ {\rm A}_2,{\rm B}_{3}\} = u \cdot \{ {\rm A}_2, {\rm B}_{1}\}.
\ee
So we define
$
\chi^o({\rm B}_{1}, {\rm A}_2,{\rm B}_{3}):= \chi_{{\rm A}_2}(u).  
$ 
Using it as a building block, we define a  positive rational function ${\cal W}$ on each 
of the spaces (\ref{mixedconf}). 

For example, to define the ${\cal W}$ on the space 
${\rm Conf}_n({\cal A})$ we start with a generic collection $\{\A_1, ..., \A_n\} \in {\cal A}^n$, 
set $\B_i:= \pi(\A_i)$, and define ${\cal W}$ as a sum, 
with the indices modulo $n$:
\be \la{theWp}
{\cal W}: {\rm Conf}_n({\cal A}) \lra {\Bbb A}^1, ~~~~
{\cal W}(\A_1, ..., \A_n):= \sum_{i=1}^n\chi^o(\B_{i-1}, \A_i, \B_{i+1}).
\ee
Note that the potential ${\cal W}$ is well-defined when each adjacent pair $\{\A_i, \A_{i+1}\}$ is generic, meaning that $\{\pi(\A_i), \pi(\A_{i+1})\}$ is generic. 
 Assigning the (decorated) flags  
to the vertices of a polygon, we picture  the potential ${\cal W}$ 
as a sum of the contributions $\chi_{\rm A}$ at the 
$\A$-vertices (shown boldface) of the polygon, see Fig \ref{polygon}. 

By construction, the potential ${\cal W}_{\G}$ on the space ${\rm Conf}_n({\cal A}_{{\rm G}})$
 is the pull back of the potential ${\cal W}_{\G'}$ 
for the adjoint group ${\rm G}'$ 
via the natural projection $p_{{\rm G}\to {\rm G'}}: {\rm Conf}_n({\cal A}_{{\rm G}}) \to {\rm Conf}_n({\cal A}_{{\rm G}'})$:
\be \la{potentialadj}
{\cal W}_{\G} = p_{{\rm G}\to {\rm G'}}^*{\cal W}_{\G'}.
\ee

\begin{figure}[ht]
\centerline{\epsfbox{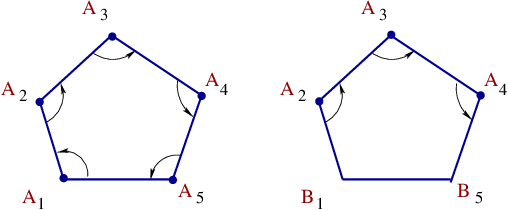}}
\caption{The potential ${\cal W}$ is a sum of the contributions $\chi_{\rm A}$  at the $\A$-vertices (boldface). 
\label{polygon}}
\end{figure}

Potentials for the other two spaces in (\ref{mixedconf}) are defined similarly, as the sums of the characters 
assigned to the 
decorated flags of a configuration. A formula similar to (\ref{potentialadj}) evidently holds.  

\paragraph{Parametrisations of canonical bases.} 
It was shown in \cite{FG1} that all of the spaces (\ref{mixedconf}) 
 have natural positive structures. We show that the potential ${\cal W}$ is a positive rational function. 

We prove that the sets parametrizing 
canonical bases admit a uniform description as the  
sets ${\cal Y}_{\cal W}^{+}(\Z^t)$ of positive integral tropical points assigned to the following positive pairs 
$({\cal Y}, {\cal W})$. To write the potential 
${\cal W}$ we use 
an abbreviation $\chi_{\A_i}:= \chi^o(\B_{i-1},\A_i, \B_{i+1})$, with indices mod $n$:
\begin{enumerate}
\item The canonical basis in $U({\cal N}^L)$: 
$$
{\cal Y}={\rm Conf}({\cal B}, {\cal A},{\cal B}), ~~~{\cal W}({\B}_1, {\A}_2, {\rm B}_3):=\chi_{\A_2}.
$$
\item The canonical basis in $V_\lambda$:
$$
{\cal Y}={\rm Conf}({\cal A}, {\cal A},{\cal B}), ~~~{\cal W}({\rm A}_1, {\rm A}_2, {\rm B}_3):=\chi_{\A_1}
+ \chi_{\A_2}.
$$
\item The canonical basis in invariants of tensor product of $n$ irreducible 
${\rm G}^L$-modules:
\be \la{ppoottee}
{\cal Y}={\rm Conf}_n({\cal A}), ~~~{\cal W}(\A_1,\ldots,\A_n) :=\sum_{i=1}^n\chi_{\A_i}.
\ee

\item The canonical basis in tensor products of $n$ irreducible ${\rm G}^L$-modules:
\be \la{ppoott}
{\cal Y}={\rm Conf}({\cal A}^{n+1}, {\cal B}), ~~~{\cal W}(\A_1,\ldots,\A_{n+1}, \B) :=\sum_{i=1}^{n+1}\chi_{\A_i}.
\ee
\end{enumerate}
 Natural decompositions of these sets, like decompositions into weight subspaces  
in 1) and 2),  are easily described in terms of the corresponding configuration space, see Section \ref{sec2.3.2}.

Let us emphasize that the canonical bases in tensor products are 
not the tensor products of canonical bases in irreducible representations. 
Similarly, in spite of the natural decomposition
$$
V_{\lambda_1}\otimes ... \otimes V_{\lambda_n} = \oplus_{\lambda} V_{\lambda} \otimes (V^*_{\lambda} \otimes V_{\lambda_1}\otimes ... \otimes V_{\lambda_n})^{{\rm G}^L},
$$
the canonical basis on the left is not a product of the canonical bases on the right.

\vskip 2mm
Descriptions of the sets parametrizing the canonical bases 
were known in different but equivalent formulations in the following cases: 

In the cases 1)-2) there is  the 
original parametrization of Lusztig \cite{L}.  
 
In the case 3) for $n=3$, 
there is Berenstein-Zelevinsky's parametrization \cite{BZ}, referred to as the BZ data. 
We produce in Appendix 
 an isomorphism between our parametrization and the 
BZ data. The cyclic symmetry, 
evident in our approach, 
is obscure for the BZ data.

The description in the $n>3$ case in 3) seems to be new.  

The cases 1), 2) and 4) were investigated by Berenstein and Kazhdan \cite{BK1},\cite{BK2}, who  introduced and studied 
{\it geometric crystals} as algebraic-geometric avatars of Kashiwara's {\it crystals}. 
In particular, they describe the sets 
parametrizing canonical bases in  the form (\ref{postropintp}), without using, however, 
configuration spaces. 
We show in Appendix \ref{sec9} that the space of generic configurations ${\rm Conf}^*({\cal A}^n,{\cal B})$ 
with the potential ${\cal W}$  
is a {\it positive decorated geometric crystal} in the sense of  
\cite{BK3}. 
Interpretation of geometric crystals relevant to representation theory 
as moduli spaces of  mixed configurations of  flags 
makes, to our opinion, the story more transparent. 

\vskip 2mm
To define canonical bases in representations, one needs to choose a maximal torus 
in ${\rm G}^L$ 
and a positive Weyl chamber. 
Usual descriptions of the sets parametrizing canonical bases require the same choice. 
Unlike this, working with configurations we do not require  such choices.\footnote{{
We would like to stress that the 
positive structures and potentials on configuration spaces which we employ 
for parametrization of canonical bases 
do not depend on any extra choices, like pinning etc., in the group. See Section \ref{proofmth1}.}}

\vskip 2mm

Most importantly, our parametrization of the canonical basis in tensor products invariants 
leads immediately to a similar set which parametrizes a linear basis in the space of functions on the moduli space 
${\rm Loc}_{{\rm G^L}, S}$ of 
${\rm G}^L$-local systems on a decorated surface $S$. Here the approach via configurations of decorated flags, 
and in particular its transparent  cyclic 
invariance, are essential.  See the example when $G=SL_2$  in Section \ref{LAMCB}. 

\vskip 2mm
Summarizing, we understood the sets parametrizing 
canonical bases as the sets of positive integral tropical points of 
various configuration spaces. Let us show now how this, combined 
with the geometric Satake correspondence \cite{L4}, \cite{G}, \cite{MV}, \cite{BD}, leads  
to a natural uniform construction of canonical bases in the cases 2)-4).  

We explain in Section \ref{sec1.1.2} the construction in the case of tensor products invariants.  
A cyclically invariant canonical basis  in this case was defined  
by Lusztig  \cite[Chapter 28]{L5}. However Lusztig's construction does not provide a 
description of the set parametrizing the basis. 
Our basis in tensor products is new -- it generalizes the MV basis in $V_\lambda$. 
We explain this in Section \ref{tensor}.

\subsubsection{Constructing canonical bases in tensor products invariants}
\la{sec1.1.2}

We start with a simple general construction. 
Let ${\cal Y}$ be a positive space, understood just 
as a collection of split tori glued by positive birational maps \cite{FG1}. 
 Since it is a birational notion, there is no set of $F$-points of ${\cal Y}$, where $F$ is a field.  
Let ${\cal K}:=\C((t))$. 
In Section \ref{sec2.2.1} we introduce a set 
$
{\cal Y}^\circ({\cal K}). 
$ 
We call it the set of 
{\it transcendental ${\cal K}$-points of ${\cal Y}$}. 
It is a set making sense of ``generic ${\cal K}$-points of ${\cal Y}$''. In particular, 
if ${\cal Y}$ is given by  a variety $Y$ with a positive rational atlas, 
then ${\cal Y}^\circ({\cal K})\subset Y({\cal K})$. 
The set ${\cal Y}^\circ({\cal K})$  comes with a natural {\it valuation map}:
$$
{\rm val}: {\cal Y}^\circ({\cal K})\lra {\cal Y}(\Z^t). 
$$ 
For any $l\in {\cal Y}(\Z^t)$, we define the {\it transcendental cell} ${\cal C}^\circ_l$ assigned to $l$:
$$
{\cal C}^\circ_l:= {\rm val}^{-1}(l)\subset {\cal Y}^\circ({\cal K}), ~~~~
{\cal Y}^\circ({\cal K}) = \coprod_{l\in {\cal Y}(\Z^t)}{\cal C}^\circ_l.
$$

Let us now go to canonical bases   in invariants of tensor products of 
${\rm G}^L$-modules  (\ref{tmima}). 
The relevant configuration space is  
$
{\rm Conf}_n({\cal A}).
$ 
The tropicalized potential 
$
{\cal W}^t: {\rm Conf}_n({\cal A})(\Z^t) \to \Z 
$  
  determines the subset of 
{positive integral tropical points}:
\be \la{121212as}
{\rm Conf}^+_n({\cal A})(\Z^t):= \{l \in {\rm Conf}_n({\cal A})(\Z^t)~|~ {\cal W}^t(l)\geq 0\}.
\ee
We construct 
a canonical basis in (\ref{tmima})  parametrized by the  set (\ref{121212as}).

Let  ${\cal O}:=\C[[t]]$. In Section \ref{sec2.2.2} we introduce a moduli subspace
\be \la{O-int}
{\rm Conf}^{\cal O}_n({\cal A})\subset {\rm Conf}_n({\cal A})({\cal K}).
\ee
We call it the space  of {\it ${\cal O}$-integral configurations 
of decorated flags}. 
Here are its crucial properties: 

\begin{enumerate}

\item A transcendental cell ${\cal C}^\circ_l$ of ${\rm Conf}_n({\cal A})$ is contained in   ${\rm Conf}_n^{\cal O}({\cal A})$ 
if and only if it corresponds to a positive tropical point. Moreover, 
given a point $l\in {\rm Conf}_n({\cal A})(\Z^t)$, one has 
\be \la{features}
l\in {\rm Conf}^+_n({\cal A})(\Z^t)  ~{\Longleftrightarrow} ~{\cal C}^\circ_l \subset {\rm Conf}^{\cal O}_n({\cal A}) 
~{\Longleftrightarrow}~ {\cal C}^\circ_l \cap {\rm Conf}^{\cal O}_n({\cal A})\not =\emptyset. 
\ee

\item Let ${\rm Gr}:= {\rm G}({\cal K})/{\rm G}({\cal O})$ 
be the affine Grassmannian. 
It follows immediately from the very definition of the subspace (\ref{O-int}) that there 
 is a canonical map
$$
\kappa: {\rm Conf}^{\cal O}_n({\cal A})\lra {\rm Conf}_n({\rm Gr}):= 
{\rm G}({\cal K})\backslash ({\rm Gr})^n.
$$
\end{enumerate}

These two properties of ${\rm Conf}_n^{\cal O}({\cal A})$ allow to 
transport points $l\in {\rm Conf}^+_n({\cal A})(\Z^t)$ 
into the top components of the stack ${\rm Conf}_n({\rm Gr})$. 
Namely, given a point $l\in {\rm Conf}^+_n({\cal A})(\Z^t)$, we define a cycle 
$$
{\cal M}_l := \mbox{closure of ${\cal M}_l ^\circ\subset {\rm Conf}_n({\rm Gr}),~~~~\mbox{where } {\cal M}_l^\circ:= \kappa({\cal C}^\circ_l)$}.
$$
The  cycle ${\cal C}^\circ_l$ is defined for any 
$l\in {\rm Conf}_n({\cal A})(\Z^t)$. However, as it clear from (\ref{features}),
 the map $\kappa$ can be applied to it if and only if $l$ is positive: otherwise 
${\cal C}^\circ_l$ is not in the domain of the map $\kappa$. 

We prove that  the map $l\lms {\cal M}_l$ 
provides a bijection
\be \la{ltoml}
{\rm Conf}^+_n({\cal A})(\Z^t) \stackrel{\sim}{\lra} \{\mbox{\rm closures of the top dimensional components of 
the stack ${\rm Conf}_n({\rm Gr})$}\}. 
\ee
Here the very notion of a ``top dimensional'' component of a stack requires clarification. 
For now, we will bypass this question in a moment by passing to more traditional varieties. 

We use a very general argument to show the injectivity
of the map $l \lms {\cal M}_l$. 
Namely, given a positive rational function $F$ on ${\rm Conf}_n({\cal A})$, we define a 
$\Z$-valued function $D_F$ on ${\rm Conf}_n({\rm Gr})$. 
It generalizes the function on the affine Grassmannian for $\G={\rm GL}_m$ and its products 
defined  by Kamnitzer \cite{K}, \cite{K1}. We prove that 
the restriction of $D_F$ to ${\cal M}^\circ_l$ is equal to the value 
$F^t(l)$ of the tropicalization $F^t$ of $F$ at the  point $l\in  {\rm Conf}^+_n({\cal A})(\Z^t)$. 
Thus the map (\ref{ltoml}) is injective.

\vskip 2mm
Let us reformulate our result in a more traditional language.  The orbits of ${\G}({\cal O})$ acting on ${\rm Gr}\times {\rm Gr}$ are labelled by dominant weights of ${\G}^L$. We write ${\lr_1}\stackrel{\lambda}{\lra}\lr_2$ if $(\lr_1, \lr_2)$ is in the orbit labelled by $\lambda$. Let $[1]$ be the identity coset in ${\rm Gr}$.
 A set $\underline{\lambda}=(\lambda_1,\ldots, \lambda_{n})$ of dominant weights of ${\rm G}^L$ 
determines  a {\it cyclic convolution variety}, better known as a {\it fiber of the convolution map}:
\be \la{convvarvar}
{\rm Gr}_{c(\underline{\lambda})}:=\{(\lr_1,\ldots, \lr_{n})~|~ \lr_1\stackrel{\lambda_1}{\lra }\lr_2 \stackrel{\lambda_2}{\lra }\ldots \stackrel{\lambda_{n}}{\lra} \lr_{n+1},~\lr_1=\lr_{n+1}=[1]\} \subset 
[1]\times {\rm Gr}^{n-1}.
\ee
These varieties  provide a ${\rm G}({\cal O})$-equivariant decomposition
\be \la{cvcvcv}
[1] \times {\rm Gr}^{n-1} = \coprod_{\underline{\lambda}=(\lambda_1, ..., \lambda_{n})}{\rm Gr}_{c(\underline{\lambda})}.
\ee
Since ${\rm G}({\cal O})$ is connected, it preserves each component of ${\rm Gr}_{c(\underline{\lambda})}$. Thus the components of ${\rm Gr}_{c(\underline{\lambda})}$ live naturally on the stack
$$
{\rm Conf}_{n}({\rm Gr}) = {\rm G}({\cal O})\backslash ([1]\times {\rm Gr}^{n-1}).
$$
We prove that the  cycles ${\cal M}_l$ assigned to the points  $l \in {\rm Conf}^+_n({\cal A})(\Z^t)$ 
are closures of the top dimensional components 
of the cyclic convolution varieties. The latter, due to the geometric Satake correspondence, 
give rise to a canonical basis in (\ref{tmima}). 
We already know that the map (\ref{ltoml}) is injective. 
We show that the $\underline{\lambda}$-components of the sets related by the map (\ref{ltoml}) are
 finite sets of the same cardinality, respected by the map. 
Therefore the map (\ref{ltoml}) is an isomorphism.  
\vskip 3mm

Our  result generalizes a theorem of 
 Kamnitzer \cite{K1}, who  used hives \cite{KT} to parametrize top 
components of convolution varieties for ${\rm G=GL}_m$, $n=3$.   

Our construction generalizes Kamnitzer's  construction of 
parametrizations of Mirkovi\'{c}-Vilonen cycles \cite{K}. 
At the same time, it gives a coordinate free description of Kamnitzer's  construction. 

When ${\rm G=GL}_m$, there is a special coordinate system on the space ${\rm Conf}_3({\cal A})$,  
introduced in Section 9 of \cite{FG1}. We show in Section \ref{KT} that it provides an isomorphism of sets
$$
{\rm Conf}^+_3({\cal A})(\Z^t) ~\stackrel{\sim}{\lra} ~\{\mbox{Knutson-Tao's hives \cite{KT}}\}. 
$$ 
Using this, we get a one line proof 
of Knutson-Tao-Woodward's theorem \cite{KTW} in Section \ref{sec2.1.6}.

For ${\rm G=GL}_m$, $n>3$, we prove Kamnitzer conjecture \cite{K1}, describing the top 
components of convolution varieties
via a generalization of hives -- we identify the latter with the set ${\rm Conf}^+_n({\cal A})(\Z^t)$ via
the special positive coordinate systems on   ${\rm Conf}_n({\cal A})$
from \cite{FG1}.

\subsection{Positive tropical points and top components}
\la{sec1.2}

\subsubsection{Our main example}\la{sec1.2.2}

Denote by 
${\rm Conf}^\times_n({\cal A})$ the subvariety of ${\rm Conf}_n({\cal A})$ 
parametrizing configurations of decorated flags $(\A_1, ..., \A_n)$ such that the flags 
$(\pi(\A_i), \pi(\A_{i+1}))$ are in generic position for each $i=1, ..., n$ modulo $n$. 
The potential ${\cal W}$ was defined in (\ref{theWp}). It 
is evidently a regular function on ${\rm Conf}^\times_n({\cal A})$. 

Let ${\rm P}^+$ be the cone of dominant coweights.
There are canonical 
isomorphisms
\be \la{agr}
{\alpha}: {\rm Conf}^\times_2({\cal A}) \stackrel{\sim}{\lra} {\rm H}, ~~~~ {\rm Conf}_2({\rm Gr}) = {\rm P}^+.
\ee

Configurations $(\A_1, ..., \A_n)$ sit at the vertices of a polygon, 
as on Fig \ref{polygon1}. 
Let $\pi_{E}: {\rm Conf}_n({\cal A}) \to {\rm Conf}_2({\cal A})$ be the projection 
corresponding to a side $E$ of the polygon. 
Denote by $\pi^\times_{E}$ its restriction to ${\rm Conf}^\times_n({\cal A})$.  
The collection of the maps $\{\pi^\times_{E}\}$, followed by
 the first isomorphism  in (\ref{agr}) provides a map
$$
\pi: {\rm Conf}^\times_n({\cal A}) \lra {\rm Conf}^\times_2({\cal A})^n \stackrel{{\alpha}}{=} {\rm H}^n.
$$
Using similarly the second isomorphism in (\ref{agr}), we get a map




\begin{figure}[ht]
\centerline{\epsfbox{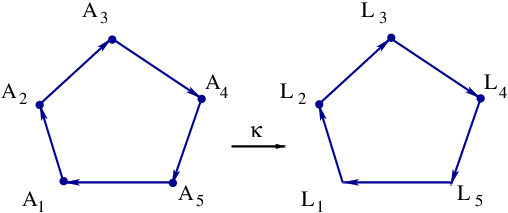}}
\caption{Going from an ${\cal O}$-integral  configuration of decorated flags to a configuration of lattices.}
\label{polygon1}
\end{figure}

$$
\pi_{\rm Gr}: {\rm Conf}_n({\rm Gr}) \lra {\rm Conf}_2({\rm Gr})^n {=}  ({\rm P}^+)^n.
$$

Let $\{\omega_i\}$ be a basis of the cone of positive dominant weights  of ${\rm H}$.  
The functions $\pi_{E}^*\omega_i$ 
are equations of the irreducible components of the divisor 
$D:= {\rm Conf}_n({\cal A}) - {\rm Conf}^\times_n({\cal A})$:
$$
D:= {\rm Conf}_n({\cal A}) - {\rm Conf}^\times_n({\cal A}) = \cup_{E, i}D^E_{i}. 
$$
Equivalently, the component $D^E_{i}$ is determined by the condition that the 
pair of flags at the endpoints of the edge $E$ 
belongs to the codimension one $\G$-orbit corresponding to the simple reflection $s_{i}\in W$. 
\footnote{Indeed,  $\omega_i({\alpha}(\A_1, \A_2))=0$ 
if and only if the corresponding pair of flags belongs to the codimension 
one $\G$-orbit corresponding to a simple reflection $s_i$.}

The space ${\rm Conf}_n({\cal A})$ has a cluster ${\cal A}$-variety structure, described for 
$\G=SL_m$ in  \cite{FG1}, Section 10. 
An important fact \cite{FG5} is that any cluster ${\cal A}$-variety  ${\cal A}$ 
has a canonical cluster volume form $\Omega_{\cal A}$, which in any cluster ${\cal A}$-coordinate system 
$(\A_1, \ldots , \A_n)$ is given by 
$$
\Omega_{\cal A} = \pm d\log \A_1 \wedge \ldots \wedge d\log \A_n. 
$$

The functions $\pi_{E}^*\omega_i$  are  
the {\it frozen ${\cal A}$-cluster coordinates} in the sense of Definition \ref{9.12.14.2}. 
This is equivalent to the claim 
that the canonical  volume form $\Omega_{\cal A}$ on 
${\rm Conf}_n({\cal A})$ has non-zero residues 
precisely at the irreducible components of the divisor $D$.\footnote{Indeed, it follows from 
Lemma \ref{nonfr} and an explcit description of cluster structure on ${\rm Conf}_n({\cal A})$ 
 that the form $\Omega_{\cal A}$ can not have non-zero residues anywhere 
else the divisors $D^E_i$. One can show that the residues at these divisors are non-zero.} 

All this data is defined for any split semi-simple group ${\rm G}$ over $\Q$. 
Indeed, the form $\Omega$ on ${\rm Conf}_n({\cal A})$ for the simply-connected group 
is invariant under the action of the center the group, and thus its integral multiple descends to a 
form on ${\rm Conf}_n({\cal A}_{\rm G})$. The potential ${\cal W}_{\rm G}$ 
is defined by pulling back the potential ${\cal W}_{\rm G'}$ for the adjoint group ${\rm G'}$. 
We continue discussion of this example in Section \ref{sec1.4}, where it is casted as an example 
of the mirror symmetry.

\paragraph{The simplest example.} Let $(V_2, \omega)$ be a two dimensional vector space 
with a symplectic form. Then ${\rm SL_2} = {\rm Aut}(V_2, \omega)$, and ${\cal A}_{SL_2}=V_2 -\{0\}$. 
Next, ${\rm Conf}_n({\cal A}_{\rm SL_2}) = {\rm Conf}_n(V_2)$ is the space of configuration $(l_1, ..., l_n)$ of $n$ non-zero vectors 
in $V_2$. Set $\Delta_{i,j}:= \langle \omega, l_i \wedge l_j\rangle$. Then the potential is given by the following 
formula, where the indices are mod $n$: 
\be \la{potsl2}
{\cal W}:= \sum_{i=1}^n\frac{\Delta_{i, i+2}}{\Delta_{i, i+1}\Delta_{i+1, i+2}}. 
\ee
The boundary divisors are given by equations $\Delta_{i, i+1}=0$. 
To write the volume form, pick a triangulation $T$ of the polygon whose vertices a labeled by the vectors. 
Then, up to a sign,  
$$
\Omega:= \bigwedge_{E} d\log \Delta_E.
$$
where $E$ are the diagonals and sides of the $n$-gon, and $\Delta_E:= \Delta_{i, j}$ if $E=(i,j)$.  
The function (\ref{potsl2}) is invariant under $l_i\to -l_i$, and thus descends to 
${\rm Conf}_n({\cal A}_{\rm PGL_2}) = {\rm Conf}_n(V_2/\pm 1)$.

\subsubsection{The general framework}\la{sec1.2.1}

Let us explain main features of 
the geometric picture
 underlying our construction in most general 
terms, which we later on elaborate in details in every particular situation.   
First, there are three main ingredients: 

\begin{enumerate}

\item A positive space ${\cal Y}$ with a positive rational 
function ${\cal W}$ called the {\it potential}, 
 and a volume form $\Omega_{\cal Y}$ 
with logarithmic singularities. This determines  
the set ${\cal Y}^+_{\cal W}(\Z^t)$ of positive integral tropical points  -- the set parametrizing a canonical basis.\footnote{The set ${\cal Y}(\Z^t)$, the tropicalization ${\cal W}^t$, and thus the subset 
${\cal Y}^+_{\cal W}(\Z^t)$  can also be determined by the volume form 
$\Omega_{\cal Y}$, without using the positive structure on ${\cal Y}$.}

\item A subset of ${\cal O}$-integral points ${\cal Y}^{\cal O} \subset {\cal Y}^\circ({\cal K})$. 
Its key feature is that, given an $ l\in {\cal Y}(\Z^t)$,
\be \la{kf1}
l \in {\cal Y}^+_{\cal W}(\Z^t) ~{\Longleftrightarrow} ~ {\cal C}^\circ_l \subset {\cal Y}^{\cal O}~{ \Longleftrightarrow} ~
{\cal C}^\circ_l \cap {\cal Y}^{\cal O}\not = \emptyset.
\ee

\item A moduli space ${\rm Gr}_{{\cal Y}, {\cal W}}$, together with  a canonical map 
\be \la{kf2}
\kappa: {\cal Y}^{\cal O} \lra {\rm Gr}_{{\cal Y}, {\cal W}}.
\ee
\end{enumerate}


These ingredients are related as follows:

\begin{itemize}

\item Any  positive rational function $F$ on ${\cal Y}$ gives rise to a $\Z$-valued function $D_F$ on 
${\rm Gr}_{{\cal Y}, {\cal W}}$, such that for any $l\in  {\cal Y}^+_{\cal W}(\Z^t)$, the 
restriction of $D_F$ to $\kappa({\cal C}^\circ_l)$ equals  $F^t(l)$.  

\end{itemize}

So we arrive at a collection of irreducible cycles 
$$
{\cal M}^\circ _l := \kappa({\cal C}^\circ_l)\subset {\rm Gr}_{{\cal Y}, {\cal W}}, ~~~~
{\cal M}_l := \mbox{closure of ${\cal M}^\circ_l$},~~~l \in {\cal Y}^+_{\cal W}(\Z^t). 
$$
Thanks to the $\bullet$, the assignment $l \lms {\cal M}_l$ is injective.

\vskip 3mm
Consider the set $\{D_c\}$  of all irreducible divisors in ${\cal Y}$ such that the residue 
of the form $\Omega_{\cal Y}$ at $D_c$ is 
non-zero. We call them 
the {\it boundary divisors} of ${\cal Y}$.  
We define 
\be \la{Yspace}
{\cal Y}^\times:= {\cal Y} - \cup D_c. 
\ee
By definition, the form $\Omega_{\cal Y}$ is regular on ${\cal Y}^\times$.  
In all examples the potential ${\cal W}$ is regular on ${\cal Y}^\times$.

There is a split torus ${\Bbb H}$, and a positive regular surjective projection 
$$
\pi: {\cal Y}^\times\lra  {\Bbb H}.
$$ 
The  map $\pi$ is determined by  the form $\Omega_{\cal Y}$. 
For example,  assume that each boundary divisor $D_c$ is defined by a global equation $\Delta_c=0$. 
Then the regular functions $\{\Delta_c\}$ define the map $\pi$, i.e. $\pi(y) = \{\Delta_c(y)\}$. 

\vskip 2mm
Next, there is  a semigroup ${\Bbb H}^{\cal O}\subset {\Bbb H}({\cal K})$ containing ${\Bbb H}({\cal O})$, 
defining a cone
$$
{\Bbb P}:= {\Bbb H}^{\cal O}/{\Bbb H}({\cal O}) \subset {\Bbb H}(\Z^t):= {\Bbb H}({\cal K})/{\Bbb H}({\cal O}) = X_*({\Bbb H}), 
$$
such that the tropicalization of the map $\pi$  provides a map 
$\pi^t: {\cal Y}^+_{\cal W}(\Z^t) \to {\Bbb P}$, and there is a surjective map 
$\pi_{\rm Gr}: {\rm Gr}_{{\cal Y}, {\cal W}} \to {\Bbb P}$. 
Denote by $\pi^{\cal O}$ 
restricting of  $\pi\otimes {\cal K}$  
to ${\cal Y}^{\cal O}$. 
These maps fit into a commutative diagram
\be \la{mcdia}
\begin{array}{ccccc}
{\cal Y}^+_{\cal W}(\Z^t)&\stackrel{\rm val}{\longleftarrow} &{\cal Y}^{\cal O} 
& \stackrel{\kappa}{\lra}&{\rm Gr}_{{\cal Y}, {\cal W}}  \\
&&&&\\
\pi^t \downarrow &&\pi^{\cal O}\downarrow &&\downarrow  \pi_{\rm Gr}\\
&&&&\\
{\Bbb P}&\stackrel{\rm val}{\longleftarrow} &{\Bbb H}^{\cal O} &\stackrel{\rm val}{\longrightarrow} &{\Bbb P}
\end{array}
\ee
We define  ${\rm Gr}_{{\cal Y}, {\cal W}}^{(\lambda)}$ and ${\cal Y}^+_{\cal W}(\Z^t)_\lambda$ as the fibers of the maps 
$\pi_{\rm Gr}$ and $\pi^t$ over a 
$\lambda \in {\Bbb P}$.  
So we have
\be \la{stratilam}
{\rm Gr}_{{\cal Y}, {\cal W}} = \coprod_{\lambda \in {\Bbb P}}{\rm Gr}_{{\cal Y}, {\cal W}}^{(\lambda)}, ~~~~ 
{\cal Y}^+_{\cal W}(\Z^t) = \coprod_{\lambda \in {\Bbb P}}{\cal Y}^+_{\cal W}(\Z^t)_\lambda.
\ee

The following is a key property of our picture:
\begin{itemize}

\item $\bullet$
 The map $l \lra {\cal M}_l $ provides a bijection
$$
{\cal Y}^+_{\cal W}(\Z^t)_\lambda~~{ \longleftrightarrow} ~~
\{\mbox{\rm Closures of top dimensional components of  ${\rm Gr}_{{\cal Y}, {\cal W}}^{(\lambda)}$}\}.
$$
\end{itemize}
Although the space ${\rm Gr}_{{\cal Y}, {\cal W}}$ is usually infinite dimensional, it is 
nice. 
The map $\pi_{\rm Gr}: {\rm Gr}_{{\cal Y}, {\cal W}} \to {\Bbb P}$ slices it into 
highly singular and reducible pieces. However the slicing makes the perverse sheaves geometry clean and beautiful. 
It allows to relate the 
positive integral tropical points  to the top components of the slices.

\paragraph{Example.} In our main example, discussed in Section \ref{sec1.1} we have
$$
{\cal Y} = {\rm Conf}_n({\cal A}), ~~~~{\cal Y}^\times = {\rm Conf}^\times_n({\cal A}), ~~~~
{\cal Y}^{\cal O} = {\rm Conf}^{\cal O}_n({\cal A}), ~~~~
{\rm Gr}_{{\cal Y}, {\cal W}} = {\rm Conf}_n({\rm Gr}), ~~~~ {\Bbb H}= {\rm H}^n, 
~~~~ {\Bbb P}= ({\rm P}^+)^n.
$$ 
The potential ${\cal W}$ is defined in (\ref{theWp}), and 
decomposition
(\ref{stratilam}) is described by cyclic convolution varieties 
(\ref{cvcvcv}).

\subsubsection{Mixed configurations and a generalization of Mirkovi\'{c}-Vilonen cycles}
\la{sec1.2a}

Let us briefly discuss other examples relevant to representation theory. 
All of them follow the set-up of Section \ref{sec1.2}. The obtained cycles ${\cal M}_l$
 can be viewed as generalisations of Mirkovi\'{c}-Vilonen  cycles. Let us list first
the spaces ${\cal Y}$ and  ${\rm Gr}_{{\cal Y}, {\cal W}}$. 
The notation ${\rm Conf}_{w_0}$ indicates that the pair of the first and the last flags in configuration is in generic position. 

i)  {\it Generalized Mirkovi\'{c}-Vilonen  cycles}: 
$$
{\cal Y}:= {\rm Conf}_{w_0}({\cal A}, {\cal A}^n,{\cal B}), ~~~~ 
{\rm Gr}_{{\cal Y}, {\cal W}}:= {\rm Conf}_{w_0}({\cal A}, {\rm Gr}^n,{\cal B}) = {\rm Gr}^n.
$$
If $n=1$, we recover the   
Mirkovi\'{c}-Vilonen cycles in the affine Grassmannian \cite{MV}. 

ii) {\it Generalized stable Mirkovi\'{c}-Vilonen  cycles}: 
$$
{\cal Y}:= {\rm Conf}_{w_0}({\cal B}, {\cal A}^n,{\cal B}), ~~~~ 
{\rm Gr}_{{\cal Y}, {\cal W}}:= {\rm Conf}_{w_0}({\cal B}, {\rm Gr}^n,{\cal B}) = {\rm H}({\cal K}) \backslash {\rm Gr}^n.
$$
If $n=1$, we recover the 
stable Mirkovi\'{c}-Vilonen cycles in the affine Grassmannian. In our interpretation  they are 
top components of the stack 
$$
{\rm Conf}_{w_0}({\cal B}, {\rm Gr}, {\cal B}) = {\rm H}\backslash {\rm Gr}.
$$

iii) {\it The cycles providing canonical bases in tensor products}  
$$
{\cal Y}:= {\rm Conf}({\cal A}^{n+1}, {\cal B}), ~~~~ 
{\rm Gr}_{{\cal Y}, {\cal W}}:= {\rm Conf}({\rm Gr}^{n+1}, {\cal B}) = {\rm B}^-({\cal O}) \backslash {\rm Gr}^n.
$$

The spaces ${\cal Y}$ in examples i) and iii) are essentially the same. 
However the potentials are different: in the case iii) it is the sum of contributions of all 
decorated flags, while in the case i) we skip the first one. 
Passing from ${\cal Y}$ to ${\rm Gr}_{{\cal Y}, {\cal W}}$ we replace those 
${\cal A}$'s which contribute to the potential by ${\rm Gr}$'s, but keep the ${\cal B}$'s and the  
${\cal A}$'s  which do not contribute to the potential intact. 

We picture 
configurations at the vertices of a convex polygon, as on Fig \ref{polygon}. 
Some of the ${\cal A}$-vertices are shown boldface. 
The potential ${\cal W}$
is a sum of the characters assigned to the boldface ${\cal A}$-vertices, generalizing (\ref{theWp}). 
The decorated polygons in the cases ii) and iii) are depicted on the right of Fig \ref{tpm20.5} and  on Fig \ref{tpm21.5}.  
We discuss these examples in detail in Sections \ref{sec2.3} - \ref{tensor}. 

\subsection{Examples related to decorated surfaces}

\subsubsection{Laminations on decorated surfaces and canonical basis for $\G=SL_2$} \la{LAMCB}

\paragraph{1. Canonical basis in the tensor products invariants.} This example can be traced back to XIX century. 
We relate it to laminations on a polygon. 

\bd An integral lamination $l$ on {an} $n$-gon $P_n$ is a collection $\{\beta_j\}$ of simple nonselfintersecting 
intervals  
ending on the boundary of $P_n- \{\mbox{vertices}\}$, modulo isotopy.  
\ed

\begin{figure}[ht]
\centerline{\epsfbox{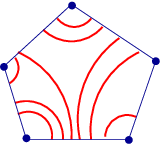}}
\caption{An integral lamination on a pentagon of type $(4,4,1,6,3)$.}
\label{polygon2}
\end{figure}

Pick a vertex of $P_n$, and number the sides clockwise. 
Given a collection of positive integers $a_1, ..., a_n$, consider the set ${\cal L}_n(a_1, ..., a_n)$ 
of all integral laminations $l$ on the polygon $P_n$ 
such that the number of endpoints of $l$ on the $k$-th side is $a_k$. 
Let $(V_2, \omega)$ be a two dimensional $\Q$-vector space with a symplectic form. 
Let us assign to an $l\in {\cal L}_n(a_1, ..., a_n)$ an $SL_2$-invariant map 
$$
{\Bbb I}_l: (\otimes^{a_1}V_2) \otimes \ldots \otimes (\otimes^{a_n}V_2) \lra \Q.
$$
We assign the factors in the tensor product to the endpoints of $l$, so that the order of the factors 
match the clockwise order of the endpoints. 
 Then  
for each interval $\beta$ in   $l$ we evaluate  
the form $\omega$ on the pair of vectors in the two factors of the tensor product 
labelled by the endpoints of $\beta$, and take the product over all intervals $\beta$ in $l$. 
Recall that the $SL_2$-modules $S^{a}V_2$, $a> 0$, provide all non-trivial irreducible 
finite dimensional $SL_2$-modules up to isomorphism.

\bt \la{9.19.13.19} Projections of the maps ${\Bbb I}_l$,  $l\in {\cal L}_n(a_1, ..., a_n)$, 
to $S^{a_1}V_2 \otimes \ldots \otimes S^{a_n}V_2$  
form a basis in ${\rm Hom}_{SL_2}(S^{a_1}V_2 \otimes \ldots \otimes S^{a_n}V_2, \Q)$.
\et

\paragraph{2. Canonical basis in the space of functions on the moduli space of $SL_2$-local systems.} 

\bd \la{9.19.13.1} Let $S$ be a surface with boundary. An integral lamination $l$ on $S$ 
is a collection of simple, mutually non intersecting,  
non isotopic 
loops $\alpha_i$ 
 with positive integral multiplicities
$$
l = \sum_i n_i[\alpha_i]  ~~~~ n_i \in \Z_{>0}, 
$$
considered modulo isotopy. The set of all integral laminations on $S$ is denoted by ${\cal L}_\Z({S})$.\footnote{Laminations on 
decorated surfaces were investigated in \cite{FG1}, Section 12, and \cite{FG3}. 
However the two types of laminations considered there, 
the ${\cal A}$- and ${\cal X}$-laminations, are different then the ones 
in Definition \ref{9.19.13.1}. Indeed, they parametrise canonical bases in ${\cal O}({\cal X}_{PGL_2, S})$ and, respectively,  
${\cal O}({\cal A}_{SL_2, S})$, 
while the latter parametrise a canonical basis in ${\cal O}({\rm Loc}_{SL_2, S})$. Notice that a lamination 
in Definition \ref{9.19.13.1} can not end on a boundary circle.} 
\ed
\begin{figure}[ht]
\epsfxsize130pt
\centerline{\epsfbox{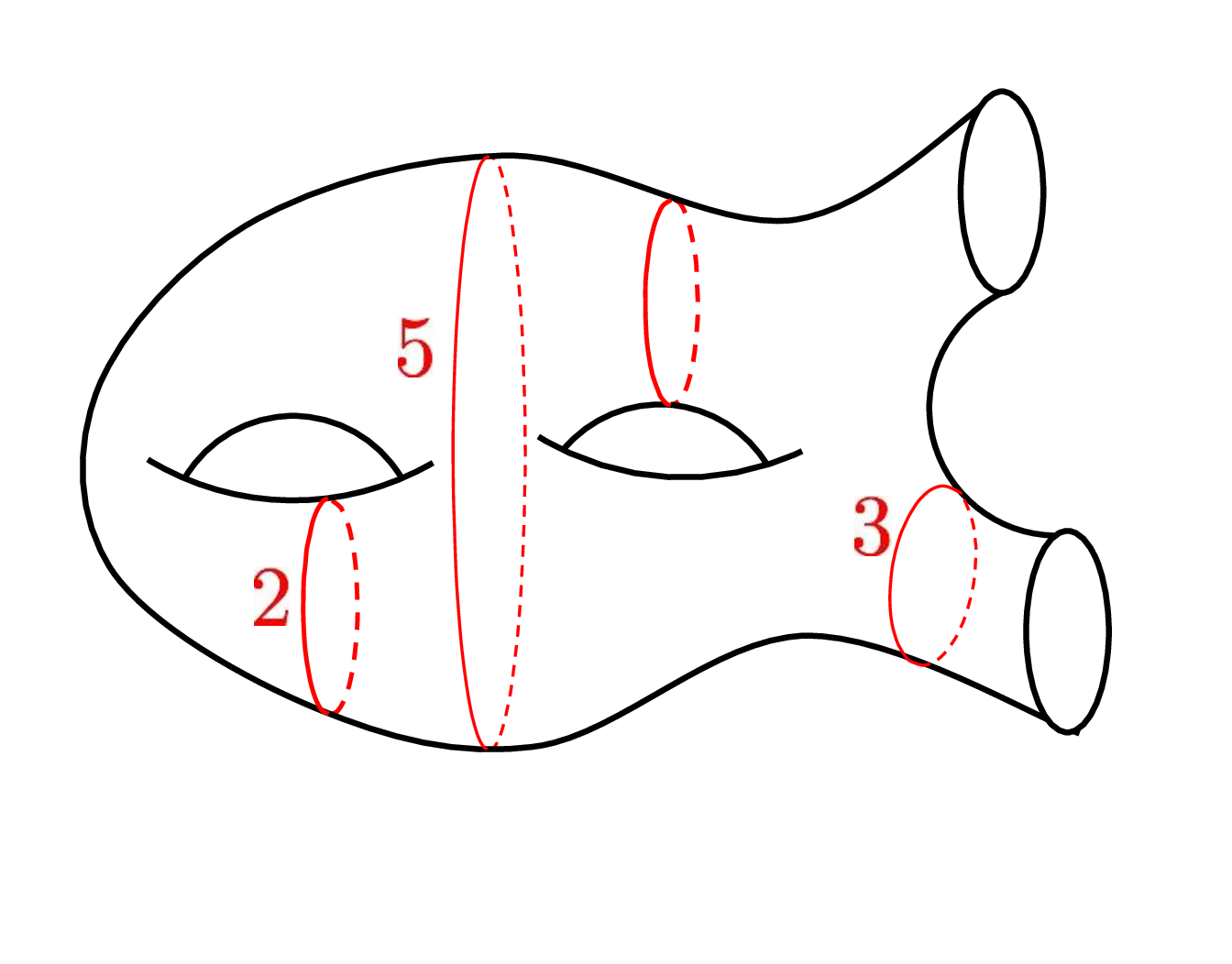}}
\caption{An integral lamination on a surface with two holes, and no special points.}
\label{fish}
\end{figure}

In the case when $S$ is a surface without boundary we get Thurston's integral 
laminations.

Given an integral lamination $l$ on $S$, let us define a regular function $M_l$ on the 
moduli space ${\rm Loc}_{SL_2, S}$ of $SL_2$-local systems on $S$. Denote by 
${\rm Mon}_{\alpha}({\cal L})$ the monodromy of an $SL_2$-local system ${\cal L}$ over a loop $\alpha$ on $S$. 
The value of the function $M_l$ on ${\cal L}$ is given by 
$$
M_l({\cal L}):= \prod_i {\rm Tr} ({\rm Mon}^{n_i}_{\alpha_i}({\cal L})).
$$

\bt \la{9.19.13.20} (\cite{FG1}, Proposition 12.2). The functions $M_l$, $l\in {\cal L}_\Z({S})$, form a linear basis 
in the space ${\cal O}({\rm Loc}_{SL_2, S})$. 
\et

Recall that a {\it decorated surface}  $S$ is an oriented surface  with boundary, 
and a finite, possibly empty, collection $\{s_1, ..., s_n\}$ of {\it special} 
points on the boundary,  considered modulo isotopy. 

 We define a 
moduli space ${\rm Loc}_{SL_2, S}$ for any decorated surface $S$, so that laminations on $S$ provide a canonical basis 
${\cal O}({\rm Loc}_{SL_2, S})$, generalising both Theorem \ref{9.19.13.19} (when $S$ is a polygon) and 
Theorem \ref{9.19.13.20}, see Section \ref{sec10.3n}. 

Let us discuss now how to generalize constructions 
of Section \ref{sec1.1.2} to the  decorated surfaces.

\subsubsection{Positive $\G$-laminations and top components of surface affine Grassmannians}
\la{sec1.3}

A pair $({\rm G}, S)$ gives rise to a moduli space ${\cal A}_{{\rm G}, S}$ \cite{FG1}. 
Here are two basic examples.
\begin{itemize}

\item 
When $S$ is a disc with $n$ special points on the boundary, we recover the space 
${\rm Conf}_n({\cal A})$. 
 
\item 
When $S$ is just a surface, without special points, the moduli 
space ${\cal A}_{{\rm G}, S}$ is a twisted version 
of the moduli space of {{\rm G}-local systems with unipotent monodromy around boundary components} on $S$ equipped with a covariantly constant
 decorated flag near every boundary component of $S$. 
\end{itemize}
The space ${\cal A}_{{\rm G},   S}$ has a positive structure \cite{FG1}. 
We define in Section \ref{sec11} a {\it potential} ${\cal W}$ on the space ${\cal A}_{{\rm G},   S}$. It is a  
rational positive function, with the  tropicalization 
$
{\cal W}^t: {\cal A}_{{\rm G}, S}(\Z^t) \lra \Z.
$ 
 
The condition ${\cal W}^t\geq 0$ determines a subset of {\it positive integral  {\rm G}-laminations on $  S$}:
\be \la{121212a}
{\cal A}^+_{{\rm G}, S}(\Z^t):= \{l \in {\cal A}_{{\rm G}, S}(\Z^t)~|~ {\cal W}^t(l)\geq 0\}.
\ee

For any decorated surface $S$, the set 
${\cal A}^+_{SL_2, S}(\Z^t)$ is canonically isomorphic to the set of integral laminations on $S$, see Section \ref{sec10.3n}. An interesting approach 
to a geometric definition of laminations for $\G=SL_m$, which 
employs the affine Grassmannian, was suggested by Ian Le \cite{Le}.

There is a canonical volume form $\Omega$ on the 
space ${\cal A}_{{\rm G}, S}$, which can be defined by using an ideal triangulation of $S$ 
and the volume forms on ${\rm Conf}_n({\cal A})$. When $G$ is simply-connected, it is also 
the cluster volume form $\Omega_{\cal A}$. 

We also assign to a  pair $({\rm G}, S)$ a stack ${\rm Gr}_{{\rm G}, S}$, which we call
 the {\it surface affine Grassmannian}. 
When $S$ is a disc with $n$ special points on the boundary, 
we recover the stack ${\rm Conf}_n({\rm Gr})$. 
In general it is an infinite dimensional stack.

The components of the punctured boundary  $\partial S - \{s_1, ..., s_n\}$ isomorphic to intervals are called 
boundary intervals. We define the torus ${\Bbb H}$  and the lattice ${\Bbb P}$ by
$$
{\Bbb H}:= {\rm H}^{\{\mbox{boundary intervals on $S$}\}}, ~~~~{\Bbb P}:= ({\rm P}^+)^{\{\mbox{boundary intervals on $S$}\}}. 
$$
The map $\pi$ is defined by assigning to a boundary interval ${\rm I}$ the element 
$i(\A_+, \A_-)\in {\rm H}$, see (\ref{agr}),  
where $(\A_-,\A_+)$ are the decorated flags at the 
ends of the interval ${\rm I}$, ordered by the orientation of $S$, 
provided by the very definition of the space ${\cal \A}_{\G, S}$. 

Given a point $l\in {\cal A}^+_{\G, S}(\Z^t)$, we define a cycle 
$
{\cal M}^o_l \subset {\rm Gr}_{G, S}.
$ 
Given an element $\lambda \in {\Bbb P}$, we prove that the map $l\lms {\cal M}^\circ_l$  gives rise to a bijection of sets 
\be \la{map}
{{\cal A}}^+_{\G, S}(\Z^t)_\lambda \stackrel{\sim}{\lra} \{\mbox{\rm closures of top dimensional components of 
${\rm Gr}^{(\lambda)}_{\G, S}$}\}. 
\ee
However in this case we can no longer bypass the question 
what are the ``top components'' of an infinite dimensional 
stack, as we did in Section \ref{sec1.1.2}. 
So we define in Section \ref{sec11.3.1} ``dimensions'' of certain relevant stacks with values in certain {\it dimension $\Z$-torsors}. 
As a result, although the ``dimension'' is no longer an integer, 
the difference of two ``dimensions'' from the same dimension $\Z$-torsor is an integer, and so the 
notion of ``top dimensional components'' does make sense.

\vskip 2mm
To define the analog of the space of tensor product invariants for a decorated surface $S$, we introduce 
in Section \ref{sec11} a moduli space ${\rm Loc}_{\G^L, S}$. 
If $S$ has no special points, it 
is the moduli space of  ${\rm G}^L$-local systems on $S$. 
If $S$ is a disc with $n$ points on the boundary, it is the space 
${\rm Conf}_n({\cal A}_{\G^L})$.  
We prove there
that the set ${\cal A}^+_{{\rm G}, S}(\Z^t)$ parametrizes a linear basis in ${\cal O}({\rm Loc}_{\G^L, S})$.

\subsection{Canonical bases, canonical pairings, and homological mirror symmetry} \la{sec1.4}

 Below we write 
 ${\cal A}$ for ${\cal A}_{{\rm G}}$ etc., and use notation ${\cal A}_L$ for ${\cal A}_{{\rm G^L}}$ etc. 

For any split reductive group ${\rm G}$, the space ${\cal O}({\cal A}_{L})$ of regular functions on the 
principal affine space ${\cal A}_{L}$ of $\G^L$ is a model of 
representations of ${\rm G}^L$: every irreducible ${\rm G}^L$-module appears there once.  
This allows us to organize the direct sum of all vector spaces of a given kind 
where the canonical bases live into a vector space of regular functions on a single space. 
For example:
\be \la{model}
\bigoplus_{ (\lambda_1, \ldots , \lambda_n) \in ({\rm P}^+)^n}
V_{\lambda_1} \otimes \ldots \otimes V_{\lambda_n}
= {\cal O}({\cal A}_{L}^n). 
\ee
\be \la{model1}
\bigoplus_{ (\lambda_1, \ldots , \lambda_n)\in ({\rm P}^+)^n}
\Bigl(V_{\lambda_1} \otimes \ldots \otimes V_{\lambda_n}\Bigr)^{{\rm G}^L} 
= {\cal O}({\cal A}_{L}^n)^{{\rm G}^L} = {\cal O}({\rm Conf}_n({\cal A}_{L})). 
\ee
Using this, let us interpret the statement that a canonical basis of a given kind is parametrized 
by positive integral tropical points of a certain space as existence of a {\it canonical pairing}. 

\subsubsection{Tensor product invariants and mirror symmetry} \la{s1.4.1}

For any split reductive group ${\rm G}$, 
the set ${\rm Conf}^+_n({\cal A})(\Z^t)$  
parametrizes a canonical basis in the space (\ref{model1}). 
So there are canonical pairings
\be \la{canpar}
{\bf I}_{\rm G}: {\rm Conf}^+_n({\cal A})(\Z^t) \times {\rm Conf}_n({\cal A}_{L})\lra 
{\Bbb A}^1.
\ee
\be \la{canparL}
{\bf I}_{\rm G^L}: {\rm Conf}_n({\cal A})\times {\rm Conf}^+_n({\cal A}_{L})(\Z^t) \lra 
{\Bbb A}^1.
\ee
So the story becomes completely symmetric. The idea that the set parametrizing canonical bases in tensor product invariants is a subset 
of ${\rm Conf}_n({\cal A})(\Z^t)$  goes back to 
Duality Conjectures from \cite{FG1}. 
It is quite surprising that taking into account the potential we get a canonical basis in the 
space of regular functions on 
the {\it same kind of space}, ${\rm Conf}_n({\cal A}_{L})$, 
for the Langlands dual group.


\begin{figure}[ht]
\epsfxsize100pt
\centerline{\epsfbox{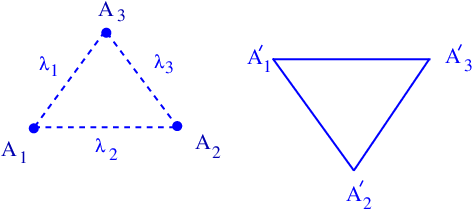}}
\caption{Duality between configurations spaces of decorated flags for ${\rm G}$ and ${\rm G}^L$. 
The potential is a sum of contributions at the boldface vertices. 
Pairs of decorated flags at the dashed sides are in generic position. 
No condition on the pairs of decorated flags at the solid sides.}
\label{tpm5.5}
\end{figure}

To picture this symmetry, consider  a convex $n$-gon $P_n$ on the left of Fig \ref{tpm5.5}, and 
assign a configuration $(\A_1, ..., \A_n)\in {\rm Conf}^\times_n({\cal A})$ to its vertices. 
The potential ${\cal W}$ is a sum of the vertex contributions; so the vertices are shown boldface. 
The pair of decorated flags at each side is generic; so all sides are dashed. 
Tropicalizing the data at the vertices, and using the isomorphism 
${\rm Conf}^+_2({\cal A})(\Z^t) ={\rm P}^+$, 
we assign a dominant weight $\lambda_k$  of $\G^L$ to each side of the left polygon. 
Consider now the dual $n$-gon $\ast P_n$ on the right, 
and a configuration of decorated flags $(\A'_1, ..., \A'_n)$ in ${\rm G}^L$ 
at its vertices. 
The dominant weight $\lambda_k$ on the left  corresponds to the irreducible representation $V_{\lambda_k}$,  
realised in the model ${\cal O}({\cal A}_{L})$ assigned to the dual vertex of $\ast P_n$. 



\vskip 3mm
Tropical points live naturally at the  
boundary of a positive space, compactifying the set of its real positive points \cite{FG4}. 
An example is given 
by  Thurston's  
boundary of Teichm\"uller spaces, realized as
 the space of projective measured laminations. 

It is tempting to think that canonical pairings (\ref{canpar}) and  (\ref{canparL}) are 
 manifestations of a symmetry involving both spaces simultaneously, 
rather then relating the tropical points of one space to 
the regular functions on the other space. 
We conjecture that this elusive symmetry is the mirror symmetry, and 
the function ${\cal W}$ 
is the potential for the Landau-Ginzburg 
model. 

To formulate precise conjectures, let us 
start with a  general set-up.

\paragraph{The A-model.} Let 
 ${\cal M}$ be a complex affine variety. So it has an affine embedding  $i: {\cal M}\hra \C^N$.
 The Kahler form $\sum_i dz_id\bar z_i$ on $\C^N$ induces a Kahler form  on 
${\cal M}(\C)$ with  an exact 
symplectic form $\omega$. The wrapped Fukaya category 
${\cal F}_{\rm wr}({\cal M}, \omega)$ \cite{AS} 
does not depend on the embedding $i$. We denote it by 
${\cal F}_{\rm wr}({\cal M})$. 
A potential ${\cal W}$ on $({\cal M}, \omega)$ allows to define the 
wrapped Fukaya-Seidel category 
$
{\cal F}{\cal S}_{\rm wr}({\cal M}) = {\cal F}{\cal S}_{\rm wr}({\cal M}, \omega, {\cal W}). 
$ The case of a potential with only Morse singularities is treated in \cite{S08}.
It also does not depend on the choice of  affine embedding. 
A volume form $\Omega$ provides a $\Z$-grading on ${\cal F}{\cal S}_{\rm wr}({\cal M})$ \cite{S}.

\paragraph{The positive A-brane.} In our examples 
${\cal M}$ is a {positive space} over $\Q$. So it  
has a submanifold ${\cal M}(\R_{>0})$ of real positive points. 
It is a Lagrangian submanifold for the symplectic form $\omega$ induced by any affine embedding. 
The form $\Omega$ is defined over $\Q$, 
and so 
${\cal M}(\R_{>0})$ is a special Lagrangian submanifold since it 
restricts to a real volume form on ${\cal M}(\R_{>0})$. The potential 
${\cal W}$ is a positive function on ${\cal M}$. 
So the special Lagrangian submanifold ${\cal M}(\R_{>0})$ 
should give rise to an 
 object of the wrapped Fukaya-Seidel category of ${\cal M}$, which we call the {\it 
positive $A$-brane}, denoted by ${\cal L}_+$.

\paragraph{The projection / action data.} 
In all our examples we have a mirror dual pair ${\cal M} \leftrightarrow {\cal M}_L$ equipped with the 
following data: a projection $\pi: {\cal M} \lra {\Bbb H}$ onto a split torus ${\Bbb H}$, an action of the 
split torus ${\Bbb T}$ on ${\cal M}$ preserving the volume form and the potential, 
and a similar pair of tori ${\Bbb H}_L, {\Bbb T}_L$ for ${\cal M}_L$. These tori are in duality: 
$$
X_*({\Bbb T}_L) = X^*({\Bbb H}), ~~~~X_*({\Bbb H}_L) = X^*({\Bbb T}). 
$$
This projection / action data gives rise to the following additional structures 
on the categories. 

\vskip 2mm

i) The group ${\rm Hom}(X_*({\Bbb H}), \C^*) = \widehat {\Bbb H}(\C)$ 
of $\C^*$-local systems on the complex torus ${\Bbb H}(\C)$ acts 
on the category ${\cal F}{\cal S}_{\rm wr}({\cal M})$. 
Namely, we assume that the objects of the category are given by Lagrangian 
submanifolds in ${\cal M}(\C)$ with ${\rm U}(1)$-local systems. 
Then a ${\rm U}(1)$-local system ${\cal L}$ on ${\Bbb H}(\C)$ acts by 
 the tensor product with $\pi^*({\cal L})$, providing an action 
of the subgroup ${\rm Hom}(X_*({\Bbb H}), {\rm U}(1))$ on the category. 
We assume that the action extends to an algebraic action of 
the complex torus 
$$
{\rm Hom}(X_*({\Bbb H}), \C^*) = X^*({\Bbb H})\otimes \C^* = X_*(\widehat 
{\Bbb H})\otimes \C^* = \widehat {\Bbb H}(\C). 
$$

\vskip 2mm

ii) Let  ${\Bbb T}_K$  be the maximal compact subgroup 
of the torus ${\Bbb T}(\C)$. We assume that the 
action of the group ${\Bbb T}_K$  on the symplectic manifold 
$({\cal M}, \omega)$ is Hamiltonian.\footnote{In our main examples the symplectic structure 
is exact, $\omega = d\alpha$. So avereging the form $\alpha$ by the action of the compact group ${\Bbb T}_K$ 
we can assume that it is ${\Bbb T}_K$-invariant. Therefore the action is Hamiltonian: 
the Hamiltonian at $x$ for a
 one parametric subgroup $g^t$  is given by the 
formula $\alpha(\frac{d}{dt}g^t(x))$.}
Then any subgroup $S^1\subset {\Bbb T}_K$ provides a family of symplectic maps 
$r_t$, $t \in \R/\Z = S^1$. The map 
$r_1$ provides an invertible 
functorial automorphism of Hom's of the  category 
${\cal F}{\cal S}_{\rm wr}({\cal M})$, and thus 
an invertible element of the center of the category. 
So the group algebra $\Z[X_*({\Bbb T})]= {\cal O}(\widehat {\Bbb T})$ is mapped into the center:
$$
{\cal O}(\widehat {\Bbb T})\lra {\rm Center}({\cal F}{\cal S}_{\rm wr}({\cal M})). 
$$
\vskip 2mm

iii) Clearly, there is a map 
${\cal O}({\Bbb H})\lra {\rm Center}(D^b{\rm Coh}({\cal M}))$, and the group ${\Bbb T}$ 
acts on $D^b{\rm Coh}({\cal M})$.

\paragraph{The potential / boundary divisors.} It was anticipated by Hori-Vafa \cite{HV} and 
Auroux \cite{Au1} 
that adding a potential on a space ${\cal M}$ amounts to a partial compactification 
of its mirror  ${\cal M}_L$ by a divisor. More precisely,
denote by ${\cal M}^\times$ and ${\cal M}_L^\times$ the regular loci  of the 
forms $\Omega$ and $\Omega_{L}$. 
The potential is  a sum 
${\cal W} = \sum_c{\cal W}_c$. 
Its components ${\cal W}_c$ are expected to match the irreducible divisors  
$D_c$ of  ${\cal M}_L - {\cal M}_L^\times$. 
The divisors $D_c$ are defined as the divisors 
on ${\cal M}_L$ where  ${\rm Res}_{D_c}(\Omega_{L})$ is non-zero. So we should have
\be \la{12.1.14.10}
{\cal W} = \sum_c{\cal W}_c, ~~~ {\cal M}_L - {\cal M}_L^\times= \cup_cD_c, ~~~{\cal W}_c
\stackrel{?}{\leftrightarrow} D_c.
\ee

There are several ways to explain how  this correspondence should work. 
\vskip 2mm
i) The potential ${\cal W}_c$ determines an element
$
[{\cal W}_c] \in {\rm H}{\rm H}^0({\cal M}), 
$ 
which defines a deformation 
of the category $D^b{\rm Coh}({\cal M})$ 
as a $\Z/2\Z$-category. On the dual side it corresponds to 
a deformation of the Fukaya category obtained by adding to the symplectic form on ${\cal M}_L$ 
a multiple of the 2-form $\omega_c$,  
whose cohomology class is the cycle class  $[D_c]\in H^2({\cal M}_L, \Z(1))$ of the divisor $D_c$.
\vskip 2mm

ii) The Landau-Ginzburg potential ${\cal W}_c$ should be obtained 
by counting the holomorphic discs touching the divisor $D_c$, 
as was demonstrated by Auroux \cite{Au1} in examples. 
\vskip 2mm

iii) In the cluster variety set up the correspondence is much more precise, see Section \ref{seccluster}. 

\vskip 3mm

{\bf Example}. To illustrate the set-up, let us specify the data on the moduli space 
${\rm Conf}_n({\cal A}_{})$. 

\begin{itemize}

\item A regular positive function, the potential ${\cal W}: {\rm Conf}^\times_n({\cal A}) \lra {\Bbb A}^1$. 

\item A regular 
volume form $\Omega$ on  ${\rm Conf}^\times_n({\cal A})$, with logarithmic singularities at infinity. 

\item A regular projection $\pi: {\rm Conf}^\times_n({\cal A}) \lra {\Bbb H}$ onto a torus $
{\Bbb H}:=  {\rm H}^{\{\mbox{sides of the $n$-gon $P_n$}\}}.
$

\item An action $r$ of the torus ${\Bbb T}:= {\rm H}^{\{\mbox{vertices of $P_n$}\}}$ on 
${\rm Conf}_n({\cal A})$ by rescaling decorated flags. 

\end{itemize}

Changing $\G$ to $\G^L$ we interchanges the action with the projection: 
\begin{itemize}
\item The torus ${\Bbb T}_L$ is dual to the torus ${\Bbb H}$, i.e.  there is a canonical isomorphism 
$X_*({\Bbb T}_L) = X^*({\Bbb H})$. 
\end{itemize}

By construction, 
the potential is  a sum 
\be \la{12.1.14.11}
{\cal W} = \sum_v\sum_{i\in {\rm I}}{\cal W}^v_{i} 
\ee
over the vertices $v$ of the polygon $P_n$, 
parametrising configurations $(\A_1, ..., \A_n)$, and 
 the set ${\rm I}$ of simple positive roots for $\G$.  
Indeed, a non-degenerate character $\chi$ of ${\rm U}$ is naturally a sum 
$\chi =\sum_i\chi_i$. 

 On the other hand, the set of irreducible components of the divisor 
${\rm Conf}_n({\cal A}_L) - {\rm Conf}^\times_n({\cal A}_L)$ 
is parametrised by the pairs $(E, i)$ 
where $E$ are the edges of the dual polygon $\ast P_n$, see Section \ref{sec1.2.2}: 
\be \la{12.1.14.12}
{\rm Conf}_n({\cal A}_L) - {\rm Conf}^\times_n({\cal A}_L) = 
\cup_{E}\cup_{i\in {\rm I}} D^E_{i}.
\ee 

Since vertices of the polygon $P_n$ match the sides of the dual polygon $\ast P_n$, 
the components of the potential (\ref{12.1.14.11}) match the 
irreducible components of the divisor at infinity (\ref{12.1.14.12}) on the  dual space. 

\vskip 3mm
We start with the most basic form of our mirror conjectures, which does 
not involve the potential.

\bcon \la{MIRRORDUAL} For any split semisimple group 
${\rm G}$ over $\Q$, there is  a mirror duality 
\be \la{8.6.13.2} 
({\rm Conf}^\times_n({\cal A}), \Omega) ~~\mbox{is mirror dual to}~~ 
({\rm Conf}^\times_n({\cal A}_L), \Omega_{L}). 
\ee
This means in particular that one has an equivalence of $\A_\infty$-categories
\be \la{8.6.13.3} 
{\cal F}_{\rm wr}({\rm Conf}^\times_n({\cal A})(\C)) 
\stackrel{\sim}{\lra} 
D^b{\rm Coh}({\rm Conf}^\times_n({\cal A}_L)). 
\ee

This equivalence maps the positive $A$-brane  ${\cal L}_+$ to the structure sheaf ${\cal O}$. 

It 
identifies the action of the group 
$\widehat {\Bbb H}(\C)$ on the category ${\cal F}_{\rm wr}({\rm Conf}^\times_n({\cal A})(\C))$ with the 
action of the group ${\Bbb T}_L(\C)$ on $D^b{\rm Coh}({\rm Conf}^\times_n({\cal A}_L))$, and identifies the subalgebras 
$$
{\cal O}(\widehat {\Bbb T})\subset 
{\rm Center}({\cal F}_{\rm wr}({\rm Conf}^\times_n({\cal A})(\C))) ~~\mbox{\rm and}~~
{\cal O}({\Bbb H}_L)\subset 
{\rm Center}(D^b{\rm Coh}({\rm Conf}^\times_n({\cal A}_L))).
$$ 
\econ
The projection / action data for the pair (\ref{8.6.13.2}) is given by 
$$
{\Bbb H} = {\rm H}^{n}, ~~~~{\Bbb H}_L = {\rm H}^{n}_L, ~~~~
{\Bbb T} = {\rm H}^{n}, ~~~~{\Bbb T}_L = {\rm H}_L^{n}. 
$$

The pair (\ref{8.6.13.2}) is symmetric:  
 interchanging the group $\G$ with the Langlands dual group $\G^L$ amounts 
to exchanging the  $\A$-model with the  $\B$-model. 

\vskip 2mm

Using the mirror pair (\ref{8.6.13.2}) as a starting point, 
we can now turn on the  potentails at all  vertices of the left polygon $P_n$. 
This amounts to a partial compactification of the dual space. Namely, 
we take the space ${\rm Conf}_n({\cal A}_L)$, and consider its affine closure 
${\rm Conf}_n({\cal A}_L)_{\bf a}:= {\rm Spec}\Bigl({\cal O}({\cal A}^n_L)^{\G^L}\Bigr)$. 

Since the action of the group ${\rm H}^{n}$ on ${\rm Conf}^\times_n({\cal A})$ alters the potential ${\cal W}$, 
and the projection $\pi_L$ onto ${\rm H}_L^n$ does not extend to ${\rm Conf}_n({\cal A}_L)_{\bf a}$, 
the projection / action data for the pair (\ref{8.6.13.2b}) is 
$$
{\Bbb H} = {\rm H}^{n}, ~~~~{\Bbb H}_L = \{e\}, ~~~~
{\Bbb T} = \{e\}, ~~~~{\Bbb T}_L = {\rm H}_L^{n}. 
$$
Therefore by turning on the potentials we arrive at the following Mirror Conjecture:

\bcon \la{MIRRORDUALa} For any split semisimple group ${\rm G}$ over $\Q$, 
there is  a mirror duality 
\be \la{8.6.13.2a} 
({\rm Conf}^\times_n({\cal A}), {\cal W}, \Omega) ~~\mbox{is mirror dual to}~~ 
{\rm Conf}_n({\cal A}_L)_{\bf a}. 
\ee
This means in particular that there is an equivalence of $\A_\infty$-categories
\be \la{8.6.13.1b} 
{\cal F}{\cal S}_{\rm wr}({\rm Conf}^\times_n({\cal A})(\C), {\cal W}, \Omega) 
~~\stackrel{\sim}{\lra}~~ 
D^b{\rm Coh}({\rm Conf}_n({\cal A}_L)_{\bf a}). 
\ee
It maps the positive $A$-brane  ${\cal L}_+$ to the structure sheaf ${\cal O}$, and 
identifies the action of the group 
$\widehat {\Bbb H}(\C)$ on the category ${\cal F}{\cal S}_{\rm wr}({\rm Conf}^\times_n({\cal A})(\C))$ with the 
action of ${\Bbb T}_L(\C)$ on $D^b{\rm Coh}({\rm Conf}_n({\cal A}_L)_{\bf a})$.  
\econ

\vskip 2mm
The geometry  of mirror dual objects in Conjectures \ref{MIRRORDUAL} and \ref{MIRRORDUALa} 
is {\it essentially} dictated by representation theory. 
Indeed, the tropical points are determined by birational types of the spaces, and  
canonical bases tell the algebras of functions on the dual affine varieties:
\be \la{dcI}
\mbox{ The set ${\rm Conf}^+_n({\cal A})(\Z^t)$ parametrises a canonical basis in ${\cal O}({\rm Conf}_n({\cal A}_{L}))$}.
\ee
\be \la{dcII}
\mbox{ The set ${\rm Conf}_n({\cal A}_L)(\Z^t)$ should parametrise a 
canonical basis in ${\cal O}({\rm Conf}^\times_n({\cal A}))$}.\footnote{Although the claim (\ref{dcII}) is 
not addressed in the paper, it can be deduced from (\ref{dcI}).} 
\ee

The potential ${\cal W}$ and the projection $\pi$ define a regular map 
$
(\pi, {\cal W}): {\rm Conf}^\times_n({\cal A}) \lra {\Bbb H}\times {\Bbb A}^1.
$ 
The form $\Omega$ on ${\rm Conf}^\times_n({\cal A})$ and the canonical volume forms 
on ${\Bbb H}$ and ${\Bbb A}^1$ provide a volume form $\Omega^{(a,c)}$ at the 
fiber $F_{a, c}$ of this map over a generic point $(a,c) \in {\Bbb H}\times {\Bbb A}^1$.  

\vskip 3mm
More generally, we can turn on only partial potentials at the vertices 
of the polygon $P_n$, which amounts on the dual side 
to taking partial compactifications, and then considering their affine closures. 
This way we get an array of conjecturally dual pairs, described as follows. 

For each vertex $v$ of the polygon $P_n$ 
parametrising configurations $(\A_1, ..., \A_n)$ choose an arbitrary subset 
${\rm I}_v \subset {\rm I}$ of the set parametrising the simple positive roots of ${\rm G}$. 
It determines a partial potential
\be \la{WPART}
{\cal W}_{\{{\rm I}_v\}} = \sum_v{\cal W}_{{\rm I}_v}, ~~~~{\cal W}_{{\rm I}_v}:= \sum_{i\in {\rm I}_v}{\cal W}^v_{i}. 
\ee
On the dual side, subsets $\{{\rm I}_v\}$ determine a partial compactification 
of the space ${\rm Conf}^\times_n({\cal A}_L)$, obtained by adding the divisors 
$D^{E_v}_{i}$ where $i \in {\rm I}_v$. 
 Here $E_v$ is the side of the polygon 
$\ast P_n$ dual to the vertex $v$ of $P_n$:
\be \la{SCOMPACTI}
{\rm Conf}_n({\cal A}_L)_{\{{\rm I}_v\}}:= {\rm Conf}^\times_n({\cal A}_L) 
\bigcup \cup_v\cup_{i\in {\rm I}_v} D^{E_v}_{i}.
\ee

For each vertex $v$ of $P_n$ there is a subgroup 
${\rm H}_{{\rm I}_v} \subset {\rm H}$ preserving the partial potential 
${\cal W}_{{\rm I}_v}$ at $v$. On the dual side, let ${\rm H}^{{\rm I}_v}_L$ be the dual 
quotient of the Cartan group ${\rm H}_L$. 
So we arrive at the projection / action data  
\be \la{actproj}
{\Bbb H} = {\rm H}^{n}, ~~~~{\Bbb H}_L = \prod_v{\rm H}^{{\rm I}_v}_L, ~~~~
{\Bbb T} = \prod_v{\rm H}_{{\rm I}_v}, ~~~~{\Bbb T}_L = {\rm H}_L^{n}. 
\ee
So turning on partial potentials we arrive at  Conjecture \ref{MIRRORDUALab}, interpolating  
Conjectures \ref{MIRRORDUAL} and \ref{MIRRORDUALa}:
\bcon \la{MIRRORDUALab}
For any split semisimple group ${\rm G}$ over $\Q$, 
there is  a mirror duality 
\be \la{8.6.13.2ab} 
({\rm Conf}^\times_n({\cal A}), {\cal W}_{\{{\rm I}_v\}}, \Omega) ~~\mbox{is mirror dual to the affine closure of}~~ 
{\rm Conf}_n({\cal A}_L)_{\{{\rm I}_v\}}. 
\ee
Its action / projection data is given by (\ref{actproj}).
\econ
Needless to say, the positive integral tropical points of the left space 
parametrise a basis in the space of functions on the right space. 


\vskip 3mm 
Here is another general principle to generate new mirror dual pairs.  
We start with a mirror dual pair $({\cal M}, \Omega, {\cal W}) \leftrightarrow 
 {\cal M}_L$, equipped with 
the projection / action data which involves a dual pair  
$({\Bbb T}, {\Bbb H}_L)$. So  ${\Bbb T}$ acts by automorphisms of the triple 
$({\cal M}, \Omega, {\cal W})$, and there is a dual projection $\pi_L: {\cal M}_L \to {\Bbb H}_L$. 

Choose any subgroup ${\Bbb T}' \subset {\Bbb T}$, and consider the corresponding 
${\Bbb T}'$-equivariant category. 
If the group ${\Bbb T}'$ acts freely, 
this amounts to 
taking the quotient 
of the space with potential $({\cal M}, {\cal W})$ by the action of ${\Bbb T}'$. 
A volume form on ${\Bbb T}'$ gives rise to a volume form on the quotient, obtained by contructing 
the volume form $\Omega$ with the dual polyvector field on ${\Bbb T}'$. 
The subgroup ${\Bbb T}'\subset {\Bbb T}$ determines by the duality a quotient group 
${\Bbb H}_L\lra {\Bbb H}_L'$, and therefore a projection 
$\pi_L': {\cal M}_L \to {\Bbb H}_L'$. 

\begin{itemize}

\item {\it The quotient  stack $({\cal M}/{\Bbb T}', {\cal W})$ is mirror dual to the family  
$\pi_L': {\cal M}_L \to {\Bbb H}_L'$}. 

In the examples below $({\cal M}/{\Bbb T}', {\cal W})$ is just dual to a fiber 
${\pi_L'}^{-1}(a) \subset {\cal M}_L$, $a\in {\Bbb H}_L'$.

\end{itemize}

In particular, starting from a mirror dual pair 
(\ref{8.6.13.2ab}), we can choose any subgroup ${\Bbb T}' \subset 
{\Bbb T} = \prod_v{\rm H}_{{\rm I}_v}$ acting on the space with potential on the left. 
All examples below are obtained this way. 

\vskip 3mm
{\bf Example.} We start with the space ${\rm Conf}^\times({\cal A}^{n+1})$ 
with the potential ${\cal W}_{1, ..., n}$ given by the sum of 
the full potentials at all vertices but one, the vertex ${\rm A}_{n+1}$. 
The action of the group ${\rm H}$ on the decorated flag ${\rm A}_{n+1}$ 
preserves the potential  ${\cal W}_{1, ..., n}$. 
Applying the above principle, we get  a dual pair illustrated on 
Fig \ref{tpm21.5}.  
The fiber over $a$, illustrated by the middle picture on Fig \ref{tpm21.5}, is canonically 
isomorphic to the less symmetrically defined space illustrated on the right.

\begin{figure}[ht]
\epsfxsize300pt
\centerline{\epsfbox{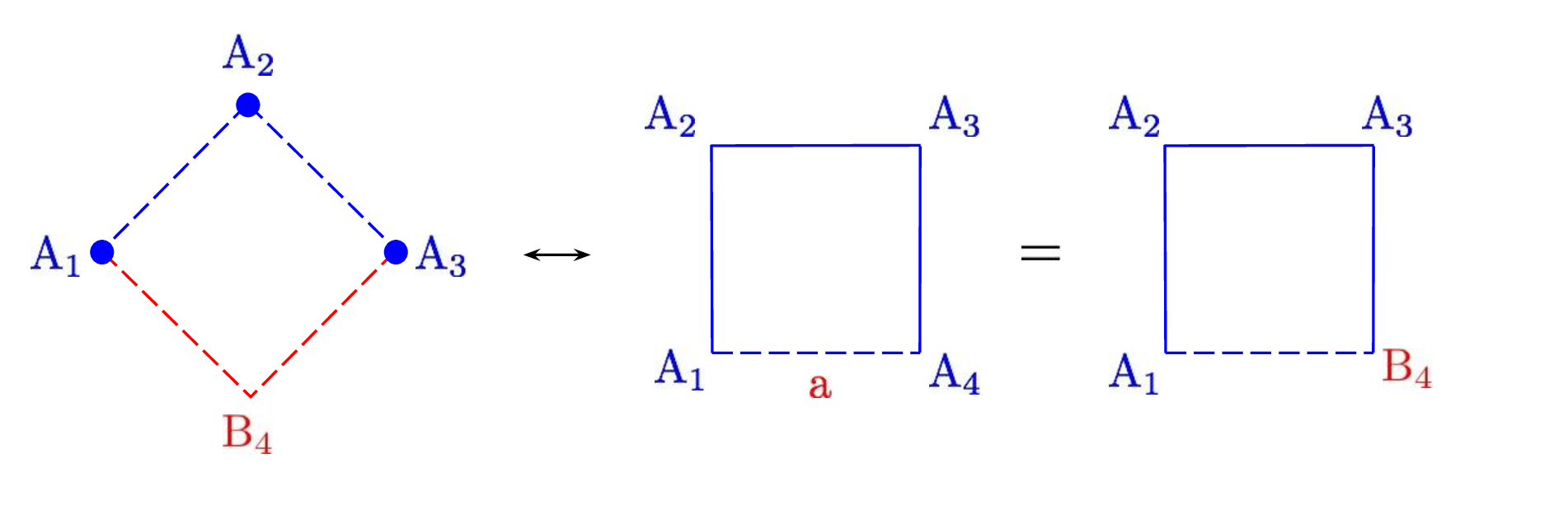}}
\caption{Dual pairs $({\rm Conf}^\times( {\cal A}^3, {\cal B}), {\cal W}_{1,2,3})$  and 
${\rm Conf}_{w_0}( {\cal A}_L^3, {\cal B}_L) = {\cal A}_L^2$. 
The ${\rm H}$-components of the projection $\lambda$ sit at 
the $\A$-decorated blue dashed edges on the left. The projection $\mu$ to ${\rm H}$ is assigned to 
the red $\A_3\B_4\A_1$. }
\label{tpm21.5}
\end{figure}

In the next Section we consider this example from a different point of view, starting 
from representation-theortic picture, just as we did with our basic example, and
 arrive to the same dual pairs. 

\subsubsection{Tensor products of representations and mirror symmetry} \la{s1.4.1a}

The set ${\rm Conf}^+({\cal A}^{n+1}, {\cal B})(\Z^t)$  defined using the potential 
${\cal W} $ from (\ref{ppoott}) 
parametrises canonical bases in $n$-fold tensor products of simple ${\rm G}^L$-modules.  
So using (\ref{model})  
we arrive  at a canonical pairing 
\be \la{canpar1}
{\bf I}: {\rm Conf}^+({\cal A}^{n+1}, {\cal B})(\Z^t) \times {\cal A}_{L}^n\lra 
{\Bbb A}^1.
\ee

Let us present  ${\cal A}_{L}^n$ as a configuration space. Recall that 
${\rm Conf}_{w_0}({\cal A}^{n+1}_{L}, {\cal B}_{L})$ parametrises configurations 
$(\A_{1}, \ldots , \A_{n+1},  \B_{n+2})$ such  that the pair
 $(\A_{n+1}, \B_{n+2})$  is generic. 
Generic pairs $\{\A, \B\}$ form a ${\rm G}^L$-torsor. Let  
 $\{\A^+,  \B^-\}$  be a standard generic pair. Then there is an isomorphism 
\be \la{mduala1a2}
{\cal A}_{L}^n \stackrel{=}{\lra} {\rm Conf}_{w_0}({\cal A}^{n+1}_{L}, {\cal B}_{L}), ~~~~\{\A_{1}, \ldots , \A_{n}\} 
\lms (\A_{1}, \ldots , \A_{n}, \A^+,  \B^-).
\ee

The subspace ${\rm Conf}^\times({\cal A}^{n+1}_{L}, {\cal B}_{L})$ parametrises configurations 
$(\A_{1}, \ldots , \A_{n+1},  \B_{n+2})$ such  that the consecutive pairs of flags are generic. 
It is the quotient of 
${\rm Conf}_{n+2}^\times({\cal A})$ by the action of the group ${\rm H}$ on the last decorated flag. 
The projection ${\rm Conf}_{n+2}^\times({\cal A}) \to {\rm H}^{n+2}$ induces a map, see (\ref{agr}),
\be \la{project}
\pi= (\lambda, \mu): {\rm Conf}^\times({\cal A}^{n+1}, {\cal B})  \lra {\rm H}^{n}\times {\rm H}.
\ee
$$
(\A_{1}, \ldots , \A_{n+1},  \B_{n+2})\lms \Bigl(\alpha(\A_1, \A_2), ..., \alpha(\A_n, \A_{n+1})\Bigr) \times 
\alpha(\A_{n+1}, \B_{n+2}) \alpha(\A_{1}, \B_{n+2})^{-1}.
$$

Then the symmetry is restored, and we can view (\ref{canpar1}) as a manifestation of a mirror duality: 
\be \la{mduala2}
({\rm Conf}^\times({\cal A}_{}^{n+1}, {\cal B}_{}), {\cal W}_{}, \Omega_{}, \pi) ~~\mbox{is mirror dual to}~~ 
({\rm Conf}_{w_0}({\cal A}^{n+1}_{L}, {\cal B}_{L}), \Omega_{L}, r_L). 
\ee
Here  $r_L$ is the action of 
 ${\rm H}^{n+1}_{L}$ by rescaling  of the decorated flags. The projection/action data is 
$$
{\Bbb H} = {\rm H}^{n+1}, ~~~~{\Bbb H}_L = \{e\}, ~~~~
{\Bbb T} = \{e\}, ~~~~{\Bbb T}_L = {\rm H}_L^{n+1}, ~~~~
$$

The analog of mirror dual pair (\ref{8.6.13.2}) and its projection/action data are given by, see Fig \ref{tpm21.75}, 
\be \la{8.6.13.2b} 
({\rm Conf}^\times({\cal A}^{n+1}, {\cal B}), \Omega) ~~\mbox{is mirror dual to}~~ 
({\rm Conf}^\times({\cal A}^{n+1}_L, {\cal B}_L), \Omega_{L}). 
\ee
$$
{\Bbb H} = {\rm H}^{n+1}, ~~~~{\Bbb H}_L = {\rm H}_L^{n+1}, ~~~~
{\Bbb T} = {\rm H}^{n+1}, ~~~~{\Bbb T}_L = {\rm H}_L^{n+1}, ~~~~
$$ 

So we arrived at the two dual pairs and (\ref{mduala2}) and (\ref{8.6.13.2b}) 
using canonical pairings as a guideline. 

As discussed in the Example in Section \ref{sec1.4}, 
we can get them 
from the basic dual pairs (\ref{8.6.13.2a}) and (\ref{8.6.13.2}) using the action / projection 
duality $\bullet$, which in this case tells that 
the quotient by the action of the group ${\rm H}$ on one side is dual to a fiber of the family of spaces 
over 
the dual group ${\rm H}_L$ over a point $a \in {\rm H}_L$.

 In particular, the dual pair (\ref{8.6.13.2}) leads to 
the  dual pair illustrated  on Fig \ref{tpm21.75}. 
Notice that configurations $(\A_1, ..., \A_{n+2})$ with $\alpha(\A_{n+1}, \A_{n+2}) =a \in {\rm H}$ are in bijection with 
configurations $(\A_1, ..., \A_{n+1}, \B_{n+2})$ where the pair $(\A_{n+1}, \B_{n+2})$ is generic. So the 
 two diagrams on the right of  Fig \ref{tpm21.75} represent isomorphic configuration spaces,  
and we get the dual pair (\ref{8.6.13.2b}) from  (\ref{8.6.13.2}). 
The dual pair (\ref{mduala2}) is obtained from (\ref{8.6.13.2b}) by adding potentials at the $\A$-vertices, 
thus
allowing arbitrary pairs of flags on the dual sides. 

We conjecture that the analogs of Conjectures \ref{MIRRORDUAL} and \ref{MIRRORDUALa} hold for the pairs (\ref{8.6.13.2b}) 
and (\ref{mduala2}).

\begin{figure}[ht]
\epsfxsize300pt
\centerline{\epsfbox{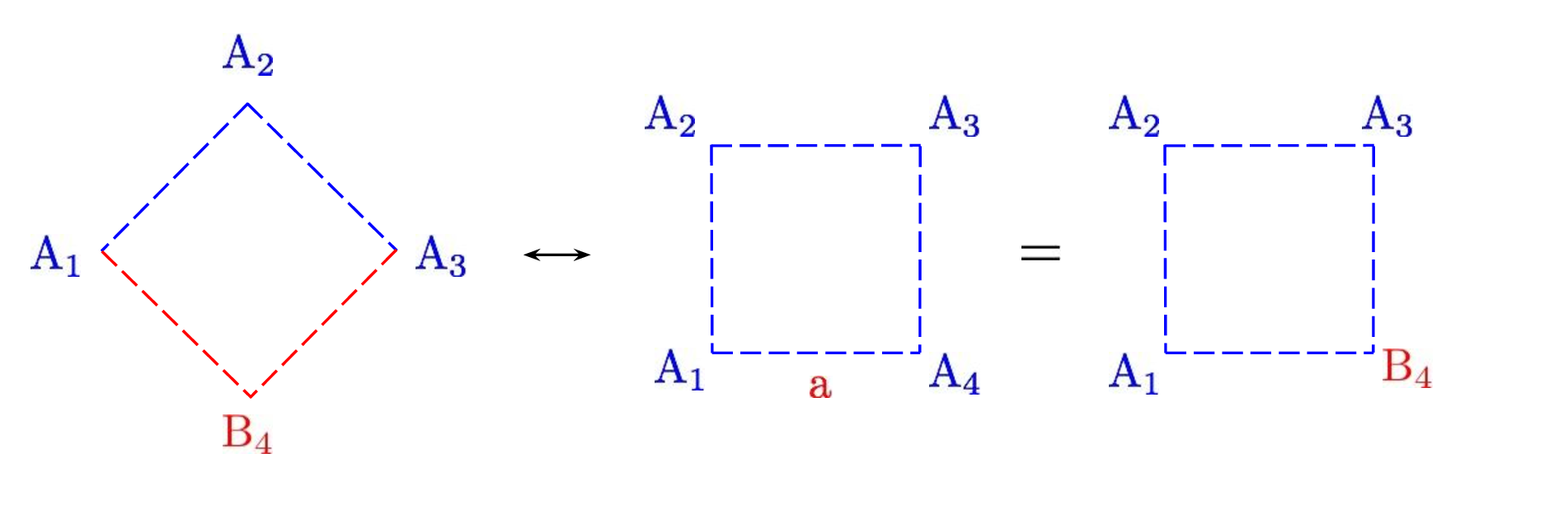}}
\caption{Dual spaces ${\rm Conf}^\times( {\cal A}^3, {\cal B})$ (left) and 
${\rm Conf}^\times( {\cal A}_L^3, {\cal B}_L) = {\cal A}_L^2$ (right). }
\label{tpm21.75}
\end{figure}

\subsubsection{Landau-Ginzburg mirror of a maximal unipotent group  ${\rm U}$ and it generalisations} \la{s1.4.1ab} 
We view  Lusztig's dual canonical basis in ${\cal O}(\U^L)$  as a 
 canonical pairing, and hence as a mirror duality:
\be \la{mduala3a}
{\bf I}: {\rm U}^+_\chi(\Z^t) \times {\rm U}^L\lra 
{\Bbb A}^1, ~~~~({\rm U}^*, \chi) ~~\mbox{is mirror dual to}~~ {\rm U}^L. 
\ee
To define ${\rm U}^*$, we realise a maximal unipotent subgroup ${\rm U}$ as a 
big Bruhat cell in the flag variety, and intersect it with the opposite big Bruhat cell. 
The $\chi$ is a non-degenerate additive character of  $\U$, restricted to ${\rm U}^*$. 
This example is explained and generalised  using configurations as follows. 

\vskip 3mm
Let ${\rm Conf}_{w_0}({\cal B}, {\cal A}^n,  {\cal B})$ be the space parametrising 
configurations $(\B_{1}, \A_{2}, \ldots , \A_{n+1}, \B_{n+2})$ such that the pairs 
{$(\B_{1}, \B_{n+2})$ and $(\A_{n+1}, \B_{n+2})$}  are generic, see the right picture on 
Fig \ref{tpm20.5}. There is an isomorphism 
\be \la{mduala1a23}
 \U_L \times{\cal A}_{L}^{n-1}= {\rm Conf}_{w_0}({\cal B}_{L}, {\cal A}^n_{L},  {\cal B}_{L}), ~~~~
\{\B_1, \A_2, ..., \A_n\} \lms (\B_1, \A_2, ..., \A_n, \A^+, \B^-).
\ee
The group 
 ${\rm H}^{n}_{L}$ acts on 
${\rm Conf}_{w_0}( {\cal B}_{L}, {\cal A}^n_{L}, {\cal B}_{L})$ by rescaling  decorated flags. 

The subspace ${\rm Conf}^\times({\cal B}, {\cal A}^n,  {\cal B})$ parametrises configurations 
where each consecutive pair of flags is generic. It is depicted on the left of Fig \ref{tpm20.5}. 
It is the quotient of ${\rm Conf}_{n+2}^\times({\cal A})$ by the action of ${\rm H}\times {\rm H}$ 
on the first and last decorated flags. 
Thus there is a map $\pi$, defined 
similarly to (\ref{project}):
\be \la{projectabb}
\pi = (\lambda, \mu): {\rm Conf}^\times( {\cal B}, {\cal A}^{n}, {\cal B}) \to {\rm H}^{n-1}\times {\rm H}.
\ee
So the projection / action data in this case is 
$$
{\Bbb H} = {\rm H}^{n}, ~~~~{\Bbb H}_L = \{e\}, ~~~~
{\Bbb T} = \{e\}, ~~~~{\Bbb T}_L = {\rm H}_L^{n}, ~~~~
$$
For example, ${\rm Conf}^\times({\cal B}, {\cal A}, {\cal B}) = {\rm U}^*$,  in agreement with 
${\rm U}^*$  in (\ref{mduala3a}). 

\bcon \la{9.22.13.1}
The set ${\rm Conf}^+({\cal B}, {\cal A}^{n},  {\cal B})(\Z^t)$ parametrises a canonical basis in 
${\cal O}(\U_L \times {\cal A}_{L}^{n-1})$. The subset  
$(\lambda^t, \mu^t)^{-1}(\lambda_1, ..., \lambda_{n-1}; \nu)$ 
parametrises a canonical basis in the 
weight $\nu$ subspace of 
$$
 {\rm U}({\cal N}^L) \otimes V_{\lambda_1}\otimes ... \otimes V_{\lambda_{n-1}}.
$$ 
  
The  analogs of 
Conjectures \ref{MIRRORDUAL} and \ref{MIRRORDUALa} hold for the following mirror dual pairs:
$$
({\rm Conf}^\times({\cal B}, {\cal A}^{n}, {\cal B}), \Omega)~~ \mbox{is mirror dual to}~~ 
({\rm Conf}^\times({\cal B}_{L}, {\cal A}^n_{L}, {\cal B}_{L}), \Omega_L),
$$ 
$$
({\rm Conf}^\times({\cal B}, {\cal A}^{n}, {\cal B}), {\cal W}, \Omega, \pi)~~ \mbox{is mirror dual to}~~ 
({\rm Conf}_{w_0}({\cal B}_{L}, {\cal A}^n_{L}, {\cal B}_{L}), r_L)
$$ 

\econ

These mirror pairs can be obtained from the basic mirror  pairs  
(\ref{8.6.13.2}) and (\ref{8.6.13.2a}) by trading, using the action / projection 
principle $\bullet$,  the quotient by  ${\rm H}^2_L$ 
to the fiber over $(a,b) \in {\rm H}^2$ on the dual side, see Fig \ref{tpm20.5}. 

\begin{figure}
\epsfxsize300pt
\centerline{\epsfbox{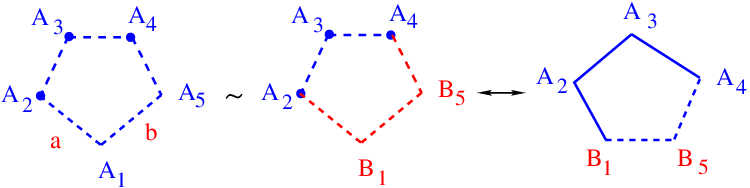}}
\caption{Duality ${\rm Conf}^\times({\cal B}^2, {\cal A}^3) \leftrightarrow {\rm Conf}_{w_0}({\cal B}_{L}^2, {\cal A}_{L}^3) = {\rm U}_{L} \times {\cal A}_{L}^2$. In the middle: 
the ${\rm H}$-components of the map $\lambda$ sit at 
the dashed blue sides. 
The map $\mu$ is assigned to $\A_2\B_1\B_5\A_4$. }
\label{tpm20.5}
\end{figure}

\subsubsection{Landau-Ginzburg mirror of a  simple split group ${\rm G}$}

In this Section we interprete a split simple group $\G$ as a configuration space, 
and using this deduce its Landau-Ginzburg mirror from Conjecture \ref{MIRRORDUALa} 
by using our standard toolbox. The  companion conjecture tells that 
the mirror of the maximal double Bruhat cell for $\G$ is the maximal double Bruhat cell for  $\G^L$. 

\vskip 3mm

Denote by ${\rm Conf}^\times({\cal B}, {\cal A}, {\cal B}, {\cal A})$ 
the space parametrising  configurations $(\B_1, \A_2, \B_3, \A_4)$ where 
all four consecutive pairs are generic. 
There is a potential  
given by the sum of the potentials at the 
$A$-vertices:
$$
{\cal W}_{2,4}(\B_1, \A_2, \B_3, \A_4):= \chi_{\A_2}(\B_1, \A_2, \B_3) + \chi_{\A_4}(\B_3, \A_4, \B_1). 
$$
The space with potential is illustrated on the left of Fig \ref{untitle}. 
Let us describe its mirror. 
\epsfxsize250pt
\begin{figure}[ht]
\centerline{\epsfbox{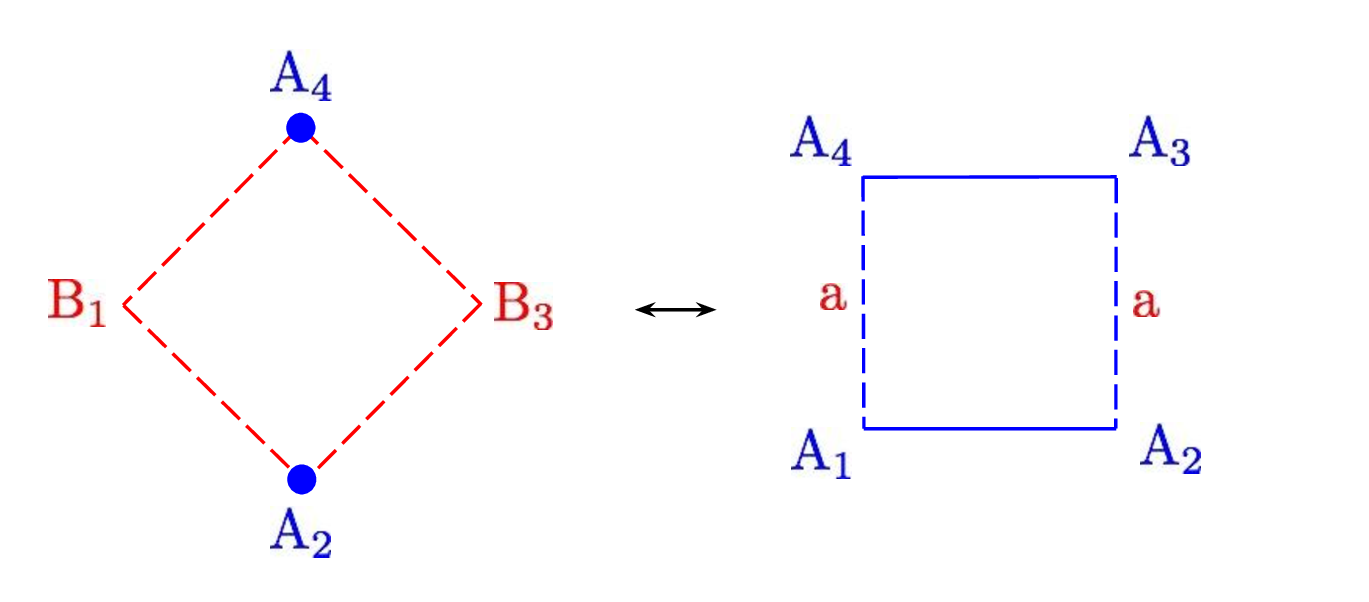}}
\caption{The Landau-Ginzburg model (left) dual to ${\G^L}$ (right).}
\label{untitle}
\end{figure}

Recall the isomorphism
$
\alpha: {\rm Conf}^\times({\cal A}, {\cal A}) \lra {\rm H}. 
$ 
Consider the moduli space of configurations
\be \la{G}
(\A_1, \A_2, \A_3, \A_4) \in {\rm Conf}_4({\cal A}_L)~|~
\mbox{$(\A_1, \A_2)$, $(\A_3, \A_4)$ are generic}; ~\alpha(\A_1, \A_2) = \alpha(\A_3, \A_4) =e.
\ee 
The picture on the right of Fig \ref{untitle} illusatrates this moduli space. 
\bl
The moduli space (\ref{G}) is isomorphic to the group $\G^L$. 
\el

\begin{proof} Pick a generic pair $\{A_1, A_2\}$ 
with $\alpha(A_1, A_2)=e$. Then 
for each $\G^L$-orbit in (\ref{G}) there is a unique representative $\{\A_1, \A_2, \A_3, \A_4\}$ 
where $\{\A_1, \A_2\}$ is the chosen pair.  
There is a unique $g \in \G^L$ such that 
$
g\{\A_1, \A_2\} = \{\A_3, \A_4\}.
$ 
The map $(\A_1, \A_2, \A_3, \A_4) \to g$ provides the isomorphism.  
\end{proof}

\bcon \la{ConG}
The mirror  to a split semisimple algebraic group $\G^L$ over $\Q$ is the pair
\be \la{G4s}
({\rm Conf}^\times({\cal B}, {\cal A}, {\cal B}, {\cal A}), {\cal W}_{2,4}).
\ee
\econ

\paragraph{Example.} Let $\G^L=PGL_2$, so $\G=SL_2$. Then ${\cal A}= {\Bbb A}^2 -\{0\}$, ${\cal B}= 
{\Bbb P}^1$, and 
\be \la{10.16.14.1}
{\rm Conf}^\times({\cal B}, {\cal A}, {\cal B}, {\cal A}) = \{(L_1, v_2, L_3, v_4)\}/SL_2. 
\ee
Here $L_1, L_3$ are one dimensional subspaces 
in a two dimensional vector space $V_2$, and $v_2, v_4$ are non-zero vectors 
in $V_2$. 
The pairs $(L_1, v_2)$, $(v_2, L_3)$, $(L_3, v_4)$, $(v_4, L_1)$ are generic, i.e. the corresponding pairs of
 lines are distinct. Pick non-zero vectors $l_1\in L_1$ and $l_3\in L_3$. 
Then 
$$
{\cal W}_{2,4} = \frac{\Delta(l_1, l_3)}{\Delta(l_1, v_2)\Delta(v_2, l_3)} + 
\frac{\Delta(l_1, l_3)}{\Delta(l_3, v_4)\Delta(l_1, v_4)}. 
$$ 
It is a regular function on (\ref{10.16.14.1}), independent of the choice of vectors $l_1, l_3$. 
To calculate it,  set
\be
l_1 = (1,0), ~~~~ v_2 = (x,1/p), ~~~~ l_3 = (1,y/p), ~~~~ v_4 = (0,1).  
\la{10.19.14h}
\ee
Then 
$$
{\rm Conf}^\times({\cal B}_L, {\cal A}_L, {\cal B}_L, {\cal A}_L) = 
\{(x, y, p) \in {\Bbb A}^1 \times {\Bbb A}^1 \times {\Bbb G}_m - (xy-1= 0)\}. 
$$
\be
\la{10.19.14.1h}
{\cal W}_{2,4} = 
\frac{y/p}{1/p\cdot (xy/p-1/p)} + 
\frac{y/p}{1\cdot 1} = \frac{yp}{xy-1} +\frac{y}{p}. 
\ee

The case $\G=PGL_2$,  $\G^L=SL_2$ is similar, except that now  
${\cal A}_{PGL_2} = {\rm A}^2-\{0\}/\pm 1$.

\vskip 3mm
Let us explain how  this conjecture can be deduced from our general conjecture. 

\epsfxsize250pt
\begin{figure}[ht]
\centerline{\epsfbox{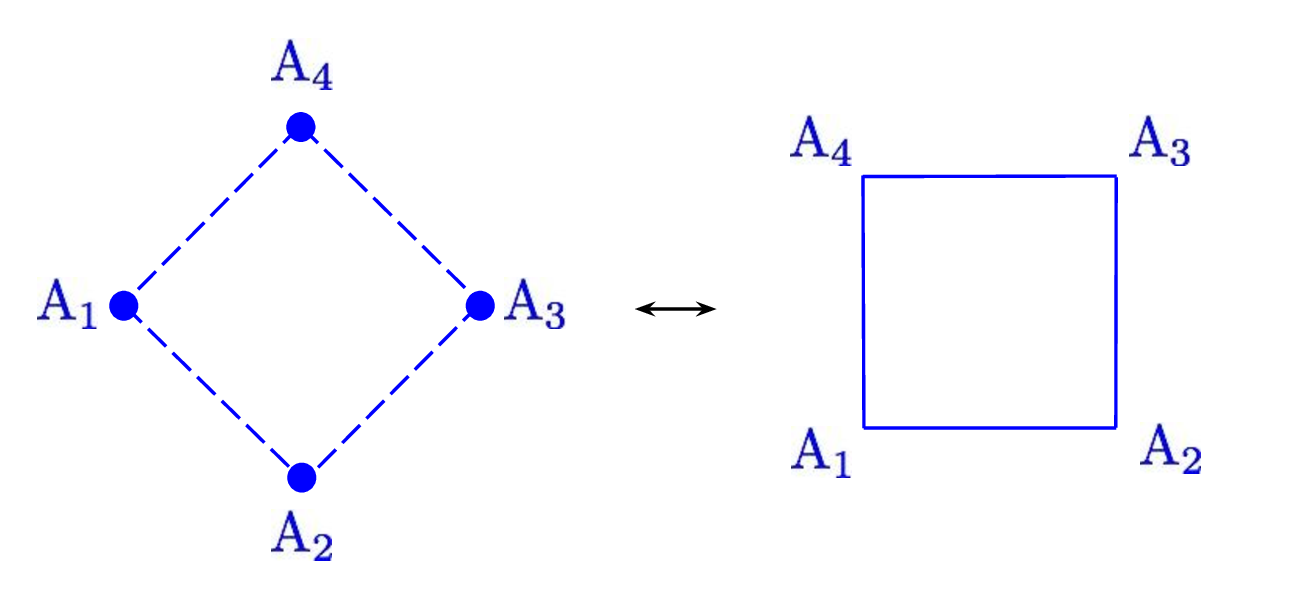}}
\caption{Duality between configurations of decorated flags for $\G$ and $\G^L$.}
\label{dual2}
\end{figure}

\paragraph{Step 1.} Conjecture \ref{MIRRORDUALa} 
tells us mirror duality, illustrated on Fig \ref{dual2}:
$$
({\rm Conf}^\times_4({\cal A}), {\cal W}_{1,2,3,4})\leftrightarrow {\rm Conf}_4({\cal A}_L).
$$

\paragraph{Step 2.} We alter the pair $({\rm Conf}^\times_4({\cal A}), {\cal W}_{1,2,3,4})$ by 
removing the potentials at the vertices $\A_1$ and $\A_3$. This reduces the potential 
${\cal W}_{1,2,3,4}$ to a new potential: 
$$
{\cal W}_{2,4}(\A_1, \A_2, \A_3, \A_4): = \chi_{\A_2}(\B_1, \A_2, \B_3) + 
\chi_{\A_4}(\B_3, \A_4, \B_1). 
$$

In the  dual picture this amounts to removing
 two divisors from  ${\rm Conf}_4({\cal A}_L)$, illustrated by two
 punctured edges on the right of Fig 
\ref{dual3}, dual to the vertices $\A_1$ and $\A_3$ on the left.  Precisely, we introduce a subspace 
$\widetilde {\rm Conf}_4({\cal A}_L)$ such that  
the pairs of decorated flags at punctured sides are generic. 
The obtained dual pair is illustrated on Fig \ref{dual3}. 
\epsfxsize250pt
\begin{figure}[ht]
\centerline{\epsfbox{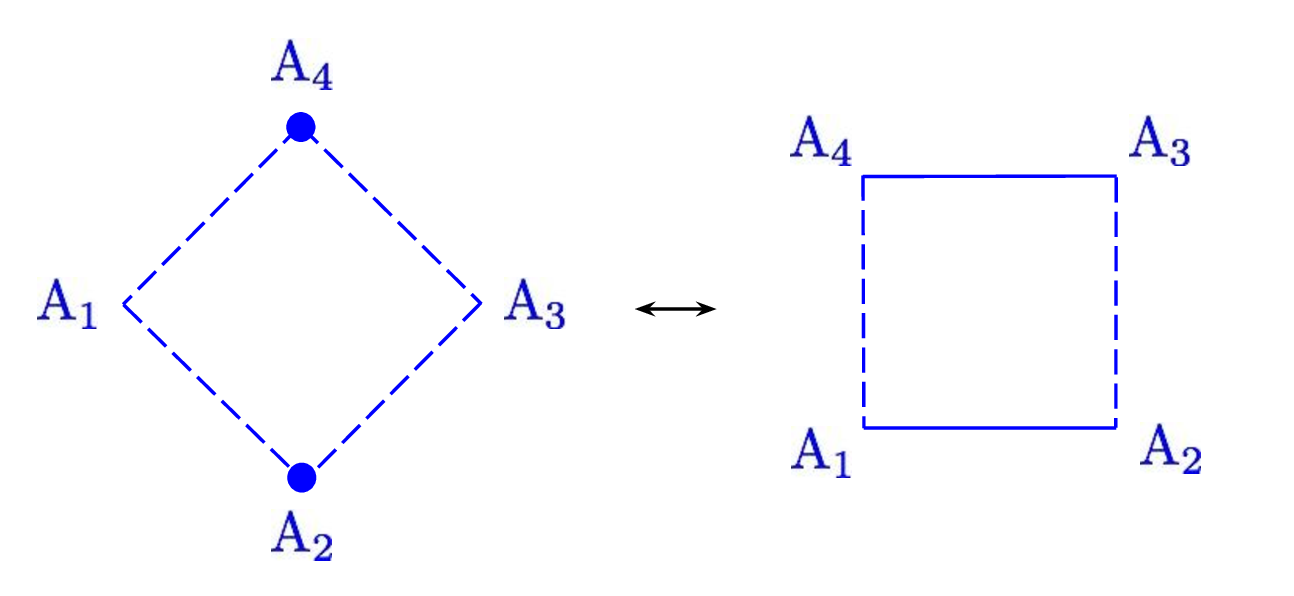}}
\caption{Dual pair of spaces obtained on Step 2.}
\label{dual3}
\end{figure}
In particular there is a projection provided by the two punctured sides:
\be \la{G2}
\widetilde {\rm Conf}_4({\cal A}_L) \lra {\rm H}_L^2. 
\ee

\paragraph{Step 3.} 
There is an action of the  group ${\rm H}\times {\rm H}$ on 
${\rm Conf}^\times_4({\cal A})$ preserves the potential ${\cal W}_{2,4}$, given by  
$(\A_1, \A_2, \A_3, \A_4) \lra (h_1 \cdot \A_1, \A_2, h_1 \cdot \A_3, \A_4).$ 
The quotient is the space (\ref{G4s}): 
$$
({\rm Conf}^\times_4({\cal A}), {\cal W}_{2,4})/ ({\rm H}\times {\rm H}) = 
({\rm Conf}^\times({\cal B}, {\cal A}, {\cal B}, {\cal A}), {\cal W}_{2,4}).
$$

\paragraph{Step 4.} The action of the group ${\rm H}\times {\rm H}$ is dual to the projection 
 (\ref{G2}). 
 The quotient by the ${\rm H}\times {\rm H}$-action   is dual to the fiber over 
$e \in {\rm H}_L\times {\rm H}_L$.  The fiber is just the space (\ref{G}). 
On the level of pictures, this is how we go from Fig \ref{dual3} to Fig \ref{untitle}. 
This way we arrived at Conjecture 
\ref{ConG}.

\paragraph{Canonical basis motivation.} Let us explain how the positive integral tropical points of the 
space from Conjecture \ref{ConG} parametrise a canonical basis in ${\cal O}(\G^L)$. 
One has 
$
{\cal O}(\G^L) = \bigoplus_{\lambda\in {\rm P}^+}V_\lambda \otimes V_{\lambda}^*. 
$ 
\epsfxsize200pt
\begin{figure}[ht]
\centerline{\epsfbox{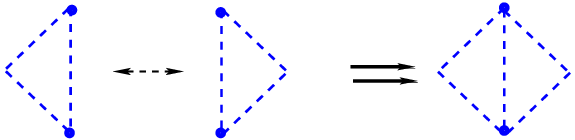}}
\caption{The (tropicalised) 
 Landau-Ginzburg model dual to $\G^L$ is obtained by gluing the two LG models dual to ${\cal A}_L$ 
along their ``vertical
sides'', as shown on the left.}
\label{G2}
\end{figure}

Recall that $
{\cal O}({\cal A}_L) = \bigoplus_{\lambda\in {\rm P}^+}V_\lambda$. 
The decomposition of ${\cal O}({\cal A}_L)$ into irreducible $\G^L$-modules 
is provided by the ${\rm H}_L$-action   on ${\cal A}_L$. According to our general picture,  
$$
{\cal A}_L = {\rm Conf}_{w_0}({\cal B}_L, {\cal A}_L, {\cal A}_L) ~~
\mbox{is mirror dual to} ~~({\rm Conf}^\times({\cal B}, {\cal A}, {\cal A}), {\cal W}_{2,3}).
$$
The canonical basis in $V_\lambda$ is parametrised 
by the fiber of the projection 
$
{\rm Conf}^+({\cal B}, {\cal A}, {\cal A})(\Z^t) \lra P^+
$ over the $\lambda \in P^+$. 
This projection is the tropicalisation of the positive rational map 
$
{\rm Conf}({\cal B}, {\cal A}, {\cal A}) \lra {\rm Conf}({\cal A}, {\cal A}).
$ 
Therefore the tensor product of canonical basis in $V_\lambda \otimes V_\lambda^*$ 
is parametrised by the fiber over $\lambda$ of the tropicalisation of the positive rational map 
$
{\rm Conf}({\cal B}, {\cal A}, {\cal B}, {\cal A}) \lra {\rm Conf}({\cal A}, {\cal A}). 
$

\bl
The space ${\rm Conf}^\times({\cal B},{\cal A}, {\cal B},{\cal A})$ is isomorphic to the open double Bruhat cell of $\G$.
\el

\begin{proof}
Note that  ${\rm Conf}^\times({\cal B},{\cal A}, {\cal B},{\cal A})$ is isomorphic to the moduli space parametrizing the configurations
$(\A_1,\A_2, \A_3, \A_4)\in {\rm Conf}_4({\cal A})$ such that 
$\alpha(\A_1, \A_2)=\alpha(\A_4, \A_3)=e$ and each consecutive pair $(\A_i, \A_{i+1})$ is generic. 
There is a unique element $g\in \G$ such that
$
\{g\cdot\A_1, g\cdot\A_2\}= \{\A_4, \A_3\}. 
$ 
Let  $\pi(\A_1)= \B$ and $\pi(\A_2)=\B^-$. Then we have
$$
\{\A_1, \A_4\}= \{\A_1, g\cdot \A_1\} \mbox{ is generic} ~~~ \Longleftrightarrow ~~~g\in \B w_0\B, 
$$
$$
\{\A_2, \A_3\}= \{\A_2, g\cdot \A_2\} \mbox{ is generic} ~~~\Longleftrightarrow ~~~g\in \B^- w_0\B^-. 
$$
So the space is isomorphic to the open double Bruhat cell $\B w_0\B\bigcap \B^- w_0\B^-$.
\end{proof}

\bcon The open double Bruhat cell of $\G$ is mirror to the open double Bruhat cell of $\G^L$. 
\begin{figure}[ht]
\epsfxsize300pt
\centerline{\epsfbox{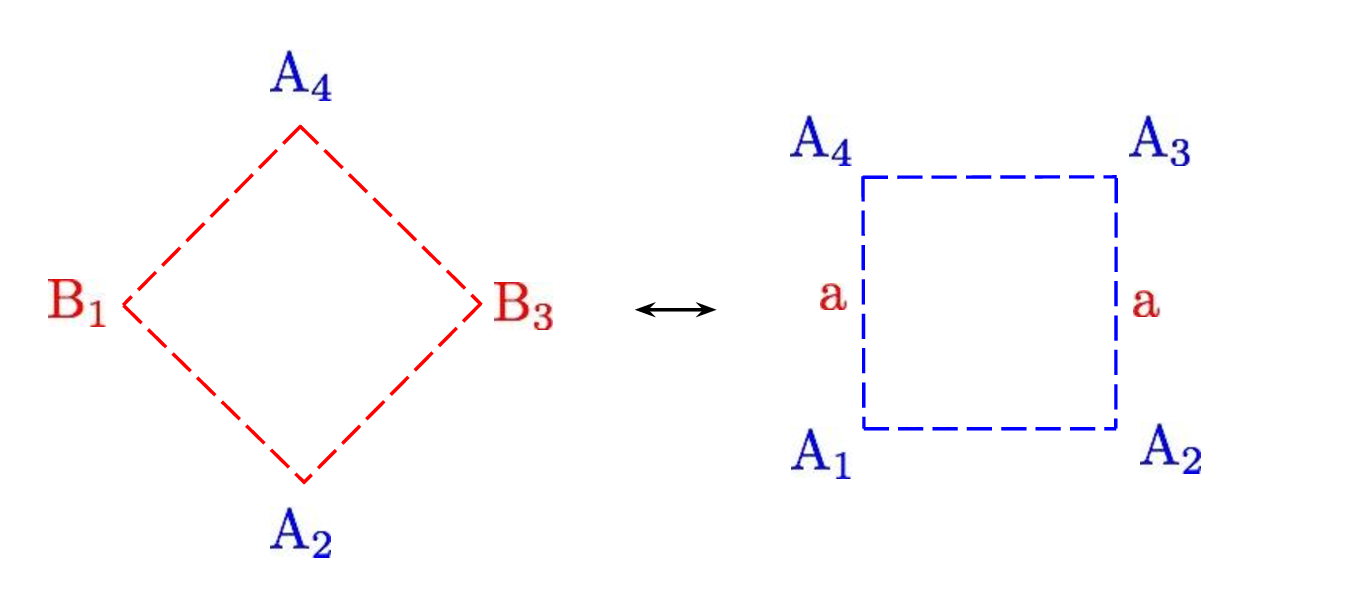}}
\label{bruhat}
\end{figure}
\econ
Below we investigate the case when $\G={\rm SL}_2$.

\paragraph{An example: the open double Bruhat cell of ${\rm SL}_2$.} It consists of elements
$$
   \begin{bmatrix} 
      x & p \\
      q & y \\
   \end{bmatrix} \in {\rm SL}_2,  \quad \quad xy-1=pq, ~~p.q \in {\Bbb G}_m.
$$
Let 
$$
D=\{xy-1=0\}\subset {\Bbb A}^2, ~~~~X={\Bbb A}^2\backslash D.
$$
The open Bruhat cell of ${\rm SL}_2$ is isomorphic to $X\times {\Bbb G}_m$. 
It is a cluster ${\cal A}$-variety with 
two seeds $p \longleftarrow x \longrightarrow q$ and $p \longleftarrow y \longrightarrow  q$. 
The  $\{p,q\}$ are the frozen variables.  The seeds are related by 
the cluster transformation
$$
x=\frac{1+pq}{y}.
$$
Note that the coordinates here are compatible with \eqref{10.19.14h}. In particular, the potential function \eqref{10.19.14.1h} becomes
$$
{\cal W}=\frac{y}{q}+\frac{y}{p}
$$

\paragraph{An example: the open double Bruhat cell of ${\rm PGL}_2$.} It consists of elements
$$
  \begin{bmatrix} 
      px & 1 \\
      p & y \\
   \end{bmatrix} \in {\rm GL}_2,  \quad \quad xy-1\neq 0,~~ p\in {\Bbb G}_m.
$$
It is again isomorphic to $X\times {\Bbb G}_m$. 
So we expect that 
$ 
{X\times {\Bbb G}_m} \mbox{ is mirror to itself.}
$ 

This open double Bruhat cell admits a cluster ${\cal X}$-structure.
We set
$$
x_1=px, \quad\quad x_2=xy-1, \quad \quad x_3=x.
$$
It gives rise to a cluster ${\cal X}$-structure
$
x_1\longleftarrow x_2 \longrightarrow x_3.
$
Mutation at $x_2$ delivers 
$
x_1'\longleftarrow x_2' \longrightarrow x_3'.
$ 
Here
$$
x_1'=x_1(1+x_2)^{-1}=p/y,
$$
$$
x_2'=x_2^{-1}=\frac{1}{xy-1},
$$
$$
x_3'=x_3(1+x_2)^{-1}=1/y.
$$

The example $X={\Bbb A}^2\backslash D$ is considered by Auroux in Section 5 of \cite{Au1}. 
See also Section 2 of \cite{P}

Finally, let us remind that in general, 
when $\G$ is simply connected, then the open double Bruhat cell is a cluster ${\cal A}$-variety.
The open double Bruhat cell of $\G^L$ is a cluster ${\cal X}$-variety.

\subsubsection{Landau-Ginzburg mirror of $G^n$} 
Since $G^n$ is a semi-simple group, the previous discussion applies. 
However our general approach leads to a slightly different mirror dual, 
which has an additional symmetry:  the group $\Z/(n+1)\Z$ acts naturally 
on each of the spaces. 
It starts  from the 
dual pair
$$
{\rm Conf}^\times_{2n+2}({\cal A}) ~~\mbox{mirror dual to}~~ {\rm Conf}^\times_{2n+2}({\cal A}_L). 
$$ 

Let
${\rm Conf}^\times_{2n+2}({\cal A}, {\cal B}, \ldots , {\cal A}, {\cal B})$  be the space parametrising  
configurations 
$(\A_1, \B_2, \A_3, \B_4, \ldots , \A_{2n+1}, \B_{2n+2})$  
such that any consequtive pair is generic. There is a potential 
$$
{\cal W}_{1,3, ..., 2n+1}(\A_1, \B_2, \ldots , \A_{2n+1}, \B_{2n+2}):= 
\sum_{i=1, 3, ..., 2n+1}\chi_{\A_i}(\B_{i-1}, \A_i, \B_{i+1}).
$$

The dual space parametrises configurations 
\be \la{Gcon2}
(\A_1, \ldots , \A_{2n+2}) \in {\rm Conf}_{2n+2}({\cal A}_L)~~\mbox{such that 
$(\A_{2k+1}, \A_{2k+2})$ are generic, and  $\alpha(\A_{2k+1}, \A_{2k+2})=1$}. 
\ee
The group $\Z/(n+1)\Z$ acts by automorphisms of this pair of spaces. 
The dual pair is illustrated on Fig \ref{ababab}.

\bl
The space (\ref{Gcon2}) is isomorphic to $(G^L)^n$.
\el

\begin{proof} For any given collection $\{\A_1, \ldots , \A_{2n+2}\}$ representing a point in the moduli space 
(\ref{Gcon2}) there is unique $g_k\in G^L$ such that $\{\A_{2k+1}, \A_{2k+2}\} = g_k \{\A_{1}, \A_{2}\}$. 
So picking a representative with the first pair $\{\A_{1}, \A_{2}\}$ provided by 
a pinning in $\G^L$, we get an isomorphism with $(G^L)^n$. 

\end{proof}

 \bcon \la{ConG}
The mirror  to $(\G^L)^n$  is the pair
\be \la{G4s}
({\rm Conf}^\times_{2n+2}({\cal A}, {\cal B}, \ldots , {\cal A}, {\cal B}), {\cal W}_{1,3, ..., 2n+1}).
\ee
\econ

As in the $n=1$ case, Conjecture \ref{ConG} can be deduced from Conjecture \ref{MIRRORDUALa} telling that 
\be \la{G4s1}
({\rm Conf}^\times_{2n+2}({\cal A}), {\cal W}_{1,2,..., 2n+2}) ~~\mbox{mirror dual to}~~ 
{\rm Conf}_{2n+2}({\cal A}_L). 
\ee
Indeed, starting with duality (\ref{G4s1}), we turn off the potentials at the even vertices. 
Then the group $H^{n+1}$ acts by automorphisms of the space with the new potential. 
The quotient is the pair (\ref{G4s}). 

On the other hand, turning off the potentials at the even vertices 
amounts on the dual side to 
removing the divisors from ${\rm Conf}_{2n+2}({\cal A}_L)$ assigned to the sides of the 
$(2n+2)$-gon dual to those vertices. The obtained space 
is fibered over ${H_L}^{n+1}$. The fiber over $e$ is the space (\ref{Gcon2}). 
\epsfxsize300pt
\begin{figure}[ht]
\centerline{\epsfbox{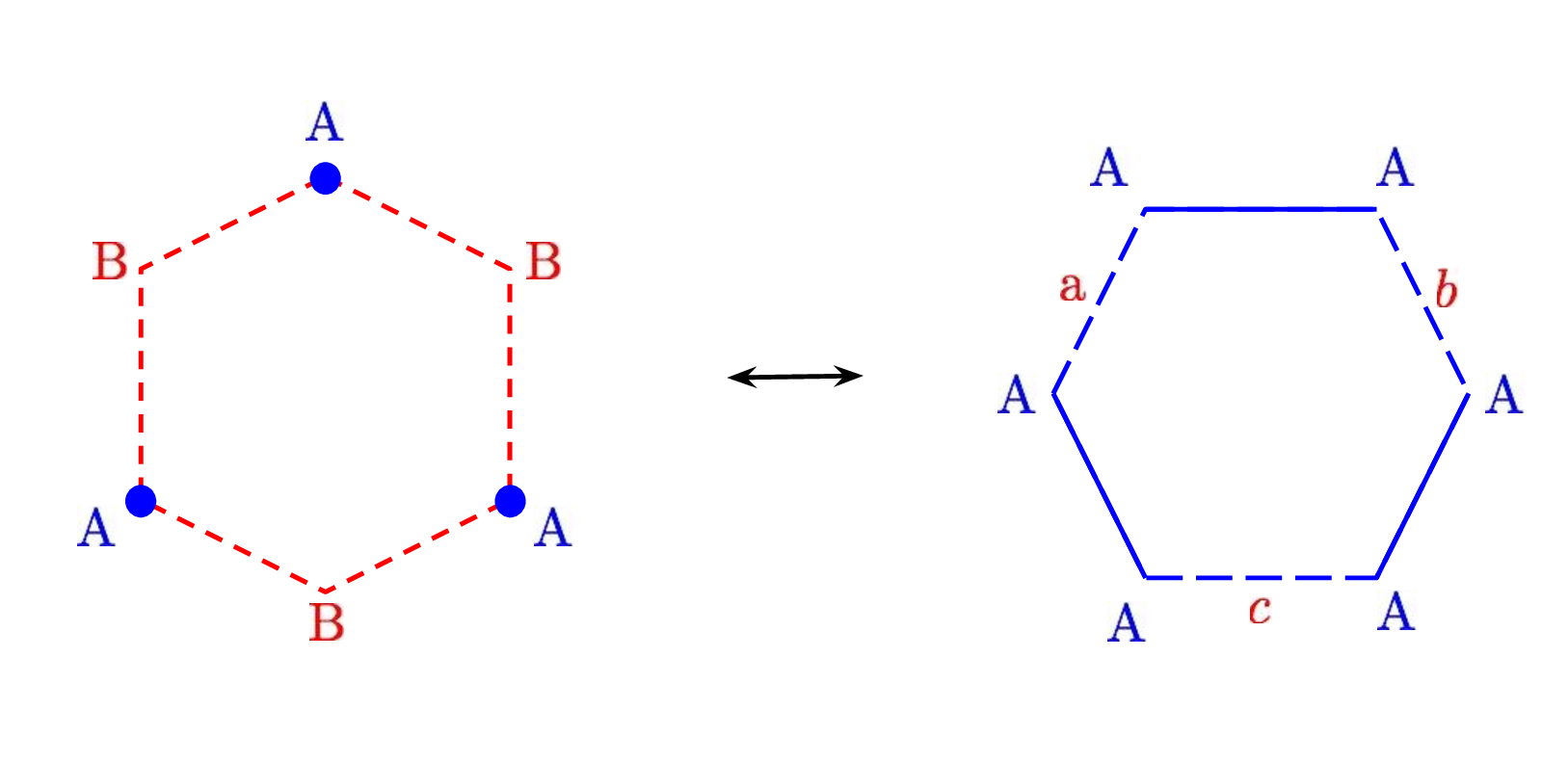}}
\caption{Getting the mirror dual for $G^n$ from configurations.}
\label{ababab}
\end{figure}

\subsection{Representation theory and examples of homological mirror symmetry for stacks} \la{sec4.2stack}

As soon as our space ${\cal M}$ is fibered over a split torus ${\Bbb H}$, 
the mirror dual space ${\cal M}_L$ acquires an action of the dual torus ${\Bbb T}_L$. 
Thus we 
want to find the mirror of the  stack ${\cal M}_L/{\Bbb T}_L$. 
Let us discuss two examples corresponding to the examples in 
Section \ref{s1.4.1a} and \ref{s1.4.1ab}. 

Let us look first at the dual pair (\ref{mduala2}). The subgroup $1\times {\rm H}_L^n$ acts freely  on the last $n$ 
decorated flags in ${\rm Conf}_{w_0}({\cal B}_L, {\cal A}^{n+1}_L)$, 
and the quotient  is ${\cal B}^{n}_L$. So one has 
\be \la{8.11.13.110}
{\rm Conf}_{w_0}({\cal B}_L, {\cal A}^{n+1}_L)/({\rm H}_L\times {\rm H}^n_L)= {\rm H}_L\backslash {\cal B}^{n}_L.
\ee
We start with the problem reflecting the A-model to this stack.

  \paragraph{1. Equivariant quantum cohomology of 
products of flag varieties.} 
There is a  way to understand 
 mirror symmetry  as 
an isomorphism of two modules over the algebra of $\hbar$-differential operators ${\cal D}_\hbar$:  
one provided by the quantum cohomology connection, 
and the other by the integral for the mirror dual Landau-Ginzburg model: 
$$
\mbox{The quantum cohomology ${\cal D}_\hbar$-module of a projective (Fano) variety ${\cal M}$}  = 
$$
$$
\mbox{The ${\cal D}_\hbar$-module for the Landau-Ginzburg mirror $(\pi: {\cal M}^\vee\to {\Bbb H}, {\cal W}, \Omega)$, 
defined by 
$\int e^{-{\cal W}/\hbar}\Omega$}. 
$$
Here the space ${\cal M}^\vee$ 
 is fibered over a torus ${\Bbb H}$, the 
$\Omega$ is a volume form on ${\cal M}^\vee$, and ${\cal W}$ is a function on ${\cal M}^\vee$,
 called the Landau-Ginzburg potential. 
The form $\Omega$ and the canonical volume form  on the torus ${\Bbb  H}$ define a volume form 
$\Omega^{(a)}$ on the fiber of the map $\pi$ over an $a\in {\Bbb  H}$. 
The integrals $\int e^{-{\cal W}/\hbar}\Omega^{(a)}$ over  cycles in the fibers are 
solutions of the ${\cal D}_\hbar$-module $\pi_*( e^{-{\cal W}/\hbar}\Omega)$ on ${\Bbb  H}$.

This approach to mirror symmetry was originated by Givental \cite{Gi}, see also Witten \cite{W} and   
\cite{EHX}, and developed further in \cite{HV} and many other works.  
See \cite{Au1}, \cite{Au2} for a discussion of examples of mirrors for   the
complements to anticanonical divisors on Fano varieties.

In our situation ${\cal M}$ is a positive space and ${\cal W}$ is a positive function, 
so there is an integral
\be \la{keyfunct}
{\cal F}_{{\cal M}}(a; \hbar):= \int_{\gamma^+(a)} e^{-{\cal W}/\hbar}\Omega^{(a)}, ~~~~\gamma^+(a):= \pi^{-1}(a) \cap {\cal M}(\R_{>0}).
\ee
If {it} converges, it defines a function on ${\Bbb H}(\R_{>0})$. 
This function as well as its partial Mellin transforms is a very important object to study. 
It plays a key role in the story. 
Below we elaborate some examples related to representation theory.   


Let $\psi_s$ be the character of ${\rm H}(\R_{>0})$ corresponding to an element 
$s \in {\rm H}_{L}(\R_{>0})$. Recall the projection
 $\mu: {\rm Conf}^\times({\cal A}^{n+1}, {\cal B})\to {\rm H}$ from (\ref{project}). 
Consider the integral
 \be \la{integralfoq}
 {\cal F}_{{\rm Conf}^\times
({\cal A}^{n+1}, {\cal B})}(a, s; \hbar):= \int_{\gamma^+(a)}\mu^*(\psi_s) e^{-{\cal W}/\hbar}\Omega^{(a)}, ~~~~
(a, s) \in ({\rm H}^{n}\times {\rm H}_L)(\R_{>0}).
\ee
It is the Mellin transform of the function (\ref{keyfunct}) along the torus 
$1 \times {\rm H}\subset {\rm H}^{n+1}$.  If $n=1$, one can identify integral (\ref{integralfoq}) with an 
integral presentation for the Whittaker-Bessel function of the 
principal series representation of $\G(\R)$ corresponding to the character $\psi_s$. 
The latter solves the quantum Toda lattice integrable system \cite{Ko}.

Therefore  it provides, generalising 
Givental's work  \cite{Gi2} for $\G=GL_m$ in non-equivariant setting,  
the integral presentation 
of the special solution of 
equivariant quantum cohomology ${\cal D}_\hbar$-module for the flag variety ${\cal B}_L$ 
studied in \cite{GiK}, \cite{Gi3}, \cite{GKLO}, \cite{GLO1}-\cite{GLO3}, \cite{L}, \cite{R1}, \cite{R2}.

Recall the special cluster coordinate system on  ${\rm Conf}_{3}({\cal A})$ for 
$\G={\rm GL}_m$ from \cite{FG1}. 
It has a slight modification providing a rational coordinate 
system on ${\rm Conf}_{w_0}({\cal B}, {\cal A}, {\cal A})$,
 see Section \ref{KT}. 
\bt \la{4.23.13.999}
i) Let $\G=GL_m$. Then the potential ${\cal W}$, expressed in the special cluster coordinate system on 
${\rm Conf}_{w_0}({\cal B}, {\cal A}, {\cal A})$, is precisely Givental's potential from \cite{Gi2}. 

The value of the integral ${\cal F}_{{\rm Conf}^\times({\cal B}, {\cal A}, {\cal A})}(a; s, \hbar)$ at $s=e$ coincides with Givental's 
integral for a solution of the quantum 
cohomology ${\cal D}_\hbar$-module ${\rm QH}^*({\cal B}_L)$ \cite{Gi2}.

ii) For any group $\G$, the integral ${\cal F}_{{\rm Conf}^\times({\cal B}, {\cal A}, {\cal A})}(a; s)$ is a solution 
of  the ${\cal D}_\hbar$-module ${\rm QH}^*_{\rm H_L}({\cal B}_L)$. 
\et

\begin{proof} i) It is proved in Section \ref{sec4.2Gi}. 

ii) Since integral (\ref{integralfoq}) provides an integral presentation for the Whittaker function, 
it is equivalent to the results of \cite{GLO1}, \cite{R1}. Observe that the parameter 
$a \in {\rm H}(\C)$ is interpreted as the parameter on $H^2({\cal B}_L, \C^*)$, 
which is the base of the small quantum cohomology connection, while the parameter 
$s\in {\rm H}_L(\R_{>0})$ is the parameter of the ${\rm H_L}$-equivariant cohomology. 
\end{proof}

For arbitrary $n$, integral (\ref{integralfoq}) determines 
the equivariant quantum cohomology ${\cal D}_\hbar$-module of ${\cal B}_L^n$. The 
latter lives on ${\rm H}^{n}\times {\rm H}_L$, it is a ${\cal D}_\hbar$-module on 
${\rm H}^{n}$, but only ${\cal O}$-module along ${\rm H}_L$. Integral (\ref{integralfoq}) 
is a solution of this ${\cal D}_\hbar$-module.

\paragraph{2. Mirror of equivariant B-model on ${\cal B}_L^n$.} 
The integral (\ref{integralfoq}) admits an analytic continuation 
in $s$ provided by the analytic continuation of the character $\psi_s$ in the integrand. 
The complex integrand lives on an analytic space defined as follows. 
Let $\widetilde {\rm H}(\C)$ is the universal cover of ${\rm H}(\C)$. 
 Denote by $({\cal B}\times \ldots\times  {\cal B})_n^{\ast, a}$ the fiber of the map $\lambda$ in 
(\ref{project}) over an $a \in {\rm H}^n$. It is a Zariski open subset of ${\cal B}^n$. 
Consider the fibered product
$$
\begin{array}{ccc}
\widetilde {({\cal B}\times \ldots\times  {\cal B})_n^{\ast, a}}(\C)&\stackrel{\widetilde 
{\rm exp}}{\lra} &({\cal B}\times \ldots\times  {\cal B})_n^{\ast, a}(\C)\\
\widetilde\mu\downarrow  &&\mu\downarrow\\
\widetilde {\rm H}(\C)& \stackrel{{\rm exp}}{\lra}&{\rm H}(\C)
\end{array}
$$

Let $\widetilde {\cal W}$ and $\widetilde \Omega$ be the lifts of ${\cal W}$ and 
$\Omega$ by the map $\widetilde {\rm exp}$. We get a locally constant family of categories 
${\cal F}{\cal S}_{\rm wr}(\widetilde {({\cal B}\times \ldots\times  {\cal B})_n^{\ast, a}}(\C),
 \widetilde {\cal W}, \widetilde \Omega)$ over ${\rm H}^n(\C)$. So the fundamental 
group $\pi_1({\rm H}^n(\C))$ acts on the category for any given $a$. The group $\pi_1({\rm H}(\C))$ 
also acts on it by the deck transformations induced from the universal cover 
$\widetilde {\rm H}(\C)\lra{\rm H}(\C)$. 

On the other hand, the Picard group of the stack 
${\rm H}_L\backslash {\cal B}_L^{n}$,  
$$
{\rm Pic}({\rm H}_L\backslash {\cal B}_L^{n}) = 
X^*({\rm H}_L) \times {\rm Pic}({\cal B}_L^n) = 
X^*({\rm H}_L)\times X^*({\rm H}^n_L)
$$ acts by autoequivalences of the category $D^b{\rm Coh}_{{\rm H}_L}({\cal B}_L^{n})$.

\bcon \la{4.27.13.1aas} There is an  equivalence of $\A_\infty$-categories
\be \la{4.27.13.1as}
{\cal F}{\cal S}_{\rm wr}(
\widetilde {({\cal B}\times \ldots\times  {\cal B})_n^{\ast, a}}(\C),
 \widetilde {\cal W}, \widetilde \Omega)~~ \sim ~~
D^b{\rm Coh}_{{\rm H}_L}({\cal B}_L^{n}).
\ee
It intertwines the deck transformation action of $\pi_1({\rm H}(\C))$ $\times$ 
the monodromy action of $\pi_1({\rm H}^n(\C))$ 
on the Fukaya-Seidel category with the action of 
$X^*({\rm H}_L) \times {\rm Pic}({\cal B}_L^n)$ 
on the category $D^b{\rm Coh}_{{\rm H}_L}({\cal B}_L^{n})$. 

The integral (\ref{integralfoq}) over Lagrangian submanifolds supporting objects of the Fukaya-Seidel category 
is a central charge for a stability condition on the category.  
\econ

Kontsevich argued \cite{K13} that there is a smaller class of stability conditions, which he called 
``physical stability conditions''. 
Stability conditions above should be from that class.  

\vskip 3mm
{\bf Examples}. 1. Let $n=1$. Then 
${\cal B}_1^{\ast, a}$ is the intersection ${\cal B}^{\ast}$ of two big Bruhat cells in the flag variety ${\cal B}$. It
 parametrising flags in generic position to two 
generic flags, say $(B^+, B^-)$.  

2. Let $\G=SL_2$, $n=1$. Then 
${\cal B}_1^{\ast, a} = \C^*$ 
with the  coordinate $u$,  
$\widetilde {\cal B}_1^{\ast, a} = \C$ with the  coordinate $t$,  $u=e^t$, 
and $\widetilde {\cal W} = a^{-1}(e^{t} + e^{-t})$ where $a\in \C^*$ is a parameter. Next, 
${\cal B}_L= \C{\Bbb P}^1$, with the natural $\C^*$-action preserving $0, \infty$. 
Conjecture \ref{4.27.13.1aas} predicts an equivalence
$$
{\cal F}{\cal S}_{\rm wr}(\C; a^{-1}(e^{t}+ e^{-t}), dt)~~\stackrel{}{\sim}~~ 
D^b{\rm Coh}_{{\C^*}}(\C{\Bbb P}^1), ~~~~ a\in \C^*.
$$
The equivalence is a trivial exercise for the experts. 
It can be checked by using the Kontsevich combinatorial model \cite{K09}, \cite{A09}, \cite{STZ}, \cite{DK}  
for the Fukaya-Seidel category as a category of locally constant sheaves on the Lagrangian skeleton  
for a surface with potential in the case 
of $(\C, e^{t}+ e^{-t})$, shown on Fig \ref{gcb2}. 

Varying the parameter $a\in \C^*$ 
in the potential we get a locally constant family of the Fukaya-Seidel categories. 
Its monodromy is an autoequivalence 
corresponding to the action of a generator of the 
group ${\rm Pic}({\Bbb P}^1)$. 
The translation $t \lms t+2\pi i$ is another autoequivalence 
corresponding to the action of a generator of the character group 
$X^*(\C^*) =\Z$ on $D^b{\rm Coh}_{{\C^*}}(\C{\Bbb P}^1)$. 
\begin{figure}[ht]
\epsfxsize120pt
\centerline{\epsfbox{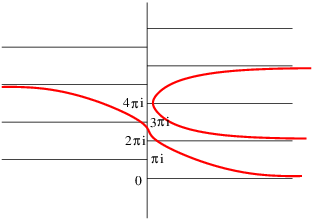}}
\caption{
Horisontal rays are the rays of fast decay of the potential. Together with the vertical line, they form 
the Lagrangian skeleton of the Kontsevich model of ${\cal F}{\cal S}_{\rm wr}(\C; a^{-1}(e^{t}+ e^{-t}), dt)$. }
\label{gcb2}
\end{figure}

Let us consider now the oscillatory integral   
$$
\int_{L} {\rm exp}(\frac{1}{\hbar}(-a^{-1}(e^{t} + e^{-t})-st))dt = 
\int_{{\rm exp}(L)} e^{-a^{-1}(u+u^{-1})/\hbar}u^{s/\hbar}\frac{du}{u}.
$$
Here $L$ is a path which goes to infinity along the line of fast decay of the integrand. 
This is an integral for the Bessel function. 
It defines a family of stability conditions on the Fukaya-Seidel category depending on $s\in \C$ -- it is
the value of the central charge  on the $K_0$-class of the object supported on  $L$. 
The parameter 
$s$ reflects the equivariant parameter for the $\C^*$-action.

\paragraph{3. Mirror of equivariant B-model on ${\cal B}_L^n\times {\rm U}_L$.} 

There is an integral very similar to  (\ref{integralfoq}): 
\be \la{integralfo}
 {\cal F}_{{\rm Conf}^\times( {\cal B}, {\cal A}^{n}, {\cal B})}(a, s):= 
\int_{\gamma^+(a)}\mu^*(\psi_s) e^{-{\cal W}/\hbar}\Omega^{(a)}, ~~~~
(a, s) \in ({\rm H}^{n-1}\times {\rm H}_L)(\R_{>0}).
\ee
Denote {by} $\lambda_{(\ref{projectabb})}$ the map $\lambda$ onto ${\rm H}_L^{n-1}$ from (\ref{projectabb}). 
The integrand has an analytic continuation in $s$ which lives on the fibered product 
$$
\begin{array}{ccc}
\widetilde {\lambda^{-1}_{(\ref{projectabb})}}(a)(\C)&\stackrel{\widetilde 
{\rm exp}}{\lra} &\lambda^{-1}_{(\ref{projectabb})}(a)(\C)\\
\widetilde\mu\downarrow  &&\mu\downarrow\\
\widetilde {\rm H}(\C)& \stackrel{{\rm exp}}{\lra}&{\rm H}(\C)
\end{array}
$$
There is a conjecture similar to Conjecture \ref{4.27.13.1aas} describing 
the category $D^b{\rm Coh}_{{\rm H}_L}({\cal B}_L^{n-1}\times {\rm U}_L)$. For example, 
when $n=1$ it reads as follows. 

\bcon \la{9.10.13.1} There is an  equivalence of $\A_\infty$-categories
\be \la{9.10.13.2}
{\cal F}{\cal S}_{\rm wr}(
\widetilde {{\rm U}^{\ast}}(\C),
 \widetilde {\cal W}, \widetilde \Omega)~~ \sim ~~
D^b{\rm Coh}_{{\rm H}_L}({\rm U}_L).
\ee
It intertwines the deck transformation action of $\pi_1({\rm H}(\C))$ 
on the Fukaya-Seidel category with the action of 
$X^*({\rm H}_L)$ 
on the category $D^b{\rm Coh}_{{\rm H}_L}({\rm U}_L)$. 
\econ
{\bf Example}. If $\G=SL_2$ and $n=1$, then ${\rm Conf}_{w_0}( {\cal B}, {\cal A}, {\cal B})=\C$ with the $\C^*$-action. 
On the dual side, ${\rm Conf}^\times( {\cal B}, {\cal A}, {\cal B})=\C^*$, 
$\pi=\mu$ is the identity map, ${\cal W}=u$,  $\Omega = du/u$. The universal cover of $\C^*$ is $\C$ with the coordinate 
$t$ such that $u=e^t$. 
The integral is 
$$
{\cal F}(s) = \int^\infty_{0}e^{-u}u^s du/u = \Gamma(s). 
$$
The equivalence of categories predicted by Conjecture \ref{9.10.13.1} is 
\be \la{9.10.13.3}
{\cal F}{\cal S}_{\rm wr}(
\C, e^t, dt)~~ \sim ~~
D^b{\rm Coh}_{\C^*}(\C).
\ee
It can be checked by using the Kontsevich combinatorial model for the Fukaya-Seidel category. 

\subsection{Concluding remarks} 

\paragraph{1. Mirror dual of the moduli spaces of ${\rm G}^L$-local systems on $S$.} The true analog of 
the moduli space of $G^L$-local systems  
for a decorated surface $S$ is the moduli space ${\rm Loc}_{{\rm G^L}, S}$. 
We view the function ${\cal W}$ on the space ${\cal A}_{{\rm G},   S}$
as the Landau-Ginzburg potential on ${\cal A}_{{\rm G},   S}$, and suggest 

\bcon
\be \la{12.1.14.1}
\mbox{\rm $({\cal A}^\times_{{\rm G},   S}, {\cal W}_{}, \Omega_{}, \pi)$ is mirror dual to 
 $({\rm Loc}_{{\rm G^L},   S}, \Omega_{L}, r_L)$}. 
\ee
\econ

It would be interesting to compare this mirror duality conjecture 
with the mirror duality  conjectures 
of Kapustin-Witten \cite{KW} and Gukov-Witten \cite{GW1}, which do not involve a potential, 
and refer to families of moduli spaces, which are somewhat different then the moduli spaces we consider.  

\vskip 3mm
Notice also that if  each boundary component of $S$ has at least one special point, 
then  ${\rm Loc}_{{\rm G^L},   S} = {\cal A}_{{\rm G^L},   S}$, and so in this case we have a more 
 symmetric picture: 
\be \la{12.1.14.2}
\mbox{\rm $({\cal A}^\times_{{\rm G},   S}, {\cal W}_{}, \Omega_{}, \pi)$ is mirror dual to 
 $({\cal A}_{{\rm G^L},   S}, \Omega_{{L}}, r_L)$}. 
\ee
\be \la{12.1.14.2a}
\mbox{\rm $({\cal A}^\times_{{\rm G},   S}, \Omega_{}, \pi)$ is mirror dual to 
 $({\cal A}^\times_{{\rm G^L},   S}, \Omega_{{L}}, r_L)$}. 
\ee


\paragraph{2. Oscillatory integrals.} 
The analog of integral (\ref{keyfunct}) in the surface case  is an integral 
\be \la{stapiS}
{\cal F}_{\G, S}(a):= \int_{\gamma^+(a)/\Gamma_S} e^{-{\cal W}/\hbar}\Omega^{(a)}.
\ee
Since the integrand  is $\Gamma_S$-invariant, 
the integration cycles are defined by intersecting
the fibers with ${\cal A}_{{\rm G},   S}(\R_{>0})/\Gamma_S$. 
Notice that ${\cal A}_{{\rm G},   S}(\R_{>0})$ is the decorated Higher 
Teichmuller space \cite{FG1}. 
If $\G=SL_2$, the integral converges. For other groups convergence is a problem.

Notice also that the three convergent oscillatory integrals 
$$
{\cal F}_{{\rm Conf}^\times_n({\cal A}, {\cal B}, {\cal B})}(s),~~~ 
{\cal F}_{{\rm Conf}^\times({\cal A}, {\cal A}, {\cal B})}(a;s), ~~~
{\cal F}_{{\rm Conf}^\times_3({\cal A})}(a_1, a_2, a_3), ~~~~a_i \in {\rm H}(\R_{>0}),~ 
s\in  {\rm H}_L(\R_{>0})
$$  are continuous analogs of the 
Kostant partition function, weight multiplicities and dimensions of triple tensor product invariants 
for the Langlands dual group $\G^L(\R)$.

\paragraph{3. Relating our dualites to cluster Duality Conjectures \cite{FG2}.} 
The latter study dual pairs $({\cal A}, {\cal X}^\vee)$, where  ${\cal A}$ is 
 a cluster ${\cal A}$-variety,  and ${\cal X}^\vee$ is the Langlands 
 dual cluster ${\cal X}$-variety:
$$
{\cal A} ~~\mbox{is dual to}~~  {{\cal X}^\vee}.
$$
There is a discrete group $\Gamma$ 
acting by automorphisms of each of the  spaces ${\cal A}$ and ${\cal X}^\vee$, 
called the {\it cluster modular group}. So it acts on the sets of tropical points 
${\cal A}(\Z^t)$ and ${\cal X}^\vee(\Z^t)$. 
Cluster Duality Conjectures predict canonical $\Gamma$-equivariant pairings
\be \la{duaIIcon}
{\bf I}_{\cal A}: {\cal A}(\Z^t) \times {\cal X}^\vee \lra  {\Bbb A}^1, ~~~~
{\bf I}_{\cal X}: {\cal A} \times {\cal X}^\vee(\Z^t) \lra {\Bbb A}^1. 
\ee
As the work \cite{GHK13} shows, in general the functions assigned to the tropical points may 
exist only as formal universally Laurent power series rather then universally Laurent polynomials.

There are cluster volume forms 
 ${\Omega}_{\cal A}$ and ${\Omega}_{\cal X}$ on the ${\cal A}$ and ${\cal X}$ spaces 
\cite{FG5}, see    Section \ref{seccluster}. 

{We suggest that, in a { rather general situation}, there is a natural $\Gamma$-invariant 
positive potential ${\cal W}_{\cal A}$ on the space ${\cal A}$, a similar potential ${\cal W}_{\cal X}$ 
on the space ${\cal X}$, and a certain ``alterations'' $\widehat {{\cal X}^\vee}$ and $\widehat {{\cal A}^\vee}$ 
of the spaces ${\cal X}^\vee$ and ${\cal A}^\vee$ providing mirror dualities underlying 
canonical pairings (\ref{duaIIcon}):
\be \la{duality321}
({\cal A}, {\cal W}_{\cal A}, {\Omega}_{\cal A}, \pi_{\cal A}) ~~\mbox{is mirror dual to}~~ 
(\widehat {{\cal X}^\vee}, {\Omega}_{{\cal X}^\vee}, r_{{\cal X}^\vee}).\ee
\be
({\cal X}, {\cal W}_{{\cal X}}, {\Omega}_{{\cal X}}, \pi_{{\cal X}}) ~~\mbox{is mirror dual to}~~ (\widehat {{\cal A}^\vee}, {\Omega}_{{\cal A}^\vee}, r_{{\cal A}^\vee}).
\ee
Canonical pairings (\ref{duaIIcon}) should induce canonical 
pairings related to the potentials and alterations:  
$$
{\bf I}_{({\cal A}, {\cal W}_{\cal A})}: {\cal A}_{{\cal W}_{\cal A}}^+(\Z^t) \times \widehat {{\cal X}^\vee} \lra {\Bbb A}^1, ~~~~
 {\bf I}_{({\cal X}, {\cal W}_{\cal X})}: {\cal X}_{{\cal W}_{\cal X}}^+(\Z^t) 
\times \widehat {{\cal A}^\vee} \lra {\Bbb A}^1. 
$$
\vskip 2mm

This should provide a cluster generalisation of our  examples. For instance, 
there is a split torus ${\Bbb H}_{{\cal A}}$ associated to a cluster variety ${\cal A}$, 
coming with a canonical basis of characters, 
given by the  {\it frozen ${\cal A}$-coordinates}. 
They describe the  projection $\pi_{\cal A}: 
{\cal A} \to {\Bbb H}_{{\cal A}}$, see   Section \ref{seccluster}. 


An alteration $\widehat {\cal A}^\vee$, given by a partial compactification of the space ${\cal A}^\vee$,   
and a conjectural definition of the potential ${\cal W}_{\cal X}$ are given in   Section \ref{potfrv}. 

\paragraph{4. Conclusion.} {\it A parametrisation of a canonical basis, casted as 
a canonical pairing ${\bf I}$, should  be understood as a manifestation of a mirror duality 
between a space with a Landau-Ginzburg potential and a similar space for the Langlands dual group}. 
\vskip2mm
Our main evidence is that canonical pairing (\ref{canpar1}) 
describing a parametrisation of canonical basis in tensor products of 
$n$ irreducible $\G^L$-modules is  related  via an integral presentation to 
the ${\cal D}_\hbar$-module describing the 
 equivariant quantum cohomology of $({\cal B}_L)^n$. 

\vskip 3mm
There is a remarkable  mirror conjecture of Gross-Hacking-Keel \cite{GHK11}, who 
start with a  maximally degenerate log Calabi-Yau $Y$ and conjecture 
that the Gromov-Witten theory of $Y$ gives rise to a commutative ring $R(Y)$, with a basis. 
Its spectrum is an affine variety which is conjectured to be the mirror of $Y$.

Notice that in our conjectures we give an {\it a priori} description of the mirror dual pair of spaces, while 
in \cite{GHK11} the mirror space is encrypted in the conjecture. For example, mirror 
conjecture (\ref{8.6.13.2}) is expected to be an example of the Gross-Hacking-Keel conjecture, but 
we do not know how to deduce, starting from the pair $({\rm Conf}^\times_n({\cal A}), \Omega)$, 
the former from the latter, and in particular why the Langlands dual group appears in the description of the mirror.  

We want to stress that in our mirror conjectures we usually 
 deal with mirror dual pairs where at least one 
is a Landau-Ginzburg model, i.e. is represented by a space with a potential. 
In particular  canonical bases in representation theory  and their generalisations 
related to moduli spaces of $\G$-local systems on decorated surfaces $S$  
always require the dual space to be 
a space with a non-trivial potential, unless $S$ is a closed surface without boundary.   

Finally, in applications to representation theory we 
are forced to deal with stacks rather then varieties, as discussed in 
Section \ref{sec4.2stack}. This is a less explored 
chapter of the homological mirror symmetry. See also 
a recent paper of C. Teleman \cite{Te14} in this direction.

\vskip 2mm
The space  ${\cal M}({\cal K})$ of ${\cal K}$-points of a space ${\cal M}$ 
is a  cousin   of the loop space $\Omega{\cal M}(\C)$. Heuristically, 
the quantum cohomology ${\cal D}_\hbar$-module is best seen in the (ill defined) 
$S^1$-equivariant Floer cohomology of the loop space $\Omega{\cal M}(\C)$ \cite{Gi}, which are sort of 
``semi-infinite cohomology'' of the loop space. 
It would be interesting to relate this to the infinite dimensional cycles 
${\cal C}^\circ_l \subset {\cal M}^\circ({\cal K})$.  
\vskip 2mm
It would be very intersecting to relate our approach to 
the construction of canonical bases via cycles ${\cal M}^\circ_l$ to the work in progress 
of Gross-Hacking-Keel-Kontsevich 
on construction of canonical bases on cluster varieties via scattering diagrams. 


\paragraph{Organization of the paper.} In Section \ref{sec2} we present 
main definitions and results relevant to representation theory. We start from a detailed discussion 
of the geometry of the tensor product invariants in Sections \ref{sec2.1}-\ref{sec2.2}. 
We discuss more general examples  in Sections \ref{sec2.3}. 
In Section \ref{tensor} we construct a canonical basis in tensor products of finite dimensional 
${\rm G}^L$-modules, and its parametrization. 
In Sections \ref{sec2} we give all definitions and 
complete descriptions of the results, 
but include a proof only if it is 
very simple. The only exception is a proof of Theorem \ref{mmvvth} in Section \ref{tensor}. 
 The rest of the proofs occupy the next Sections. 
In Section \ref{sec11} we discuss the general case related to a decorated surface. 
In the  Section \ref{seccluster}
 we discuss the volume form and the potential in the cluster set-up. 

\paragraph{Acknowledgments.} 
This work was 
supported by the  NSF grants DMS-1059129 and DMS-1301776.
A.G. is grateful to IHES and  
Max Planck Institute fur Mathematic (Bonn) for the support. 
We are grateful to  Mohammed Abouzaid, Joseph Bernstein, Alexander Braverman, 
Vladimir Fock, Alexander Givental, David Kazhdan, Joel Kamnitzer,  Sean Keel, 
Ivan Mirkovic, and Sergey Oblezin for many useful discussions. We are especially grateful 
to Maxim Kontsevich for fruitful conversations on mirror symmetry during the Summer of 2013 in IHES. 
We are very grateful to the referee for many fruitful comments, 
remarks and suggestions. 
\section{Main definitions and  results: the disc case}
\label{sec2}

\subsection{Configurations of decorated flags, the  potential ${\cal W}$, and tensor product invariants}
\label{sec2.1}

\subsubsection{Positive spaces and their tropical points} 
\la{sec2.1.2}

Below we recall  briefly the main definitions, following \cite[Section 1]{FG2}.
\paragraph{Positive spaces.} 

A positive rational function on a split algebraic torus ${\rm T}$ is a nonzero rational function on ${\rm T}$
which in a coordinate system, given by a set of characters of ${\rm T}$, can be presented as a ratio of 
two polynomials with positive integral coefficients.  

A {\it positive rational morphism} $\varphi: {\rm T}_1\ra {\rm T}_2$ of two split tori is a morphism such 
that for each character $\chi$ of ${\rm T}_2$ the function $\chi\circ \varphi$ is a positive rational function.

\vskip 2mm
A {\it positive atlas} on an irreducible space (i.e. variety / stack)  ${\cal Y}$ over $\Q$ is given by 
a non-empty collection $\{{\bf c}\}$ of birational isomorphisms over $\Q$  
$$
\alpha_{\bf c}: {\rm T} \lra {\cal Y}, 
$$
where ${\rm T}$ is a split algebraic torus, satisfying the following conditions: 
\begin{itemize}

\item
For any pair ${\bf c}, {\bf c'}$ the map 
$\varphi_{\bf c, \bf c'} := \alpha_{\bf c}^{-1}\circ \alpha_{\bf c'}$ 
is a positive birational isomorphism of ${\rm T}$. 

\item Each map $\alpha_{\bf c}$ is regular on a complement to a divisor given by 
positive rational function. 
\end{itemize}

A {\it positive space} is a space with a positive atlas. 
A split algebraic torus ${\rm T}$ is the simplest example of a positive space. It has 
a single positive coordinate system, given by the torus itself. 

A {\it positive rational function} $F$ on ${\cal Y}$ is a rational function  
given by a subtraction free rational function  in one, and hence in all  
coordinate systems of the positive atlas on ${\cal Y}$. 

A {\it positive rational map} ${\cal Y} \to {\cal Z}$ is a rational map given by positive 
rational functions in one, and hence in all positive coordinate systems. 

\paragraph{Tropical points.} The tropical semifield $\Z^t$ is the set $\Z$ equipped 
 with tropical addition and multiplication given by
$$
a+_t b=\min\{a,b\}, \quad a\cdot_t b=a+b, \quad a, b\in \Z.
$$  
This definition can be motivated as follows. 
Consider the semifield $\R_{+}((t))$ of Laurent series $f(t)$ with {\it positive} leading coefficients: there is no 
``$-$" operation in $\R_{+}((t))$. 
Then the valuation map $f(t) \mapsto {\rm val}(f)$ is a homomorphism of semifields 
${\rm val}: \R_{+}((t)) \to \Z^t$. 

Denote by $X_*({\rm T})= {\rm Hom}({\Bbb G}_m, {\rm T})$ and $X^*({\rm T})= {\rm Hom}({\rm T}, {\Bbb G}_m)$ the lattices of cocharacters and characters of a split algebraic torus ${\rm T}$. 
There is a pairing $\langle\ast, \ast\rangle: X^*({\rm T}) \times X_*({\rm T}) \to \Z$. 

The set of $\Z^t$-points of a split torus ${\rm T}$ 
is defined to be  its lattice of cocharacters: 
$$
{\rm T}(\Z^t):=  X_*({\rm T}).
$$ 
A positive rational function $F$ on ${\rm T}$ gives rise to  its tropicalization $F^t$, 
which is a $\Z$-valued  function on the set 
${\rm T}(\Z^t)$. Its definition is clear from the following example:
$$
F = \frac{x_1x_2^2 + 3 x_2x_3^5}{x_2x_4}, ~~~ F^t= {\rm min}\{x_1+2x_2, x_2+5x_3\} - {\rm min}\{x_2+x_4\}.
$$
Similarly, a positive morphism $\varphi:{\rm T}\ra {\rm S}$ of two split tori gives rise to a piecewise linear morphism $\varphi^t: {\rm T}(\Z^t)\ra {\rm S}(\Z^t)$. 

\vskip 2mm
There is a unique way to assign to a positive space ${\cal Y}$ a set ${\cal Y}(\Z^t)$ 
of its $\Z^t$-points such that 

\begin{itemize}
\item Each of the coordinate systems ${\bf c}$ provides a canonical isomorphism
$$
\alpha^t_{\bf c}: {\rm T}(\Z^t) \stackrel{\sim}{\lra}  
 {\cal Y}(\Z^t).
$$
\item These isomorphisms are related by piecewise-linear 
isomorphisms $\varphi^t_{\bf c, \bf c'}$:
$$
\alpha^t_{\bf c'}(l)  = \alpha^t_{\bf c}\circ   \varphi^t_{\bf c, \bf c'}(l).
$$
\end{itemize}

\vskip 2mm
We raise the above process to the category of positive spaces. It gives us a functor called {\it tropicalization} from the category of positive spaces to the category of sets of tropical points. For each positive morphism $f:{\cal Y}\ra {\cal Z}$, denote by $f^t: {\cal Y}(\Z^t)\ra {\cal Z}(\Z^t)$ its corresponding tropicalized morphism.

Pick a basis of cocharacters of ${\rm T}$. Then, 
assigning to each positive coordinate system ${\bf c}$
a set of integers $(l^{\bf c}_1, ..., l^{\bf c}_n)\in \Z^n$ related by piecewise-linear 
isomorphisms $\varphi^t_{\bf c, \bf c'}$, we get 
an element 
$$
l = \alpha^t_{\bf c}(l^{\bf c}_1, ..., l^{\bf c}_n)\in {\cal Y}(\Z^t).
$$

For a variety ${\cal Y}$ with a positive atlas, the set ${\cal Y}(\Z^t)$ can be interpreted as the set of 
{\it transcendental cells} of the infinite dimensional variety ${\cal Y}\big(\C((t))\big)$, as we will 
explain in Section \ref{sec2.2.1}.

\paragraph{The set of positive tropical points.} 
Let $({\cal Y}, {\cal W})$ be a pair given by a positive space ${\cal Y}$ equipped with a 
positive rational function ${\cal W}$. 
Let us tropicalize this function, getting a map
$$
{\cal W}^t: {\cal Y}(\Z^t)\lra \Z.
$$
We define the set of {\it positive tropical points}:
$$
{\cal Y}_{\cal W}^{+}(\Z^t):=\{l\in {\cal Y}(\Z^t)~|~{\cal W}^t(l)\geq 0\}.
$$

{\bf Example.} The Cartan group ${\rm H}$  of ${\rm G}$ 
is a split torus and hence has a standard positive structure. The set ${\rm H}(\Z^t)=X_*({\rm H})$ is  
the coweight lattice of   $\G$. 
Let $\{\alpha_i\}$ the set of simple positive roots indexed by $I$. 
We define
\be\la{n=2pot}
{\cal W}: {\rm H}\lra {\Bbb A}^1,~~~h\lms \sum_{i\in I}\alpha_i(h).
\ee
The set of positive tropical points is the positive Weyl chamber in $X_*({\rm H})$:
$$
{\rm H}^+(\Z^t):={\rm H}_{\cal W}^+(\Z^t)=\{\lambda\in X_*({\rm H})~|~\langle \lambda, \alpha_i\rangle\geq 0, ~\forall i\in I\}.
$$

\subsubsection{Basic notations for a split reductive group ${\rm G}$} 
\la{sec2.1.1}

Denote by ${\rm H}$ the Cartan group of ${\rm G}$, and  
by ${\rm H}^L$ the Cartan group of the Langlands dual group ${\rm G}^L$. There is a canonical isomorphism
$X^*({\rm H}^L) = X_*({\rm H}). $ 
Denote by $\Delta^+ \subset X^*({\rm H})$ the set of positive roots for ${\rm G}$, and by 
$\Pi:= \{\alpha_i\}\subset \Delta^+$ the subset of simple positive roots, indexed by a finite set ${I}$. 
We sometimes use ${\rm P}$ instead of $X_*({\rm H})$.
Denote by ${\rm P}^+$ the positive Weyl chamber  in ${\rm P}$. It is also the cone of dominant weights 
for the dual group ${\rm G}^L$. Denote by $V_\lambda$ the irreducible finite dimensional ${\rm G}^L$-modules parametrized by $\lambda\in {\rm P}^+$. 

Let ${\rm U}^{\pm}_{i}~ (i\in I)$ be the simple root subgroup of ${\rm U}^{\pm}$.   
Let $\alpha_i^{\vee}: \mathbb{G}_m \ra {\rm H}$ be the simple coroot 
corresponding to the root $\alpha_i: {\rm H}\ra \mathbb{G}_m$. 
For all $i\in I$, there are isomorphisms 
$x_i: \mathbb{G}_{a}\ra {\rm U}_{i}^{+}$ and $y_{i}: \mathbb{G}_a\ra {\rm U}_i^{-}$ such that the maps
\be \la{pinning}
   \begin{pmatrix} 
      1 & a \\
      0 & 1 \\
   \end{pmatrix}
   \lms x_i(a), \quad
      \begin{pmatrix} 
         1 & 0 \\
         b & 1 \\
      \end{pmatrix}
    \lms y_i(b),   \quad 
       \begin{pmatrix} 
          t & 0 \\
          0 & t^{-1} \\
       \end{pmatrix}
     \lms \alpha_i^{\vee}(t)  
\ee
provide homomorphisms $\phi_i: {\rm SL}_2 \ra {\rm G}.$

Let $s_i ~(i\in I)$ be the simple reflections generating the Weyl group. Set  
$\overline{s}_i:=y_i(1)x_i(-1)y_i(1).$ 
The elements $\overline{s}_i$ satisfy the braid relations. So we can associate to each $w\in W$ its representative $\overline{w}$ in such a way that for any reduced decomposition $w=s_{i_1}\ldots s_{i_k}$ one has $\overline{w}=\overline{s}_{i_1}\ldots \overline{s}_{i_k}$. 

Denote by $w_0$ be the longest element of the Weyl group.
Set $s_{\rm G}:=\overline{w}_0^2$.  It is an order two central element in {\rm G}. 
For ${\rm G=SL}_2$ it is the element $-{\rm Id}$. For an arbitrary  reductive {\rm G} the element $s_{\rm G}$ 
is the image of 
the element $s_{{\rm SL}_2}$ under a principal embedding ${\rm SL}_2 \hra {\rm G}$. 
For example, $s_{{\rm SL}_m} = (-1)^{m-1}{\rm Id}$. See \cite[Section 2.3]{FG1} for proof.

\subsubsection{Lusztig's positive atlas of $\U$ and the character $\chi_\A$}  
\la{sec4.1}
Let $w_0=s_{i_1}\ldots s_{i_m}$ be a reduced decomposition. 
It is encoded by the sequence ${\bf i}=(i_1, i_2,\ldots, i_m)$.
 It provides a regular map  
\be \la{11.20.11.191}
\phi_{\bf i}: ({\Bbb G}_m)^m \lra {\rm U}, ~~~ (a_1, ..., a_m)\lms x_{i_1}(a_1)\ldots x_{i_m}(a_m).
\ee
The map $\phi_{\bf i}$ is an open embedding \cite{L}, and a birational isomorphism. Thus it provides a rational coordinate system on ${\rm U}$. It was shown in {\it loc.cit.} that the collection of these rational coordinate systems form a positive atlas of ${\rm U}$, which we call  
{\it Lusztig's positive atlas}. There is a similar positive atlas on ${\U^{-}}$ 
provided by the maps $y_i$. 

\vskip 2mm

The choice of the maps $x_i$, $y_i$ in (\ref{pinning}) provides the 
standard character: 
\be \la{10.1.chi}
\chi: {\rm U}\lra {\Bbb A}^1,~~~x_{i_1}(a_1)\ldots x_{i_m}(a_m)\lms \sum_{j=1}^m a_j.
\ee
It is evidently a positive function in Lusztig's positive atlas. Moreover it is independent of the the sequence ${\bf i}$ chosen.
Similarly, there is a character $
\chi^{-}: \U^{-} \to {\Bbb A}^1$, $y_{i_1}(b_1)\ldots y_{i_m}(b_m) \mapsto \sum_{j=1}^m b_j$,
which is positive in the positive atlas on ${\U^{-}}$.

Let $\A:=g\cdot \U$ be a decorated flag. Its stabilizer is $\U_{\A}=g\U g^{-1}$. The associated  character is 
$$
\chi_{\A}: \U_{\A}\lra {\Bbb A}^{1}, ~~~u\lms \chi(g^{-1}ug).
$$
For example, for an $h\in {\rm H}$, the character $\chi_{h \cdot \U}$ is given by 
$
x_{i_1}(a_1)\ldots x_{i_m}(a_m) \lms \sum_{j=1}^m a_j/\alpha_{i_j}(h).
$

\subsubsection{The potential ${\cal W}$ on the moduli space ${\rm Conf}_n({\cal A})$.} 
\la{sec2.1.4}

Given a group ${\rm G}$ and  ${\rm G}$-sets $X_1, ..., X_n$,  
orbits of the diagonal ${\rm G}$-action   on $X_1\times ...\times X_n$ are called {\it configurations}. 
Denote
by $\{x_1,\ldots, x_n\}$ a collection of points, and by $(x_1,\ldots, x_n)$ its configuration.

We usually denote a decorated flag by $\A_i$ and the corresponding flag $\pi(\A_i)$ by $\B_i$. Denote the set $\{1,\ldots, n\}$ of consecutive integers by $[1,n]$.

\bd \la{6.3.12.1} 
A pair $\{\B_1, \B_2\}\in {\cal B}\times{\cal B}$ of Borel subgroups is {\em generic} if 
$\B_1 \cap \B_2$ is a Cartan subgroup in ${\rm G}$. A collection $\{\A_1,\ldots, \B_{m+n}\}\in {\cal A}^n\times {\cal B}^m$ is {\em generic} if for any distinct $i,j\in [1, m+n]$,   
the pair  $\{\B_i, \B_j\}$ is generic.
\ed

Set ${\rm Conf}({\cal A}^n, {\cal B}^m):={\G}\backslash ({\cal A}^n\times {\cal B}^m)$. 
Note that if $\{\A_1,\ldots, \B_{m+n}\}$ is generic, then so is $g\cdot \{\A_1, \ldots, \B_{m+n}\}$ for any $g\in \G$. 
Denote by ${\rm Conf}^*({\cal A}^n, {\cal B}^m)$ the subset of generic configurations.

\bd \la{torsorF} 
A frame for a split reductive algebraic group $\G$  over $\Q$ is a generic pair $\{\A,\B\}\in {\cal A}\times {\cal B}$. Denote by ${\cal F}_{{\rm G}}$ the moduli space of frames.
\ed

The space ${\cal F}_{{\rm G}}$ is a left ${\rm G}$-torsor. If ${\rm G} = {\rm SL}_m$, then a
$K$-point of 
${\cal F}_{{\rm G}}$ 
is the same thing as a unimodular frame in a vector space over $K$ of dimension $m$ with a volume form. If ${\rm G}$ is an adjoint group, then a frame is the same thing as a pinning.

\vskip 3mm

Let $\{{\rm A}_1, \ldots, {\rm A}_n\}$ be a generic collection of decorated flags. For each $j\in [1,n]$, take the triple $\{{\rm B}_{j-1}, {\rm A}_j, {\rm B}_{j+1}\}$. Since ${\cal F}_\G$ is a $\G$-torsor, 
there is a unique $u_j\in {\rm U}_{{\rm A}_j}$  such that
\be
\la{7.20.9.8}
\{ {\rm A}_j,{\rm B}_{j+1}\} = u_j \cdot \{ {\rm A}_j, {\rm B}_{j-1}\}.
\ee
Consider the following rational function on ${\cal A}^n$, whose definition is illustrated on 
Fig \ref{cal-1}:
\be \la{potential}
{\cal W}({\rm A}_1,\ldots, {\rm A}_n): = \sum_{j=1}^n \chi_{{\rm A}_j}(u_j).
\ee
\bl \la{9.21.17.00h}
For any $g\in \G$, we have ${\cal W}(g\A_1,\ldots, g\A_n)={\cal W}(\A_1,\ldots, \A_n)$.
\el
\begin{proof} Clearly $\{g\A_j, g\B_{j+1} \}=gu_jg^{-1}\cdot\{g\A_j, g\B_{j-1}\}$. 
The Lemma follows from (\ref{obvious}).
\end{proof}

\begin{figure}[ht]
\epsfxsize130pt
\centerline{\epsfbox{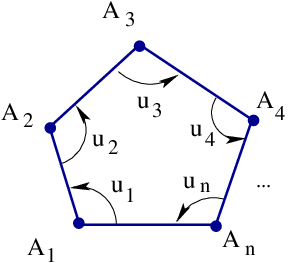}}
\caption{The potential is a sum of the contribution at the vertices.}
\label{cal-1}
\end{figure}

Since  ${\cal W}$ is invariant under the ${\rm G}$-diagonal action on ${\cal A}^n$, we define 
\bd The potential ${\cal W}$ is a rational function on ${\rm Conf}_n({\cal A})$, 
given by (\ref{potential}).
\ed

\bt \la{mth1}
The potential ${\cal W}$ is a positive rational 
function on the space ${\rm Conf}_n({\cal A})$, $n>2$.
\et

Theorem \ref{mth1} is a non-trivial result.  It is based on two facts: the character 
$\chi$ is a positive function on ${\rm U}$, and the positive structure 
on ${\rm Conf}_n({\cal A})$ is twisted cyclic invariant, see Section \ref{sec2.1.6}.
 We prove Theorem \ref{mth1} in {Section \ref{sec6.4}}.
\vskip 3mm

Therefore we arrive at the  set of positive tropical points of ${\rm Conf}_n({\cal A})$:
\be \la{11.15.11.1}
{\rm Conf}^+_n({\cal A})(\Z^t):= \{l \in {\rm Conf}_n({\cal A})(\Z^t) ~|~ 
{\cal W}^t(l) \geq 0\}, \quad n>2. 
\ee

\vskip 3mm
{\bf Example}. Let ${\rm G}={\rm SL}_2$. The space ${\rm Conf}_3({\cal A})$ parametrizes 
configurations $(v_1, v_2, v_3)$ of vectors 
in a two dimensional vector space 
with a volume form $\omega$. Set $\Delta_{i,j}:= \langle v_i \wedge v_j, \omega\rangle$. 
Then 
\be \la{FW}
{\cal W}(v_1, v_2, v_3) = \frac{\Delta_{1,3}}{\Delta_{1,2} ~\Delta_{2,3}}+ 
\frac{\Delta_{1,2}}{\Delta_{2,3}~\Delta_{1,3}}+
\frac{\Delta_{2,3}}{\Delta_{1,3}~\Delta_{1,2}}.
\ee
Therefore tropicalizing the function (\ref{FW}) we get 
$$
{\rm Conf}^+_3({\cal A}_{{\rm SL}_2})(\Z^t) = \{a,b,c\in \Z ~~|~~ a \geq b+c, ~b \geq a+c, ~c \geq a+b\}.
$$
Notice that the inequalities imply $a,b,c\in \Z_{\leq 0}$.

\subsubsection{Parametrization of  a canonical basis in 
tensor products invariants}
\la{sec2.1.5}

By Bruhat decomposition, for each $(\A_1,\A_2)\in {\rm Conf}_2^*({\cal A})$, there is a unique  $h_{\A_1,\A_2}\in {\rm H}$ such that
$$
(\A_1, \A_2)=(\U, h_{\A_1,\A_2}\overline{w}_0\cdot \U).
$$
It provides an isomorphism, which induces a positive structure on ${\rm Conf}_2({\cal A})$:
\be \la{conf2A}
\alpha: {\rm Conf}_2^*({\cal A}) \stackrel{\sim}{\lra} {\rm H}, ~~~~(\A_1,\A_2)\lra h_{\A_1,\A_2}.
\ee
We extend definition (\ref{11.15.11.1}) to  $n=2$  
using the potential \eqref{n=2pot}, 
so that one has an isomorphism
$$
\alpha^t: {\rm Conf}^+_2({\cal A})(\Z^t) \stackrel{\sim}{\lra} {\rm H}^{+}(\Z^t)={\rm P}^+.
$$
See more details in Section \ref{proofmth1}, formula (\ref{8.28.10.28h}), and \cite{FG1}.

\paragraph{The restriction maps $\pi_{ij}$.} 
We picture configurations  $({\rm A}_1, ..., {\rm A}_n)$ at the labelled vertices $[1,n]$ 
of a convex $n$-gon $P_n$. Each pair of distinct $i,j\in [1,n]$   
gives rise to a map
$$
\pi_{ij}: {\rm Conf}_n({\cal A}) \lra {\rm Conf}_2({\cal A}), ~~~ ({\rm A}_1, ..., {\rm A}_n) \lra \left\{ \begin{array}{cl}(\A_i, \A_j)~~~~&\text{if $i<j$},\\
(s_{\G}\cdot \A_i, \A_j) &\text{if $i>j$}.
\end{array} \right.
$$
The maps $\pi_{ij}$ are positive \cite{FG1}, and therefore can be tropicalized:
$$
\begin{array}{ccc}
{\rm Conf}_n({\cal A})(\Z^t) &\stackrel{\pi_{ij}^t}{\lra} &{\rm Conf}_2({\cal A})(\Z^t) = {\rm P}\\
\cup &&\cup\\
{\rm Conf}^+_n({\cal A})(\Z^t) &\stackrel{\pi_{ij}^t}{\lra} &
 {\rm Conf}^+_2({\cal A})(\Z^t) = {\rm P}^+
\end{array}
$$
The fact that $\pi_{ij}^t\big({\rm Conf}_n^+({\cal A})(\Z^t)\big)\subseteq {\rm P}^{+}$ is due to Lemma \ref{9.21.17.56h}. 

In particular, the oriented sides of the polygon $P_n$ give rise to a positive map
\be
\pi=(\pi_{12},\pi_{23},\ldots, \pi_{n,1}):{\rm Conf}_{n}({\cal A})\lra\big({\rm Conf}_2({\cal A})\big)^n\simeq {\rm H}^n.
\ee

\paragraph{A decomposition of ${\rm Conf}_n^{+}({\cal A})(\Z^t)$.}
Given $\underline{\lambda}:= (\lambda_{1}, \ldots , \lambda_n)\in ({\rm P}^{+})^n$, define
\be \la{11.20.11.1}
{\bf C}_{\underline{\lambda}} = \{l\in {\rm Conf}^+_n({\cal A})(\Z^t)~|~ \pi^t(l) = \underline{\lambda}\}. 
\ee
The weights $\underline{\lambda}$ of $\G^L$ are assigned to the oriented sides of  $P_n$, as shown on 
Fig \ref{tmp1}. 
Such sets provide a canonical decomposition
\be \la{9.21.18.17h}
{\rm Conf}_n^+({\cal A})(\Z^t)=\bigsqcup_{\underline{\lambda}\in ({\rm P}^{+})^n} {\bf C}_{\underline{\lambda}}.
\ee

\begin{figure}[ht]
\epsfxsize130pt
\centerline{\epsfbox{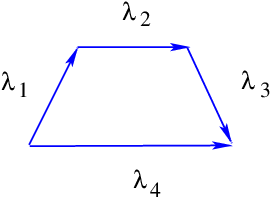}}
\caption{Dominant weights labels of the polygon sides for the set 
${\bf C}_{\lambda_1, \lambda_2, \lambda_{3}, \lambda_4}$.}
\label{tmp1}
\end{figure}
\paragraph{Tensor products invariants.} Here is one of our main results.

\bt \la{11.18.11.1} 
Let $\lambda_{1}, \ldots , \lambda_n\in {\rm P}^+$.  
The set ${\bf C}_{\lambda_1, ..., \lambda_{n}}$ parametrizes a canonical basis 
in the space of invariants
$
\big(V_{\lambda_1} \otimes \ldots \otimes V_{\lambda_n}\big)^{{\rm G}^L}.
$
\et
\noindent
Theorem \ref{11.18.11.1} follows from Theorem \ref{5.8.10.45b} and geometric Satake correspondence, 
see Section \ref{sec2.2.4}. 

\vskip 2mm

{Alternatively, there is a similar set, defined by reversing the order of the side $(1,n)$:
\be \la{11.20.11.2}
{\bf C}_{\lambda_1, ..., \lambda_{n-1}}^{\lambda_n}:= \{l\in {\rm Conf}^+_n({\cal A})(\Z^t)~|~ \pi^t_{i, i+1}(l) = \lambda_i, 
~~ i=1, ..., n-1, ~
\pi_{1, n}^t(l) = \lambda_n\}. 
\ee
Then 
$$
{\bf C}_{\lambda_1, ..., \lambda_{n}} = {\bf C}_{\lambda_1, ..., \lambda_{n-1}}^{-w_0(\lambda_n)}.
$$
The set ${\bf C}_{\lambda_1, ..., \lambda_{n-1}}^{\lambda_n}$ parametrizes 
a basis in the space of tensor product multiplicities
\be \la{11.18.11.2a}
{\rm Hom}_{{\rm G}^L}(V_{\lambda_n}, V_{\lambda_1} \otimes \ldots \otimes V_{\lambda_{n-1}}).
\ee}

\subsubsection{Some features of the set ${\rm Conf}^+_n({\cal A})(\Z^t)$.}
\la{sec2.1.6}

Here are some features of the set  ${\rm Conf}^+_n({\cal A})(\Z^t)$. All of them 
follow immediately from the  definition of the potential ${\cal W}$ 
and basic facts about the positive structure on ${\rm Conf}_n(\cal A)$. 
One of the most crucial is twisted cyclic invariance, so we start from it.

\paragraph{The twisted cyclic shift.} 

 It was proved in \cite[Section 8]{FG1}
that the defined there positive atlas on ${\rm Conf}_n({\cal A})$ is invariant under the {\it twisted cyclic shift}
$$
t: {\rm Conf}_n({\cal A}) \lra {\rm Conf}_n({\cal A}), ~~~({\rm A}_1, ...,{\rm A}_n) \lms ({\rm A}_2, ..., {\rm A}_n, {\rm A}_1\cdot s_{\rm G}). 
$$
Its tropicalization is a cyclic shift on the space of the tropical points:
\be \la{11.18.11.9}
t: {\rm Conf}_n({\cal A})(\Z^t) \lra {\rm Conf}_n({\cal A})(\Z^t). 
\ee
\begin{itemize}

\item {\bf Twisted cyclic shift invariance}. {\it The potential  ${\cal W}$ is evidently invariant  
 under the twisted cyclic shift. 
Therefore the set (\ref{11.15.11.1})  is invariant under the tropical 
cyclic shift (\ref{11.18.11.9})}. 
\end{itemize}
\vskip 2mm

Given 
a triangle $t = \{i_1< i_2<i_3\}$ inscribed into the polygon $P_n$, there is a positive map 
$$
\pi_t: {\rm Conf}_n({\cal A}) \lra {\rm Conf}_3({\cal A}), ~~~ ({\rm A}_1, \ldots, {\rm A}_n) \lms ({\rm A}_{i_1}, {\rm A}_{i_2}, {\rm A}_{i_3}).
$$
Each  triangulation  $T$ of  $P_n$ gives rise to a positive injection 
$ 
\pi_T: {\rm Conf}_n({\cal A}) \to \prod_{t \in T}{\rm Conf}_3({\cal A}), 
$ 
where the product is over all triangles $t$ of  $T$. Set its image
\be \la{5.12.31.1}
{\rm Conf}_T({\cal A}):= {\rm Im}\pi_T \subset \prod_{t \in T}{\rm Conf}_3({\cal A}). 
\ee
For each pair $(t,d)$, where $t \in T$ and $d$ is a side of $t$, there is 
a map given by obvious projections
$$
p(t,d): \prod_{t\in T}{\rm Conf}_3({\cal A})\stackrel{{\rm pr}_{t}}{\lra} {\rm Conf}_3({\cal A}) \stackrel{{\rm pr}_d}{\lra} {\rm Conf}_2({\cal A}). 
$$ 
For each diagonal $d$ of $T$, there are 
two triangles, $t^d_1$ and $t^d_2$, sharing $d$. 
A point $x$ of ${\rm Conf}_{T}({\cal A})$ is described by the condition that 
 $p(t^d_1,d)(x) = p(t^d_2,d)(x)$ for all diagonals $d$ of $T$.  
\bp {\bf \cite{FG1}} There is an isomorphism of positive moduli spaces 
$$
\pi_{T}: {\rm Conf}_n({\cal A}) \stackrel{\sim}{\lra} {\rm Conf}_{T}({\cal A}).
$$
\ep
It leads to an isomorphism of sets of their $\Z$-tropical points:
\be \la{6.1.12.1}
\pi^t_{T}: {\rm Conf}_n({\cal A})(\Z^t) \stackrel{\sim}{\lra} {\rm Conf}_{T}({\cal A})(\Z^t).
\ee
\vskip 3mm
Some important features of the potential ${\cal W}$ are the following: 

\begin{itemize}

\item {\bf Scissor congruence invariance}.  {\it For any triangulation $T$ of the polygon, 
the potential ${\cal W}_n$ on ${\rm Conf}_n({\cal A})$ is a sum 
over the triangles $t$ of $T$}:
\be \la{11.18.11.5}
{\cal W}_n = \sum_{t \in T} {\cal W}_3\circ \pi_t.
\ee
\end{itemize} 
This  follows immediately from the fact that $\chi_\A$ is a character of the subgroup $\U_\A$. 
Combining this with the isomorphism (\ref{6.1.12.1}) we get
\begin{itemize}

\item \la{DI} {\bf Decomposition isomorphism}.  {\it Given a triangulation $T$ of $P_n$, one has an isomorphism} 
$$
i^{t,+}_{T}: {\rm Conf}^+_n({\cal A})(\Z^t) \stackrel{\sim}{\lra} {\rm Conf}^+_{T}({\cal A})(\Z^t).
$$
\end{itemize}

So one can think of the data describing a 
point of ${\rm Conf}^+_n({\cal A})(\Z^t)$ 
as of a collection of similar data assigned to triangles $t$ of a 
triangulation $T$ of the polygon, 
which match at the diagonals. 
Therefore each triangulation $T$ provides a further decomposition of the set \eqref{11.20.11.1}.
By Lemma \ref{9.21.17.56h},  the weights of ${\rm G}^L$ 
assigned to the sides and edges of the polygon 
are dominant. 

\vskip 2mm
Consider an algebra with a linear basis $e_\lambda$ parametrized by  
dominant weights $\lambda$ of $\G^L$ with the 
structure constants given by the  cardinality 
of the set ${\bf C}_{\lambda_1, \lambda_2}^{\mu}$:
\be \la{11.18.11.8}
e_{\lambda_1} \ast e_{\lambda_2} = \sum_{\mu \in {\rm P}^+}|{\bf C}_{\lambda_1, \lambda_2}^{\mu}|e_{\mu}.
\ee
The following basic property is evident from our definition of the set ${\bf C}_{\lambda_1, \lambda_2}^{\mu}$:

\begin{itemize}
\item {\bf Associativity}. {\it The product $\ast$ is associative}. 
\end{itemize}
The associativity is equivalent to the fact that there are two different decompositions 
of the set  ${\rm Conf}^+_4({\cal A})(\Z^t)$  corresponding to two different triangulations of the $4$-gon. 
\begin{figure}[ht]
\centerline{\epsfbox{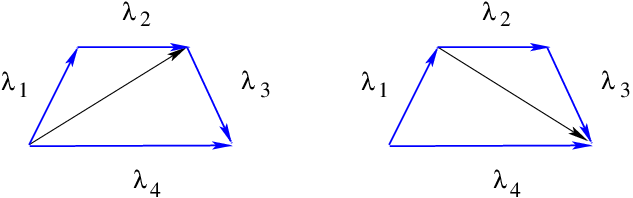}}
\caption{The associativity.}
\label{tpm2}
\end{figure}

\paragraph{A simple proof of Knutson-Tao-Woodward's theorem \cite{KTW}.} 
That theorem asserts the associativity of the similar $\ast$-product 
whose structure constants are given by  the number of hives. 
The associativity in our set-up, where the structure constant are given by 
the number of positive integral tropical points, 
 is obvious for any group ${\rm G}$. So to prove the theorem we just need to 
relate hives to positive integral tropical points for ${\rm G=GL}_m$, which is done in Section \ref{KT}.

\subsection{Parametrization of top components of fibers of convolution morphisms}
\la{sec2.2}

\subsubsection{Transcendental cells and integral tropical points} 
\la{sec2.2.1}

For a non-zero $C=\sum_{k\geq p} c_kt^k\in {\cal K}$ such that $c_p$ is not zero, we define its {\it valuation} and {\it initial term}:
$$
{\val}(C):= p, ~~~ {\rm in}(C):=c_p.
$$

\paragraph{A decomposition of ${\rm T}({\cal K})$.} For each split torus ${\rm T}$,
there is a natural projection, which we call the  valuation map:
$$
{\rm val}: {\rm T}({\cal K})\lra {\rm T}({\cal K})/{\rm T}({\cal O})={\rm T}(\Z^t).
$$
Given an isomorphism ${\rm T}=({\Bbb G}_m)^k$, the map is expressed as 
$(C_1,\ldots, C_k)\mapsto \big({\rm val}(C_1),\ldots, {\rm val}(C_k)\big)$. 

Each $l \in {\rm T}(\Z^t)$ gives rise to  a cell
$$
{\rm T}_l:= \{x\in {\rm T}({\cal K}) ~|~ \val(x) =l\}. 
$$
 It is a projective limit of irreducible algebraic varieties:
each of them is isomorphic to 
$
({\Bbb G}_m)^k \times{ \Bbb A}^N. 
$ 
Therefore ${\rm T}_l$ is an irreducible proalgebraic variety, and ${\rm T}({\cal K})$ is a disjoint union of them:
$$
{\rm T}({\cal K}) = \coprod_{l \in {\rm T}(\Z^t)} {\rm T}_l.
$$

\paragraph{Transcendental ${\cal K}$-points of ${\rm T}$.} 
Let us define an initial term map for ${\rm T}({\cal K})$ in coordinates:
$$
{\rm in}: {\rm T}({\cal K})\lra {\rm T}(\C),~~~(C_1,\ldots, C_k)\lms \big({\rm in}(C_1),\ldots, {\rm in}(C_k)\big).
$$
A subset $\{c_1,\ldots, c_q\}\subset\C$ is {\em algebraically independent} if $P(c_1,\ldots, c_q)\not =0$ 
for any $P\in \Q(X_1,\ldots, X_q)^*$. 

\bd A point $C\in {\rm T}({\cal K})$ is {\em transcendental} if its initial term 
${\rm in}(C)$ is algebraically independent as a subset of ${\C}$.
 Denote by ${\rm T}^{\circ}({\cal K})$ the set of transcendental points in ${\rm T}({\cal K})$. Set 
$$
{\rm T}_l^{\circ}:={\rm T}_l\bigcap {\rm T}^{\circ}({\cal K}).
$$ 
\ed

\bl \la{9.21.14.8h}
Let $F$ be a positive rational function on ${\rm T}$. For any $C\in {\rm T}^{\circ}({\cal K})$, we have
$$
\val \big(F(C)\big)=F^t\big(\val(C)\big).
$$
\el
\begin{proof}
It is clear.
\end{proof}

\paragraph{Transcendental ${\cal K}$-cells of a positive space ${\cal Y}$.}
 \bd
A birational isomorphism $f: {\cal Y} \ra {\cal Z}$ of positive spaces  is a 
{\em positive birational isomorphism} if it is a positive morphism, and its inverse is also a positive morphism.  
\ed

\bt \la{8.23.7.32hh}
Let $f:{\rm T}\ra{\rm S}$ be a positive birational isomorphism of split tori. Then $$
f({\rm T}_{l}^{\circ})={\rm S}_{f^t(l)}^{\circ}.
$$
\et
We prove Theorem \ref{8.23.7.32hh} in Section \ref{sec4}.
It is crucial  that the inverse of $f$ is also a positive morphism. 
As a counterexample, the map
$
f: {\Bbb G}_m\ra {\Bbb G}_m, ~x\mapsto x+1
$ is a positive morphism,
but its inverse $x\mapsto x-1$ is not.
Let $l \in {\Bbb G}_m(\Z^t) =\Z$. If $l>0$, then  Theorem \ref{8.23.7.32hh} fails: the points in $f({\rm T}_l^\circ)$ are not transcendental since  ${\rm in}(f({\rm T}_l^\circ))\equiv 1.$ 

\bd
Let $\alpha_{\bf c}:{\rm T}\ra {\cal Y}$ be a coordinate system from a positive atlas 
on  ${\cal Y}$. 
The set of transcendental ${\cal K}$-points of ${\cal Y}$ is 
$$
{\cal Y}^{\circ}({\cal K}):=\alpha_{\bf c}({\rm T}^{\circ}({\cal K})).
$$
For each $l\in {\cal Y}(\Z^t)$, the transcendental $l$-cell \footnote{By abuse of notation, such a cell will always be denoted by ${\cal C}_l^\circ$. The tropical point $l$ tells which space it lives.} of ${\cal Y}$ is
$$
{\cal C}_{l}^{\circ}:=\alpha_{\bf c}({\rm T}_{\beta^t(l)}^{\circ}),\quad \mbox{where $\beta=\alpha_{\bf c}^{-1}$}.
$$
\ed
By Theorem \ref{8.23.7.32hh}, this definition is independent of the coordinate system $\alpha_{\bf c}$ chosen. 
Similarly one can upgrade the valuation map to positive spaces: given a positive space ${\cal Y}$, 
there is a unique map 
\be \la{val12}
\val: {\cal Y}^{\circ}({\cal K})\lra {\cal Y}(\Z^t)
\ee
such that
$$
{\cal C}_l^{\circ}=\{y\in {\cal Y}^{\circ}({\cal K})~|~\val(y)=l\}.
$$
The valuation map (\ref{val12}) is functorial under positive birational isomorphisms of positive spaces. 
Therefore the transcendental cells are also  functorial under positive birational isomorphisms.

Thus there is a canonical decomposition parametrized by the set ${\cal Y}(\Z^t)$:
$$
{\cal Y}^{\circ}({\cal K})=\bigsqcup_{l\in {\cal Y}(\Z^t)} {\cal C}_{l}^{\circ}.
$$
Thanks to the following Lemma, one can identify each tropical point $l$ with ${\cal C}_l^{\circ}$.
\bl \la{thm10.1.1.2}
Let $F$ be a positive rational function on ${\cal Y}$. For any $C\in {\cal Y}^{\circ}({\cal K})$, we have
$$
\val \big (F(C)\big)=F^t\big( \val(C)\big). 
$$
\el

\begin{proof}
It follows immediately from Lemma \ref{9.21.14.8h} and Theorem \ref{8.23.7.32hh}.
\end{proof}

\subsubsection{ ${\cal O}$-integral configurations of decorated flags and the affine Grassmannian}
\la{sec2.2.2}

Recall  the affine Grassmannian
${\rm Gr}$. Recall the moduli space ${\cal F}_{\G}$ of frames from Definition \ref{torsorF}.

\blc \la{LCAG}
There is a canonical onto map
\be \la{10.13.12.1}
\lr: {\cal F}_{\G}({\cal K})\lra {\rm Gr}, ~~~~\{\A_1,\B_2\}\lms \lr(\A_1,\B_2)
\ee
\elc

\begin{proof} 
Let $\{\U, \B^-\}\in {\cal F}_{{\rm G}}(\Q)$ be a standard frame. There is a unique $g_{\{\A_1, \B_2\}} \in {\rm G}({\cal K})$ such that 
$$
\{\A_1, \B_2\}= g_{\{\A_1, \B_2\}}\cdot \{\U, \B^-\}.
$$
It provides an isomorphism ${\cal F}_\G({\cal K})\stackrel{\sim}{\ra}\G({\cal K})$. 
Composing it with the projection $[\cdot]: \G({\cal K})\ra {\rm Gr}$, 
\be \la{8.11.12.100}
\lr(\A_1, \B_2):= [g_{\{\A_1, \B_2\}}]\in {\rm Gr}.
\ee
Note that  ${\cal F}_{{\rm G}}(\Q)$  
is a ${\rm G}(\Q)$-torsor. 
So choosing a different frame in  ${\cal F}_{{\rm G}}(\Q)$ we get another
 representative of the coset $g_{\{\A_1, \B_2\}} \cdot {\rm G}(\Q)$. Since ${\rm G}(\Q)\subset {\rm G}({\cal O})$, 
the resulting lattice (\ref{8.11.12.100}) will be the same. Therefore the map $\lr$ is canonical.
\end{proof}
\paragraph{Symmetric space and affine Grassmannian.}
The affine Grassmannian is the non-archimedean version of the 
symmetric space ${\rm G}({\R})/{\rm K}$, where ${\rm K}$ is a maximal compact subgroup in $\G(\R)$. 
A generic pair of flags  $\{\B_1, \B_2\}$ over $\R$
gives rise to an ${\rm H}({\R}_{>0})$-torsor in the 
symmetric space --   
the projection of $\B_1\cap \B_2$.  \footnote{Here is a non-archimedean analog: A generic pair of flags  $\{\B_1, \B_2\}$ over ${\cal K}$ gives rise to an 
$H({\cal K})/H({\cal O})$-torsor in the affine Grassmannian -- 
the projection of $\B_1({\cal K})\cap \B_2({\cal K})$ 
to ${\rm G}({\cal K})/{\rm G}({\cal O})$.} 
 Notice that ${\rm H}({\R}_{>0}) = {\rm H}(\R)/({\rm H}(\R)\cap {\rm K}) $.
A generic pair $\{\A_1, \B_2\}$ determines a 
point\footnote{In the archimedean case, a maximal compact subgroup ${\rm K}$ 
is defined by using the Cartan involution. A generic pair $\{\A, \B\}$ determines  a pinning, and hence a Cartan involution.}  
$
Q(\A_1, \B_2)\in {\rm G}(\R)/{\rm K}.
$ 
So we get the archimedean analog of the map (\ref{10.13.12.1}):
\be \la{10.13.12.2}
Q: {\cal F}_{\G}({\R})\lra {\rm G}(\R)/{\rm K}, ~~~~\{\A_1,\B_2\}\lms Q(\A_1, \B_2).
\ee

\paragraph{Decorated flags and horospheres.} 
For the adjoint group ${\rm G}'$, the principal affine space ${\cal A}$ can be interpreted as the moduli space of horospheres
in the symmetric space ${\rm G}'(\R)/K$ in the archimedean case, or in the affine Grassmannian ${\rm Gr}$. 
The horosphere ${\cal H}_\A$ assigned to a decorated flag $\A$ 
is an orbit of the maximal unipotent subgroup ${\rm U}_\A$. 
Let ${\cal B}^*_\A$ be the open Schubert cell of flags in generic position to a given decorated flag $\A$. 
Then there is an isomorphism
$$
{\cal B}^*_{\A} \lra {\cal H}_\A, ~~~~ \B\lms \lr(\A, \B)~~ \mbox{or}~~ \B\lms Q(\A,\B).
$$

\begin{figure}[ht]
\epsfxsize80pt
\centerline{\epsfbox{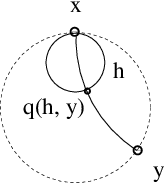}}
\caption{The metric $q(h,y)$ determined by a horocycle $h$ 
and a boundary point $y$.}
\label{tpm0}
\end{figure}

{\bf Examples}. 1. Let $\G(\R)={\rm SL}_2(\R)$.  Its maximal compact subgroup ${\rm K}={\rm SO}_2(\R)$. The symmetric space ${\rm SL}_2(\R)/{\rm SO}_2(\R)$ is the hyperbolic plane ${\cal H}^2$. A decorated flag ${\A}_1\in {\cal A}_{{\rm PGL}_2}(\R)$  corresponds to a horocycle $h$ based as a point $x$ at the boundary. A flag ${\B}_2$ corresponds to another point $y$ at the boundary. Let $g(x,y)$ be the geodesic connecting $x$ and $y$.  The point $Q(\A_1,\B_2)$ is 
the intersection of $h$ and
$g(x, y)$, see Fig \ref{tpm0}:
$$
q(h, y):= h\cap g(x, y) \in {\cal H}^2.
$$ 

2. Let $\G={\rm GL}_n$. Recall that 
a flag $F_\bullet$ in an $n$-dimensional vector space $V_n$ over a field 
is a data $F_1 \subset \ldots \subset F_n$, ${\rm dim}F_i=i$. 
A generic pair of flags $(F_\bullet, G_\bullet)$ 
in $V_n$ is the same thing as  
a decomposition of  $V_n$ into a 
direct sum of one dimensional subspaces
\be \la{L}
V_n= L_1 \oplus \ldots \oplus L_n,
\ee
where
$
L_i = F_{i} \cap G_{n+1-i}.$  Conversely,  $F_{a}= L_1 \oplus 
\ldots \oplus L_a$ and $G_{b}= L_{n-b+1} \oplus \ldots \oplus L_n.$

Over the field $\R$, this decomposition gives rise to a
$(\R^*_{>0})^n$-torsor in the symmetric space, given by a family of 
positive definite metrics on $V_n$ with the principal axes $(L_1, \ldots , L_n)$:
\be \la{M}
a_1 x_1^2 + \ldots + a_n x_n^2, \quad a_i > 0. 
\ee
Here $(x_1, ..., x_n)$ is a coordinate system  for which the 
lines $L_i$ are the coordinate lines.

A decorated flag $\A$ in $V_n$  is a flag $F_\bullet$ plus a collection 
of non-zero vectors $l_i \in F_{i}/F_{i-1}$. 
A frame in ${\cal F}_{{\rm GL}_n}$ is equivalent to a generic pair of flags $(F_\bullet, G_\bullet)$ and 
a decorated flag $\A$ over the flag $F_\bullet$.
It determines a basis $(e_1, ..., e_n)$ in $V_n$ and vice verse.  Here $e_i\in L_i$ and $e_i=l_i$ under the projection $L_i\lra F_i/F_{i-1}$. 
This basis determines a metric -- 
the positive definite metric with the principal axes 
$L_i$ such that the vectors $e_i$ are unit vectors.

3. Over the field ${\cal K}$, decomposition (\ref{L}) gives rise to an 
${\rm H}({\cal K})/{\rm H}({\cal O}) = \Z^n$-torsor in ${\rm Gr}$, given by the following collection
 of 
lattices in $V_n$. 
$$
{\cal O}t^{k_1}e_1 \oplus \ldots \oplus {\cal O}t^{k_n}e_n, ~~k_i \in \Z. 
$$
These lattices are the non-archimedean version of the unit balls of the metrics 
(\ref{M}).

\paragraph{${\cal O}$-integral configurations of decorated flags.}
\bd \la{13.3.12.502h}
A collection of decorated flags $\{\A_1,\ldots, \A_n\}$ over ${\cal K}$ is ${\cal O}$-integral if it is generic and  for any $i\in [1,n]$ the lattice $\lr(\A_i, \B_j)$ does not depend on the choice of $j$ different then $i$. 
\ed

 Let $g\in \G({\cal K})$. Note that  $\lr(g\A_i, g\B_j)=g\cdot \lr(\A_i, \B_j)$. Therefore if $\{\A_1,\ldots, \A_n\}$ is ${\cal O}$-integral, so is $g\cdot \{\A_1,\ldots, \A_n\}$.  Thus we define

\bd \la{8.8.12.1}
A  configuration in ${\rm Conf}_n({\cal A})({\cal K})$ is ${\cal O}$-integral 
if it is a ${\rm G}({\cal K})$-orbit of an ${\cal O}$-integral collection of decorated flags. 
Denote by ${\rm Conf}_n^{\cal O}({\cal A})$ the space 
of  such configurations. 
\ed

The archimedean version of Definition \ref{8.8.12.1} 
is trivial. For example, let $\G= {\rm SL}_2(\R)$. Then 
there are no  horocycles $(h_1, h_2, h_3)$ 
such that  their boundary points $(x_1, x_2, x_3)$ are distinct, 
and the intersection of the horocycle $h_i$ with the geodesic 
$g(x_i,x_j)$ do not depend on $j \not = i$. 

In contrast with this, we demonstrate below that the non-archimedean version is very rich.
The difference stems from the fact that in the archimedean case 
the intersection $K \cap {\rm U}=e$ is trivial, while in the non-archimedean 
${\rm G}({\cal O}) \cap {\rm U}({\cal K}) = {\rm U}({\cal O})$. 

\paragraph{Transcendental cells and ${\cal O}$-integral configurations.}
The following  fact is crucial.
\bt \la{8.27.17.08hh}
If $l \in {\rm Conf}_n^{+}({\cal A})(\Z^t)$, then there is an inclusion 
$
{\cal C}_l^{\circ}\subset {\rm Conf}^{\cal O}_n({\cal A}). 
$ 
Otherwise ${\cal C}_l^{\circ}\cap {\rm Conf}^{\cal O}_n({\cal A})$ is an empty set. 
\et
Theorem \ref{8.27.17.08hh} gives an alternative conceptual definition of the set of 
positive integral tropical 
points of the space ${\rm Conf}_n({\cal A})$, which {refers neither 
to the potential ${\cal W}$, nor to a specific positive coordinate system.} 
However to show that the set ${\rm Conf}_n^{+}({\cal A})(\Z^t)$ is ``big'', or even non-empty, we 
use the potential ${\cal W}$ and  its properties, which imply, for example, 
that the set ${\rm Conf}_n^{+}({\cal A})(\Z^t)$ is obtained by amalgamation of similar sets assigned to triangles 
of a triangulation of the polygon. 
We prove Theorem \ref{8.27.17.08hh} in Section \ref{sec6.4}.

\subsubsection{The canonical map $\kappa$ and cycles on ${\rm Conf}_n({\rm Gr})$}
\la{sec2.2.3}

\paragraph{The canonical map $\kappa$.}
Recall
the configuration space
$$
{\rm Conf}_n({\rm Gr}):= {\rm G}({\cal K})\backslash ({{\rm Gr}\times \ldots \times {\rm Gr}}).
$$
Given an ${\cal O}$-integral collection $\{\A_1, ..., \A_n\}$ of  decorated flags, 
we get a collection of lattices 
$\{\lr_1, ..., \lr_n\}$ by setting $\lr_i:= \lr(\A_i, \B_j)$ for some $j \not = i$. 
By definition, 
the lattice $\lr_i$  is independent of $j$ chosen. 
This construction descends to configurations, providing  a canonical map 
\be \la{4.18.12.41qx}
\kappa: {\rm Conf}^{\cal O}_n({\cal A}) \lra {\rm Conf}_n({\rm Gr}), \quad 
(\A_1, \ldots, \A_n) \lms (\lr_1, ..., \lr_n).
\ee
The map is evidently cyclic invariant, and  commutes with the restriction 
to subconfigurations: 
$$
\kappa({\rm A}_{i_1}, ..., {\rm A}_{i_k}) = 
({\rm L}_{i_1}, ..., {\rm L}_{i_k})\quad  
\mbox{for any $1\leq i_1 <  ...< i_k\leq n$}.
$$

\paragraph{The cycles ${\cal M}_l$ in ${\rm Conf}_n({\rm Gr})$.}
Let $l\in {\rm Conf}^+_n({\cal A})(\Z^t)$.  Thanks to Theorem \ref{8.27.17.08hh}, 
we can combine the inclusion there 
with the canonical map (\ref{4.18.12.41qx}): 
\be \la{4.18.12.444}
{\cal C}_l^{\circ}\subset  {\rm Conf}^{\cal O}_n({\cal A}) \stackrel{\kappa}{\lra} {\rm Conf}_n({\rm Gr}).
\ee

\bd \la{9.22.12.13h}
The cycle ${\cal M}_l \subset {\rm Conf}_n({\rm Gr})$ is a substack  
given by the closure of $\kappa({\cal C}^{\circ}_l)$:
\be \la{4.18.12.444}
{\cal M}_l:= \overline{{\cal M}_l^\circ}, ~~~~{\cal M}_l^\circ:= {\kappa({\cal C}^{\circ}_{l})} \subset {\rm Conf}_n({\rm Gr}),\qquad l \in {\rm Conf}_n^{+}({\cal A})(\Z^t).
\ee
\ed

\bl
The cycle ${\cal M}_l$ is irreducible. 
\el

\begin{proof}
For a split torus ${\rm T}$, the cycle ${\rm T}_{l}$ is irreducible. So the cycles ${\cal C}^{\circ}_{l}$ and 
${\cal M}_l$ are irreducible. 
\end{proof}

In other words, ${\cal M}_l$ is a ${\rm G}(\cal K)$-invariant closed subspace in ${\rm Gr}^n$. 
There is a bijection
\be \la{MTH11}
\{\mbox{${\rm G}(\cal K)$-orbits in ${\rm Gr}^n$}\} 
\stackrel{1:1}{\longleftrightarrow}  \{\mbox{${\rm G}(\cal O)$-orbits in $[1] \times{\rm Gr}^{n-1}$}\}. 
\ee
Therefore one can also view the cycles ${\cal M}_l$ as ${\rm G}(\cal O)$-invariant closed subspaces in 
$[1]\times{\rm Gr}^{n-1}$. 
Let us describe them using this point of view.

\subsubsection{Top components of the fibers of the convolution morphism}
\la{sec2.2.4}

Given $\underline{\lambda}=(\lambda_1,\ldots, \lambda_n)\in ({\rm P}^+)^n$, recall the cyclic convolution variety
$$
{\rm Gr}_{c(\underline{\lambda})}:=\{({\rm L}_1, \ldots, {\rm L}_n)\in {\rm Gr}^n ~|~ {\rm L}_1 \stackrel{\lambda_1}{\lra}{\rm L}_2\stackrel{\lambda_2}{\lra}\ldots\stackrel{\lambda_n}{\lra}{\rm L}_{n+1}, ~{\rm L}_1={\rm L}_{n+1} = [1]\}.
$$  
It is a finite dimensional reducible variety of top dimension
$$
{\rm ht}(\underline{\lambda}):= \langle \rho, \lambda_1+ \ldots +\lambda_n\rangle.
$$
It is the fiber of the convolution morphism, and therefore,
thanks to the geometric 
Satake correspondence \cite{L4} \cite{G} \cite{MV},
there is a canonical isomorphism 
\be \la{5.31.12.2}
{\rm IH}^{{\rm ht}(\underline{\lambda})}({\rm Gr}_{c(\underline{\lambda}) }) = 
\Bigl(V_{\lambda_1} \otimes \ldots \otimes V_{\lambda_n}\Bigr)^{{\rm G}^L}.
\ee
Each top dimensional component of ${\rm Gr}_{c(\underline{\lambda}) }$ provides an element 
in the space (\ref{5.31.12.2}). These elements form a canonical basis in (\ref{5.31.12.2}). 
Let ${\bf T}_{\underline{\lambda}}$ be the set of top dimensional components of ${\rm Gr}_{c(\underline{\lambda})}$. 
Recall the set ${\bf C}_{\underline{\lambda}}$ of positive tropical points (\ref{11.20.11.1}), and the cycle 
${\cal M}_l$ from Definition \ref{9.22.12.13h}.
 
 \bt \la{5.8.10.45b}
Let $l\in {\bf C}_{\underline{\lambda}}$. Then the cycle ${\cal M}_{l}$ is the closure of a top dimensional component of 
${{\rm Gr}_{c(\underline{\lambda})}}$. The map $l \lms  {\cal M}_l$ provides a canonical bijection from  ${\bf C}_{\underline{\lambda}}$ to  ${\bf T}_{\underline{\lambda}}$.
 \et
Theorem \ref{5.8.10.45b} is proved in Section \ref{sec9.4}. It implies Theorem \ref{11.18.11.1}.

\subsubsection{Constructible equations for the top dimensional components}
\la{sec2.2.5}

We have defined the cycles ${\cal M}_{l}$ as the closures of the images of the cells ${\cal C}_{l}^\circ$.
Now let us define the cycles ${\cal M}_{l}$ by equations,
given by certain  constructible functions on the space ${\rm Conf}_n({\rm Gr})$.
These functions  generalize Kamnitzer's functions $H_{i_1,\ldots, i_n}$  for ${\rm G=GL}_m$ (\cite{K1}).

\paragraph{Constructible function $D_F$.}{
Let $R$ be a reductive algebraic group over $\C$. We assume that there is a rational left algebraic action of $R$ on $\C^n$. Let $\C(x_1,\ldots, x_n)$ be the field of rational functions on $\C^n$. We get a right algebraic action of $R$ on $\C(x_1,\ldots, x_n)$ denoted by $\circ$.

Let ${\cal K}(x_1,\ldots, x_n)$ be the field of rational functions with ${\cal K}$-coefficients. The valuation of ${\cal K}^\times$ induces a natural valuation map
$$
\val:~~ {\cal K}(x_1,\ldots, x_n)^\times \lra \Z.
$$
Let $F, G\in {\cal K}(x_1,\ldots, x_n)^\times$. The valuation map has two basic properties
\begin{align}
\val(FG)&=\val(F)+\val(G),\la{9.7.14.1h}\\
\val(F+G)&=\val(F), \quad \mbox{if }\val(F)<\val(G).\la{9.7.14.2h}
\end{align}

The group $R({\cal K})$ acts on ${\cal K}(x_1,\ldots, x_n)$ on the right. We have the following 
\bl\la{lem1.9.7.14}
Let $F\in {\cal K}(x_1,\ldots, x_n)^\times$. If $h\in R({\cal O})$, then $\val(F\circ h)=\val(F)$.
\el
\begin{proof} For any $k\in {\cal K}^\times$, we have
$(kF)\circ h=k(F\circ h).$
Therefore it suffices to prove the case when $\val(F)=0$.

Note that the group $R$ is reductive. It is generated by
$$
x_i(a)\in \U,\quad y_i(b)\in \U^{-},\quad \alpha(c)\in {\rm H}, \quad \quad \mbox{where $i\in I$ and $\alpha\in {\rm Hom}({\Bbb G}_m, {\rm H})$.}
$$
Since the action of $R$ is algebraic, for any $f\in \C(x_1,\ldots, x_n)^\times$, we have $f\circ x_i(a)\in \C(x_1,\ldots, x_n, a)^\times$. Note that $f\circ x_i(0)=f$. Therefore we get
\be\la{9.7.14.4h}
f\circ x_i(a)=
\frac{f+af_1+\ldots+a^l f_l}{1+ag_1+\ldots+a^mg_m}, \quad \mbox{where $f_j, g_j \in \C(x_1,\ldots, x_n)$.}
\ee
If $a\in \C$, then $f\circ x_i(a)\in \C(x_1,\ldots,x_n)$. Moreover $f\circ x_i(a)$ is non zero. Otherwise, 
$
f=(f\circ x_i(a))\circ x_i(-a)=0.
$ 
If $a\in t{\cal O}$, then by the basic property \eqref{9.7.14.2h}, we get 
$\val(f\circ x_i(a))=\val(f)=0.$ 

Let $a=a_0+b=a_0+a_1t+a_2t^2+\ldots  \in {\cal O}$. Then 
$
f\circ x_i(a)=\big(f\circ x_i(a_0)\big)\circ x_i(b).
$ 

Note that $f\circ x_i(a_0)\in \C(x_1,\ldots, x_n)^\times$ and $b\in t{\cal O}$. Combining the above arguments we get
\be \la{9.7.14.40h}
\val(f\circ x_i(a))=\val(f\circ x_i(a_0))=0=\val(f), ~~~\forall a\in {\cal O}.
\ee

Now let $F\in {\cal K}(x_1, ..., x_n)^\times$ such that $\val(F)=0$. Then $F$ can be expressed as
$$
F=\frac{f_0+b_lf_1+\ldots+b_lf_l}{1+c_1g_1+\ldots+c_mg_m}.
$$
Here
$f_0, f_p, g_q\in \C(x_1,\ldots, x_n)^\times$, $b_p, c_q\in {\cal K}^\times$, $\val(b_p)>0, \val(c_q)>0, 
$. 
By definition, we have
$$
F\circ x_i(a)=\frac{f_0\circ x_i(a)+b_1f_1\circ x_i(a)+
\ldots+b_lf_l\circ x_i(a)}{1+c_1g_1\circ x_i(a)+\ldots+c_mg_m\circ x_i(a)}.
$$
Let $a\in {\cal O}$. {By \eqref{9.7.14.40h}}, we get
\begin{align}
\val(f_0\circ x_i(a))&=0,\nonumber\\
\val(b_pf_p\circ x_i(a))&=\val(b_p)+\val(f_p\circ x_i(a))=\val(b_p)>0,\nonumber\\
\val(c_qg_q\circ x_i(a))&=\val(c_q)+\val(g_q\circ x_i(a))=\val(c_q)>0. \nonumber
\end{align}
By the basic property \eqref{9.7.14.2h}, we get $\val(F\circ x_i(a))=\val(f_0\circ x_i(a))=0$. Hence we prove that
$$
\val(F\circ x_i(a))=\val(F),\quad \forall a\in {\cal O}.
$$
By the same argument, we show that
$$
\val(F\circ y_i(b))=\val(F),\quad \forall b\in {\cal O}, ~~~~
\val(F\circ \alpha(c))=\val(F),\quad \forall c\in {\cal O}^\times.
$$
Note that $R({\cal O})$ is generated by the elements
$
x_i(a), y_i(b), \alpha(c),  \quad a,b\in {\cal O}, c\in {\cal O}^\times.
$ 
The Lemma is proved.
\end{proof}

Let $\mathfrak{X}$ be rational space over $\C$, i.e.,
$\C(\mathfrak{X})\stackrel{\sim}{=}\C(x_1,\ldots, x_n).$
Similarly, there is a valuation map $\val: {\cal K}(\mathfrak{X})^\times \rightarrow \Z$. 
We assume that there is left algebraic action of $R$ on $\mathfrak{X}$. Lemma \ref{lem1.9.7.14} implies 
\bl \la{13.1.26.11.18h}
Let $F\in {\cal K}(\mathfrak{X})^\times.$ If $h\in R({\cal O})$, then $\val(F\circ h)=\val(F)$.
\el 



\paragraph{Constructible equations for top components.} 
 Let $\mathfrak{X} := {\cal A}^n$ and let $R:=\G^n$. Let $F\in \C({\cal A}^n)$ and let $(g_1,\ldots, g_n)\in \G^n$. Then $\G^n$ acts on $\C({\cal A}^n)$ on the right:
 \be \la{rightaction}(F\circ(g_1, ..., g_n))(\A_1, ..., \A_n):= F(g_1\cdot\A_1, ..., g_n\cdot \A_n),~~~\forall (\A_1,\ldots, \A_n)\in {\cal A}^n.\ee 
 By definition, a nonzero rational function $F\in \C({\rm Conf}_n({\cal A}))$ is also a $\G$-diagonal invariant function on ${\cal A}^n$
$$
F(g\A_1, ..., g\A_n) = F(\A_1, ..., \A_n).
$$
 There  is a $\Z$-valued function
\be \la{6.4.12.1}
D_F: \G({\cal K})^n\lra \Z, ~~~~
D_F(g_1(t), ..., g_n(t)):= {\rm val}\Bigl(F\circ (g_1(t), ..., g_n(t))\Bigr). 
\ee
\blc
The function $D_F$ is invariant under the left diagonal action of the group $\G({\cal K})$  
on ${\rm G}({\cal K})^n$, and the 
right action of the subgroup ${\rm G}({\cal O})^n \subset {\rm G}({\cal K})^n$. Therefore $D_F$ descends to a function ${\rm Conf}_n({\rm Gr})\ra \Z$ which we also denote by $D_F$.
\elc

\begin{proof} The first property is clear 
since $F \in \C({\cal A}^n)^{\rm G}$. 
The second property is by Lemma  \ref{13.1.26.11.18h}. \end{proof}


Let $\Q_{+}({\rm Conf}_n({\cal A}))$ be the semifield of positive 
rational functions on  ${\rm Conf}_n({\cal A})$. 
Take a non-zero function $F\in \Q_{+}({\rm Conf}_n({\cal A}))$. Therefore it gives rise to a function $D_F$ on ${\rm Conf}_n({\rm Gr})$. Meanwhile,  its tropicalization $F^t$ is a function on ${\rm Conf}_n({\cal A})(\Z^t)$.

\bt \la{5.8.10.45a}
Let $l\in{\rm Conf}_n^+({\cal A})(\Z^t)$ and $F\in {\Q}_{+}({\rm Conf}_n({\cal A}))$. Then 
$
D_F\big(\kappa({\cal C}_l^{\circ})\big)\equiv F^t(l).
$ 
\et
}
Theorem \ref{5.8.10.45a} is proved in Section \ref{sec7}. 
It implies that the map  in Theorem \ref{5.8.10.45b} is injective. It can be reformulated as follows: 
\be \la{MHV1111}
\mbox{For any $l$ and $F$ as above, 
the generic value of $D_F$ on the cycle ${\cal M}_l$ is $F^t(l)$}.
\ee

When ${\rm G=GL}_m$, one can describe the set ${\bf C}_{\underline{\lambda}}$ by using 
the special collection of functions on the space ${\rm Conf}_n({\cal A})$ defined in Section 9 of \cite{FG1}. 
The obtained description coincides with Kamnitzer's  
generalization of hives \cite{K1}. He conjectured in \cite{K1} 
that the latter set parametrizes the components of the convolution variety for ${\rm GL}_m$. 
Therefore Theorems \ref{5.8.10.45b}  and \ref{5.8.10.45a} imply Conjecture 4.3 in \cite{K1}.



\subsection{Mixed configurations and a generalization of Mirkovi\'{c}-Vilonen cycles}
\la{sec2.3}

In this Section we discuss several other examples. Each of them fits in the general scheme 
of Section \ref{sec1.2}. We show how to encode all the data in a polygon. 

\subsubsection{Mixed configurations and the map $\kappa$} 
\la{sec2.3.1}

\bd \la{10.14.12.1a} i) Given a subset ${\rm I} \subset [1,n]$, 
the moduli space ${\rm Conf}_{\rm I}({\cal A}; {\cal B})$ parametrizes
configurations $(x_1, ..., x_n)$, where $x_i\in {\cal A}$ if $i\in {\rm I}$, otherwise 
$x_i\in {\cal B}$.

ii) Given subsets ${\rm J} \subset {\rm I} \subset [1,  n]$, the moduli space ${\rm Conf}_{{\rm J} \subset {\rm I}}({\rm Gr}; {\cal A}, {\cal B})$ 
parametrizes 
configurations  $(x_1, ..., x_n)$ where 
$$
x_i\in {\rm Gr}~~\mbox{ if}~~ i\in {\rm J},~~~~ 
x_i\in {\cal A}({\cal K})~~\mbox{ if}~~ i\in {\rm I} -{\rm J}, ~~~~
x_i\in {\cal B}({\cal K}) ~~\mbox{ otherwise}. 
$$
We set ${\rm Conf}_{\rm I}({\rm Gr}; {\cal B}):= {\rm Conf}_{{\rm I} \subset {\rm I}}({\rm Gr}; {\cal B}).$ 
\ed
A positive structure on the space ${\rm Conf}_{\rm I}({\cal A}; {\cal B})$ is defined 
in Section \ref{proofmth1}. This positive structure 
is invariant under a cyclic twisted shift. See Lemma \ref{13.1.10.1h} for the precise statement. 

\bd \la{aointegral}

Let ${\rm J} \subset {\rm I} \subset [1,  n]$.
A configuration in ${\rm Conf}_{\rm I}({\cal A}; {\cal B})({\cal K})$ is called ${\cal O}$-integral relative to 
${\rm J}$ if   
\begin{enumerate}
\item For all $j\in {\rm J}$ and $k\not=j$, the pairs $(\A_j, \B_k)$ are generic. Here $\B_k=\pi(\A_k)$ if $k\in {\rm I}$.
\item The lattices $\lr_j:=\lr(\A_j, \B_k)$ given by the above pairs only depend on $j$. 
\end{enumerate}
Denote by
${\rm Conf}^{\cal O}_{{\rm J} \subset {\rm I}}({\cal A}; {\cal B})$ the moduli space of such configurations. 
\ed

By the very definition, there is a canonical map
\be \la{5.12.12.2}
\kappa: {\rm Conf}^{\cal O}_{{\rm J} \subset {\rm I}}({\cal A}; {\cal B}) \lra 
 {\rm Conf}_{{\rm J} \subset {\rm I}}({\rm Gr}; {\cal A}, {\cal B}).
\ee
It assigns to  $\A_j$ the lattice $\lr_j$ when $j\in {\rm J}$ and keeps the rest intact . 
%
\vskip 2mm

Recall $u_j\in \U_{\A_j}$ in (\ref{7.20.9.8}). 
The potential ${\cal W}_{\rm J}$ on ${\rm Conf}_{{\rm I}}({\cal A}; {\cal B})$ is a  function
\be \la{13.3.1.527h}
{\cal W}_{\rm J}:=\sum_{j\in {\rm J}}\chi_{\A_j}(u_j).
\ee
Positivity of ${\cal W}_{\rm J}$ is proved in Section \ref{sec6.4}. 

Next Theorem generalizes  Theorem \ref{8.27.17.08hh}. Its proof is the same. See Section \ref{sec6.4}.
\bt \la{13.2.22.2226h} Let $l\in {\rm Conf}_{\rm I}({\cal A}; {\cal B})(\Z^t)$.
A configuration in ${\cal C}_l^\circ$ is ${\cal O}$-integral relative to 
${\rm J}$  if and only if ${\cal W}_{\rm J}^t(l)\geq 0$.
\et

Denote by ${\rm Conf}^+_{{\rm J}\subset {\rm I}}({\cal A}; {\cal B})(\Z^t)$ the set of points $l\in{\rm Conf}_{\rm I}({\cal A};{\cal B})(\Z^t)$ such that ${\cal W}_{\rm J}^t(l)\geq 0$. 
Set 
\be \la{MVGCyyyxs}
{\cal M}^\circ_l := \kappa({\cal C}^\circ_l)\subset {\rm Conf}_{{\rm J}\subset {\rm I}}({\rm Gr};{\cal A}, {\cal B}), ~~~~ 
l\in {\rm Conf}^+_{{\rm J} \subset {\rm I}}({\cal A}; {\cal B})(\Z^t).
\ee
These cycles generalize the Mirkovi\'{c}-Vilonen cycles, as we will see in Section \ref{mdgmc}. 


\subsubsection{Basic invariants} \la{sec2.3.2}

Recall the isomorphism (\ref{conf2A}):
\be \la{conf2A2}
\alpha: {\rm Conf}^*({\cal A}, {\cal A}) \stackrel{\sim}{\lra} {\rm H}, ~~~~\alpha(\A_1\cdot h_1, \A_2\cdot h_2) = h_1^{-1}w_0(h_2)
\alpha(\A_1, \A_2). 
\ee

Given a generic triple $\{\A_1, \B_2, \A_3\}$, we choose a decorated flag $\A_2$ over the flag $\B_2$, and set 
$$
\mu(\A_1, \B_2, \A_3):= \alpha(\A_1, \A_2)\alpha(\A_3, \A_2)^{-1}\in {\rm H}.
$$
Due to (\ref{conf2A2}), it does not depend on the choice of  $\A_2$. 
We illustrate the invariant $\mu$ by a 
pair of red dashed arrows on the left in Fig \ref{tpm14}. 

Given a generic configuration $(\A_1, \B_2, \B_3, \A_4)$, see the right of Fig \ref{tpm14}, 
choose decorated flags $\A_2, \A_3$ over the flags $\B_2, \B_3$, and set
 $$
\mu(\A_1, \B_2, \B_3, \A_4):= \alpha(\A_2, \A_1)\alpha_2(\A_2, \A_3)^{-1}\alpha(\A_4, \A_3)\in {\rm H}.
$$
These invariants  coincide with a similar {\rm H}-valued $\mu$-invariants from Section \ref{sec1.4}. 

\begin{figure}[ht]
\centerline{\epsfbox{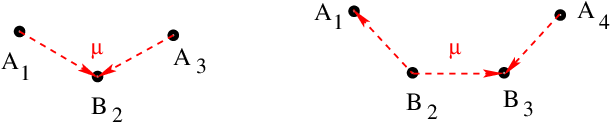}}
\caption{The invariants $\mu(\A_1, \B_2, \A_3)\in {\rm H}$ and $\mu(\A_1, \B_2, \B_3, \A_4)\in {\rm H}$.}
\label{tpm14}
\end{figure}

\vskip 3mm

There are
canonical isomorphisms:
\begin{align}\la{caniso}
&\pi_{\rm Gr}: {\rm Conf}({\rm Gr}, {\rm Gr}) \stackrel{=}{\lra}{\rm P}^+,\nonumber\\
&\alpha_{\rm Gr}: {\rm Conf}({\cal A}, {\rm Gr}) \stackrel{=}{\lra}  {\rm P},\nonumber\\
&\alpha_{\rm Gr}': {\rm Conf}({\rm Gr}, {\cal A})\stackrel{=}{\lra} {\rm P}.
\end{align}
The first map uses the decomposition  ${\rm G}({\cal K}) = {\rm G}({\cal O})\cdot {\rm H}({\cal K})\cdot {\rm G}({\cal O})$: 
$$
{\rm Conf}({\rm Gr}, {\rm Gr}) =  {\rm G}({\cal O})\backslash  {\rm G}({\cal K})/{\rm G}({\cal O}) =  
W\backslash{\rm H}({\cal K})/{\rm H}({\cal O}) = {\rm P}^{+}. 
$$
The second map uses the Iwasawa decomposition  ${\rm G}({\cal K}) = {\rm U}({\cal K})\cdot {\rm H}({\cal K})\cdot {\rm G}({\cal O})$: 
$$
{\rm Conf}({\cal A}, {\rm Gr}) = {\rm G}({\cal K})\backslash \Bigl({\rm G}({\cal K})/{\U}({\cal K}) \times 
{\rm G}({\cal K})/{\rm G}({\cal O})\Bigr) = {\rm U}({\cal K})\backslash  {\rm G}({\cal K})/{\rm G}({\cal O}) = 
{\rm H}({\cal K})/{\rm H}({\cal O}) = {\rm P}. 
$$
The third map is a cousin of the second one:
$$
\alpha_{\rm Gr}'(\lr, \A):=- w_0\big(\alpha_{\rm Gr}(\A, \lr)\big).
$$

{\bf Remark.} These isomorphisms parametrize  ${\rm G}({\cal O})$, $\U({\cal K})$ and  $\U^-({\cal K})$-orbits of ${\rm Gr}$. Each coweight $\lambda \in {\rm P}={\rm H}(\Z^t)={\rm H}({\cal K})/{\rm H}({\cal O})$ corresponds to an element $t^\lambda$ of ${\rm Gr}$. Then
\begin{align} \la{13.3.13.246h}
&\pi_{\rm Gr}([1], g\cdot t^\lambda)=\lambda, ~~~\forall g\in {\rm G}({\cal O}); \nonumber\\
&\alpha_{\rm Gr}(\U, u\cdot t^{\lambda})=\lambda,~~\forall u\in {\U}({\cal K}); \nonumber\\
&\alpha_{\rm Gr}'(v\cdot t^{-\lambda}, \overline{w}_0\cdot \U)=\lambda,~~~\forall v\in \U^{-}({\cal K}).
\end{align}

We define Grassmannian versions of $\mu$-invariants:
$$
\mu_{\rm Gr}: {\rm Conf}({\rm Gr}, {\cal B}, {\rm Gr}) \lra {\rm P}, ~~~~
\mu_{\rm Gr}: {\rm Conf}({\rm Gr}, {\cal B}, {\cal B}, {\rm Gr}) \lra {\rm P}
$$
$$
\mu_{\rm Gr}(\lr_1, \B_2, \lr_3):= \alpha_{\rm Gr}'(\lr_1, \A_2)- \alpha_{\rm Gr}'(\lr_3, \A_2)\in {\rm P}.
$$
$$
\mu_{\rm Gr}(\lr_1, \B_2, \B_3, \lr_4):= \alpha_{\rm Gr}(\A_2, \lr_1) - {\rm val}\circ \alpha(\A_2, \A_3) +\alpha_{\rm Gr}'(\lr_4, \A_3)\in {\rm P}.
$$
Let ${\rm pr}:\B^-({\cal K})\ra {\rm H}({\cal K})\ra {\rm P}$ be the composite of standard projections. The first map  has an equivalent description:
$$\mu_{\rm Gr} ([b_1],\B^{-}, [b_2])={\rm pr}(b_1^{-1}b_2),~~~b_1, b_2 \in \B^{-}({\cal K}).
$$


\vskip 3mm

More generally, take a chain of flags starting and ending by a decorated flag, pick
an alternating sequence of arrows, and write an alternating product of the $\alpha$-invariants.  
We get regular maps 
\be \la{invmu1}
\mu: {\rm Conf}^*({\cal A}, {\cal B}^{2n+1}, {\cal A}) \lra {\rm H}, 
\ee
$$
(\A_1, \B_2,  ..., \B_{2n+2}, \A_{2n+3}) \lms \frac{\alpha(\A_1, \A_2)}{\alpha(\A_3, \A_2)}
\frac{\alpha(\A_3, \A_4)}{\alpha(\A_5, \A_4)} \ldots 
\frac{\alpha(\A_{2n+1}, \A_{2n+2})}{\alpha(\A_{2n+3}, \A_{2n+2})} . 
$$
\be \la{invmu2}
\mu: {\rm Conf}^*({\cal A}, {\cal B}^{2n}, {\cal A}) \lra {\rm H},
\ee
$$
(\A_1, \B_2,  ..., \B_{2n+1}, \A_{2n+2}) \lms \frac{\alpha(\A_2, \A_1)}{\alpha(\A_2, \A_3)}
\frac{\alpha(\A_4, \A_3)}{\alpha(\A_4, \A_5)}\ldots 
\alpha(\A_{2n+2}, \A_{2n+1}). 
$$
Given a cyclic collection of an even number of flags, there is an invariant which 
for $n=2$ and ${\rm G}={\rm SL}_2$ recovers the cross-ratio: 
$$
{\rm Conf}^*_{2n}({\cal B}) \lra {\rm H}, ~~~~(\B_1, ..., \B_{2n}) \lms 
\frac{\alpha(\A_1, \A_2)}{\alpha(\A_3, \A_2)}\frac{\alpha(\A_3, \A_4)}{\alpha(\A_5, \A_4)} \ldots  
\frac{\alpha(\A_{2n-1}, \A_{2n})}{\alpha(\A_1, \A_{2n})}  . 
$$
One gets Grassmannian versions by replacing ${\cal A}$ by ${\rm Gr}$, and 
$\alpha$ by one of the maps (\ref{caniso}).

\vskip 3mm

\begin{figure}[ht]
\centerline{\epsfbox{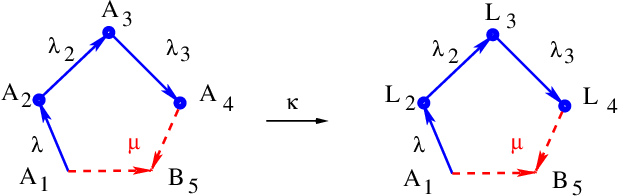}}
\caption{Generalized MV cycles 
${\cal M}_{l}\subset {\rm Gr}^3 = {\rm Conf}_{w_0}({\cal A}, {\rm Gr}^3, {\cal B})$.}
\label{tpm11}
\end{figure}

These invariants provide decompositions for both spaces in \eqref{MVGCyyyxs}.

Let us encode all the data in a polygon, as illustrated on Fig \ref{tpm11}. 
Let $l\in {\rm Conf}^+_{{\rm J} \subset {\rm I}}({\cal A}; {\cal B})(\Z^t)$. 
We show on the left an element of  ${\cal C}_{l}^{\circ}$. 
Flags or decorated flags are assigned to the vertices of a convex polygon. The vertices labeled by ${\rm J}$ are boldface. Note that although we order the vertices by choosing a reference vertex, due to the twisted cyclic invariance the story does not depend on its choice.

The solid blue sides are labeled by a pair of decorated flags. 
There is an invariant $\lambda_E \in {\rm P}$ assigned to such a side $E$. 
It is provided by the tropicalization of the isomorphism (\ref{conf2A2}) evaluated on ${l}$. 
The collection of dashed edges determines an invariant $\mu \in {\rm P}$.

{Recall the cone ${\rm R}^+ \subset {\rm P}$ generated by positive coroots}. The ${\cal O}$-integrality imposes restrictions on basic invariants, summarized in Lemma \ref{restr}, and illustrated on Fig \ref{tpm15}. 
\bl \la{restr}i) Let $(\A_1, \A_2, \B_3) \in {\cal C}_l^{\circ}\subset{\rm Conf}^{\cal O}({\cal A}, {\cal A}, {\cal B})$. Then ${\rm val}\circ \alpha(\A_1, \A_2) \in {\rm P}^+. $ 

ii) Let $(\B_1, \A_2, \B_3) \in {\cal C}_l^{\circ} \subset {\rm Conf}^{\cal O}({\cal B}, {\cal A}, {\cal B})$. Then 
${\rm val}\circ \mu(\A_2, \B_1, \B_3, \A_2) \in {\rm R}^+. $ 
\el
{\begin{proof}
Here i) follows from Lemma \ref{9.21.17.56h}, and ii) follows from Lemmas \ref{8.13.1.20h} \& \ref{12.12.11.h}(4). 
\end{proof}
}

\begin{figure}[ht]\centerline{\epsfbox{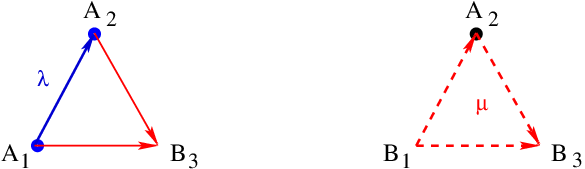}}
\caption{One has $\lambda \in {\rm P}^+$ and $\mu \in {\rm R}^+$.}\label{tpm15} \end{figure}

Applying the map $\kappa$, we replace the  
decorated flag at each boldface vertex by the corresponding lattice. 
Others remain intact. 
We use the notation $\underline{\cal A}$ for the  decorated flags 
which do not contribute the character $\chi_{\A}$ to the potential -- they are 
assigned to the unmarked vertices. 
For example, we associate to the polygons on Fig \ref{tpm11} 
the following maps
$$
\kappa: {\cal C}_{l}^\circ \lra {\rm Conf}({\cal A}, {\rm Gr}^3, {\cal B}), ~~l\in {\rm Conf}^+(\underline{\cal A}, {\cal A}^3, {\cal B})(\Z^t).
$$
$$
\pi: {\rm Conf}^*(\underline{{\cal A}}, {\cal A}^3, {\cal B}) \lra {\rm H}^3, ~~~
\mu: {\rm Conf}^*(\underline{{\cal A}}, {\cal A}^3, {\cal B}) \lra {\rm H},
$$
\be \la{restr1}
(\pi^t, \mu^t): {\rm Conf}^{+}(\underline{{\cal A}}, {\cal A}^3, {\cal B})(\Z^t) \lra {\rm P}\times ({\rm P}^+)^2 \times {\rm P},
\ee
$$
(\pi_{\rm Gr}, \mu_{\rm Gr}): {\rm Conf}(\underline{{\cal A}}, {\rm Gr}^3, {\cal B}) \lra {\rm P}\times ({\rm P}^+)^2\times {\rm P}.
$$
It is easy to check  that the targets of the invariants 
assigned to configurations of flags are the same as the targets 
of their Grassmannian counterparts.

\subsubsection{Generalized Mirkovi\'{c}-Vilonen cycles}
\la{mdgmc}
Let us recall the standard definition of  {\it Mirkovi\'{c}-Vilonen cycles}  following \cite{MV}, \cite{A}, \cite{K}. 

For $w\in W$, let ${\rm U}_{w}=w{\rm U}w^{-1}$. For $w\in W$ and $\mu\in {\rm P}$
define the {\it semi-infinite cells} 
\be
{\rm S}_{w}^{\mu}:={\rm U}_w({\cal K})t^{\mu}.
\ee 
Let $\lambda, \mu \in {\rm P}$. The closure 
 $\overline{{\rm S}_{e}^{\lambda}\cap {\rm S}_{w_0}^{\mu}}$ is non-empty  if and only if  $\lambda-\mu\in {\rm R}^+$. 
In that case, it is also well known that  $\overline{{\rm S}_{e}^{\lambda}\cap {\rm S}_{w_0}^{\mu}}$ has pure dimension ${\rm ht}(\lambda-\mu):=\langle\rho, \lambda-\mu\rangle$.
\bd
A component of $\overline{{\rm S}_{e}^{\lambda}\cap {\rm S}_{w_0}^{\mu}} \subset {\rm Gr}$ 
is called an {\it MV cycle} of coweight $(\lambda,\mu)$. 
\ed
Since ${\rm H}$ normalizes ${\rm U}_w$, for each $h\in {\rm H}({\cal K})$ such that $[h]=t^{\nu}$, we have $h \cdot {\rm S}_{w}^{\mu}={\rm S}_{w}^{\mu+\nu}$. Therefore if $V$ is an MV cycle of coweight $(\lambda,\mu)$, then $h\cdot V$ is an MV cycle of coweight $(\lambda+\nu,\mu+\nu)$. The ${\rm H}({\cal K})$-orbit of  an MV cycle of coweight $(\lambda, \mu)$ is called a {\it stable MV cycle} of coweight $\lambda-\mu$. 

\vskip 2mm

Let $\underline{\lambda}=(\lambda_1,\ldots, \lambda_n)\in ({\rm P}^+)^n$. Consider the convolution variety 
\be \la{con.var.247}
{\rm Gr}_{\underline{\lambda}}=\{(\lr_1, \lr_2, \ldots, \lr_n)~|~[1]\stackrel{\lambda_1}{\lra}\lr_1\stackrel{\lambda_2}{\lra}\ldots\stackrel{\lambda_n}{\lra}\lr_n\}\subset {\rm Gr}^n.
\ee 
Let ${\rm pr}_n: {\rm Gr}^n\ra {\rm Gr}$ be the projection onto the last factor. Set
\be \la{13.3.13.407h}
{\rm Gr}_{\underline{\lambda}}^{\mu}:={\rm Gr}_{\underline{\lambda}}\cap {\rm pr}_n^{-1}\big({\rm S}_{w_0}^\mu\big).
\ee 
When n=1, under the geometric Satake correspondence, the components of ${\rm Gr}_{\lambda}^\mu$ give a basis (the MV basis) for the weight space $V_{\lambda}^{(\mu)}$, see \cite[Corollary 7.4]{MV}. It is easy to see that they are precisely  MV cycles of coweight $(\lambda,\mu)$ contained in $\overline{{\rm Gr}_{\lambda}}$, see \cite[Proposition 3]{A}.

\vskip 3mm

Now we restrict constructions in preceding subsections to four main examples associated to an (n+2)-gon.  
The $n=1$ case recovers the above three versions of {\rm MV} cycles. In this sense,  the following can be viewed as a generalization of ${\rm MV}$ cycles.

\paragraph{Example 1: ${\rm J}=[2,n+1]\subset{\rm I}=[1,n+1]$.}
Let ${\rm Conf}_{w_0}({\cal A}, {\rm Gr}^n, {\cal B})\subset {\rm Conf}_{{\rm J}\subset {\rm I}}({\rm Gr};{\cal A}, {\cal B})$ be the substack parametrizing configurations $(\A_1,\lr_2,\ldots, \lr_{n+1}, \B_{n+2})$ where $(\A_1, \B_{n+2})$ is generic. 

Recall ${\cal F}_{\G}$ in Definition \ref{torsorF}. Then
$$
{\rm Conf}_{w_0}({\cal A}, {\rm Gr}^n,  {\cal B})={\rm G}({\cal K})\backslash\big({\cal F}_{\rm G}({\cal K})\times {\rm Gr}^n\big).
$$
Since ${\cal F}_{\rm G}$ is a ${\rm G}$-torsor, we get an isomorphism
\be \la{10.9.12.2a}
i: {\rm Gr}^n \stackrel{=}{\lra}{\rm Conf}_{w_0}({\cal A}, {\rm Gr}^n, {\cal B}),~~~(\lr_1, \ldots, \lr_{n})\lms ({\rm U}, \lr_1,\ldots,\lr_{n}, \B^-). 
\ee
From now on we identify ${\rm Gr}^n$ with ${\rm Conf}_{w_0}({\cal A}, {\rm Gr}^n, {\cal B})$.

There is a map, whose construction is illustrated on the right of Fig \ref{tpm11}:
$$
\pi_{\rm Gr}: {\rm Conf}_{w_0}({\cal A}, {\rm Gr}^{n}, {\cal B}) \lra {\Bbb P}:= {\rm P}\times ({\rm P}^+)^{n-1}\times {\rm P}.
$$
Its fibers are finite dimensional subvarieties ${\rm Gr}_{\lambda; \underline{\lambda}}^\mu$:
\be \la{MVlm}
{\rm Gr}^n = \coprod {\rm Gr}_{\lambda; \underline{\lambda}}^\mu, ~~~~\mbox{where} ~~
(\lambda, \underline{\lambda},\mu)\in {\rm P}\times ({\rm P}^+)^{n-1}\times {\rm P}.
\ee
By \eqref{13.3.13.246h} we see that 
$$
{\rm Gr}_{\lambda;\underline{\lambda}}^{\mu}=\{(\lr_1,\ldots, \lr_n)\in {\rm Gr}^n~|~ \lr_1\stackrel{\lambda_2}{\lra}\ldots\stackrel{\lambda_n}{\lra}\lr_n,~\lr_1\in {\rm S}_e^{\lambda},~ \lr_n\in {\rm S}_{w_0}^{\mu}\},~~~~~\underline{\lambda}:=(\lambda_2,\ldots, \lambda_n).
$$
When $n=1$, it is the intersection ${\rm S}_{e}^{\lambda}\cap {\rm S}_{w_0}^{\mu}$. 
Note that the very notion of MV cycles depends on the choice of the pair ${\rm H}\subset \B$.
We transport the MV cycles to ${\rm Conf}_{w_0}({\cal A},{\rm Gr},{\cal B})$ by the isomorphism \eqref{10.9.12.2a}. It is then independent of the pair chosen.  In general we define
\bd \la{MVlm1ayxsh}
The irreducible components of $\overline{{\rm Gr}_{\lambda; \underline{\lambda}}^\mu}$ are called the {\it generalized  Mirkovi\'{c}-Vilonen cycles} of coweight $(\lambda, \underline{\lambda}, \mu)$. 
\ed

Similarly the left of Fig \ref{tpm11} provides a map
\be \la{13.2.22.852h}
\pi^t: {\rm Conf}^+(\underline{\cal A}, {\cal A}^{n}, {\cal B})(\Z^t) \lra  {\rm P}\times ({\rm P}^+)^{n-1}\times {\rm P}.
\ee
Let ${\bf P}_{\lambda; \underline{\lambda}}^{\mu}:={\rm Conf}^+(\underline{\cal A}, {\cal A}^{n}, {\cal B})(\Z^t)_{\lambda; \underline{\lambda}}^\mu$ be the fiber of map \eqref{13.2.22.852h} over $(\lambda, \underline{\lambda}, \mu)$. Then
\be \la{13.2.22.902h}
{\rm Conf}^+(\underline{\cal A}, {\cal A}^{n}, {\cal B})(\Z^t)= \coprod {\bf P}_{\lambda; \underline{\lambda}}^{\mu}~~~~\mbox{where} ~~
(\lambda, \underline{\lambda},\mu)\in {\rm P}\times ({\rm P}^+)^{n-1}\times {\rm P}.
\ee
By definition $\pi^t\circ \val$ and  $\pi_{\rm Gr}\circ \kappa$ deliver the same map from ${\cal C}_{l}^\circ$ to ${\Bbb P}$. 
Thus we arrive at
\be \la{MVlm1}
{\cal M}_l:=\overline{{\cal M}^\circ_l} \subset \overline{{\rm Gr}_{\lambda; \underline{\lambda}}^\mu}, ~~~~
l\in {\bf P}_{\lambda; \underline{\lambda}}^{\mu}:={\rm Conf}^+_{w_0}({\cal A}, {\cal A}^{n}, {\cal B})(\Z^t)_{\lambda; \underline{\lambda}}^\mu.
\ee
\bt \la{MVlm1a}
The cycles (\ref{MVlm1}) are precisely the generalized MV cycles of coweight $(\lambda, \underline{\lambda},\mu)$. 
\et

\paragraph{Example 2: ${\rm J}={\rm I}=[2,n+1]$.} 
Let ${\rm Conf}_{w_0}({\cal B}, {\rm Gr}^n, {\cal B})\subset {\rm Conf}_{{\rm J}\subset {\rm I}}({\rm Gr};{\cal A}, {\cal B})$ be the substack parametrizing configurations $(\B_1,\lr_2,\ldots, \lr_{n+1}, \B_{n+2})$ where $(\B_1, \B_{n+2})$ is generic. 

 Similarly, we get an isomorphism of stacks
\be \la{10.9.12.2}
i_s: {\rm H}({\cal K})\backslash{\rm Gr}^n\stackrel{=}{\lra}{\rm Conf}_{w_0}({\cal B},{\rm Gr}^n, {\cal B}),~~~(\lr_1,\ldots, \lr_n)\lms (\B, \lr_1,\ldots, \lr_n, \B^-).
\ee
Here the group ${\rm H}({\cal K})$ acts diagonally  on ${\rm Gr}^n$. Let $h\in {\rm H}({\cal K})$. If $[h]=t^{\mu}$, then $h\cdot \overline {{\rm Gr}_{\lambda;\underline{\lambda}}^\nu}=\overline {{\rm Gr}_{\lambda+\mu;\underline{\lambda}}^{\nu+\mu}}$. It provides an isomorphism between the sets of components of both varieties.
\bd The ${\rm H}({\cal K})$-orbit of a generalized $MV$ cycle of coweight $(\lambda,\underline{\lambda}, \nu)$ is called a {\it generalized stable MV cycle} of coweight $(\underline{\lambda},\lambda-\nu)$. 
\ed
When $n=1$, it recovers the usual stable MV cycles.
The generalized stable MV cycles live naturally on the stack ${\rm H}({\cal K})\backslash {\rm Gr}^n$.
The isomorphism \eqref{10.9.12.2}
transports them to 
${\rm Conf}_{w_0}({\cal B},{\rm Gr}^n, {\cal B})$.

\vskip 2mm

The solid blue arrows and the triple of dashed reds on Fig \ref{tpm13a} provide  a canonical projection
$$(\pi^t,\mu^t): {\rm Conf}^+({\cal B}, {\cal A}^{n}, {\cal B})(\Z^t) \lra ({\rm P}^+)^{n-1}\times {\rm P}.$$
Let ${\bf A}_{\underline{\lambda}}^\mu:= {\rm Conf}^+({\cal B}, {\cal A}^{n}, {\cal B})(\Z^t)_{\underline{\lambda}}^\mu$ be its fiber over $(\underline{\lambda}, \mu)$. Then 
\be  
{\rm Conf}^+(\underline{\cal B}, {\cal A}^{n}, {\cal B})(\Z^t)= \coprod {\bf A}_{\underline{\lambda}}^{\mu}~~~~\mbox{where} ~~
\underline{\lambda}\in ({\rm P}^+)^{n-1}, ~~\mu \in {\rm P}.
\ee
On the other hand, our general construction provides us with the irreducible cycles
\be \la{MVlm1s}
{\cal M}_l:=\overline{{\cal M}_l^\circ} \subset {\rm H}({\cal K})\backslash{\rm Gr}^n= {\rm Conf}_{w_0}({\cal B}, {\rm Gr}^{n}, {\cal B}), ~~~~
l\in {\bf A}_{\underline{\lambda}}^\mu.
\ee

\bt \la{MVlm1sa}
The cycles (\ref{MVlm1s}) are precisely the generalized stable MV cycles of coweight $(\underline{\lambda},\mu)$. 
\et

\begin{figure}[ht]
\centerline{\epsfbox{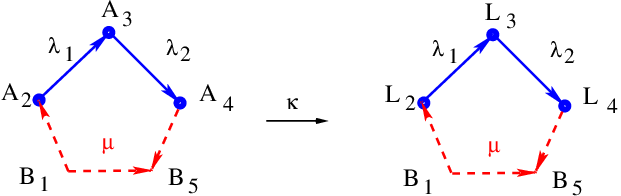}}
\caption{Generalized stable MV cycles 
${\cal M}_{l}\subset {\rm Conf}({\cal B}, {\rm Gr}^3, {\cal B}) = {\rm H}({\cal K})\backslash {\rm Gr}^3$.}
\label{tpm13a}
\end{figure}

\paragraph{Example 3:  ${\rm J}={\rm I}=[1,n+1]$.} By Iwasawa decomposition 
we get an isomorphism
\be \la{stackk1}
i_b: \B^-({\cal O})\backslash {\rm Gr}^n \stackrel{=}{\lra} {\rm Conf}({\rm Gr}^{n+1},{\cal B}),~~~(\lr_1,\ldots, \lr_n)\lms ([1], \lr_1,\ldots, \lr_n, \B^{-}).
\ee

There are two projections, illustrated on Fig \ref{tpm12}:
\be \la{lm}
(\pi_{\rm Gr}, \mu_{\rm Gr}): {\rm Conf}({\rm Gr}^{n+1}, {\cal B}) \lra ({\rm P}^+)^n \times {\rm P}, 
\ee
\be \la{mpp}
(\pi^t, \mu^t): {\rm Conf}^{+}({\cal A}^{n+1}, {\cal B})(\Z^t) \lra ({\rm P}^+)^n \times {\rm P}.
\ee
Their fibers over $(\underline \lambda, \mu)\in ({\rm P}^+)^n\times {\rm P}$ provide decompositions
\be \la{decompo}
{\rm Conf}({\rm Gr}^{n+1}, {\cal B}) = 
\coprod_{\underline \lambda, \mu }{\rm Conf}({\rm Gr}^{n+1}, {\cal B})^\mu_{\underline \lambda}.
\ee
\be
{\rm Conf}^+({\cal A}^{n+1}, {\cal B})(\Z^t) = 
\coprod_{\underline \lambda, \mu }{\rm Conf}^+({\cal A}^{n+1}, {\cal B})(\Z^t)^\mu_{\underline \lambda}.
\ee
By definition, these decompositions are compatible under the map $\kappa$.
We get irreducible cycles
\be\la{MVn+1bb}
{\cal M}_l:=\overline{{\cal M}^\circ_l} \subset  \B^-({\cal O})\backslash {\rm Gr}^n={\rm Conf}({\rm Gr}^{n+1}, {\cal B}), ~~~~
l \in {\bf B}^\mu_{\underline \lambda}:={\rm Conf}^+({\cal A}^{n+1}, {\cal B})(\Z^t)^\mu_{\underline \lambda}.
\ee
\begin{figure}[ht]
\centerline{\epsfbox{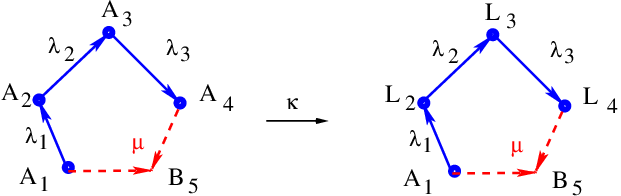}}
\caption{Generalized MV cycles 
${\cal M}_{l}\subset {\rm Conf}({\rm Gr}^4, {\cal B}) = {\rm B}^-({\cal O})\backslash {\rm Gr}^3$.}
\label{tpm12}
\end{figure}

The connected group ${\rm B}^-({\cal O})$ acts diagonally ${\rm Gr}^n$. It preserves components of  subvarieties $\overline{{\rm Gr}_{\underline{\lambda}}^\mu}$  in \eqref{13.3.13.407h}. 
Hence these components live naturally on the stack $\B^-({\cal O})\backslash {\rm Gr}^n$. We transport them to ${\rm Conf}({\rm Gr}^{n+1},{\cal B})$ by  \eqref{stackk1}. 
\bt \la{MVn+1bbb}
The cycles  (\ref{MVn+1bb}) are precisely the components of  $\B^{-}({\cal O})\backslash\overline{{\rm Gr}_{\underline{\lambda}}^\mu}$. 
\et

\paragraph{Example 4: ${\rm J}={\rm I}=[1,n+2]$.} 
There is an isomorphism
\be
i_g: {\rm G}({\cal O})\backslash {\rm Gr}^{n+1}\stackrel{=}{\lra}{\rm Conf}_{n+2}({\rm Gr}),~~~(\lr_1,\ldots, \lr_{n+1})\lra ([1],\lr_1,\ldots, \lr_{n+1}).
\ee
 We arrive at irreducible cycles defined in Definition \ref{9.22.12.13h}:
$$
{\cal M}_l:=\overline{{\cal M}^\circ_l} \subset  {\rm G}({\cal O})\backslash {\rm Gr}^{n+1}= {\rm Conf}_{n+2}({\rm Gr}) ~~~~
l \in {\bf C}_{\underline \lambda}:= {\rm Conf}^{+}_n({\cal A})(\Z^t)_{\underline \lambda}.
$$ 
This example recovers Theorem \ref{5.8.10.45b}.

\vskip 2mm

\begin{figure}[ht]
\centerline{\epsfbox{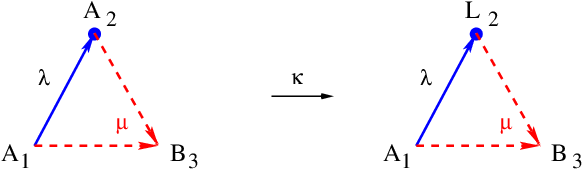}}
\caption{Mirkovi\'{c}-Vilonen cycles ${\cal M}_{l}\subset {\rm Conf}_{w_0}({\cal A}, {\rm Gr}, {\cal B}) =  {\rm Gr}$.}
\label{tpm7}
\end{figure}

\begin{figure}[ht]
\centerline{\epsfbox{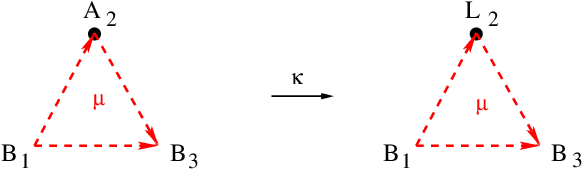}}
\caption{Stable Mirkovi\'{c}-Vilonen cycles 
${\cal M}_{l}\subset {\rm Conf}_{w_0}({\cal B}, {\rm Gr}, {\cal B}) = {\rm H}({\cal K})\backslash {\rm Gr}$.}
\label{tpm8}
\end{figure}

\begin{figure}[ht]
\centerline{\epsfbox{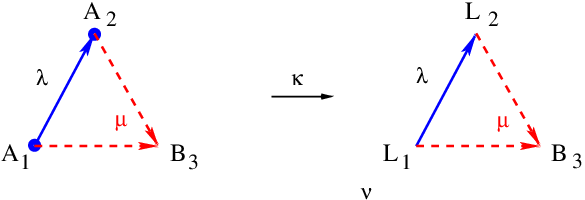}}
\caption{MV cycles which lie in ${\rm Gr}_\lambda$ are the cycles 
${\cal M}^\circ_{l}\subset {\rm Conf}({\rm Gr}, {\rm Gr}, {\cal B})_\lambda = {\rm B}^-({\cal O})\backslash {\rm Gr}_\lambda$.}
\label{tpm9}
\end{figure}

\begin{figure}[ht]
\centerline{\epsfbox{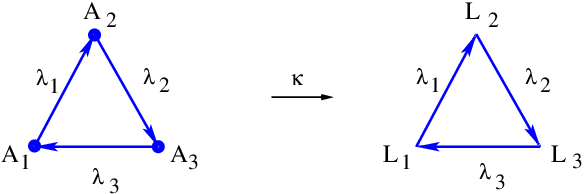}}
\caption{Generalized MV cycles 
${\cal M}_{l}\subset {\rm Conf}({\rm Gr}, {\rm Gr}, {\rm Gr})$.}
\label{tpm10}
\end{figure}

Specializing Theorems \ref{MVlm1a}-\ref{MVn+1bbb} to $n=1$, we get
\bt \la{kth}
1) Mirkovi\'{c}-Vilonen cycles of coweight $(\lambda, \mu)$ are precisely the cycles
$$
{\cal M}_l \subset {\rm Gr}, ~~~~l\in 
{\bf P}_{\lambda}^{\mu}:={\rm Conf}^+(\underline {\cal A}, {\cal A},  {\cal B})(\Z^t)^\mu_\lambda~~\mbox{ for}~~~ {\cal W} = \chi_{\A_2}.$$

2) Stable Mirkovi\'{c}-Vilonen cycles of coweight $\mu$ are precisely the cycles 
$$
{\cal M}_l \subset {\rm H}({\cal K})\backslash {\rm Gr}, ~~~~l\in {\bf A}_{\mu}:=
{\rm Conf}^+({\cal B}, {\cal A},  {\cal B})(\Z^t)^\mu~~\mbox{ for}~~~ {\cal W} = \chi_{\A_2}
.$$

3) Mirkovi\'{c}-Vilonen cycles of coweight $(\lambda, \mu)$ which lie in 
$\overline {\rm Gr}_\lambda \subset {\rm Gr}$
are precisely the cycles 
$$
{\cal M}_l\subset {\rm B}^-({\cal O})\backslash {\rm Gr}, ~~~~l\in {\bf B}_{\lambda}^{\mu}:=
{\rm Conf}^+({\cal A}, {\cal A},  {\cal B})(\Z^t)_\lambda^\mu~~\mbox{ for}~~~ {\cal W} = \chi_{\A_1}+\chi_{\A_2}
$$ 
\et
Theorem \ref{kth} is proved in Section \ref{sec12.1.1}.

Note that there is a positive birational isomorphism ${\rm Conf}({\cal B}, {\cal A},   {\cal B}) \stackrel{\sim}{=} {\rm U}$. Thus we identify ${\rm Conf}^+({\cal B}, {\cal A},   {\cal B})(\Z^t)$ with the subset of 
${\rm U}(\Z^t)$ used by Lusztig \cite{L}, \cite{L1} to parametrize the canonical basis in Lemma \ref{8.20.3.42h}. 
Then Theorem \ref{kth} is equivalent to the 
main results of Kamnitzer's paper \cite{K}. Our approach, using the moduli 
space ${\rm Conf}({\cal B}, {\cal A},   {\cal B})$ rather than ${\rm U}$, 
makes  parametrization of the MV cycles more natural and transparent, and 
puts it into the general framework of this paper.  

\vskip 3mm

To summarize, there are four different versions of the  cycles relevant to representation theory 
related to mixed configurations of triples, as illustrate on Fig \ref{tpm7}-\ref{tpm10}. 


\subsubsection{Constructible equations for the cycles ${\cal M}_l^\circ$} 
\la{sec2.3.2-} 
 Let $F$ be a rational function on the stack ${\rm Conf}_{\rm I}({\cal A};{\cal B})$. We generalize the construction of $D_F$ from Section \ref{sec2.2.5}. As an application, it implies that the cycles {${\cal M}_l^{\circ}$ in \eqref{MVGCyyyxs} are disjoint. }
 
Given ${\rm J}\subset {\rm I}\subset [1,n]$, let $m$ be the cardinality of ${\rm J}$. We assume ${\rm J}=\{j_1,\ldots, j_m\}$.

Consider the space 
$$
\mathfrak{X}:=X_1\times \ldots \times X_{n},~~~\mbox{where }
X_i =\left\{\begin{array}{ll}{\rm G} ~~~&\text{ if }i\in {\rm J}, \\ {\cal A} &\text{ if } i\in {{\rm I}-{\rm J}},\\ {\cal B} &\text{ otherwise}.\end{array}\right.
$$ 
Let 
$\mathfrak{X}_*$ be its subset consisting of collections $\{x_1,\ldots, x_{n}\}$ 
 whose subcollections  $\{x_{i_1},\ldots, x_{i_{n-m}}\}$, $i_s\not \in {\rm J}$,  are generic.

Given a rational function  $F$ on ${\rm Conf}_{\rm I}({\cal A};{\cal B})$,
each $x=\{x_1,\ldots, x_{n}\}\in \mathfrak{X}_*({\cal K})$ provides a  function $F_x$ on ${\cal A}^m$, whose value  on  $\{\A_{j_1},\ldots, \A_{j_m}\} \in {\cal A}^m$ is  
\be \la{13.1.30.8.43h}
F_x(\A_{j_1},\ldots, \A_{j_m}):=F(x_1',\ldots, x_{n}')\in {\cal K},~~~x_i'=\left\{\begin{array}{cl} x_j\cdot \A_{j} &\text{ if } 
j\in {\rm J},\\ x_i &\text{ otherwise}. \end{array}\right.
\ee
Then $F_x\in {\cal K}({\cal A}^m)$

Recall the map {$\val: {\cal K}({\cal A}^m)^\times \to \Z$}. We get a $\Z$-valued function
\be \la{13.1.26.10.21hhh}
D_F: \mathfrak{X}_*({\cal K}) \lra \Z,  ~~~~D_F(x):=\val(F_x).
\ee

{Recall the right action of ${\rm G}^m$ on $\C({\cal A}^m)$}. Thanks to Lemma \ref{13.1.26.11.18h} and the fact  that $F\in \Q({\rm Conf}_{\rm I}({\cal A};{\cal B}))$, we have
\be \la{13.2.22.1854h}
\forall g\in {\rm G}({\cal K}),~\forall h\in {\rm G}({\cal O})^m,~~~\val(F_{g\cdot x}\circ h)=\val(F_x).
\ee
Thus $D_F$ descends to 
\be
D_F: {\rm Conf}_{{\rm J}\subset {\rm I}}^*({\rm Gr}; {\cal A},{\cal B})\lra \Z.
\ee
Here ${\rm Conf}_{{\rm J}\subset{\rm I}}^*({\rm  Gr};{\cal A}, {\cal B})$ is a subspace of ${\rm Conf}_{{\rm J}\subset{\rm I}}({\rm  Gr};{\cal A}, {\cal B})$ consisting of the configurations whose subconfigurations of flags and decorated flags are generic. 

By definition,  ${\cal M}_{l}^{\circ}$ in {\eqref{MVGCyyyxs}} are contained in ${\rm Conf}_{{\rm J}\subset {\rm I}}^*({\rm Gr}; {\cal A},{\cal B})$. The following Theorem is a generalization of Theorem \ref{5.8.10.45a}. See Section \ref{sec7} for its proof.
\bt \la{13.1.30.742h}
Let $l\in {\rm Conf}_{{\rm J}\subset{\rm I}}^+({\cal A}; {\cal B})(\Z^t)$. Let $F\in \Q_+({\rm Conf}_{\rm I}({\cal A};{\cal B}))$. Then  $D_{F}({\cal M}_l^{\circ})\equiv F^t(l)$.
\et
 

\subsection{Canonical bases in tensor products and ${\rm Conf}({\cal A}^{n}, {\cal B})$} \la{tensor}

Recall that a collection of dominant coweights $\underline {\lambda}= (\lambda_1, ..., \lambda_n)$ gives rise to a convolution variety ${\rm Gr}_{\underline {\lambda}}\subset {\rm Gr}^n$. 
It  is open and smooth. Its dimension is calculated inductively:
\be \la{dimgr}
{\rm dim}~{\rm Gr}_{\underline {\lambda}} = 2{\rm ht}(\underline {\lambda}):= 2\langle \rho, \lambda_1 + \ldots + \lambda_n\rangle.
\ee
The subvarieties ${\rm Gr}_{\underline {\lambda}}$ form a stratification ${\cal S}$ of ${\rm Gr}^n$. 
Let ${\rm IC}_{\underline {\lambda}}$ be the ${\rm IC}$-sheaf of 
$\overline {{\rm Gr}_{\underline {\lambda}}}$.  
By the geometric Satake correspondence, 
\be \la{gSC}
{\rm H}^*({\rm IC}_{\underline {\lambda}}) = V_{\underline{\lambda}}:= V_{\lambda_1} \otimes \ldots \otimes V_{\lambda_n}. 
\ee

Let ${\rm pr}_n:{\rm Gr}^n \to {\rm Gr}$ be the projection onto the last factor. 
Recall the point $t^\mu \in{\rm Gr}$. Set 
$$
{\rm S}_\mu:= {\rm pr}_n^{-1}({\rm U}({\cal K})t^\mu) \subset {\rm Gr}^n, ~~~~
 {\rm T}_\mu:= {\rm pr}_n^{-1}({\rm U}^-({\cal K})t^\mu) \subset {\rm Gr}^n.
$$
The sum of positive coroots is a cocharacter $2\rho^{\vee}: {\Bbb G}_m \to {\rm H}$. 
It provides an action of the  group ${\Bbb G}_m$ on ${\rm Gr}^n$ 
given by the action on the last factor. The subvarieties  ${\rm S}_\mu$ 
and ${\rm T}_\mu$ are attracting and repulsing subvarieties for this action. Set
$$
{\rm Gr}^\mu_{\underline {\lambda}}:= {\rm Gr}_{\underline {\lambda}} \cap {\rm S}_\mu.
$$

\bl
If ${\rm Gr}^\mu_{\underline {\lambda}}$ 
is non-empty, then it is a subvariety of pure dimension
\be \la{dimgrs}
{\rm dim }~{\rm Gr}^\mu_{\underline {\lambda}} = {\rm ht}(\underline {\lambda};\mu):= 
\langle \rho, \lambda_1 + \ldots + \lambda_n +\mu\rangle.
\ee
\el

Denote by ${\rm Irr}(X)$ the set of top dimensional components of a variety $X$, and by  $\Q[{\rm Irr}(X)]$ 
the vector space with the bases parametrised by the set ${\rm Irr}(X)$. 

\bt \la{mmvvth} There are canonical isomorphisms 
$$
{\rm H}^*({\rm Gr}^n, {\rm IC}_{\underline {\lambda}}) =\oplus_{\mu}{\rm H}^{2{\rm ht}(\mu)}_c({\rm S}_\mu, {\rm IC}_{\underline {\lambda}}) = \oplus_{\mu}\Q[{\rm Irr}(\overline {{\rm Gr}^\mu_{\underline {\lambda}}})].
$$
\et

\begin{proof} Theorem \ref{mmvvth} for $n=1$ is proved in \cite[Section 3]{MV}. 
The proof for arbitrary $n$  follows the same line. For convenience of the reader we provide 
a complete proof.

Let $m: \C^* \times X \to X$ be a map defining an action of the group $\C^*$ on $X$. 
Let ${\cal D}(X)$ be the bounded derived category of constructible sheaves on $X$. 
 An object ${\cal F} \in {\cal D}(X)$ is {\it weakly $\C^*$-equivariant}, 
if $m^*{\cal F} = L\boxtimes {\cal F}$ for some locally constant sheaf $L$ on $\C^*$.

Recall the action of ${\rm G}_m$ on ${\rm Gr}^n$ 
defined above. Denote by 
${\rm P}_{\cal S}({\rm Gr}^n)$ the category of weakly $\C^*$-equivariant 
perverse sheaves on ${\rm Gr}^n$ 
which are constructible with respect to the stratification ${\cal S}$. 

\bl
The sheaf ${\rm IC}_{\underline{\lambda}}$ 
is locally constant along the stratification ${\cal S}$. It belongs to the category 
${\rm P}_{\cal S}({\rm Gr}^n)$.
\el 

\begin{proof}
 Given a subgroup ${\rm G}'\subset {\rm G}$, 
denote by ${\rm G}'_{[k,n]}\subset {\rm G}^n$ the subgroup  of 
elements $(e, ..., e, g, ..., g)$, with $(n-k+1)$ of $g \in {\rm G}'$.   
Denote by ${\rm G}(L)$ the subgroup stabilising a point $L\in {\rm Gr}$. 
The group  ${\rm G}(L)_{[k,n]}$ preserves the category 
${\rm P}_{\cal S}({\rm Gr}^n)$. 
Take two collections $(L_1, ..., L_n), (M_1, ..., M_n)\in {\rm Gr}^n$, with $L_1=M_1=[1]$ and
  in the same stratum. 
We can move $(L_1, ..., L_n)$ by an element of 
${\rm G}(L_1)_{[1,n]}$, getting $(M_1, M_2, L_3', ..., L'_n)$. Then we move it by an element  
of ${\rm G}(M_2)_{[2,n]}$, getting $(M_1, M_2, M_3, ..., L''_n)$, and so on, using 
subgroups ${\rm G}(L_n)_{[k,n]}$ for $k=3, 4, ...n-1$. 
In the last step we get $(M_1, ..., M_n)$. The $\C^*$-equivariance is evident.
\end{proof}

\bp
For all ${\cal P} \in {\rm P}_{\cal S}({\rm Gr}^n)$ 
we have a canonical isomorphism
\be \la{hyploc}
{\rm H}^k_c({\rm S}_\mu, {\cal P}) \stackrel{\sim}{\lra} {\rm H}^k_{{\rm T}_\mu}({\rm Gr}^n, {\cal P}). 
\ee
Both  sides vanish if $k \not = 2{\rm ht}(\mu)$. 
The functors $F_\mu:= {\rm H}_c^{2{\rm ht}(\mu)}(S_\mu, -) : {\rm P}_{\cal S}({\rm Gr}^n) \lra {\rm Vect}$ are exact.
\ep

\begin{proof} 
Isomorphism (\ref{hyploc}) follows from the hyperbolic localisation theorem of Braden 
\cite{Br}. Let us briefly recall 
how it works. 

Let $X$ be a  normal complex variety on which the group $\C^*$ acts. 
Let $F$ be the stable points variety. It is a union of components $F_1, ..., F_k$. 
Consider the attracting and repulsing subvarieties
$$
X_k^+ = \{x\in X~|~ {\rm lim}_{t\to 0}t\cdot x \in F_k\}, ~~~~
X_k^- = \{x\in X~|~ {\rm lim}_{t\to \infty}t\cdot x \in F_k\}, 
$$
Let $X^+$ (resp. $X^-$) be the disjoint union of all the $X_k^+$ (resp. $X_k^-$). 
There are projections
$$
\pi^{\pm}: X^{\pm} \to F, ~~~~ \pi^+(x)={\rm lim}_{t\to 0}t\cdot x , ~~~\pi^-(x)={\rm lim}_{t\to \infty}t\cdot x.  
$$
Let $g^{\pm}: X^{\pm} \hra X$ be the natural inclusions. 
Given an object ${\cal F} \in {\cal D}(X)$, define hyperbolic localisation functors
$$
{\cal F}^{!*}:= (\pi^+)_!(g^+)^*{\cal F}, ~~~~{\cal F}^{*!}:= (\pi^-)_*(g^-)^!{\cal F}.
$$
Combining Theorem 1  and Section 3 of \cite{Br}, we have the following result, which implies 
(\ref{hyploc}).
\bp 
If ${\cal F}$ is weakly $\C^*$-equivariant,  the natural map 
${\cal F}^{!*}\to {\cal F}^{*!}$ is an isomorphism. 
\ep

Let us prove the vanishing. 
One has  ${\rm H}^k_c({\rm Gr}^\mu_{\underline {\lambda}}, \Q) =0$ for 
$k> 2{\rm dim}{\rm Gr}^\mu_{\underline {\lambda}} = 2{\rm ht}(\underline {\lambda}; \mu)$.
 Due to perversity, the restriction of any $ {\cal P} \in {\rm P}_{\cal S}({\rm Gr}^n)$ 
to ${\rm Gr}_{\underline {\lambda}}$ 
lies in degrees $\leq - {\rm dim}{\rm Gr}_{\underline {\lambda}} = -2{\rm ht}(\underline {\lambda})$. 
So 
\be \la{esti}
{\rm H}^k_c({\rm Gr}^\mu_{\underline {\lambda}}, {\cal P}) =0 ~~~\mbox{if $k>2{\rm ht}(\mu)$}. 
\ee
Although ${\rm S}_\mu$ is infinite dimensional, we can slice it by its intersections with 
the strata ${\rm Gr}_{\underline {\lambda}}$. Since the estimate (\ref{esti}) 
on each strata does not depend on $\underline {\lambda}$, a devissage 
using exact triangles $j_!j^* {\cal A} \to{\cal A} \to i_!i^* {\cal A}$ tells that 
$$
{\rm H}^k_c({\rm S}_\mu, {\cal P}) =0 ~~~\mbox{if $k>2{\rm ht}(\mu)$}. 
$$
Applying the duality, and using the fact that $ \ast{\cal P} = {\cal P}$, we get the dual estimate
$$
{\rm H}^k_{{\rm T}_{\mu}}({\rm Gr}^n, {\cal P}) =0 ~~~\mbox{if $k<2{\rm ht}(\mu)$}. 
$$
Combining with the isomorphism (\ref{hyploc}), we get the proof. The last claim is then obvious. 
\end{proof}

\bp \la{mainprop}
We have  natural equivalence of functors 
$$
{\rm H}^* \stackrel{\sim}{=} 
\oplus_{\mu \in {\rm P}}{\rm H}_c^{2{\rm ht}(\mu)}(S_\mu, -) : {\rm P}_{\cal S}({\rm Gr}^n) \lra {\rm Vect}.
$$
\ep

\begin{proof} The proof of Theorem 3.6 in \cite{MV} works in our case. 
Namely, the two filtrations of ${\rm Gr}^n$ by the closures of ${\rm S}_\mu$ and ${\rm T}_\mu$ 
give rise to two filtrations of ${\rm H}^*$, 
given by the kernels of ${\rm H}^* \to {\rm H}^*_c(\overline {\rm S}_\mu, -)$ and 
the images of ${\rm H}^*_{\overline {\rm T}_\mu}({\rm Gr}^n, -) \to {\rm H}^*$. The vanishing 
implies ${\rm H}^{2{\rm ht}(\mu)}_c(\overline {\rm S}_\mu, -) = {\rm H}^{2{\rm ht}(\mu)}_c({\rm S}_\mu, -)$ and 
 ${\rm H}^{2{\rm ht}(\mu)}_{\overline {\rm T}_\mu}({\rm Gr}^n, -) = {\rm H}^{2{\rm ht}(\mu)}_{{\rm T}_\mu}({\rm Gr}^n, -)$, 
and the composition ${\rm H}^{2{\rm ht}(\mu)}_{{\rm T}_\mu}({\rm Gr}^n, -) \to {\rm H}^{2{\rm ht}(\mu)} \to 
{\rm H}^{2{\rm ht}(\mu)}_c({\rm S}_\mu, -)$ is an isomorphism. So the two filtrations split each other. 
\end{proof}

\bc
The global cohomology functor ${\rm H}^*: {\rm P}_{\cal S}({\rm Gr}^n) \lra {\rm Vect}$ is faithful and exact. 
\ec

Denote by ${\rm H}^p_{\rm per}{\cal F}$ the cohomology of an ${\cal F} \in D^b_{\cal S}({\rm Gr}^n)$ 
for the the perverse $t$-structure. 
Let $j: {\rm Gr}_{\underline{\lambda}}\hra \overline {{\rm Gr}_{\underline{\lambda}}}$ be the natural embedding,  
 ${\cal J}_!(\underline {\lambda}, \Q):= {\rm H}^0_{\rm per}(j_!\Q[{\rm dim}{\rm Gr}_{\underline{\lambda}}])$, and 
${\cal J}_*(\underline {\lambda}, \Q):= {\rm H}^0_{\rm per}(j_*\Q[{\rm dim}{\rm Gr}_{\underline{\lambda}}])$. 
The following Lemma is a generalisation of Lemma 7.1 of \cite{MV}. 
\bl
The category ${\rm P}_{\cal S}({\rm Gr}^n)$ is semi-simple. The sheaves ${\cal J}_!(\underline {\lambda}, \Q)$, 
${\cal J}_*(\underline {\lambda}, \Q)$, and ${\cal J}_{!*}(\underline {\lambda}, \Q)$ are isomorphic. 
\el

\begin{proof} 
Let us prove first the parity vanishing for 
the stalks of the sheaf ${\cal J}_{!*}(\underline {\lambda}, \Q)$: 
the stalks could have non-zero cohomology only at even degrees. 
For $n=1$ it is proved in \cite{L4}. It can also be proved by using the Bott-Samelson resolution 
of the Schubert cells in the affine (i.e. Kac-Moody) case, as was explained to us by A. Braverman.
 Let ${\cal F}$ be a Kac-Moody flag variety. 
Take an element $w=w_{1} ... w_{n}$ of the affine Weyl group such that 
$l(w) = l(w_{1}) + ... + l(w_{n})$. Denote by ${\cal F}_{w_1, ..., w_n}$ the variety parametrising flags 
$(F_1=[1], F_2, ..., F_n)$ such that the pair $(F_i, F_{i+1})$ is in the incidence relation $w_i$. 
Choose reduced decompositions $[w_{1}], ...  , [w_{n}]$ of the elements 
$w_{1}, ...  , w_{n}$. 
 Their product is a reduced decomposition $[w]$ of $w$.  
It gives rise to the Bott-Samelson variety $X_{[w]}$. By its very definition, it is a tower 
of fibrations
$$
X([w_{1}], ...  , [w_{n}]) \lra X([w_{1}], ...  , [w_{n-1}])  \lra \ldots \lra X([w_{1}]).  
$$
The Bott-Samelson resolution 
of the affine Schubert cell ${\rm Gr}_\lambda$ 
is a smooth projective variety $X_\lambda$ with a map $\beta_\lambda: X_\lambda \to {\rm Gr}_\lambda$ which is 
$1:1$ at the open stratum, and which, according to \cite{Gau1}, \cite{Gau2}, has the following property.  
For each of the strata ${\rm Gr}_\mu\subset {\rm Gr}_\lambda$, there exists a point 
$p_\mu\in {\rm Gr}_\mu$ such that the fiber $\beta_\lambda^{-1}(p_\mu)$ of the 
Bott-Samelson resolution has a cellular decomposition with the cells being complex vector spaces. 
Therefore the stalk of the push forward  $\beta_{\lambda\ast}\Q_{X_\lambda}$ of the constant sheaf on $X_\lambda$ 
at the point $p_\mu$ satisfies the parity vanishing.  
By the decomposition theorem \cite{BBD}, the sheaf ${\rm IC}_\lambda$ is a direct summand 
of the push forward  $\beta_{\lambda\ast}\Q_{X_\lambda}$ of the constant sheaf on $X_\lambda$. 
Indeed, the latter is a direct sum of shifts of perverse sheaves, and 
it is the constant sheaf over the open stratum. Therefore the stalk of the sheaf ${\rm IC}_\lambda$ 
at the point $p_\mu$ satisfies the parity vanishing. Since the cohomology of ${\rm IC}_\lambda$ 
is locally constant over each of the stratum ${\rm Gr}_\mu$, we get the parity vanishing. 
The general case of ${\rm Gr}_{\underline{\lambda}}$ is treated very similarly to the 
case of ${\rm Gr}_{{\lambda}}$.

The rest is pretty standard, and goes as follows. 
 The strata ${\rm Gr}_{\underline{\lambda}}$ 
are simply connected: this is 
well known for $n=1$, and the strata ${\rm Gr}_{\underline{\lambda}}$ is fibered over 
   ${\rm Gr}_{\underline{\lambda'}}$ with the fiber ${\rm Gr}_{{\lambda_n}}$, where 
$\underline{\lambda} = (\underline{\lambda'},\lambda_n)$.  
Since the strata are  even dimensional over $\R$, this plus the parity vanishing implies 
that there are no extensions between the simple objects in ${\rm P}_{\cal S}({\rm Gr}^n)$. 
Indeed, by devissage this claim reduces to calculation of extensions 
between constant sheaves concentrated on two open strata. Thus there are no extensions 
in the category  ${\rm P}_{\cal S}({\rm Gr}^n)$, i.e. it is semi-simple. 

Let us show now that 
${\cal J}_!(\underline {\lambda}, \Q) = {\cal J}_{!*}(\underline {\lambda}, \Q)$. 
Since ${\rm H}^p_{\rm per}(j_!\Q_{{\rm Gr}_{\underline{\lambda}}})=0$ for $p>0$, 
there is a map $j_!\Q_{{\rm Gr}_{\underline{\lambda}}}\to {\rm H}^0_{\rm per}(j_!\Q_{{\rm Gr}_{\underline{\lambda}}}) = 
{\cal J}_{!}(\underline {\lambda}, \Q)$. 
If ${\cal J}_!(\underline {\lambda}, \Q) \not = {\cal J}_{!*}(\underline {\lambda}, \Q)$, 
since the category ${\rm P}_{\cal S}({\rm Gr}^n)$ is semisimple, 
there is a non-zero direct summand ${\cal B}$ 
of  ${\cal J}_!(\underline {\lambda}, \Q)$ supported at a lower stratum. 
Composing these two maps,  we get a non-zero map $j_!\Q_{{\rm Gr}_{\underline{\lambda}}}\to {\cal B}$. 
On the other hand, 
given a space $X$ and complexes of sheaves  ${\cal A}$ and ${\cal B}$ supported at disjoint subsets  
$A$ and $B$ respectively, one has 
${\rm Hom}(j_!{\cal A}, {\cal B})=0$, where $j:A \hra X$.  
Contradiction. The statement about 
${\cal J}_\ast$ follows by the duality. 
\end{proof} 

\bl \la{gen3.5}
There are canonical isomorphisms
$$
F_{\mu}[{\cal J}_!(\underline {\lambda}, \Q)] = \Q[{\rm Irr}(\overline {{\rm Gr}^\mu_{\underline {\lambda}}})] = 
F_{\mu}[{\cal J}_*(\underline {\lambda}, \Q)]. 
$$
\el
 \begin{proof} We prove the first claim. The second is similar. 
We follow closely the proof of Proposition 3.10 in \cite{MV}. 
Set ${\cal F}:= {\cal J}_!(\underline {\lambda}, \Q)$. Let ${\rm Gr}_{\underline{\eta}}$ 
be a stratum in the closure of ${\rm Gr}_{\underline{\lambda}}$. Let $i_{\underline{\eta}}: {\rm Gr}_{\underline{\eta}} \hra 
\overline{{\rm Gr}_{\underline{\lambda}}}$ 
be the natural embedding. Then  
$i_{\underline{\eta}}^*{\cal F} \in D^{\leq -{\rm dim}{\rm Gr}_{\underline{\eta}} -2}({\rm Gr}_{\underline{\eta}})$. 
Indeed, we use $i_\eta^*j_!\Q=0$, and ${\rm H}^p_{\rm per}j_!\Q[{\rm dim}{\rm Gr}_{\underline{\lambda}}]=0$ for $p>0$ and 
apply $i_{\underline{\eta}}^*$ to the exact triangle 
$$
\lra {\tau}_{\rm per}^{\leq -1}(j_!\Q[{\rm dim}{\rm Gr}_{\underline{\lambda}}]) \lra j_!\Q[{\rm dim}{\rm Gr}_{\underline{\lambda}}] \lra {\rm H}_{\rm per}^{0}(j_!\Q[{\rm dim}{\rm Gr}_{\underline{\lambda}}]) \lra \ldots .
$$ 
Due to dimension counts (\ref{dimgr}) and (\ref{dimgrs}), we have 
${\rm H}^k_c({\rm Gr}_{\underline{\eta}} \cap {\rm S}_\mu, {\cal F})=0$ if $k> 2{\rm ht}(\mu)-2$. Thus the devissage 
associated to the filtration of ${\rm Gr}^n$ by ${\rm Gr}_{\underline{\eta}}$  tells 
that  
there is no contribution from the lower strata ${\rm Gr}_{\underline{\eta}}$ to ${\rm H}^{2{\rm ht}(\mu)}_c$, i.e. 
${\rm H}^{2{\rm ht}(\mu)}_c({\rm S}_\mu, {\cal F}) = {\rm H}^{2{\rm ht}(\mu)}_c({\rm Gr}_{\underline{\lambda}} \cap {\rm S_\mu}, {\cal F})$. 
Now we can conclude:
$$
{\rm H}^{2{\rm ht}(\mu)}_c({\rm Gr}_{\underline{\lambda}}^\mu, {\cal F}) = 
{\rm H}^{2{\rm ht}(\mu) +2{\rm ht}(\underline{\lambda})}_c({\rm Gr}_{\underline{\lambda}}^\mu, \Q) = 
{\rm H}_c^{2{\rm dim}({\rm Gr}^\mu_{\underline{\lambda}})}({\rm Gr}^\mu_{\underline{\lambda}}, \Q).
$$
The last cohomology group has a basis given by the top dimensional components 
of ${\rm Gr}^\mu_{\underline{\lambda}}$.
\end{proof} 
Lemma \ref{gen3.5} implies that there is a canonical isomorphism
$
{\rm H}^{2{\rm ht}(\mu)}_c({\rm S}_\mu, {\rm IC}_{\underline {\lambda}}) = \Q[{\rm Irr}(\overline {{\rm Gr}^\mu_{\underline {\lambda}}})]. 
$ 
Combined with Proposition \ref{mainprop} we arrive at Theorem \ref{mmvvth}. 
\end{proof}

\paragraph{Parametrisation of a canonical basis.} Since the group ${\B}({\cal O})$ is connected, the  
projection 
$$
p: {\rm Gr}_{\underline{\lambda}}^{\mu} \lra {\rm B}({\cal O})\backslash {\rm Gr}_{\underline{\lambda}}^{\mu}
= {\rm Conf}({\rm Gr}^{n+1}, {\cal B})^\mu_{\underline \lambda}
$$ 
identifies the top components. 
So Theorem \ref{MVn+1bbb} tells that the cycles $p^{-1}({\cal M}^\circ_l)$, 
$l \in {\bf B}^\mu_{\underline {\lambda}}$, see (\ref{MVn+1bb}),  
are the top components of ${\rm Gr}_{\underline{\lambda}}^{\mu}$.  
Theorem \ref{mmvvth} plus  (\ref{gSC}) 
implies  that they give rise to classes $[p^{-1}({\cal M}^\circ_l)] \in V_{\underline{\lambda}}$. 
Moreover, the $\mu$ is the weight of the class in $V_{\underline{\lambda}}$. 
So we get the following result. 

\bt \la{tensorproductbasis}
The set ${\bf B}^\mu_{\underline {\lambda}}$ parametrises a canonical basis in the weight $\mu$ part 
$V^{(\mu)}_{\underline{\lambda}}$ of the representation $V_{\lambda_1} \otimes \ldots \otimes V_{\lambda_n}$ of ${\rm G^L}$. 
This basis is given by the classes 
$[p^{-1}({\cal M}_l)]$, $l \in {\bf B}^\mu_{\underline {\lambda}}$.
\et

\section{The potential ${\cal W}$ in special coordinates for ${\rm GL}_m$}\la{KT}

\subsection{Potential for ${\rm Conf}_3({\cal A})$  and Knutson-Tao's rhombus inequalities}
Recall that a flag $F_\bullet$ for ${\rm GL}_m$ is a collection of subspaces 
in an $m$-dimensional vector space $V_m$:
$$
F_\bullet = F_0 \subset F_1 \subset \ldots \subset F_{m-1} \subset F_m, ~~~~ {\rm dim}F_i=i. 
$$
A decorated flag for ${\rm GL}_m$ is a flag $F_\bullet$ with a choice of non-zero vectors  
$f_i \subset  F_{i}/F_{i-1}$ for each $i=1, \ldots, m$, called {\it decorations}. 
It  determines a collection of decomposable 
$k$-vectors 
$$
f_{(1)}:= f_1, ~~f_{(2)}:= f_1\wedge f_2,~~ \ldots , ~~f_{(m)}:= f_1 \wedge ... \wedge f_m. 
$$
A decorated flag is determined by a collection of decomposable $k$-vectors 
such that each  divides the next one. A linear basis $(f_1, ..., f_m)$ in the space $V_m$ determines 
a decorated flag by setting $F_k:= \langle f_1, ..., f_k\rangle$, 
and taking the projections of $f_k$ to $F_k/F_{k-1}$ to be the decorations. 

Recall the notion of an 
$m$-triangulation of a triangle  \cite[Section 9]{FG1}. It is a graph whose vertices are parametrized by a set 
\be \la{Gamma}
\Gamma_m:=\{(a,b,c)~|~ a+b+c=m, \quad a,b,c\in \Z_{\geq 0}\}.
\ee
Let $({\rm F, G, H})\in {\rm Conf}_3({\cal A})$ be a generic configuration of three decorated flags, 
described by a triple of linear bases in the space $V_m$: 
$$
{\rm F}=(f_1,\ldots, f_m), ~{\rm G}=(g_1,\ldots,g_m),~{\rm H}=(h_1,\ldots,h_m).
$$
Let $\omega\in \det V_m^*$ be a volume form. Then 
each vertex $(a,b,c)\in (\ref{Gamma})$ gives rise to a function
$$
\Delta_{a,b,c}({\rm F, G, H})=\langle f_{(a)}\wedge g_{(b)}\wedge h_{(c)}, \omega\rangle.
$$
There is a one dimensional space ${\rm L}_a^{b,c}:={\rm F}_{a+1}\cap ({\rm G}_b\oplus {\rm H}_c)$. 

Let $e_{a}^{b,c}\in {\rm L}_a^{b,c}$ such that $e_{a}^{b,c}-f_{a+1}\in {\rm F}_a$. It is easy to see that
$
e_{a}^{b+1,c-1}-e_{a}^{b,c}\in {\rm L}_{a-1}^{b+1, c}.
$
Therefore there exists a unique scalar $\alpha_{a}^{b,c}$ such that 
$
e_{a}^{b+1,c-1}-e_{a}^{b,c}=\alpha_{a}^{b,c} e_{a-1}^{b+1,c}.
$

\begin{figure}[ht]
\epsfxsize400pt
\centerline{\epsfbox{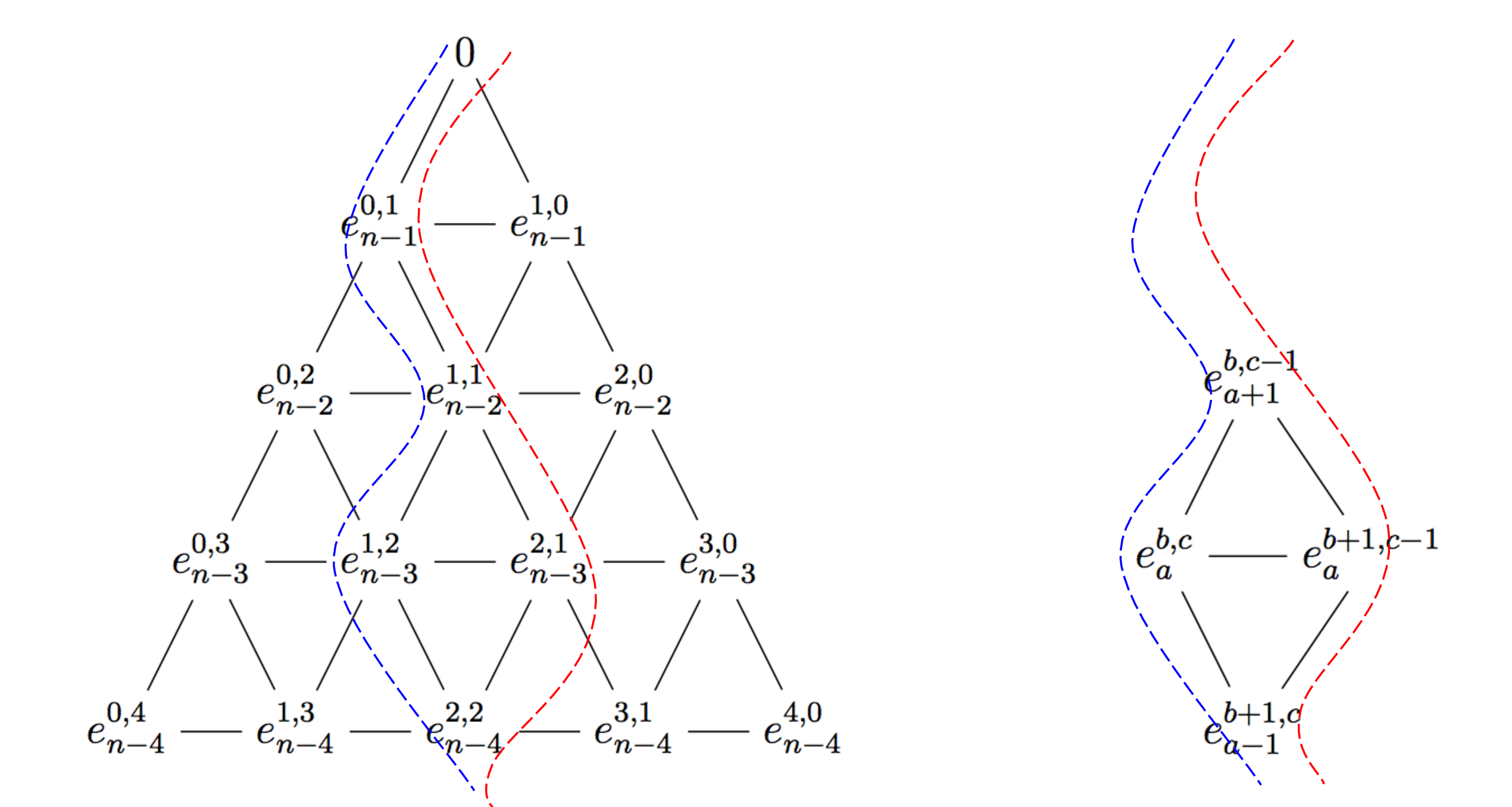}}
\caption{Zig-zag paths and bases for the decorated flag ${\rm F}$.}
\label{tri}
\end{figure}

\begin{lemma} \la{lem36}
One has 
\begin{equation}
\alpha_{a}^{b,c}=\frac{\Delta_{a-1,b+1,c}\Delta_{a+1,b,c-1}}{\Delta_{a,b,c}\Delta_{a,b+1,c-1}}.
\end{equation}
\end{lemma}
\begin{proof}
Set
\begin{equation}
\alpha:=\alpha_{a}^{b,c}, \quad \beta:=\frac{\Delta_{a,b,c}}{\Delta_{a+1,b,c-1}},\quad \gamma:=\frac{\Delta_{a,b+1,c-1}}{\Delta_{a+1,b,c-1}}.
\la{eq34}
\end{equation}
By definition,
\begin{align}
f_{(a)}&=f_{(a-1)}\wedge e_{a-1}^{b,c+1},\nonumber\\
f_{(a+1)}&=f_{(a)}\wedge e_{a}^{b,c}=f_{(a)}\wedge e_{a}^{b-1,c+1},\nonumber\\
g_{(b)}\wedge h_{(c)}&=\beta e_{a}^{b,c}\wedge g_{(b)}\wedge h_{(c-1)},\nonumber\\
g_{(b+1)}\wedge h_{(c-1)}&= \gamma e_a^{b+1,c-1} \wedge g_{(b)}\wedge h_{(c-1)}. \nonumber
\end{align}
Therefore,
\begin{align}
g_{(b+1)}\wedge h_{(c)}&=\gamma e_a^{b+1,c-1} \wedge g_{(b)}\wedge h_{(c)}\nonumber\\
                                              &=\beta \gamma e_{a}^{b+1,c-1}\wedge e_{a}^{b,c}\wedge g_{(b)}\wedge h_{(c-1)}\nonumber\\
                                              &=\beta\gamma(e_{a}^{b+1,c-1}-e_{a}^{b,c})\wedge e_{a}^{b,c}\wedge g_{(b)}\wedge h_{(c-1)}\nonumber\\
                                              &=\beta\gamma\alpha  e_{a-1}^{b+1,c}\wedge e_{a}^{b,c}\wedge g_{(b)}\wedge {h_{(c-1)}}. \nonumber
\end{align}
So
$$
f_{(a-1)}\wedge g_{(b+1)}\wedge h_{(c)}=\alpha \beta\gamma f_{(a+1)}\wedge g_{(b)}\wedge h_{(c-1)}. 
$$
Therefore, 
$$
\alpha \beta\gamma =\frac{\Delta_{a-1,b+1,c}}{\Delta_{a+1,b,c-1}}.
$$
Go back to (\ref{eq34}), the Lemma is proved.
\end{proof}

As shown on Fig \ref{tri}, each zig-zag path $p$ provides a basis ${\rm E}_{p}$ for ${\rm F}$. For example,
$$
{\rm E}_{l}:=\{e_{0}^{0, n}, e_{1}^{0,n-1},\ldots, e_{n-1}^{0,1}\},~~~ {\rm E}_{r}:=\{e_{0}^{n,0}, e_{1}^{n-1, 1},\ldots,e_{n-1}^{1,0}\}
$$
are the bases provided by the very left and very right paths.

Given two zig-zag paths, say $p$ and $q$, there is a unique  unipotent element $u_{pq}$ stabilizing  {\rm F}, transforming ${\rm E}_{p}$ to ${\rm E}_{q}$. 
Recall the character $\chi_{\rm F}$ in section 1. For each triple $(p, q,r)$ of zig-zag paths, we have 
\begin{align}
\chi_{\rm F}(u_{pq})&=-\chi_{\rm F}(u_{qp}),\nonumber\\
\chi_{\rm F}(u_{pr})&=\chi_{\rm F}(u_{pq})+\chi_{\rm F}(u_{qr}).\nonumber
\end{align}
 If $p$, $q$ are adjacent paths, see the right of Fig \ref{tri}, then by Lemma \ref{lem36}, 
$$
\chi_{\rm F}(u_{pq})=\alpha_{a}^{b,c}=\frac{\Delta_{a-1,b+1,c}\Delta_{a+1,b,c-1}}{\Delta_{a,b,c}\Delta_{a,b+1,c-1}}.
$$

One can transform the very left path to the very right by a sequence of adjacent paths.
 Let $u\in {\rm U}_{\rm F}$ transform ${\rm E}_l$ to ${\rm E}_r$. Then
$$
\chi_{\rm F}(u)=\sum_{(a,b,c)\in \Gamma_m, a\neq 0, c\neq 0}\alpha_{a}^{b,c}=\sum_{(a,b,c)\in \Gamma_m, a\neq 0, c\neq 0}\frac{\Delta_{a-1,b+1,c}\Delta_{a+1,b,c-1}}{\Delta_{a,b,c}\Delta_{a,b+1,c-1}}.
$$
Its tropicalization 
$$
\chi_{\rm F}^t=\min_{(a,b,c)\in \Gamma_n, a\neq 0, c\neq 0}\{ \Delta_{a-1,b+1,c}^t+\Delta_{a+1,b,c-1}^t-\Delta_{a,b,c}^t-\Delta_{a,b+1,c-1}^t\}
$$
delivers 1/3 of Knutson-Tao rhombus inequalities. Clearly, same holds for the other two directions. By definition,
$$
{\cal W}({\rm F,G,H})=\chi_{\rm F}+\chi_{\rm G}+\chi_{\rm H}.
$$
Our set ${\rm Conf}_3^{+}({\cal A})(\Z^t)$ coincides with the set of hives in \cite{KT}.

\begin{figure}[ht]
\epsfxsize=3in
\centerline{\epsfbox{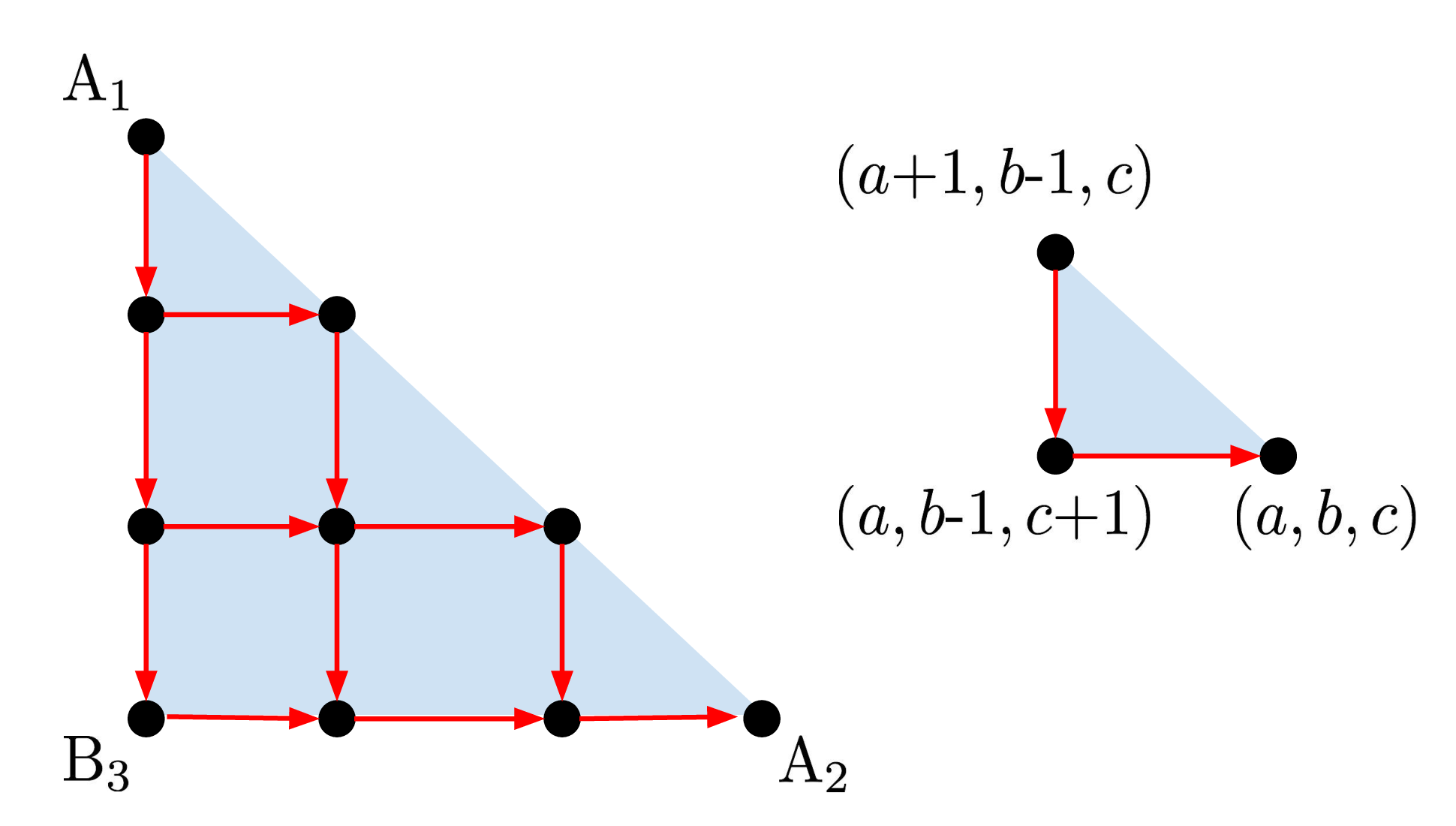}}
\caption{Calculating  the potential ${\cal W}$ on ${\rm Conf}({\cal A}, {\cal A}, {\cal B})$ 
in the special coordinates for ${\rm GL}_m$.}
\la{phase}
\end{figure}

In Sections \ref{sec4.2Gi}-\ref{sec4.2GZ} we show that the potential on the 
space ${\rm Conf}({\cal A}, {\cal A}, {\cal B})$ for ${\rm GL}_m$, written in the special coordinates there, 
recovers Givental's potential and, after tropicalization, Gelfand-Tsetlin's patterns. 

\subsection{The potential for ${\rm Conf}({\cal A}, {\cal A}, {\cal B})$ 
 and Givental's potential for ${\rm GL}_m$} \la{sec4.2Gi}

Let ${\rm G}={\rm GL}_m$. Recall the set $\Gamma_m$, see (\ref{Gamma}). 
For each triple $(a,b,c)\in \Gamma_m$, 
there is a canonical function $\Delta_{a,b,c}: {\rm Conf}_3({\cal A})\ra {\Bbb A}^1$. 
Consider the functions $\Delta_{a,b,c}$ with 
$(a,b,c)\in \Gamma_m- (0,0,m)$, illustrated by the $\bullet$-vertices on Fig \ref{giv}. 
For each triple $(a,b,c)\in \Gamma_{m-1}$, let us set
\be \la{Rratio}
{\rm R}_{a,b,c}:=\frac{\Delta_{a, b+1, c}}{\Delta_{a+1, b,c}}.
\ee

The functions ${\rm R}_{a,b,c}$ are assigned naturally to the $\circ$-vertices 
on Fig \ref{giv}. Each of them is the ratio of the $\Delta$-functions at the ends 
of the slant edge centered at a $\circ$-vertex. 
They are functions on ${\rm Conf}({\cal A},{\cal A},{\cal B})$ since 
${\rm R}_{a,b,c}(\A_1,\A_2,\A_3\cdot h)={\rm R}_{a,b,c}(\A_1,\A_2,\A_3)$ for any $h\in {\rm H}$. 
The  functions ${\rm R}_{a,b,c}$
form a coordinate system on  ${\rm Conf}({\cal A},{\cal A},{\cal B})$, 
referred to as the {\it special coordinate system}. 

The functions $\{{\rm R}_{a,b,0}\}$
provide the canonical map 
\be \la{q-map}
{\rm Conf}({\cal A},{\cal A},{\cal B}) \lra {\rm Conf}({\cal A},{\cal A}) = ({\Bbb G}_m)^{m-1}.
\ee

\begin{figure}[ht]
\epsfxsize=1.5in
\centerline{\epsfbox{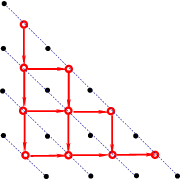}}
\caption{The Givental quiver and special coordinates 
on ${\rm Conf}({\cal A}, {\cal A}, {\cal B})$ for ${\rm GL}_4$.}
\la{giv}
\end{figure}

Consider now the {\it Givental quiver $\Gamma_{m-1}$}, whose vertices are the $\circ$-vertices, 
parametrised by the set $\Gamma_{m-1}$, 
with the arrows are going down and to the right, as shown on Fig \ref{giv}. 
For each arrow connecting two vertices, take the sourse/tail ratio 
of the corresponding functions. For example, see 
Fig \ref{phase}, the vertical arrow $\alpha$ connecting $(a+1,b-1,c)$ and $(a,b-1,c+1)$ provides
\be 
\la{13.12.28}
Q_{\alpha}=\frac{{\rm R}_{a,b-1,c+1}}{{\rm R}_{a+1, b-1,c}}=\frac{\Delta_{a,b,c+1}\Delta_{a+2,b-1,c}}{\Delta_{a+1,b-1,c+1}\Delta_{a+1,b,c}}.
\ee

\vskip 2mm 

Recall the function $\chi_{\A_1}$, $\chi_{\A_2}$ on ${\rm Conf}({\cal A},{\cal A},{\cal B})$. 
Taking the sum of $Q_{\alpha}$ over the vertical arrows $\alpha$, 
and a similar sum over the horizontal arrows $\beta$, 
and using \eqref{13.12.28}, we get
$$
\chi_{\A_1}=\sum_{\alpha \mbox{~vertical}} Q_{\alpha}, ~~~~
\chi_{\A_2}=\sum_{\beta \mbox{~horizontal}} Q_{\beta}.
$$

\paragraph{Relating to Givental's work.} 
Givental  \cite{Gi2}, pages 3-4, introduced parameters $T_{i,j}$, 
$0 \leq i\leq j \leq m$, matching the vertices of the Givental quiver: 
$$
\begin{matrix}
T_{0,0}&&&\\
T_{01}&T_{1,1}&&\\
T_{02}&T_{12}&T_{2,2}&\\
T_{03}&T_{13}&T_{2,3}&T_{3,3}\\
\end{matrix}
$$
He treats the entires on the main diagonal ${\rm a}= ({T}_{0,0}, {T}_{1,1}, ..., {T}_{m,m})$ as parameters, 
and defines the potential as a sum over the oriented edges of the quiver:
$$
{\cal W}_{\rm a} = \sum_{0 \leq i< j \leq m} \bigl({\rm exp}(T_{i,j} - T_{i,j-1}) + 
{\rm exp}(T_{i,j} - T_{i+1,j})\Bigr). 
$$
 Let $Y_{\rm a}$ be the subvariety with a given value of ${\rm a}$. Then Givental's integral is 
$$
{\cal F}({\rm a}, \hbar)= \int_{Y_{\rm a}}{\rm exp}(-{\cal W}_{\rm a}/\hbar)\bigwedge_{i=1}^n\bigwedge_{j=0}^{i-1}d T_{i.j}.
$$

Givental's variables $T_{i,j}$  match our coordinates ${\rm R}_{a,b,c}$ where $a+b+c=m-1$:
$$
{\rm R}_{m-i-1, j, i-j} = {\rm exp}(T_{i,j}).
$$
Observe that $Y_{\rm a}$ is the fiber of the map (\ref{q-map}) over a point 
${\rm a}= ({\rm R}_{m-1, 0}, {\rm R}_{m-2, 1}, ..., {\rm R}_{0, m-1})$. 
Givental's potential coincides 
with 
 $\chi_{\A_1}+\chi_{\A_2}$. Givental's volume form on $Y_{\rm a}$ coincides, up to a sign, with ours since 
$$
\bigwedge_{i=1}^n\bigwedge_{j=0}^{i-1}d T_{i.j} = \pm \bigwedge_{a+b+c=m-1, c>0}d\log {\rm R}_{a,b,c}.
$$

\subsection{The potential for ${\rm Conf}({\cal A}, {\cal A}, {\cal B})$ 
 and Gelfand-Tsetlin's patterns for ${\rm GL}_m$} \la{sec4.2GZ}

Gelfand-Tsetlin's patterns for ${\rm GL}_m$ \cite{GT1} are arrays of integers 
$\{p_{i,j}\}$, $1 \leq i \leq  j \leq m$, such that 
\be \la{ineq}
{p_{i,j+1}\leq p_{i,j}\leq p_{i+1,j+1}.}
\ee

\begin{figure}[ht]
\centerline{\epsfbox{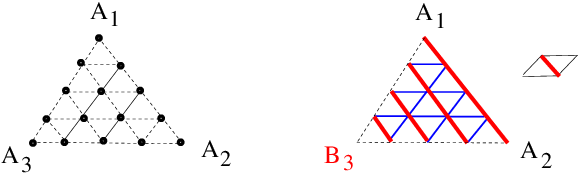}}
\caption{Gelfand-Tsetlin patterns for ${\rm GL}_4$ and the 
special coordinates for ${\rm Conf}({\cal A}, {\cal A}, {\cal B})$.}
\label{tpm50}
\end{figure}

\bt 
The special coordinate system on ${\rm Conf}({\cal A}_{{\rm GL}_m}, {\cal A}_{{\rm GL}_m}, {\cal B}_{{\rm GL}_m})$ 
together with the potential ${\cal W} = \chi_{\A_1}+\chi_{\A_2}$ 
provide a canonical isomorphism
$$
\{\mbox{\rm Gelfand-Tsetlin's patterns for ${\rm GL}_m$}\} ~~= ~~ 
{\rm Conf}^+({\cal A}_{{\rm GL}_m}, {\cal A}_{{\rm GL}_m}, {\cal B}_{{\rm GL}_m})(\Z^t). 
$$
\et

\begin{proof} 
The space ${\rm Conf}({\cal A}_{{\rm GL}_m}^3, \omega_m)$ of ${\rm GL}_m$-orbits 
on ${\cal A}_{{\rm GL}_m}^3 \times {\rm det}V_m^*$ has dimension $\frac{(m+1)(m+2)}{2}$.  
It has a coordinate system given by the functions $\Delta_{a,b,c}$, {$a+b+c=m$}, 
parametrized by the vertices of the graph $\Gamma_{m}$, shown on the left of Fig \ref{tpm50}. 
The coordinates on ${\rm Conf}({\cal A}_{{\rm GL}_m}, {\cal A}_{{\rm GL}_m}, {\cal B}_{{\rm GL}_m})$ are parametrized by 
the edges $E$ of the graph parallel to the edge $\A_1\A_2$ of the triangle. 
They are little red segments on the right of Fig \ref{tpm50}. They are given by the ratios of the 
coordinates at the ends of the edge $E$, recovering formula (\ref{Rratio}). 
 Notice that the edges $E$ are  oriented by the orientation of the side $\A_1\A_2$. The monomials of the potential 
$\chi_{\A_1}+\chi_{\A_2}$ are paramerized by the blue edges, that is by the edges of the graph 
inside of the triangle parallel to either side $\B_3\A_1$ or $\B_3\A_2$. 
We claim that the monomials of potential 
$\chi_{\A_1}+\chi_{\A_2}$ are in bijection with Gelfand-Tsetlin's inequalities. Indeed, 
a typical pair of inequalities (\ref{ineq}) is encoded by a  
part of the graph shown on Fig \ref{tpm50a}. The three coordinates $(P_1, P_2, Q)$ on 
${\rm Conf}({\cal A}_{{\rm GL}_m}, {\cal A}_{{\rm GL}_m}, {\cal B}_{{\rm GL}_m})$ assigned to the red edges are 
expressed via the coordinates $(A, B, C, D, F)$ at the vertices: 
$$
P_1 = \frac{B}{A}, ~~~~P_2 = \frac{C}{B}, ~~~~Q = \frac{E}{D}.
$$
\begin{figure}[ht]
\epsfxsize80pt
\centerline{\epsfbox{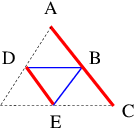}}
\caption{Gelfand-Tsetlin patterns from the potential.}
\label{tpm50a}
\end{figure}
The monomials of the potential at the two blue edges are $\frac{EA}{DB}$ and $\frac{DC}{EB}$. Their tropicalization delivers the 
inequalities 
$
p_1 \leq q, q\leq p_2.
$ 
\end{proof}

\section{Proof of Theorem \ref{8.23.7.32hh}} 
\la{9.21.12.17h}
Let ${\rm T}$ be a split torus.  Let 
$g:=\sum_{\alpha\in X^*({\rm T})} g_{\alpha} X^{\alpha}$
be a nonzero positive polynomial on ${\rm T}$, i.e.
its coefficients $g_{\alpha}\geq 0$ are non-negative. 
The integral tropical points $l\in {\rm T}(\Z^t)=X_{*}({\rm T})$ are cocharacters of ${\rm T}$.
The tropicalization of $g$ is a piecewise linear function on ${\rm T}(\Z^t)$:
$$
g^t(l)=\min_{\alpha~|~g_{\alpha}>0}\{\langle l, \alpha\rangle\}.
$$
Fix an $l\in {\rm T}(\Z^t)$. Set
$$
\Lambda_{g,l}:=\{\alpha\in X^*({\rm T})~|~ g_{\alpha}> 0, ~\langle l,\alpha\rangle =g^t(l)\}, \qquad
g_l:=\sum_{\alpha \in \Lambda_{g, l}}g_{\alpha}X^{\alpha}.
$$
The set $\Lambda_{g,l}$ is non-empty. Therefore $g_l$ is  a nonzero positive polynomial.
If $f$ and  $g$ are two such polynomials, so is the product $f\cdot g$.   We have $(f\cdot g)_l=f_{l}\cdot g_{l}$ for all $l\in {\rm T}(\Z^t).$

Let $h$ be a nonzero positive rational function on ${\rm T}$. It can be expressed as a ratio $f/g$ of two nonzero positive polynomials. Set $h_l:=f_l/g_l$. Let $h=f'/g'$ be another expression. Then
$$
f/g=f'/g' ~\Longrightarrow ~ f \cdot g'=f'\cdot g ~{\Longrightarrow} ~f_l\cdot g'_l =f'_l\cdot g_l~\Longrightarrow ~f_l/g_l =f'_l/g_l'.
$$
Hence $h_l$ is well defined.

\bl \la{8.21.14.46h} Let $h$, $l$ be as above. For each $C\in {\rm T}_l$ such that \footnote{Every transcendental point $C\in {\rm T}^{\circ}_l$ automatically satisfies such conditions.} $h_l({\rm in}(C))\in \C^*$, we have 
 \be \la{9.12.12.1}
{\val}(h(C))=h^t(l),~~~~~{\rm in}(h(C))=h_{l}({\rm in}(C)). 
\ee
\el

\begin{proof} Assume that $h$ is a nonzero positive polynomial. By definition 
$$
\forall C\in {\rm T}_l,~~~~h(C)=h_l({\rm in}(C))t^{h^t(l)}+\text{ terms with higher valuation}.
$$ 
If $h_l({\rm in}(C))\in \C^*$, then \eqref{9.12.12.1} follows. 
The argument for  a positive rational function is similar.
\end{proof}

Let $f=(f_1, \ldots, f_k):{\rm T}\ra {\rm S}$ be a positive birational isomorphism of split tori.  Let $l\in {\rm T}(\Z^t)$. We generalize the above construction by setting 
$f_l:=(f_{1,l},\ldots,f_{k,l}): {\rm T}\lra {\rm S}.$

\bl \la{9.20.13.54h} Let $f$, $l$ be as above. Let $C\in {\rm T}_l^{\circ}$. Then
\be \la{9.20.13.54h1}
\val(f(C))=f^t(l),~~~~~ {\rm in}(f(C))=f_l({\rm in}(C)). 
\ee
Let $h$ be a nonzero positive rational function on ${\rm S}$. Then 
\be \la{9.20.13.54h2}
{\rm in}\big(h\circ f (C)\big)=h_{f^t(l)}\big({\rm in}(f(C))\big).
\ee
\el

\begin{proof}
Here \eqref{9.20.13.54h1} follows directly from Lemma \ref{8.21.14.46h}. 
Note that $h_{f^t(l)}\circ f_l$ is a nonzero positive rational function on {\rm T}. Since $C$ is transcendental, we get
$$
h_{f^t(l)}\big( {\rm in}(f(C))\big)=h_{f^t(l)}\circ f_l({\rm in }(C))\in \C^*.
$$
Thus \eqref{9.20.13.54h2} follows from Lemma \ref{8.21.14.46h}.
\end{proof}

\paragraph{Proof of Theorem  \ref{8.23.7.32hh}.}
It suffices  to prove  $f({\rm T}_{l}^{\circ})\subseteq{\rm S}_{f^t(l)}^{\circ}$. The other direction is the same.
 
Let $C=(C_1,\ldots,C_k)\in {\rm T}_l^{\circ}$. Let $f(C):=(C_1',\ldots, C_k')$. 
By \eqref{9.20.13.54h1}, we get 
$
f(C)\in {\rm S}_{f^t(l)}$ 
and the field extension
$
\Q({\rm in}(C_1'),\ldots, {\rm in}(C_k'))\subseteq \Q({\rm in}(C_1),\ldots, {\rm in}(C_k)).
$

Let $g=(g_1,\ldots, g_k): {\rm S}\ra {\rm T}$ be the inverse morphism of $f$. Then $C_j=g_j\circ f(C)$ for $j\in [1,k]$.
The functions  $g_j$ are  nonzero positive rational functions on ${\rm S}$. Therefore 
$$
{\rm in}(C_j)={\rm in}(g_j\circ f(C))\stackrel{\eqref{9.20.13.54h2}}{=}g_{j,f^t(l)}({\rm in}( f(C)))\in \Q({\rm in}(C_1'),\ldots, {\rm in}(C_k')).
$$
Therefore 
$
\Q({\rm in}(C_1),\ldots, {\rm in}(C_k))\subseteq \Q({\rm in}(C_1'),\ldots, {\rm in}(C_k')).
$ 
Summarizing, we get
\be \la{9.20.2.12}
\Q({\rm in}(C_1),\ldots, {\rm in}(C_k))=\Q({\rm in}(C_1'),\ldots, {\rm in}(C_k')).
\ee
Therefore $f(C)$ is transcendental.  Theorem  \ref{8.23.7.32hh} is proved.


\section{Positive structures on the unipotent subgroups $\U$ and $\U^{-}$}
\la{sec4}

\subsection{Lusztig's data and MV cycles}
\paragraph{Lusztig's data.} 
Fix a reduced word ${\bf i}=(i_m,\ldots, i_1)$ for ${w}_0$. There are positive functions 
\be \la{13.3.27.1827h}
F_{{\bf i}, j}: \U\lra {\Bbb A}^1,~~~~x_{i_m}(a_m)\ldots x_{i_1}(a_1)\lms a_j.
\ee
 Their tropicalizations induce an isomorphism
$f_{\bf i}: \U(\Z^t)\stackrel{=}{\ra} \Z^m,~p\mapsto \{F_{{\bf i}, j}^t(p)\}.$

Let ${\Bbb N}=\{0,1,2,\ldots\}$. Lusztig proved \cite{L1} that the subset
\be \la{8.20.1.20h}
f_{\bf i}^{-1}({\Bbb N}^m)\subset \U(\Z^t)
\ee
does not depend on {\bf i}, and parametrizes the canonical basis in the quantum enveloping algebra of the Lie algebra of a maximal unipotent subgroup of the Langlands dual group ${\rm G}^L$.

\vskip 2mm
 
\bl \la{8.20.3.42h}
The subset $\U_{\chi}^{+}(\Z^t):=\{l\in \U(\Z^t)~|~ \chi^t(l)\geq 0\}$ is identified with the set \eqref{8.20.1.20h}.
\el
\begin{proof} Note that $\chi=\sum_{j=1}^mF_{{\bf i},j}$. It tropicalization is $\min_{1\leq j\leq m}\{F_{{\bf i},j}^t\}$. 
Let $l\in \U(\Z^t)$. Then 
$$
\chi^t(l)\geq 0 \Longleftrightarrow F_{{\bf i},j}^t(l)\geq 0,~ \forall j\in[1,m] \Longleftrightarrow f_{\bf i}(l)\in {\Bbb N}^m.
$$
\end{proof}

Let $l\in \U(\Z^t)$. Recall the transcendental cell ${\cal C}_{l}^{\circ}\subset \U({\cal K})$.

\bl \la{Lema10.1.1} 
Let $u\in {\cal C}_{l}^{\circ}$. Then $u\in {\U}({\cal O})$ if and only if $l\in {\U}_{\chi}^+(\Z^t)$.
\el

\begin{proof} Set $u=x_{i_m}(a_m)\ldots x_{i_1}(a_1)\in {\cal C}_l^{\circ}$. {Note that $u$ is transcendental.} Using Lemma \ref{thm10.1.1.2}, we get 
$$
\chi^t(l)=\val(\chi(u));~~~~ {F}_{{\bf i},j}^t(l)=\val(a_j), ~\forall j\in [1,m].
$$

If $l\in {\U}_\chi^+(\Z^t)$, then $\val(a_j)={F}_{{\bf i},j}^t(l)\geq 0$. 
Therefore $a_j\in {\cal O}$. Hence $u\in \U({\cal O})$. 

{Note that $\chi$ is a regular function of $\U$. So $u\in \U({\cal O})$, then $\chi(u)\in {\cal O}$.} Therefore $\chi^t(l)=\val(\chi(u))\geq 0$. Hence $l\in \U_\chi^+(\Z^t)$.
\end{proof}

\paragraph{The positive morphism $\beta$.}

Let $[g]_0:=h$ if $g=u_{+}hu_{-}$, where $u_{\pm}\in {\rm U}^{\pm}$, $h\in {\rm H}$. 
Define
\be \la{8.16.10.40h}
\beta: {\rm U}\lra {\rm H},~~u\lms [\overline{w}_0u]_0.
\ee

Let ${\bf i}=(i_m,\ldots, i_1)$ as above.  Let $w_{k}^{\bf i}:=s_{i_1}\ldots s_{i_k}\in W$. 
Let $\beta_{k}^{\bf i}:=w_{k-1}^{\bf i}(\alpha_{i_k}^{\vee})\in {\rm P}$.  {The following Lemma shows that $\beta$ is a positive map.
\bl[{\cite[Lemma 6.4]{BZ}}]
\la{8.13.1.20h}
For each $u=x_{i_m}(a_m)\ldots x_{i_1}(a_1)\in {\rm U}$, we have
$[\overline{w}_0u]_0=\prod_{k=1}^m \beta_{k}^{\bf i}(a_k^{-1}).$
\el
}
Let $l\in \U(\Z^t)$. The tropicalization $\beta^t$ becomes
$\beta^t(l)=-\sum_{k=1}^mF_{{\bf i},k}^t(l)\beta_{k}^{\bf i}.$

Note that $\beta_k^{\bf i}\in {\rm P}$ are positive coroots. If $l\in \U_{\chi}^{+}(\Z^t)$, then $-\beta^t(l) \in {\rm R}^+$. Hence
\be \la{8.23.10.22h}
\U_{\chi}^{+}(\Z^t)=\bigsqcup_{\lambda\in {\rm R}^+} {\bf A}_{\lambda}, ~~~~
{\bf A}_{\lambda}:=\{l\in \U_{\chi}^{+}(\Z^t)~|~ -\beta^t(l)=\lambda\}.
\ee
The set ${\bf A}_\lambda$ is identified with Lusztig's set parametrizing  the canonical basis of weight $\lambda$ 
\cite{L1}.

\paragraph{Kamnitzer's parametrization of MV cycles.}
Kamnitzer \cite{K} constructs a canonical bijection between Lusztig's data (i.e. $\U_{\chi}^{+}(\Z^t)$ in our set-up) and the set of stable MV cycles. Let us briefly recall Kamnitzer's result for future use.

Let $\U_{*}:={\rm U}\cap {\rm B}^{-}w_0{\rm B}^{-}$ and let $\U^{-}_{*}=\U^{-}\cap {\rm B}w_0{\rm B}$. There is an well-defined isomorphism
 \be \la{13.map.eta}
 \eta: {\rm U}_{*}\ra {\rm U}^{-}_{*},~~~u\lms \eta(u).
 \ee 
such that $\eta(u)$ is the unique element in  ${\rm U}^{-}\cap {\rm B}{w}_0u$. 
The map $\eta$ was used in  \cite{FZ1}.  Set
\be \la{kappa.kam}
\kappa_{\rm Kam}: {\rm U}_{*}({\cal K})\lra {\rm Gr},~~~u\lra [\eta(u)].
\ee

Let $l\in \U(\Z^t)$. Then ${\cal C}_l^{\circ}\subset {\rm U}_{*}({\cal K})$. Define 
\be \la{thm.kam.a.l}
{\rm MV}_l:=\overline{\kappa_{\rm Kam}({\cal C}^\circ_l)}\subset {\rm Gr}.
\ee
The following Theorem is a reformulation of Kamnitzer's result.
\bt [{\cite[Theorem 4.5]{K}}] 
\la{thm.kam}
Let $l\in {\bf A}_{\lambda}$. Then ${\rm MV}_l$ is an MV cycle of coweight $(\lambda,0)$. It gives a bijection between ${\bf A}_{\lambda}$ and the set of such MV cycles.
\et

A stable MV cycle of coweight $\lambda $ has a unique representative of coweight $(\lambda,0)$. 
Therefore 
Theorem \ref{thm.kam} tells that the set ${\bf A}_{\lambda}$ parametrizes the set of stable MV cycles of coweight $\lambda$.

\subsection{Positive functions $\chi_i, {\cal L}_i, {\cal R}_i$ on ${\U}$.}
Let $i\in  I$. We introduce positive rational functions $\chi_i$, ${\cal L}_{i}$, ${\cal R}_{i}$
on  ${\rm U}$, and $\chi_i^{-}$, ${\cal L}_i^{-}$, ${\cal R}_i^{-}$ on ${\rm U}^{-}$. 

Let  ${\bf i}=(i_1, \ldots, i_m)$ be a reduced word for $w_0$. Let
$$
x=x_{i_1}(a_1)\ldots x_{i_m}(a_m)\in {\rm U},~~~y=y_{i_1}(b_1)\ldots y_{i_m}(b_m) \in {\rm U}^{-}.
$$

Using above decompositions of $x$ and $y$, we set
$$
\chi_{i}(x):=\sum_{p~|~ i_p=i} a_p,
~~~~~
\chi_i^{-}(y):=\sum_{p~|~i_p=i}b_p.
$$
By definition the characters $\chi$ and $\chi^{-}$ have decompositions
$\chi=\sum_{i\in I}\chi_i$ and $\chi^{-}=\sum_{i\in I}\chi_i^{-}$. 

We take ${\bf i}$ 
which starts from $i_1=i$. Define the  ``left" functions:
$$
{\cal L}_{i}(x):=a_1,
~~~~~
{\cal L}_{i}^{-}(y):=b_1.
$$

We take ${\bf i}$ 
which ends by $i_m=i$. Define the ``right" functions:
$$
{\cal R}_{i}(x):=a_m,
~~~~~
{\cal R}_{i}^{-}(y):=b_m. 
$$

 It is  easy to see that the above functions are well-defined and independent of ${\bf i}$ chosen.

\vskip 2mm

For each simple reflection $s_i\in W$, set $s_{i^*}$ such that $w_0s_{i^*}=s_iw_0$.

Set ${\rm Ad}_v(g):=vgv^{-1}$. For any $u\in \U$, set $\wt{u}:={\rm Ad}_{\overline{w}_0}(u^{-1})\in \U^{-}$.

\bl 
\la{13.1.11.31h.1}
The map $u\mapsto \wt{u}$ is a positive birational isomorphism from  $\U$ to $\U^{-}$. Moreover, 
\be \la{13.1.11.31h}
\chi_i(u)=\chi_{i^*}^{-}(\wt{u}),~~~{\cal L}_i(u)={\cal R}_{i^*}(\wt{u}),~~~{\cal R}_i(u)={\cal L}_{i^*}(\wt{u})~~~~
\forall i\in I. 
\ee
\el

\begin{proof} Note that
$
{\rm Ad}_{\overline{w}_0}(x_i(-a))=y_{i^*}(a).
$ 
Let $u=x_{i_1}(a_1)\ldots x_{i_m}(a_m)\in \U$. Then
$$
\wt{u}={\rm Ad}_{\overline{w}_0}(u^{-1})
=y_{i_m^*}(a_m)\ldots y_{i_1^*}(a_1).
$$
Clearly it is a positive birational isomorphism. Identities in \eqref{13.1.11.31h} follow by definition. 
\end{proof}

\bl
\la{13.1.11.31h.2}
Let $h\in {\rm H}$, $x\in \U$ and $y\in \U^{-}$. For any $i\in I$, we have
\be \la{13.1.11.33h}
\chi_i\big({\rm Ad}_h(x)\big)=\chi_i(x)\cdot\alpha_i(h),~~~{\cal L}_i\big({\rm Ad}_h(x)\big)={\cal L}_i(x)\cdot\alpha_i(h),~~~ {{\cal R}_i\big({\rm Ad}_h(x)\big)={\cal R}_i(x)\cdot\alpha_i(h)}.
\ee
\be \la{13.1.11.34h}
{\chi_i^{-}\big({\rm Ad}_h(y)\big)=\chi_i(y)/\alpha_i(h),~~~{\cal L}_i^{-}\big({\rm Ad}_h(y)\big)={\cal L}_i^{-}(y)/\alpha_i(h),~~~  {\cal R}_i^{-}\big({\rm Ad}_h(y)\big)={\cal R}_i^{-}(y)/\alpha_i(h).}
\ee
\el

\begin{proof} Follows from the identities
$
{\rm Ad}_h\big(x_i(a)\big)=x_i(a\alpha_i(h))$ and ${\rm Ad}_h\big(y_i(a)\big)=y_i(a/\alpha_i(h)).
$ 
\end{proof}

\subsection{The positive morphisms $\Phi$ and $\eta$}\la{sec6.2.2}
We show that each $\chi_i$ is closely related to ${\cal L}_i^{-}$ by the following morphism.
\bd There exists a unique morphism $\Phi: {\rm U}^{-}\lra {\rm U}$ such that 
\be \la{13.map.Phi}
u_{-}{\rm B}= \Phi(u_{-})w_0{\rm B}.
\ee
\ed
\begin{lemma} \la{lem2}
For each $i \in I$,  one has
\be 1/{\cal L}_{i}^{-}=\chi_{i} \circ \Phi,~~~~1/\chi^{-}_i={\cal L}_i\circ \Phi
\ee
\end{lemma}

${\bf Example.}$ Let ${\rm G=SL}_3$. We have 
$$
y=y_1(b_1)y_2(b_2)y_1(b_3)=y_2(\frac{b_2b_3}{b_1+b_3})y_1(b_1+b_3)y_2(\frac{b_1b_2}{b_1+b_3}).
$$ 
$$
\Phi(y)=x_1(\frac{1}{b_1+b_3})x_2(\frac{b_1+b_3}{b_2b_3})x_1(\frac{b_3}{b_1(b_1+b_3)})=x_2(\frac{1}{b_2})x_1(\frac{1}{b_1})x_2(\frac{b_1}{b_2b_3}).
$$
$$
1/{\cal L}_{1}^{-}(y)=\chi_{1}(\Phi(y))=\frac{1}{b_1}, ~~~
1/{\cal L}_{2}^{-}(y)=\chi_{2}(\Phi(y))=\frac{b_1+b_3}{b_2b_3}.
$$
$$
1/{\chi}_{1}^{-}(y)={\cal L}_{1}(\Phi(y))=b_1+b_3, ~~~
1/{\chi}_{2}^{-}(y)={\cal L}_{2}(\Phi(y))=b_2.
$$

\vskip 2mm

The proof was suggested by the  proof of Proposition 3.2 of \cite{L2}.
\begin{proof} 
We prove the first formula. The second follows similarly by considering the inverse morphism $\Phi^{-1}:\U\ra \U^{-}$ such that
$
u\B^{-}=\Phi^{-1}(u)w_0\B^{-}.
$ 

Let $i\in I$. Let $w\in W$ such that  its length $l(w)<l(s_iw)$. We use two basic identities:
\be \la{10.1.bac1}
y_i(b)x_i(a)=x_i\big(a/(1+ab)\big)y_i\big(b(1+ab)\big)\alpha_i^\vee\big(1/(1+ab)\big).
\ee
\be \la{10.1.bac2}
y_i(b)w \B=x_i(1/b){s}_i{w} \B.
\ee
By \eqref{10.1.bac2}, one can change $y_i(b)$ on the most right to $x_i(1/b)$. By \eqref{10.1.bac1}, one can ``move" $y_i(b)$ from left to the right. After finite steps, we get 
\be
\la{phi}
y_{i_1}(b_1)y_{i_2}(b_2)...y_{i_m}(b_m){\rm B}=
y_{i_1}(b_1)x_{i_m}(a_m)x_{i_{m-1}}(a_{m-1})\ldots x_{i_{2}}(a_2){s}_{i_2}\ldots {s}_{i_m}{\rm B}.
\ee
The last step is to move the very left term $y_{i_1}(b_1)$ to the right. Let 
$$
f_s(c_1,c_2,\ldots, c_m)=x_{i_{m}}(c_m)x_{i_{m-1}}(c_{m-1})...x_{i_{s+1}}(c_{s+1})y_{i_1}(c_1)x_{i_s}(c_s)\ldots x_{i_2}(c_2){s}_{i_2}\ldots {s}_{i_m}{\rm B}.
$$
We will need the relations between $\{c_i\}$ and $\{c_i'\}$ such that 
$$ 
f_s(c_1,c_2,\ldots,c_m)=f_{s-1}(c_1',c_2',\ldots,c_m')
$$
By \eqref{10.1.bac1}-\eqref{10.1.bac2}, if $i_1\neq i_s$, then $c_p=c_p'$ for all $p$. If $i_1=i_s$, then
\begin{align}
&c_p'=c_p\quad \text{for } p=s+1,\ldots, m;\nonumber\\
& c_s'=c_s/(1+c_1c_s), \quad c_1'=c_1(1+c_1c_s);\nonumber\\
&c_p'=c_p(1+c_1c_s)^{-\langle\alpha_{i_1}^{\vee}, \alpha_{i_p}\rangle} \quad \text{for }p=2,\ldots, s-1.\nonumber
\end{align}

For each $q=f_s(c_1,c_2,\ldots,c_m)$, we set
\be \la{6.21.1.1}
h(q):=\frac{1}{c_1}+\sum_{p~|~i_p=i_1, ~p>s}c_p.
\ee
If  $i_s=i_1$, then
$$
\frac{1}{c_1'}+\sum_{p~|~i_p=i_1,~p>s-1}c_p'=\frac{1}{c_1(1+c_1c_s)}+\frac{c_s}{1+c_1c_s}+\sum_{p~|~i_p=i_1, ~p>s}c_p=\frac{1}{c_1}+\sum_{p~|~i_p=i_1,~p>s}c_p.
$$
Same is true for $i_s\neq i_1$. Therefore the function (\ref{6.21.1.1}) does not depend on $s$. 

Back to (\ref{phi}), we have
\begin{align}
u{\rm B}&=y_{i_1}(b_1)y_{i_2}(b_2)...y_{i_m}(b_m){\rm B} \nonumber\\
   &=y_{i_1}(b_1)x_{i_m}(a_m)...x_{i_{2}}(a_2){s}_{i_2}...{s}_{i_n}{\B}\nonumber\\
   &=x_{i_m}(c_m)...x_{i_2}(c_2)y_{i_1}(c_1){s}_{i_2}...{s}_{i_n}{\B}\nonumber\\
   &=x_{i_m}(c_m)...x_{i_2}(c_2)x_{i_1}(1/c_1){s}_{i_1}...{s}_{i_n}{\B}\nonumber\\
   &=\Phi(u){w}_0{\rm B}\nonumber
\end{align}
Hence $\Phi(u)=x_{i_m}(c_m)...x_{i_1}(c_2)x_{i_1}(1/c_1)$. Then
$$
\chi_{i_1}(\Phi(u))=\frac{1}{c_1}+\sum_{p~|~i_p=i_1, ~p>1}c_p=h(u{\rm B})=\frac{1}{b_1}=\frac{1}{{\cal L}_{i_1}^{-}(u)}.
$$
\end{proof}

\bl \la{10.1.phi}
The morphism $\Phi: \U^{-}\ra \U$ is a positive birational isomorphism with respect to Lusztig's positive atlases on $\U^-$ and $\U$.
\el

\begin{proof}
According to the algorithm in the proof of Lemma \ref{lem2}, clearly $\Phi$ is a positive morphism. By the same argument, one can show that $\Phi^{-1}$ is a positive morphism. The Lemma is proved.
\end{proof}

The morphism $\eta$ in \eqref{13.map.eta} is the right hand side version of $\Phi$, i.e. 
$
\B^{-} u=\B^{-} w_0 \eta(u).
$ 
Similarly, 

\begin{lemma} 
\la{10.1.eta}
The morphism $\eta: \U\ra \U^{-}$ is a positive birational isomorphism. 
Moreover,
\be \la{lem2.eta}
\forall i\in  I,~~~~~~ 1/{\cal R}_{i}=\chi_{j}^{-} \circ \eta,~~~~1/\chi_i={\cal R}_i^{-}\circ \eta.
\ee
\end{lemma}

\subsection{Birational isomorphisms $\phi_i$ of $\U$}
Let $i\in I$. Define
$$
z_i(a):=\alpha_i^{\vee}(a)y_i(-a),~~~z_i^*(a):=\alpha_i^{\vee}(1/a)y_i(1/a).
$$
Clearly $z_i(a)z_i^{*}(a)=1$. 

\blc 
\la{13.1.4.26h}
There is a birational isomorphism
\be \la{map.phi_i}
\phi_i: \U\stackrel{\sim}{\lra}\U,~~~ u\lms \overline{s}_i\cdot u \cdot z_i\big(\chi_i(u)\big).
\ee
\elc

{\bf Remark.} The map $\phi_i$ is not a positive birational isomorphism.

\begin{proof}
We need the following identities: 
\be \la{13.1.4.21h}
\overline{s}_ix_i(a)z_i(a)=x_i(-1/a).
\ee
\be
z_i^{*}(a) x_i(b-a) z_i(b)=x_i(1/a-1/b).
\ee
If $j\neq i$, then
\be \la{13.1.4.22h}
z_i^*(a)x_j(b)z_i(a)=x_j(ba^{-\langle \alpha_i^{\vee}, \alpha_j\rangle})
\ee

Let ${\bf i}=(i_1,i_2,\ldots, i_m)$ be a reduced word for $w_0$ such that $i_1=i$. For each $s\in [1,m]$, define
$$
I_s^{{\bf i},i}:=\{p\in [1,s]~|~i_p=i\}.
$$
Let $u=x_{i_1}(a_1)\ldots x_{i_m}(a_m)\in \U$. Set 
$d_s:=\sum_{k\in I_s^{{\bf i},i}} a_k.$
In particular, $d_1=a_1$, $d_m=\chi_i(u)$.

Let us assume that $u\in \U$ is generic, so that $d_s\neq 0$ for all $s\in [1,m]$. By \eqref{13.1.4.21h}-\eqref{13.1.4.22h}, we get
\begin{align}
\phi_i(u)=&\overline{s}_i\cdot x_{i_1}(a_1)x_{i_2}(a_2)\ldots x_{i_m}(a_m) \cdot z_i\big(\chi_i(u)\big)\nonumber\\
=&\big(\overline{s}_ix_{i_1}(a_1)z_i(d_1)\big)\cdot \big(z_i^*(d_1)x_{i_2}(a_2)z_i(d_2)\big)\cdot \ldots \cdot \big(z_i^{*}(d_{m-1})x_{i_m}(a_m)z_i(d_m)\big)\nonumber\\
=&x_{i_1}(a_1')x_{i_2}(a_2')\ldots x_{i_m}(a_m').
\end{align}
Here $a_1'=-1/d_1$. For $s>1$,  
\be \la{13.1.4.23h}
a_s'=\left\{\begin{array} {cc}1/d_{s-1}-1/d_s,&\text{if }i_s=i,\\a_sd_s^{-\langle \alpha_i^{\vee}, \alpha_{i_s}\rangle}, &\text{if }i_s\neq i .\\ \end{array}\right.
\ee
Thus $\phi_i(u)\in \U$. The map $\phi_i$ is well-defined.
By \eqref{13.1.4.23h}, we have $\chi_i(\phi_i(u))=-1/\chi_i(u)$. Therefore
$$
\phi_i\circ \phi_i(u)=\overline{s}_i\cdot \overline{s}_i\cdot u \cdot z_i(\chi_i(u))\cdot z_i(-1/\chi_i(u))=\overline{s}_i^2\cdot u \cdot \overline{s}_i^2.
$$
Since $\overline{s}_i^4=1$, we get $\phi_i^4={\rm id}$. Therefore $\phi_i$ is birational.
\end{proof}

Let $\lambda\in {\rm P}^+$. Recall $t^{\lambda}\in {\rm Gr}$. Recall the  ${\G}({\cal O})$-orbit ${\rm Gr}_{\lambda}$ of $t^{\lambda}$ in ${\rm Gr}$.

\bl \la{13.1.8.53h}
Let $l\in \U(\Z^t)$. For any $u\in {\cal C}_l^{\circ}$, the element
$u\cdot t^{\lambda} \in \overline{{\rm Gr}_{\lambda}}$ if and only if  $l\in \U_{\chi}^+(\Z^t).$
\el

\begin{proof}
If $l\in \U_{\chi}^+(\Z^t)$, by Lemma \ref{Lema10.1.1}, we see that $u\in \U({\cal O})$. 
Hence $u\cdot t^{\lambda} \in \overline{{\rm Gr}_{\lambda}}$.

If  $\chi^t(l)=\min_{i\in I}\{\chi_i^t(l)\}<0,$ 
then pick $i$  such that $\chi_i^{t}(l)<0$. Set $\mu:=\lambda-\chi_i^t(l)\cdot\alpha_i^{\vee}$. 
Since $y_i(t^{\langle \alpha_i, \lambda\rangle}/\chi_i(u))\in {\rm G}({\cal O})$, we get
\be \la{13.1.4.25h}
z_i^*(\chi_i(u))\cdot t^{\lambda}=\alpha_i^{\vee}(1/\chi_i(u))\cdot t^{\lambda}\cdot y_i(t^{\langle \alpha_i, \lambda\rangle}/\chi_i(u))=\alpha_i^{\vee}(1/\chi_i(u))\cdot t^{\lambda}=t^{\mu}.
\ee
Recall  the $\U_{w}({\cal K})$-orbit  ${\rm S}_w^{\nu}$ of $t^\nu$ in ${\rm Gr}$.
We have
\be
u\cdot t^{\lambda}=u z_i(\chi_i(u))\cdot z_i^*(\chi_i(u)) t^{\lambda}\stackrel{\eqref{13.1.4.25h}}{=} u z_i(\chi_i(u)) \cdot t^{\mu}\stackrel{\eqref{map.phi_i}}{=}\overline{s}_i^{-1}\phi_i(u) \overline{s}_i\cdot  t^{s_i(\mu)}\in {\rm S}_{s_i}^{s_i(\mu)}.
\ee
 It is well-known that the intersection
 ${\rm S}_{w}^{\nu}\cap \overline{{\rm Gr}_{\lambda}}$ is  nonempty if and only if  $t^{\nu}\in \overline{{\rm Gr}_{\lambda}}.$
In this case $t^{s_i(\mu)}\notin \overline{{\rm Gr}_{\lambda}}$.
Therefore  ${\rm S}_{s_i}^{s_i(\mu)}\cap \overline{{\rm Gr}_{\lambda}}$ is empty. Hence $u\cdot t^{\lambda}\notin \overline{{\rm Gr}_{\lambda}}$.
\end{proof}



\section{A positive structure on the configuration space ${\rm Conf}_{\rm I}({\cal A}; {\cal B})$}
\la{sec5}

\subsection{Left  ${\rm G}$-torsors} 
\la{sec5.1}

Let ${\rm G}$ be a group. 
Let ${X}$ be a left principal homogeneous ${\rm G}$-space, also known as a left {\rm G}-torsor. Then for any  $x, y \in {X}$ there exists a unique 
 $g_{x,y}\in {\rm G}$ such that
$
x= g_{x,y}y.
$
Clearly, 
\be \la{5.251210}
g_{x,y}g_{y,z} = g_{x,z}, \qquad g_{gx,y}=gg_{x,y},~g_{x,gy}=g_{x,y}g^{-1}, ~~g\in {\rm G}.
\ee
Given a reference point $p\in {X}$, one defines a ``$p$-distance from $x$ to $y$'':
\be \la{5.19.12.1as}
g_p(x,y):= g_{p, x}g_{y, p} \in {\rm G}. 
\ee
If $i_p: X \to {\rm G}$ is a unique isomorphism of ${\rm G}$-sets  such that $i_p(p)=e$, 
then $g_p(x,y)= i_p(x)^{-1}i_p(y)$. 

\bl \la{LEMMA.4.2}
One has:
\begin{align} 
g_p(x,y)g_p(y, z) &= g_p(x,z). \la{5.25.12.1}\\
g_p(gx,gy) &= g_p(x,y),~~~~~~g\in \G. \la{5.25.12.1.ll}\\ 
y&=g_{p}(p,y)\cdot p. \la{ll.5.25.12.1}
\end{align}
\el

\begin{proof} Indeed, 
$$
g_p(x,y)g_p(y, z) = g_{p,x}g_{y,p}g_{p,y}g_{z,p}=g_{p,x}g_{z,p}=g_{p}(x,z),$$
$$
g_p(gx,gy)   = g_{p, gx}g_{gy, p} \stackrel{(\ref{5.251210})}{=} 
g_{p, x}g^{-1} gg_{y,  p} =  g_{p, x}g_{y,  p} =  g_p(x,y),
$$
$$
y=g_{y,p}\cdot p=g_{p,p}g_{y,p}\cdot p=g_{p}(p,y)\cdot p.
$$
\end{proof}

Recall  ${\cal F}_{{\rm G}}$ 
in Definition \ref{torsorF}. 
From now on, we apply the above construction  in the set-up
$$X={\cal F}_{\rm G},\quad p=\{\U,\B^{-}\}.$$

Pick a collection $\{{\rm A}_1, ..., {\rm A}_n\}$ representing a configuration  in ${\rm Conf}_n({\cal A})$. 
We assign  ${\rm A}_i$
 to the vertices of a convex n-gon, so that they go 
clockwise around the polygon. 
Each oriented pair  $\{{\rm A}_i, {\rm A}_j\}$ provides  a 
frame $\{{\rm A}_i, {\rm B}_j\}$, shown on Fig \ref{frame} by an arrow with a white dot.

\begin{figure}[ht]
\epsfxsize=2in
\centerline{\epsfbox{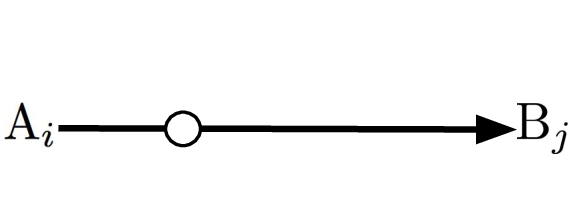}}
\caption{A frame $\{\A_i,\B_j\}$.}
\label{frame}
\end{figure}

\subsection{Basic invariants associated to a generic configuration}
\la{sec5.2}

We introduce several invariants that will be useful in the rest of this paper.
We employ $\cdot$ to denote the action of $\G$ on (decorated) flags.

\paragraph{The invariant $u^{\A_2}_{\B_1, \B_3} \in {\rm U}$.} 
Let $(\B_1, \A_2, \B_3)\in {\rm Conf}({\cal B,A,B})$ be a generic configuration. Set
\be \la{8.28.10.30h}
u_{{\B}_1,\B_3}^{\A_2}:=g_{\{\U, \B^-\}}(\{\A_2, \B_1\},\{\A_2,\B_3\}).
\ee
By \eqref{5.25.12.1.ll}, the invariant $u_{\B_1, \B_3}^{\A_2}$ is independent of  the representative chosen. 
Clearly, $u_{\B_3, \B_2}^{\A_1}\in \U$.

\paragraph{The invariant $h_{{\rm A}_1, {\rm A}_2}\in {\rm H}$.} Let $(\A_1, \A_2)$ be a generic configuration. There is a unique element $h_{{\rm A}_1, {\rm A}_2}\in {\rm H}$ such that 
\be \la{4.25.12.2}
({\rm A}_1, {\rm A}_2) = ({\rm U}, ~h_{{\rm A}_1, {\rm A}_2}\overline w_0\cdot {\rm U}).
\ee
Using the notation ({\ref{5.19.12.1as}}), we have
\be \la{8.28.10.28h}
h_{\A_1,\A_2}\overline{w}_0=g_{\{\U,\B^-\}}(\{\A_1,\B_2\},\{\A_2,\B_1\}).
\ee

\paragraph{The invariant $b_{\rm B_3}^{\rm A_1, A_2}\in {\rm B^{-}}$.}
Let $({\rm A_1, A_2, B_3})$ be a generic configuration. Define
$$
b_{\rm B_3}^{\rm A_1, A_2}:=g_{\{\U, \B^{-}\}}(\{\A_1,\B_3\},\{\A_2,\B_3\})\in {\rm B}^{-}. 
$$

\paragraph{Relations between basic invariants.} 
Let $({\rm A}_1, ..., {\rm A}_n) \in {\rm Conf}^*_n({\cal A})$. Set
\be \la{3.23.12.1}
h_{ij}:= h_{{\rm A}_i, {\rm A}_j}\in {\rm H}, \quad u_{ik}^j := u^{\A_j}_{\B_i, \B_k}\in  {\rm U}, \quad b_{k}^{ij}:=b_{\B_k}^{\A_i, \A_j}\in \B^{-}.
\ee 
We denote these invariants by dashed arrows, see Fig \ref{cal8}.

  \begin{figure}[ht]
\centerline{\epsfbox{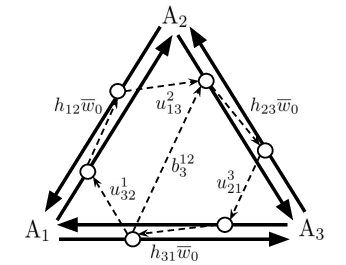}}
\caption{Invariants of a configuration.}
\label{cal8}
\end{figure}
{
\bl \la{3.23.12.1aa} The data (\ref{3.23.12.1}) satisfy the following relations:
\begin{enumerate}
\item $h_{12}\overline{w}_0h_{21}\overline{w}_0=1$.
\item $u_{23}^1u_{34}^1=u_{24}^1$, in particular $u_{23}^1u_{32}^1=1$.
\item $b_{4}^{12}b_{4}^{23}=b_{4}^{13}$.
\item $b_{3}^{12}=u^{1}_{32}h_{12}\overline{w}_0u^{2}_{13}=h_{13}\overline{w}_0u_{12}^3\overline{w}_0^{-1}h_{23}^{-1}.$
\item $u_{32}^1h_{12}\overline{w}_0u_{13}^2h_{23}\overline{w}_0u_{21}^3h_{31}\overline{w}_0=1.$ 
\end{enumerate}
\el

\begin{proof} 
We prove the first identity of 4. The others follow similarly. Let $p=\{\U,\B^{-}\}$.
Let $$x_1=\{\A_1, \B_3\},~x_2=\{\A_1, \B_2\}, ~x_3=\{\A_2, \B_1\}, ~x_4=\{\A_2, \B_3\}.$$
As illustrated by Figure \ref{cal8},
$$
b_{3}^{12}=g_{p}(x_1,x_4),~u^{1}_{32}=g_p(x_1,x_2),~h_{12}\overline{w}_0=g_p(x_2, x_3),~u_{13}^2=g_p(x_3, x_4).
$$
By \eqref{5.25.12.1},  we get $g_{p}(x_1,x_4)=g_p(x_1,x_2)g_p(x_2,x_3)g_p(x_3,x_4)$. 
\end{proof}

\bl \la{8.28.10.39hhh} 
 Let $x\in {\rm Conf}({\cal A}, {\cal A}, {\cal B})$ be a generic configuration. Then it has a unique representative $\{\A_1,\A_2,\B_3\}$  with $\{\A_1, \B_3\}=\{\U, \B^{-}\}$. Such a representative is
  \be \la{8.28.10.39hh}
  \{\U, u_{32}^1h_{12}\overline{w}_0\cdot \U, \B^{-}\}.
  \ee
\el
  \begin{figure}[ht]
\centerline{\epsfbox{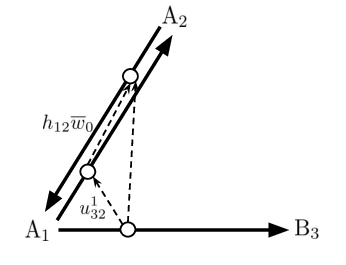}}
\caption{Invariants of a configuration $({\A}_1, {\A}_2, {\B}_3)$.}
\label{inv2}
\end{figure}

\begin{proof}
The existence and uniqueness are clear. It remains to show that it is $\eqref{8.28.10.39hh}$. 
By Fig \ref{inv2}, 
\be \la{8.28.10.39h}
g_{\{\U, \B^{-}\}}(\{\A_1,\B_3\},\{\A_2,\B_1\})=u_{32}^1h_{12}\overline{w}_0.
\ee

If $\{\A_1,\B_3\}=\{\U,\B^{-}\}$, then by \eqref{ll.5.25.12.1}, we get
$$
\{\A_2,\B_1\}=g_{\{\U, \B^{-}\}}(\{\A_1,\B_3\},\{\A_2,\B_1\})\cdot \{\U, \B^{-}\}=\{u_{32}^1h_{12}\overline{w}_0\cdot\U, \B\}.
$$
\end{proof}
}

Each $b\in \B^{-}$ can be decomposed as $b=y_l\cdot h=h\cdot y_r$ where $h\in {\rm H}$, $y_l,y_r\in {\U^{-}}$. Thus $\B^{-}$ has a positive structure induced by  positive structures  on  $\U^{-}$ and ${\rm H}$. 
There are three positive maps
\be \la{13.1.12.51h}
\pi_l, \pi_r: \B^{-}\lra \U^{-},~\pi_h : \B^{-} \lra {\rm H}, ~~\pi_l(b)=y_l, ~\pi_r(b)=y_r,~ \pi_h(b)=h.
\ee

These maps give rise to three more invariants.

\paragraph{The invariant $\mu_{\B_3}^{\A_1,\A_2}\in {\rm H}$.} For each generic $(\A_1,\A_2,\B_3)$, we define
\be
\mu_{\B_3}^{\A_1,\A_2}:= \pi_{h}(b_{\B_3}^{\A_1,\A_2}).
\ee

\paragraph{The invariant $r_{\B_3}^{\B_1,\A_2}\in \U^{-}$.} 
 For any $h\in {\rm H}$, we have
\be \la{12.12.15.1h}
b_{\B_3}^{\A_1\cdot h^{-1}, \A_2}= h\cdot b_{\B_3}^{\A_1, \A_2}.
\ee
Thus we can define 
\be \la{inv.r}
r_{\B_3}^{\B_1,\A_2}:=\pi_r(b_{\B_3}^{\A_1\cdot h^{-1}, \A_2})=\pi_r(b_{\B_3}^{\A_1, \A_2})\in \U^{-}.
\ee

\paragraph{The invariant $l_{\B_3}^{\A_1, \B_2}\in \U^{-}$.}
 For any $h\in {\rm H}$, we have
\be \la{12.12.15.2h}
b_{\B_3}^{\A_1, \A_2\cdot h}=  b_{\B_3}^{\A_1, \A_2}\cdot h.
\ee
Define 
\be
l_{\B_3}^{\A_1,\B_2}:=\pi_l(b_{\B_3}^{\A_1, \A_2\cdot h})=\pi_l(b_{\B_3}^{\A_1, \A_2})\in \U^{-}.
\ee

{
For simplicity, we set
\be \la{13.3.16.1101}
\mu_{k}^{ij}:=\mu_{\B_k}^{\A_i, \A_j}\in {\rm H}, ~~~r_{k}^{ij}:=r_{\B_k}^{\B_i, \A_j}\in \U^{-},~~~l_{k}^{ij}:=l_{\B_k}^{\A_i, \B_j}\in\U^{-}.
\ee
Recall that $\wt{u}=\overline{w}_0u^{-1}\overline{w}_0^{-1}$.
By Relations 3, 4 of Lemma \ref{3.23.12.1aa}, we get
\be
\mu_{4}^{12}\mu_{4}^{23}=\mu_{4}^{13}.
\ee
\be \la{12.12.hh}
b_{3}^{12}=l_3^{12}\mu_{3}^{12}=\mu_{3}^{12} r_{3}^{12}=u_{32}^1 h_{12} \overline{w}_0 u_{13}^2=h_{13}\wt{u_{21}^3}h_{23}^{-1}.
\ee

Recall the morphisms $\Phi$, $\eta$ and $\beta$ in Section \ref{sec4}. By the definition of these morphisms, we get
\bl \la{12.12.11.h}
We have
\begin{enumerate}
\item $u_{32}^1=\Phi\big(l_3^{12}\big)$.
\item $r_{3}^{12}=\eta\big(u_{13}^2\big)$.
\item $\wt{u_{21}^{3}}={\rm Ad}_{h_{13}^{-1}}\big(l_3^{12}\big)={\rm Ad}_{h_{23}^{-1}}\big(r_3^{12}\big)$.
\item $\mu_{3}^{12}=h_{12}\beta(u_{13}^2)=h_{13}h_{23}^{-1},~~~\beta(u_{13}^2)=h_{13}h_{23}^{-1}h_{12}^{-1}$.
\end{enumerate}
\el
\begin{proof}
By \eqref{12.12.hh}, we have
$$
l_3^{12}\mu_3^{12}=u_{32}^1(h_{13}\overline{w}_0u_{13}^2).
$$
The first identity follows. Similarly, the second identity follows from 
$$
\mu_{3}^{12} r_{3}^{12}=(u_{32}^1h_{13}\overline{w}_0)u_{13}^2
$$
The third identity follows from
$$
l_{3}^{12}\mu_{3}^{12}=h_{13}\wt{u_{21}^3}h_{23}^{-1}={\rm Ad}_{h_{13}}(\wt{u_{21}^3}) h_{13}h_{23}^{-1},\quad \quad \mu_{3}^{12}r_{3}^{12}=h_{13}h_{23}^{-1}{\rm Ad}_{h_{23}}(r_{3}^{12}).
$$
The identity $\mu_{3}^{12}=h_{12}\beta(u_{13}^2)$ follows from
$$
\mu_{3}^{12}r_{3}^{12}=u_{32}^1h_{12}\cdot (\overline{w}_0u_{13}^2).
$$
The identity $\mu_3^{12}=h_{13}h_{23}^{-1}$ follows from
$$
l_{3}^{12}\mu_{3}^{12}={\rm Ad}_{h_{13}}(\wt{u_{21}^3}) h_{13}h_{23}^{-1}.
$$
\end{proof}

\bl \la{lem1}
We have
\be \la{8.13.4a}
\chi(u_{21}^3)=\sum_{i\in I}\frac{\alpha_i(h_{13})}{{\cal L}_i(u_{32}^1)}=\sum_{i\in I}\frac{\alpha_i(h_{23})}{{\cal R}_i(u_{13}^2)}.
\ee
\be \la{13.1.12.1h}
\alpha_i(h_{12})=\alpha_{i^*}(h_{21}),~~~\forall i\in I.
\ee
\el

\begin{proof}
Use Lemmas \ref{13.1.11.31h.1}, \ref{12.12.11.h}, \ref{13.1.11.31h.2} and \ref{lem2} , we get
$$
\chi(u_{21}^3)=\chi^{-}(\wt{u_{21}^3})=\chi^{-}\big({\rm Ad}_{h_{13}^{-1}}(l_3^{12})\big)=\sum_{i\in I}\alpha_i(h_{13})\chi_i^{-}(l_{3}^{12})=\sum_{i\in I}\frac{\alpha_i(h_{13})}{{\cal L}_i(u_{32}^1)}.
$$
By the same argument, we get the other identity in \eqref{8.13.4a}. 
By Relation 1 of Lemma \ref{lem2}, we get
$$
h_{12}=\overline{w}_0h_{21}^{-1}\overline{w}_0^{-1}\cdot s_{\rm G}. 
$$
Then \eqref{13.1.12.1h} follows. 
\end{proof}
}

\subsection{A positive structure on ${\rm Conf}_{\rm I}({\cal A; {\cal B}})$}
\la{proofmth1}

Let ${\rm I}\subset [1,n]$ be a nonempty subset of cardinality $m$.  Following \cite[Section 8]{FG1}, there is a positive structure on the configuration space ${\rm Conf}_{\rm I}({\cal A};{\cal B})$. We briefly recall it below.

\vskip 3mm

Let $x=(x_1,\ldots, x_n)\in {\rm Conf}_{\rm I}({\cal A}; {\cal B})$ be a generic configuration such that 
\be \la{13.1.12.10h}
 x_i=\A_i\in {\cal A}\mbox{ when } i\in{\rm I}, \text{ otherwise }  x_i=\B_i\in{\cal B}. 
\ee

Set $\B_j:=\pi(\A_j)$ when $j\in {\rm I}$. 
Let $i\in {\rm I}$. For each $k\in [2,n]$, set
\be
u_{k}^i(x):=u_{\B_{i+k}, \B_{i+k-1}}^{\A_i},~~~~~\mbox{ where the subscript is modulo $n$}.\ee

 For each pair $i, j\in {\rm I}$, recall
\be \la{13.3.15.321h}
\pi_{ij}(x):=\left\{\begin{array}{lc}h_{\A_i, \A_j}, ~~&\text{ if } i<j,\\
h_{s_{\G}\cdot \A_i, \A_j},& \text { if }i>j.\end{array}\right.
\ee




\bl \la{7.22.1.1s}
Fix $i\in {\rm I}$. The following morphism is birational
$$
\alpha_{i}: ~{\rm Conf}_{\rm I}({\cal A}; {\cal B})\lra {\rm H}^{m-1}\times \U^{n-2},~~~x\lms (\{\pi_{ij}(x)\}, \{u_k^i(x)\}),~~j\in {\rm I}-\{i\},~k\in[2,n-1].
$$
\el

{\bf Example.} Fig \ref{alpha} illustrates the map $\alpha_1$ for ${\rm I}=\{1,3,5\}\subset[1,6]$.

 \begin{figure}[ht]
 \epsfxsize=2in
\centerline{\epsfbox{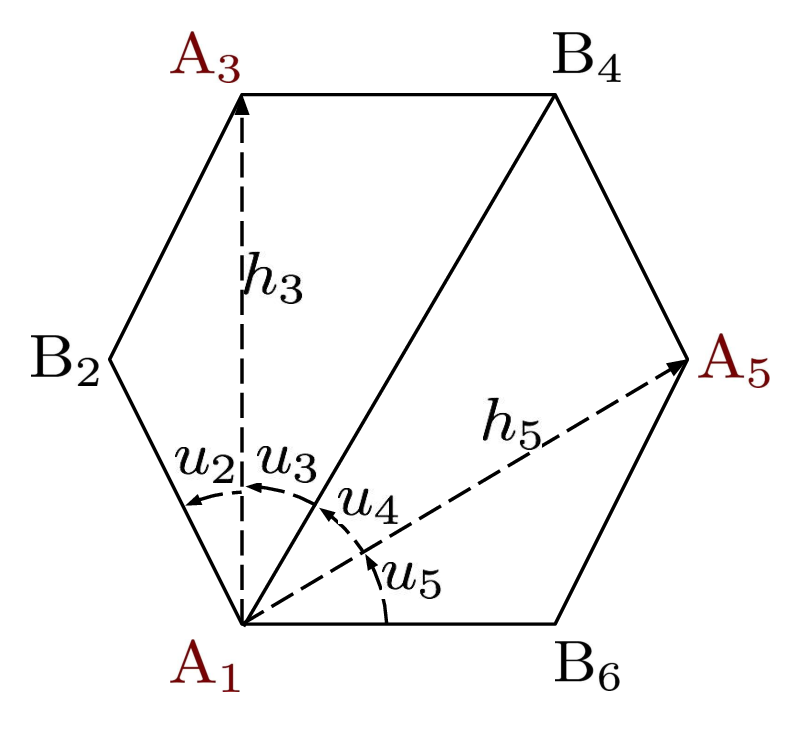}}
\caption{The map $\alpha_1$ for ${\rm I}=\{1,3,5\}\subset [1,6]$.}
\label{alpha}
\end{figure}

\begin{proof} 
Assume that $i=1\in {\rm I}$. 
Clearly $\alpha_1$ is well defined on the subspace
$$
\wt{\rm Conf}_{\rm I}({\cal A}; {\cal B}):=\{(x_1,\ldots, x_{n})~|~ (x_1, x_k) \text{ is generic for all } k\in[2,n]\}.
$$
Note that $\wt{\rm Conf}_{\rm I}({\cal A}; {\cal B})$ is dense in ${\rm Conf}_{\rm I}({\cal A}; {\cal B})$. We prove the Lemma by showing that $\alpha_1$ is a bijection from $\wt{\rm Conf}_{\rm I}({\cal A}; {\cal B})$ to ${\rm H}^{m-1}\times \U^{n-2}$,
\vskip 2mm

Let $y=(\{h_j\}, \{u_k\})\in {\rm H}^{m-1}\times \U^{n-2}$. Set $u_n':=1$. Set
$u_k':=u_{n-1}\ldots u_k$ for  $k\in [2,n-1].$ 
Let $x=(x_1,\ldots, x_n)\in \wt{\rm Conf}_{\rm I}({\cal A}; {\cal B})$ such that
\be
\la{9.3.2014.1.41h}
x_1:=\U; ~~~~ x_j:=u_j'h_j\overline{w}_0\cdot \U\in {\cal A},~ j\in {\rm I}-\{1\};~~~~x_k:=u_k'\cdot \B^{-}\in {\cal B},~ k\notin {\rm I}.
\ee
Clearly $\alpha_1(x)=y$. Hence $\alpha_1$ is a surjection. 

{Let $x\in \wt{\rm Conf}_{\rm I}({\cal A}; {\cal B})$ such that $\alpha_1(x)=y$.
Note that $x$ has a unique representative $\{x_1,\ldots, x_n\}$ such that $\{x_1, x_n\}=\{\U, \B^-\}$ if $n\notin {\rm I}$, and $\{x_1, \pi(x_n)\}=\{\U, \B^-\}$ if $n\in {\rm I}$.
By Lemma \ref{8.28.10.39hhh}, each $x_i$ is uniquely expressed by \eqref{9.3.2014.1.41h}. The injectivity of $\alpha_1$ follows.}
\end{proof}

The product ${\rm H}^{m-1}\times {\U}^{n-2}$ has a positive structure induced by the ones on ${\rm H}$ and $\U$. 

When ${\rm I}=[1,n]$, we first  introduce a positive structure on ${\rm Conf}_n({\cal A})$ such that the map $\alpha_1$ is a positive birational isomorphism. 
Such a positive structure is twisted cyclic invariant: 

\bt [{\cite[Section 8]{FG1}}] \la{12.9.FG1}
The following map is  a positive birational isomorphism
$$t: {\rm Conf}_n({\cal A})\stackrel{\sim}{\lra} {\rm Conf}_n({\cal A}), ~~~ (\A_1,\ldots, \A_n)\lms (\A_2, \ldots, \A_{n}, \A_1\cdot s_{\rm G}).$$
\et

Each $\alpha_i$ determines a positive structure on ${\rm Conf}_n({\cal A})$. Theorem  \ref{12.9.FG1}  tells us that  these positive structures coincide. We prove the same result for ${\rm Conf}_{\rm I}({\cal A};{\cal B})$, using the following Lemmas.

\bl 
\la{13.1.10.1h}
Let  ${\cal Y}$ be a space equipped with two positive structures denoted by ${\cal Y}^1$ and ${\cal Y}^2$. If for every rational function $f$ on ${\cal Y}$, we have
$$  
f \text{ is positive on } {\cal Y}^1 \Longleftrightarrow f \text{ is positive on } {\cal Y}^2,
$$ 
then ${\cal Y}^1$ and ${\cal Y}^2$ share the same positive structure.
\el
\begin{proof}
It is clear.
\end{proof}

\bl 
\la{13.1.10.2h}
Let ${\cal Y}, {\cal Z}$ be a pair of positive spaces. 
If  there are two positive maps $\gamma: {\cal Y}\ra {\cal Z}$ and $\beta:{\cal Z}\ra{\cal Y}$ such that $\beta\circ \gamma={\rm id}_{\cal Y}$, then for every rational function $f$ on ${\cal Y}$ we have
$$
f \text{ is positive on } {\cal Y} \Longleftrightarrow \beta^*(f) \text{ is positive on } {\cal Z}.
$$
\el
\begin{proof} If $f$ is positive on ${\cal Y}$, since $\beta$ is a positive morphism, then $\beta^*(f)$ is positive on ${\cal Z}$.

If $\beta^*(f)$ is positive on ${\cal Z}$, since $\gamma$ is a positive morphism, then 
$\gamma^*(\beta^*(f))=f$ is positive.
\end{proof}

\bl 
\la{13.1.10.3h}
Every $\alpha_i$ $(i\in {\rm I})$ determines the same positive structure on ${\rm Conf}_{\rm I}({\cal A}; {\cal B})$.
\el
{\bf Remark.} Lemma \ref{13.1.10.3h} is equivalent to say that for any pair $i,j\in{\rm I}$,  the map $\phi_{i,j}:=\alpha_{i}\circ\alpha_{j}^{-1}$ is a positive birational isomorphism of ${\rm H}^{m-1}\times {\U}^{n-2}$.

\begin{proof}
 Let us temporary denote the positive structure on ${\rm Conf}_{\rm I}({\cal X}; {\cal Y})$ by ${\rm Conf}_{\rm I}^i({\cal A}; {\cal B})$ such that $\alpha_i$ is a positive birational isomorphism.

There is a projection $\beta: {\rm Conf}_n({\cal A})\ra {\rm Conf}_{\rm I}({\cal A}; {\cal B})$ which maps $\A_k$ to $\A_k$ if $k\in {\rm I}$ and maps $\A_k$ to $\pi(\A_k)$ otherwise.
By Lemma \ref{12.9.FG1}, $\beta$ is a positive morphism for all ${\rm Conf}_{\rm I}^i({\cal A};{\cal B})$.

Fix $i\in {\rm I}$. Each generic $x=(x_1,\ldots, x_n)\in {\rm Conf}_{\rm I}({\cal A}; {\cal B})$ has a unique preimage $\gamma^i(x):=(\A_1,\ldots, \A_n)\in {\rm Conf}_n({\cal A})$ such that 
$$
\mbox{$\A_j=x_j$ when $j\in {\rm I}$, otherwise $\A_j$ is the preimage of $x_j$ such  that $\pi_{ij}(\gamma^i(x))=1$.}
$$
\noindent 
Clearly $\gamma^i$ a positive morphism
from ${\rm Conf}_{\rm I}^i({\cal A}; {\cal B})$ to ${\rm Conf}_n({\cal A}).$
By definition $\beta\circ \gamma^i={\rm id}$. 

Let $f$ be a rational function on  ${\rm Conf}_{\rm I}({\cal A}; {\cal B})$. Let $i, j\in {\rm I}$. By Lemma \ref{13.1.10.1h}, 
$$
f \text{ is positive on }  {\rm Conf}_{\rm I}^i({\cal A}; {\cal B}) \Longleftrightarrow \beta^*(f) \text{ is positive on }  {\rm Conf}_{n}({\cal A}) \Longleftrightarrow f \text{ is positive on }  {\rm Conf}_{\rm I}^j({\cal A}; {\cal B}). 
$$
This Lemma follows from Lemma \ref{13.1.10.2h}.
\end{proof}
Thanks to Lemma \ref{13.1.10.3h}, we introduce a canonical positive structure on ${\rm Conf}_{\rm I}({\cal A}; {\cal B})$. From now on, we view ${\rm Conf}_{\rm I}({\cal A}; {\cal B})$ as a positive space. 

\vskip 3mm
Given $k\in \Z/n$, we define the $k$-shift of the subset ${\rm I}$ by setting 
$ 
{\rm I}(k):=\{i \in [1,n]~|~ i+k\in {\rm I}\}.
$ 
The following Lemma is clear now.
\bl \la{12.9.FG1.g}
The following map is a positive birational isomorphism
$$t: {\rm Conf}_{\rm I}({\cal A};{\cal B})\stackrel{\sim}{\lra} {\rm Conf}_{{\rm I}(1)}({\cal A};{\cal B}), ~~~ (x_1,\ldots, x_n)\lms (x_2, \ldots, x_{n}, x_1\cdot s_{\rm G}).$$
\el

{
\paragraph{An invariant definition of positive structures.} 
We have defined above positive structures on the configuration spaces using 
pinning in $\G$, which allows to make calculations.  Let us explain now how to  define 
positive structures on the configurations spaces without choosing a pinning. When $\G$ is of type $A_m$, such a definition is given in \cite[Section 9]{FG1}. In general, given a decomposition of the longest Weyl group element $w_0=s_{i_1}\ldots s_{i_n}$, for each generic pair $\{\B, \B'\}$ of flags, there exists a unique chain $$\B=\B_0\stackrel{i_1}{\lra}\B_1\stackrel{i_2}{\lra}\ldots\stackrel{i_{n-1}}{\lra}\B_{n-1}\stackrel{i_n}{\lra}\B_n=\B'.$$
Here $\B_{k-1}\stackrel{i_k}{\rightarrow}\B_k$ indicates that $\{\B_{k-1}$, $\B_k\}$ is in the 
 position $s_{i_k}$. The positive structure of ${\rm Conf}({\cal B}, {\cal A}, {\cal B})$ can be defined via the birational map
$$
{\rm Conf}({\cal B}, {\cal A}, {\cal B})\lra ({\Bbb G_m})^n, \quad (\B, \A, \B')\lms (\chi^o(\B_0, \A, \B_1), \chi^o(\B_1,\A, \B_2),\ldots, \chi^o(\B_{n-1},\A, \B_n)).
$$}{
Each generic pair $\{\A, \A'\}\in {\cal A}^2$ uniquely determines a pinning for $\G$ such that 
$$x_i(a)\in \U_{\A}, ~~~\chi_\A(x_i(a))=a, ~~~ y_i(a)\in \U_{\A'}, ~~~i\in I.$$ 
The pinning gives rise to a representative $\overline{w}_0\in \G$ of $w_0$. There is a unique element $h\in \pi(\A)\cap \pi(\A')$ such that
$$
\A'=h\overline{w}_0\cdot \A.
$$
Such an element $h$ gives rise to a birational map from ${\rm Conf}_2({\cal A})$ to the Cartan group of $\G$, determining a positive structure of ${\rm Conf}_2({\cal A})$. The positive structures of general configuration spaces are defined via the positive structures of ${\rm Conf}_2({\cal A})$ and ${\rm Conf}({\cal B}, {\cal A},{\cal B})$.}

\subsection{Positivity of the potential ${\cal W}_{\rm J}$ and proof of Theorem \ref{13.2.22.2226h}}
\la{sec6.4}

Let ${\rm J}\subset {\rm I}\subset[1,n]$.
Consider the ordered triples  $\{i,j,k\}\subset[1,n]$ such that
\be \la{13.1.11.12h}
j\in {\rm J}, \text{ and } i,j,k {\rm ~seated~clockwise}.
\ee

Let $x\in {\rm Conf}_{\rm I}({\cal A}; {\cal B})$ be presented by \eqref{13.1.12.10h}. Define
$
p_{j;i,k}(x):=u_{\B_i,\B_k}^{\A_j}.$ 
In particular, we are interested in the triples $\{j-1,j, j+1\}$. Set
\be \la{fun.u.13.1.11.h}
p_j(x):=p_{j;j-1,j+1}=u_{\B_{j-1}, \B_{j+1}}^{\A_j}, ~~\forall j\in {\rm J}.
\ee

\bl \la{12.9.pos}
The following morphisms are positive morphisms
\begin{itemize}
\item [1.] $\pi_{ij}: {\rm Conf}_{\rm I}({\cal A};{\cal B})\lra {\rm H},~\forall ~i, j\in {\rm I}$.
\item [2.] $p_{j;i,k}: {\rm Conf}_{\rm I}({\cal A}; {\cal B})\lra \U, ~\forall ~
\{i,j,k\}\in \eqref{13.1.11.12h}$.
\end{itemize}
\el

\begin{proof}
The positivity of $\pi_{ij}$ is clear. 
By Relation 2 of Lemma \ref{3.23.12.1aa}, we get
$$
u_{\B_i, \B_k}^{\A_j}=u_{\B_{i}, \B_{i-1}}^{\A_j}u_{\B_{i-1},\B_{i-2}}^{\A_j}\ldots u_{\B_{k+1},\B_{k}}^{\A_j}.
$$
The product map
$ 
\U\times \U\ra \U, ~(u_1, u_2)\mapsto u_1u_2
$ 
is positive. The positivity of $p_{j;i,k}$ follows.
\end{proof}


\paragraph{Positivity of the potential ${\cal W}_{\rm J}$.} 
Recall the positive function $\chi$ on $\U$. 
Let $x\in {\rm Conf}_{\rm I}({\cal A}; {\cal B})$ be a generic configuration presented by (\ref{13.1.12.10h}). By Lemma \ref{8.28.10.39hhh}, each generic triple $({\rm B}_{j-1}, {\rm A}_j, {\rm B}_{j+1})$ has a unique representative $\{\B^-,\U, u_{\B_{j-1},\B_{j+1}}^{\A_j}\cdot \B^-\}$. In this case $u_j$ in \eqref{7.20.9.8} becomes $p_j(x)$. Therefore
$\chi_{\A_j}(u_j)=\chi\circ p_j(x).$
The potential ${\cal W}_{\rm J}$ of ${\rm Conf}_{\rm I}({\cal A}; {\cal B})$ becomes
\be \la{pontential.p.j}
{\cal W}_{\rm J}=\sum_{j\in {\rm J}} \chi\circ p_j
\ee
Since $p_j$ are positive morphisms, the positivity of ${\cal W}_{\rm J}$ follows.

\vskip 2mm


By Relation 2 of Lemma \ref{3.23.12.1aa}, we get
\be \la{13.1.11.11h}
\chi\circ p_j=\chi\circ p_{j; j-1,i}+\chi\circ p_{j; i,k}+\chi\circ p_{j; k,j+1}
\ee
All summands on right side are positive functions. 
By \eqref{pontential.p.j},
the set ${\rm Conf}_{{\rm J}\subset{\rm I}}^{+}({\cal A};{\cal B})(\Z^t)$ of tropical points such that ${\cal W}_{\rm J}^t\geq 0$ is the set
\be \la{13.1.11.21h}
\{l\in {\rm Conf}_{\rm I}({\cal A};{\cal B})(\Z^t)~|~ p_{j;i,k}^t(l)\in \U_{\chi}^+(\Z^t) \mbox{ for all } \{i,j,k\} \in \eqref{13.1.11.12h}\}.
\ee

\paragraph{Proof of Theorem \ref{13.2.22.2226h}.} 
Recall the moduli space  ${\rm Conf}_{{\rm J}\subset{\rm I}}^{\cal O}({\cal A}; {\cal B})$ in  
Definition \ref{aointegral}. 

\bl \la{Lema1}
A generic configuration in ${\rm Conf}_{\rm I}({\cal A}; {\cal B})({\cal K})$ is ${\cal O}$-integral relative to ${\rm J}$ if and only if 
 $u^{{\rm A}_j}_{{\rm B}_i,{\rm B}_k} \in {\rm U}({\cal O})$ for all $\{i,j,k\}\in\eqref{13.1.11.12h}$.
\el

\begin{proof} 
By definition $\lr(\A_j,\B_k)=[g_{\{\A_j, \B_k\}, \{\U,\B^{-}\}}]\in {\rm Gr}$.
Let $\{i,j,k\}\in$\eqref{13.1.11.12h}. Then
$$
\lr(\A_j,\B_k)=\lr(\A_j,\B_i) \Longleftrightarrow g_{\{\A_j, \B_i\}, \{\U,\B^{-}\}}^{-1}g_{\{\A_j, \B_k\}, \{\U,\B^{-}\}}=u^{{\rm A}_j}_{{\rm B}_i,{\rm B}_k}\in {\rm G}({\cal O}).
$$
The Lemma is proved.
\end{proof}
 Let $l\in{\rm Conf}_{\rm I}({\cal A}; {\cal B})$. Let $x\in {\cal C}_{l}^{\circ}$ be  presented by \eqref{13.1.12.10h}. By Lemma \ref{Lema10.1.1}, $u^{{\rm A}_j}_{{\rm B}_i,{\rm B}_k} \in {\rm U}({\cal O})$ if and only if $p_{j;i,k}^t(l)\in \U_{\chi}^{+}(\Z^t).$
Theorem \ref{13.2.22.2226h} follows from  Lemma \ref{Lema1} and \eqref{13.1.11.21h}.

\vskip 2mm
Tropicalizing the morphism \eqref{13.3.15.321h}, we get $\pi_{ij}^t: {\rm Conf}_{\rm I}({\cal A}; {\cal B})(\Z^t)\to {\rm H}(\Z^t)={\rm P}.$ 
 \bl \la{9.21.17.56h} Let $i,j \in {\rm J}$.
If $l\in {\rm Conf}_{{\rm J}\subset{\rm I}}^{+}({\cal A}; {\cal B})(\Z^t)$, then $\pi_{ij}^t(l) \in {\rm P}^+$.
\el 

\begin{proof}
Since $\pi_{ij}^t(l)=-w_0(\pi_{ji}^t(l)),$ we can
assume that there exists $k$ such that $\{i,j, k\}\in \eqref{13.1.11.12h}$.
Otherwise  we switch $i$ and $j$. 
Set $\lambda:=\pi_{ij}^t(l)$, $u_1:=p_{i;k,j}^t(l)$, $u_2:=p_{j;i,k}^{t}(l)$.   
We tropicalize \eqref{8.13.4a}:
\be \la{8.21.10.56h}
\chi^t(u_2)=\min_{r\in I}\{\langle \lambda, \alpha_r\rangle-{\cal R}_{r}^t(u_1)\}.
\ee
If $l\in \eqref{13.1.11.21h}$, then $
\chi^t(u_1)\geq 0$, $\chi^t(u_2)\geq 0.$ {By the definition of ${\cal R}_r$ and $\chi$, we get  ${\cal R}_r^t(u_1)\geq \chi^t(u_1)$. Therefore ${\cal R}^t(u_1)\geq 0$.}
Hence
$$
\forall r\in I,~~~
\langle \lambda, \alpha_r\rangle\geq \langle \lambda, \alpha_r\rangle-{\cal R}_{r}^t(u_1) \geq \chi^t(u_2)\geq 0~\Longrightarrow ~\lambda\in {\rm P}^+.
$$

\end{proof}


\section{Main examples of configuration spaces} \la{sec7sec}

As discussed in Section \ref{sec1}, the  pairs of 
configuration spaces especially 
important in representation theory are:
$$
\big\{{\rm Conf}_n({\cal A}), ~~{\rm Conf}_n({\rm Gr})\big\},
~~~
\big\{{\rm Conf}({\cal A}^n, {\cal B}), ~~{\rm Conf}({\rm Gr}^n, {\cal B})\big\},
~~~
\big\{{\rm Conf}({\cal B}, {\cal A}^n, {\cal B}), ~~{\rm Conf}({\cal B}, {\rm Gr}^n, {\cal B})\big\}.
$$ 
In Section \ref{sec7sec} we express the potential ${\cal W}$ and the map $\kappa$ in these cases under explicit coordinates.

\subsection{The configuration spaces ${\rm Conf}_n({\cal A})$ and ${\rm Conf}_n({\rm Gr})$} 

Recall $h_{ij}, u_{ij}^k$ in \eqref{3.23.12.1}.
Recall the positive birational isomorphism 
\be \la{7.22.1.1s.h}
\alpha_1: {\rm Conf}_n({\cal A})\stackrel{\sim}{\lra} {\rm H}^{n-1}\times{\rm U}^{n-2}, 
~~~(\A_1, \ldots, \A_n)\lms(h_{12},\ldots, h_{1n}, u_{3,2}^1,\ldots, u_{n,n-1}^1).
\ee
The potential ${\cal W}$ on ${\rm Conf}_n({\cal A})$ induces a positive  function
${\cal W}_{\alpha_1}:={\cal W}\circ {\alpha_1}^{-1}$ on ${\rm H}^{n-1}\times \U^{n-2}$.
\bt \la{13.2.1.6.05h}
The function
\be \la{12.12.ptWn}
{\cal W}_{\alpha_1}(h_2,\ldots, h_n, u_2,\ldots, u_{n-1})=\sum_{j=2}^{n-1}\big(\chi (u_j)+\sum_{i\in I}\frac{\alpha_i(h_{j})}{{\cal R}_i(u_j)}+\sum_{i\in I}\frac{\alpha_{i}(h_{j+1}}{{\cal L}_{i}(u_{j})}\big).
\ee
\et

\begin{proof}
By the scissor congruence invariance \eqref{11.18.11.5}, we get ${\cal W}(\A_1,\ldots, \A_n)=\sum_{j=2}^{n-1}{\cal W}(\A_1, \A_{j}, \A_{j+1}).$
The rest follows from \eqref{pontential.p.j} and Lemma \ref{lem1}.  
\end{proof}

Let us choose a map without stable points which is not necessarily a bijection:
$$
\alpha:[1,n]\lra [1,n],~~~\alpha(k)\neq k.
$$
Let $x=(\A_1,\ldots, \A_n)\in {\rm Conf}^{\cal O}_n({\cal A})$.  Define
\be \la{13.1.24.41h}
\omega_{k}(x):=[g_{\{\U, \B^{-}\}}(\{\A_1, \B_n\}, \{\A_k, \B_{\alpha(k)}\})]\in {\rm Gr}.
\ee
By the definition of ${\rm Conf}^{\cal O}_n({\cal A})$, the map $\omega_k$ is  independent of the map $\alpha$ chosen. 
Define
\be \la{13.1.23.23hh}
\omega:=(\omega_2,\ldots, \omega_n): {\rm Conf}_n^{\cal O}({\cal A})\lra {\rm Gr}^{n-1}, ~~~~x\lms (\omega_2(x),\ldots, \omega_n(x)).
\ee

Consider the  projection 
$$
i_1: {\rm Gr}^{n-1}\lra  {\rm Conf}_n({\rm Gr}),~~~~\{\lr_2,\ldots, \lr_n\}\lms ([1],\lr_2,\ldots, \lr_n) 
$$
\bl \la{13.1.24.1.hhh}
 The map $\kappa$ in \eqref{4.18.12.41qx}  is $i_1\circ \omega$.
\el
\begin{proof}
Here
$\omega_k(x)=g_{\{\U, \B^{-}\}, \{\A_1,\B_n\}} \lr(\A_k, \B_{\alpha(k)}).$
In particular $\omega_1(x)=[1]$. The Lemma follows. 
\end{proof}

 \begin{figure}[ht]
 \epsfxsize=5in 
\centerline{\epsfbox{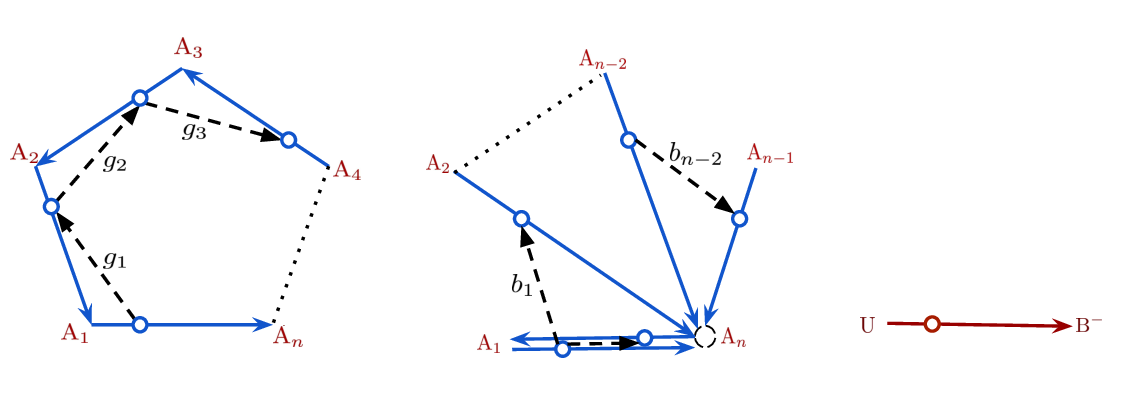}}
\caption{The map $\omega$ expressed by two different choices of frames $\{\A_{i}, \B_{\alpha(i)}\}$}
\label{cal6}
\end{figure}

 Below we give two explicit expressions of $\omega$ based on different choices of the map  $\alpha$.
We emphasize that although the expressions look entirely different from each other, they are the same map. 
As before, set  $x=(\A_1,\ldots, \A_n)\in {\rm Conf}^{\cal O}_n({\cal A})$.

\vskip 2mm

1. Let $\alpha(k)=k-1$. It provides  frames $\{\A_i, \B_{i-1}\}$, see the first graph of Fig \ref{cal6}. 
Set
\be \la{8.27.10.33h}
g_k: = g_{\{\U, \B^{-}\}}(\{{\rm A}_{k}, {\rm B}_{k-1}\}, \{{\rm A}_{k+1}, {\rm B}_{k}\})\stackrel{*}{=} {u^{{\rm A}_k}_{{\rm B}_{k-1}, {\rm B}_{k+1}}}h_{{\rm A}_k, {\rm A}_{k+1}}\overline w_0.
\ee
See Fig \ref{inv2} for proof of $*$. By \eqref{5.25.12.1}, we get 
\be
\omega_k(x)=[g_{\{\U,\B^{-}\}}(\{\A_1,\B_n\}, \{{\rm A}_{k}, {\rm B}_{k-1}\} )]=[g_1\ldots g_{k-1}],~~~ k\in [2,n]
\ee
Therefore
\be \la{3.23.12.3}
\omega(x)= 
([g_1],\ldots,
[g_1 ... g_{n-1}])\in {\rm Gr}^{n-1}.
\ee

\vskip 2mm

2. Let $\alpha(k)=n$ when $k\neq n$. Let $\alpha(n)=1$. See the second graph of Fig \ref{cal6}. 
Set
$$
b_k:=b_{\B_n}^{\A_k,\A_{k+1}}, ~k\in [1,n-2];~~~~~ h_{n}:=h_{\A_1,\A_n}.
$$
Then
\be \la{3.23.12.3.lemh.i}
\omega_{k}(x)=[g_{\{\U,\B^-\}}(\{\A_1,\B_n\},\{\A_{k},\B_n\})]=[b_1\ldots b_{k-1}], ~ k\in [2,n-1];~~~~\omega_n(x)=[h_n].
\ee
Therefore
\be \la{3.23.12.3.lemh}
\omega(x)=([b_1],\ldots, [b_1\ldots b_{n-2}], [h_n])\in {\rm Gr}^{n-1}.
\ee

\subsection{The configuration spaces ${\rm Conf}({\cal A}^n, {\cal B})$ and ${\rm Conf}({\rm Gr}^n, {\cal B})$} \la{sec7.1h}

Consider the scissoring morphism
\begin{align}
\la{8.18h.cut.map}
 s: {\rm Conf}({\cal A}^{m+n+1},{\cal B})&\lra {\rm Conf}({\cal A}^{m+1},{\cal B})\times {\rm Conf}({\cal A}^{n+1},{\cal B}),\nonumber\\
({\A_1,\ldots, \A_{m+n+1}, \B_{0}})&\lms ({\A_{1},\ldots,  \A_{m+1}, \B_{0}})\times (\A_{m+1},\ldots, \A_{m+n+1},\B_{0}).
\end{align}
By Lemmas \ref{7.22.1.1s}, \ref{13.1.10.3h}, the  morphism $s$ is a positive birational isomorphism.

 In fact, the inverse map of $s$ can be defined by  ``gluing" two configurations:
\be \la{5.7.12.1}
*: {\rm Conf}^*({\cal A}^{m+1}, {\cal B}) \times   {\rm Conf}^*({\cal A}^{n+1}, {\cal B})
\lra  {\rm Conf}({\cal A}^{m+n+1}, {\cal B}),~~~(a,b)\lms a*b.
\ee
By Lemma \ref{8.28.10.39hhh},  $a$ has a unique representative  $\{\A_1, \ldots, \A_{m}, \U, \B^{-}\}$, $b$ has a unique representative   $\{\U, \A_1',\ldots,  {\A}_{n}', {\B}^{-}\}$.  
We define the  {\it convolution product}  $a *b  := (\A_1,\ldots,\A_m, \U, \A_1',\ldots, \A_n' , \B^-).$ 
 The associativity of the convolution product is clear.

 \begin{figure}[ht]
\epsfxsize250pt
\centerline{\epsfbox{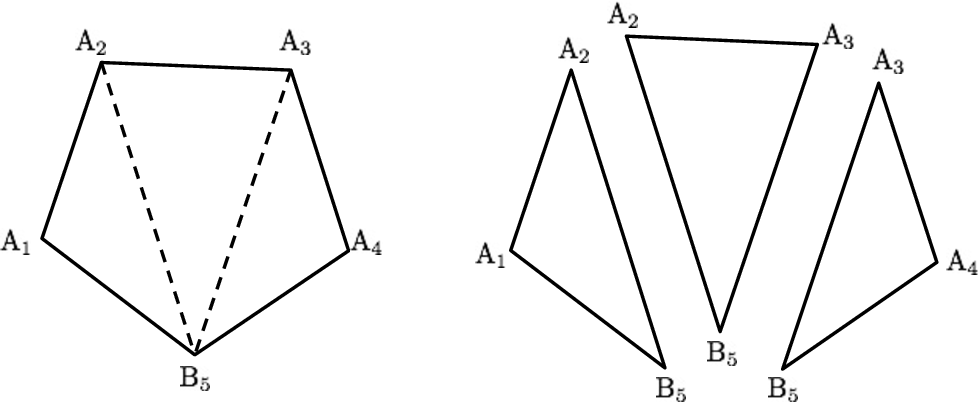}}
\caption{A map given by scissoring a convex pentagon.}
\label{cut}
\end{figure}

\vskip 2mm

Recall $b_{k}^{ij}$ in \eqref{3.23.12.1}. Recall the morphisms $\pi_r, \pi_l$ in \eqref{13.1.12.51h}.
\bt \la{12.17.thm7h}
The following morphism is a positive birational isomorphism
\be
c: {\rm Conf}({\cal A}^n, {\cal B})\lra (\B^{-})^{n-1}, ~~~(\A_1,\ldots, \A_n,\B_{n+1})\lms (b_{n+1}^{1,2},\ldots, b_{n+1}^{i,i+1},\ldots, b_{n+1}^{n-1,n}).
\ee
\et

\begin{proof}
Scissoring the convex ($n$+1)-gon along diagonals emanating from $n$+1, see Fig \ref{cut}, we get a positive birational isomorphism
${\rm Conf}({\cal A}^n, {\cal B})\stackrel{\sim}{\ra} \big({\rm Conf}({\cal A}^2, {\cal B})\big)^{n-1}.$ 
The Theorem is therefore reduced to $n=2$.
Recall  $\alpha_2$ in Lemma \ref{7.22.1.1s}. By Lemma \ref{12.12.11.h}, it is equivalent to prove that 
${\rm H}\times \U {\ra}{\rm H}\times \U^{-},~(h, u)\mapsto (\beta(u)h, \eta(u))$
is a positive birational isomorphism. Since $\eta$ is a positive birational isomorphism, {and} $\beta$ is a positive map, the Theorem follows.
\end{proof}

The potential ${\cal W}$ on ${\rm Conf}({\cal A}^n,{\cal B})$ induces a positive function
${\cal W}_{c}={\cal W}\circ {c}^{-1}$ on $({\B}^{-})^{n-1}$.

\bl
The function
\be \la{12.12.11.2h}
{\cal W}_{c}(b_1,\ldots, b_{n-1})=\sum_{j=1}^{n-1}\sum_{i\in I}\big(\frac{1}{{\cal L}_i^{-}\circ\pi_l(b_j)}+\frac{1}{{\cal R}_i^{-}\circ\pi_r(b_j)}\big)
\ee
\el
\begin{proof}
Note that
$$
{\cal W}(\A_1,\ldots, \A_n,\B_{n+1})=\sum_{j=1}^{n-1}{\cal W}(\A_{j}, \A_{j+1},\B_{n+1})=\sum_{j=1}^{n-1}\big(\chi(u_{n+1,j+1}^j)+\chi(u_{j, n+1}^{j+1})\big).
$$
The Lemma follows directly from Lemma \ref{lem2}, \eqref{lem2.eta} and Lemma \ref{12.12.11.h}.
\end{proof}

  \begin{figure}[ht]
\epsfxsize=4in  
\centerline{\epsfbox{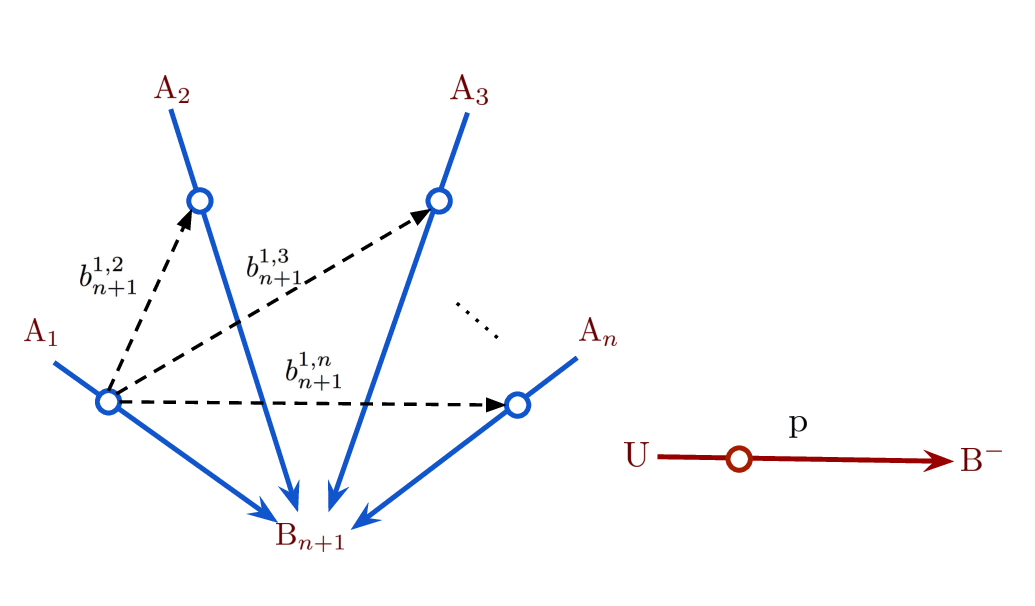}}
\caption{Frames assigned to $(\A_1,\ldots, \A_n, \B_{n+1})$.}
\label{AAB}
\end{figure}

\vskip 2mm

Define
\be \la{13.1.26.337h}
\tau: {\rm Conf}^{\cal O}({\cal A}^n, {\cal B})\lra {\rm Gr}^{n-1},~~~(\A_1,\ldots, \A_n)\lms 
\{[b_{n+1}^{1,2}],\ldots, [b_{n+1}^{1,n}]\}. 
\ee

Consider the projection
$$
i_b: {\rm Gr}^{n-1}\lra{\rm Conf}({\rm Gr}^{n}, {\cal B}),~~~\{\lr_2,\ldots, \lr_{n}\}\lms ([1], \lr_2,\ldots, \lr_{n}, \B^{-}). 
$$

Recall the map ${\kappa}$ in \eqref{5.12.12.2}. 
As illustrated by Fig \ref{AAB}, we get
\bl \la{12.15.12h35m}
When ${\rm J}={\rm I}=[1,n]\subset[1,n+1]$,  we have $\kappa=i_b\circ \tau$.
\el

\subsection{The configuration spaces ${\rm Conf}({\cal B}, {\cal A}^n, {\cal B})$ and ${\rm Conf}({\cal B}, {\rm Gr}^n, {\cal B})$}

Recall $r^{ij}_k$ in \eqref{13.3.16.1101}.  Similarly, there is a positive birational isomorphism
\be
p: {\rm Conf}({\cal B}, {\cal A}^{n}, {\cal B})\lra \U^{-}\times({\B}^{-})^{n-1}, ~~~(\B_1, \A_2, \ldots,\A_{n+1},\B_{n+2})\lms (r_{n+2}^{1, 2}, b_{n+2}^{2, 3},\ldots, b_{n+2}^{n,n+1}).
\ee
The potential ${\cal W}$ on ${\rm Conf}({\cal B},{\cal  A}^{n}, {\cal B})$ induces a positive function
${\cal W}_p:={\cal W}\circ p^{-1}$ on ${\U}^{-}\times ({\B}^{-})^{n-1}$.
We have
\be \la{12.12.1.2ll}
{\cal W}_p(r_1, b_2,\ldots, b_{n})=\sum_{i\in I}\frac{1}{{\cal R}_i^{-}(r_1)}+ \sum_{2\leq j\leq n}\sum_{i\in I}\big(\frac{1}{{\cal L}_i^{-}\circ \pi_l(b_j)}+\frac{1}{{\cal R}_i^{-}\circ \pi_r(b_j)}\big).
\ee

  \begin{figure}[ht]
\epsfxsize=4in  
\centerline{\epsfbox{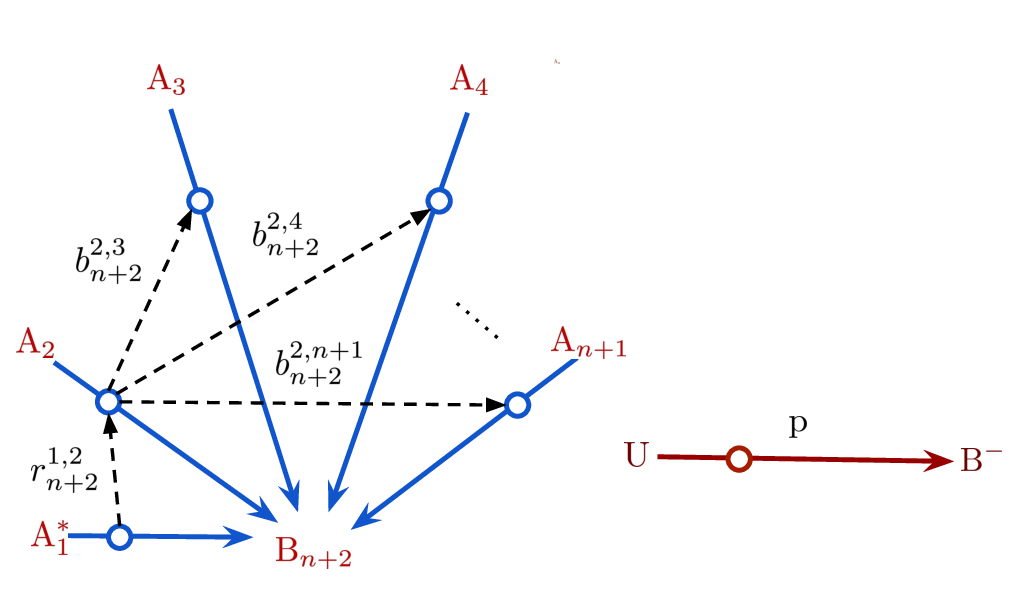}}
\caption{Frames assigned to $(\B_1,\A_2,\ldots, \A_{n+1},\B_{n+2})$. Here $\pi(\A_{1}^*)=\B_{1}$.}
\label{ABB}
\end{figure}

\vskip 2mm

Recall the map ${\kappa}$ in \eqref{5.12.12.2}. Define
\be \la{13.1.26.352h}
\tau_s: {\rm Conf}_{w_0}^{\cal O}({\cal A}, {\cal B}^n,{\cal A})\lra {\rm Gr}^{n}, ~~~~(\B_1,\A_2,\ldots, \A_{n+1}, \B_{n+2})\lms ([r_{n+2}^{1,2}],[r_{n+2}^{1,2}b_{n+2}^{2,3}],\ldots, [r_{n+2}^{1,2}b_{n+2}^{2,n+1}]).
\ee

Consider the projection
$$
i_s: {\rm Gr}^{n}\lra {\rm Conf}_{w_0}({\cal B}, {\rm Gr}^n, {\cal B}),~~~\{\lr_2,\ldots, \lr_{n+1}\}\lms (\B,\lr_2,\ldots, \lr_{n+1}, \B^{-}). 
$$

Let $x=(\B_1,\A_2,\ldots, \A_{n+1}, \B_{n+2})\in {\rm Conf}_{w_0}^{\cal O}({\cal A}, {\cal B}^n,{\cal A})$.
Let $\A_{1}^*\in {\cal A}$ be the preimage of $\B_1$ such that $b_{\B_{n+2}}^{\A_{1}^*, \A_2}=r_{n+2}^{1,2}.$
As illustrated by Fig \ref{ABB}, we get
\bl \la{12.15.kappa.j}
When ${\rm J}={\rm I}=[2,n+1]\subset[1,n+2]$,  we have $\kappa=i_s\circ \tau_s$.
\el

\section{Proof of Theorems \ref{5.8.10.45a} and \ref{13.1.30.742h}}\la{sec7}

\subsection{Lemmas}

Let ${\cal Y}={\cal Y}_1\times\ldots\times{\cal Y}_k$ be a product of positive spaces. The positive structure on ${\cal Y}$ is induced by positive structures  on ${\cal Y}_i$. Let $y_i\in {\cal Y}_i^{\circ}({\cal K})$.  Let $(y_{i,1}, \ldots, y_{i,n_i})$ be the coordinate of $y_i$ in a positive coordinate system ${\bf c}_{i}$. Define the field extension
 \be \la{13.2.12.3.38hh}
\Q(y_1,\ldots, y_k):=\Q\big({\rm in}(y_{1,1}), \ldots, {\rm in}(y_{1,n_1}),\ldots, {\rm in}(y_{k,n_k})\big).
 \ee
Thanks to \eqref{9.20.2.12}, such an extension is  independent of the positive coordinate systems chosen.

\vskip 2mm

Recall the morphisms $\pi_l$, $\pi_r$ in \eqref{13.1.12.51h}. 
\bl
\la{5.7.1}
Fix $i\in I$.
Let $(b, c)\in (\B^{-}\times {\Bbb G}_m)^{\circ}({\cal K})$. {Recall $y_i(c)\in \U^-({\cal K})$}. Then $b':=b\cdot y_i(c)\in (\B^-)^{\circ}({\cal K})$. 

Moreover, if $\val({\cal R}_i^{-}\circ \pi_r(b))\leq \val(c)$, then 
$\val(b')=\val(b)$ and $\Q(b',c)=\Q(b,c).$
\el
\begin{proof} Let $b=h\cdot y$. Fix a reduced word for $w_0$ which ends with $i_m=i$. It provides a decomposition $y=y_{i_1}(c_1)\ldots y_{i_m}(c_m)$. Then $b'=h\cdot y_{i_1}(c_1)\ldots y_{i_m}(c_m+c).$ The rest is clear. 
\end{proof}

\bl 
\la{5.7.8}
Let $(b,h)\in (\B^-\times {\rm H})^{\circ}({\cal K})$. 
Then $b':=b\cdot h\in (\B^{-})^{\circ}({\cal K})$.  

Moreover, if $h\in {\rm H}(\C)$,  then $\val(b')=\val(b)$ and $\Q(b', h)=\Q(b,h).$
\el
\begin{proof} Let $b=y\cdot h_b$. The rest is clear.
\end{proof}

\bl \la{12.12.1.1h}
Let $(b,p)\in ({\B^{-}}\times \B^{-})^{\circ}({\cal K})$.  Assume $p\in {\B^-}(\C)$.
\begin{itemize}
\item[1.] If $ \val({\cal R}_i^-\circ\pi_r(b))\leq 0$ for all $i\in I$, then $b\cdot p$ is a transcendental point. Moreover 
$$\val (b\cdot p)=\val(b),~~~~~~\Q(b\cdot p, p)=\Q(b,p).$$
\item[2.] If $\val({\cal L}_i^{-}\circ\pi_l(b))\leq 0$ for all $i\in I$, then $p^{-1}\cdot b$ is a transcendental point. Moreover 
$$\val(p^{-1}\cdot b)=\val (b),~~~~~~\Q(p^{-1}\cdot b, p) = \Q(b,p).
$$ 
\end{itemize}
\el

\begin{proof}
Combining Lemmas \ref{5.7.1}-\ref{5.7.8}, we prove 1. Analogously 2 follows.
\end{proof}

\subsection{Proof of  Theorem \ref{13.1.30.742h}.}
Our first task is to prove Theorem \ref{13.1.30.742h} for the cases when ${\rm I}=[1,n]\subset [1,n+1]$. 

Let ${\rm J}=\{j_1,\ldots, j_m\}\subset {\rm I}$. Recall
${\cal W}_{\rm J}$ in \eqref{13.3.1.527h}. Let $l\in {\rm Conf}({\cal A}^n,{\cal B})(\Z^t)$  be   such that ${\cal W}_{\rm J}^t(l)\geq 0$.

 Let ${\bf x}\in {\cal C}_l^{\circ}$.
Recall the map $c$ in Theorem \ref{12.17.thm7h}.  Set
$c({\bf x}):=(b_1,\ldots, b_{n-1})\in ({\rm B}^{-})^{n-1}({\cal K}).$ 

\bl  \la{13.1.26.10.3111h}
For every $i\in I$, we have
\begin{itemize}
\item[1.] $\val({\cal L}_i^{-}\circ \pi_l(b_j))\leq 0$ if  $j\in[1,n-1]\cap {\rm J}$,
\item[2.] $\val({\cal R}_i^-\circ \pi_r(b_{k-1}))\leq 0$ if  $k\in[2,n]\cap {\rm J}$.
\end{itemize}
\el
\begin{proof}
Let $j\in[1,n-1]\cap {\rm J}$. By definition $b_j=b_{\B_{n+1}}^{\A_j,\A_{j+1}}$. By Lemmas  \ref{lem2}, \ref{12.12.11.h}, we get
$$\val({\cal L}_i^-\circ\pi_l(b_j))=-\val(\chi_i(u_{\B_{n+1}, \B_{j+1}}^{\A_j}))\leq -\chi_{\A_j}^t(l)\leq 0.$$
The second part follows similarly. 
\end{proof}
As illustrated by  Fig \ref{AAB}, we see that
$$
{\bf x}=(g_1\cdot\U, g_2\cdot \U,\ldots, g_n\cdot \U, \B^-),~~~~g_1:=1,~g_j:=b_1\ldots b_{j-1},~j\in [2,n].
$$
If $j\in {\rm J}$, then $\lr_j:=\lr(g_j\cdot \U, \B^-)=[g_j]\in {\rm Gr}$. Therefore
 $$
\kappa({\bf x})=(x_1,\ldots, x_n, \B^-),~~~~x_j=\left\{\begin{array}{ll}[g_j] &\text{ if }j \in {\rm J},~~\\ g_j\cdot \U &\text{ otherwise}.\end{array}\right.
$$

Let  $\{\A_{j_1},\ldots, \A_{j_m}\}\in {\cal A}^m(\C)$ be a generic point  in the sense of algebraic geometry.
Define
$$
{\bf y}:=(\A_1', \A_2',\ldots, \A_n', \B^-) \in {\rm Conf}({\cal A}^n; {\cal B}),~~~~\A_j'=\left\{\begin{array}{ll}g_j\cdot \A_j &\text{ if }j \in {\rm J},~~\\ g_j\cdot \U &\text{ otherwise}.\end{array}\right.
$$
Let $F\in \Q_+\big({\rm Conf}({\cal A}^n;{\cal B})\big)$. By the very definition of $D_F$, we have
$D_F\big({\kappa({\bf x})}\big)=\val\big(F({\bf y})\big)$.

\vskip 2mm
Since $\{\A_{j_1},\ldots, \A_{j_m}\}$ is generic, it can be presented by
\be \la{13.1.26.10.4111h}
\{\A_{j_1},\ldots, \A_{j_m}\}:=\{p_{j_1}\cdot \U,\ldots, p_{j_m}\cdot \U\},~~~~{\bf p}=\{p_{j_1},\ldots, p_{j_m}\}\in (\B^-)^m(\C).
\ee
We can also assume that $({\bf x},{\bf p})$ is a transcendental point, so that
\be
(c({\bf x}), {\bf p})\in \big((\B^{-})^{m+n-1}\big)^\circ ({\cal K}).
\ee
Set $p_j=1$ for $j\notin{\rm J}$. Keep the same $p_j$ for $j\in {\rm J}$.
Then
$${\bf y}=(g_1p_1\cdot \U, \ldots, g_{n}p_n\cdot \U, \B^{-});~~~~~~ c({\bf y})=(\tilde{b}_1,\ldots, \tilde{b}_{n-1}),~\tilde{b}_j:=p_j^{-1}b_jp_{j+1}\in {\B}^{-}({\cal K}).  
$$
By Lemmas \ref{12.12.1.1h}-\ref{13.1.26.10.3111h}, we get
\begin{align}
\Q(c({\bf x}), {\bf p})&=\Q(b_1,\ldots, b_{n-1}, p_{i_1},\ldots, p_{i_m})=\Q(\tilde{b}_1,\ldots, b_{n-1}, p_{i_1},\ldots, p_{i_m})=\ldots \nonumber\\
&=\Q(\tilde{b}_1,\ldots, \tilde{b}_{n-1}, p_{i_1},\ldots, p_{i_m})=\Q(c({\bf y}), {\bf p}).\\
\val({b}_j)&=\val(\tilde{b}_j),~~~~\forall j\in [1,n-1]. \la{13.1.30.9.44h}
\end{align}
Therefore 
$(c({\bf y}), {\bf p})\in \big((\B^{-})^{m+n-1}\big)^\circ ({\cal K}).$
Thus $c({\bf y})$ is a transcendental point. Since  $\val(c({\bf y}))=\val(c({\bf x}))=c^t(l)$, we get  ${\bf y}\in {\cal C}_{l}^{\circ}$. By Lemma \ref{thm10.1.1.2},  $\val\big(F({\bf y})\big)=F^t(l)$.
Theorem \ref{13.1.30.742h}  is proved.

\vskip 2mm

Now consider the general cases when ${\rm J}\subset{\rm I}\subset[1,n]$. 
Consider the positive projection
$$
d_{\rm I}=p_{\rm I}\circ d: {\rm Conf}({\cal A}^n;{\cal B})\stackrel{d}{\lra} {\rm Conf}_n({\cal A})\stackrel{p_{\rm I}}{\lra} {\rm Conf}_{\rm I}({\cal A}; {\cal B}).
$$
Here the map $d$ kills the last flag $\B_{n+1}$. The map ${p}_{\rm I}$  keeps ${\rm A}_i$ intact  when $i\in {\rm I}$, and takes $\A_i$ to $\pi(\A_i)$ otherwise.

\bl
Let $l\in {\rm Conf}_{{\rm J}\subset{\rm I}}^+({\cal A};{\cal B})(\Z^t)$. There exists $l'\in {\rm Conf}({\cal A}^n; {\cal B})(\Z^t)$ such that ${\cal W}_{\rm J}^t(l')\geq 0$ and $d_{\rm I}^t(l')=l$.
\el
\begin{proof}
We prove the case when ${\rm J}$ contains $\{1, n\}$. In fact, the other cases are easier.  
Let $x=(\A_1,\ldots, \A_n, \B_{n+1})$. 
Consider a map
$u: {\rm Conf}({\cal A}^n;{\cal B})\ra \U$ given by $x\mapsto u_{\B_{n+1},\B_n}^{\A_1}.$  Then 
\begin{align} 
\la{13.1.14.1h04}
{\cal W}_{\rm J}(x)&={\cal W}_{\rm J}(\A_1,\ldots, \A_n)+{\cal W}(\A_1,\A_n, \B_{n+1}) ={\cal W}_{\rm J}(d_{\rm I}(x))+\chi(u_{\B_{n+1},\B_n}^{\A_1})+\chi (u_{\B_1,\B_{n+1}}^{\A_n}) \nonumber\\
                                                                      &={\cal W}_{\rm J}(d_{\rm I}(x))+\chi\big(u(x)\big)+\sum_{i\in I}\frac{\pi_{1,n}(d_{\rm I}(x))}{{\cal R}_i(u(x))}.
\end{align}

By Lemma \ref{9.21.17.56h}, we have $\lambda:=\pi_{1,n}^t(l)\in {\rm P}^+$.
Clearly there exists  $l'\in {\rm Conf}({\cal A}^n; {\cal B})(\Z^t)$ such that
$d_{\rm I}^t(l')=l$ and $u^t(l')=0\in \U(\Z^t). $ 
We tropicalize \eqref{13.1.14.1h04}:
$$
{\cal W}_{\rm J}^t(l')=\min\{{\cal W}_{\rm J}^t(l), ~\chi^t(0),~\min_{i\in I}\{\langle \lambda, \alpha_i\rangle-{\cal R}_i^t(0)\}\}
=\min\{{\cal W}_{\rm J}^t(l), ~0,~\min_{i\in I}\{\langle \lambda, \alpha_i\rangle\}\}=0.
$$
\end{proof}

Let $l, l'$ be as above.  
Let  ${\bf x}\in {\cal C}_l^\circ$.  Clearly there exists  ${\bf z}\in {\cal C}_{l'}^\circ$ such that $d_{\rm I}({\bf z})={\bf x}$.
For any $F\in {\Q}_+({\rm Conf}_{\rm I}({\cal A};{\cal B}))$,  we have
$$
D_F(\kappa({\bf x}))=D_{F\circ d_{\rm I}}(\kappa({\bf z}))=(F\circ d_{\rm I})^t(l')=F^t\circ d_{\rm I}^t(l')=F^t(l). 
$$
The second identity is due to the special cases discussed before. The rest are by definition.

\section{Configurations and generalized Mircovi\'{c}-Vilonen cycles}\la{sec11}

\subsection{Proof of Theorem \ref{kth} }
\la{sec12.1.1}
In this Section we use extensively the notation from Section \ref{sec5.2}, such as $u_{\B_1,\B_3}^{\A_2}$, $r_{\B_3}^{\B_1,\A_2}\in \U^-$. 
We identify the subset ${\bf A}_{\nu}$  in Theorem \ref{kth} with the subset ${\bf A}_{\nu}\subset \U_{\chi}^+(\Z^t)$ in \eqref{8.23.10.22h} by  tropicalizing 
\be \la{12.12.30.1hh}
\alpha: {\rm Conf}({\cal B}, {\cal A},{\cal B})\stackrel{\sim}{\lra} {\U}, ~~(\B_1, \A_2, \B_3)\lms u_{\B_1,\B_3}^{\A_2}.
\ee
Thanks to identity 4 of Lemma \ref{12.12.11.h}, the index $\nu$ for both definitions match.

\begin{proof}[Proof of Theorem \ref{kth}]

2). Let $l\in {\bf A}_\nu$. Let $x=(\B_1, \A_2, \B_3)\in {\cal C}_{l}^{\circ}$. By Lemma \ref{12.12.11.h}, 
$r_{\B_3}^{\B_1, \A_2}=\eta(u_{\B_1, \B_3}^{\A_2}).$
Recall  $\kappa_{\rm Kam}$ in \eqref{kappa.kam}. Recall $i_s$ in \eqref{10.9.12.2}. By Lemma \ref{12.15.kappa.j}, we get
\be \la{13.1.8.1h}
\kappa (x)=(\B, [r_{\B_3}^{\B_1,\A_2}], \B^{-})=(\B, \kappa_{\rm Kam}(u_{\B_1, \B_3}^{\A_2}), \B^{-})=i_s(\kappa_{\rm Kam}(\alpha(x))).
\ee
Recall ${\rm MV}_l$ in  \eqref{thm.kam.a.l}. Then
${\cal M}_l=i_s({\rm MV}_l).$
Thus 2)  is a reformulation of Theorem \ref{thm.kam}.

\vskip 2mm

1). Recall the map
\be
p_i: {\rm Conf}({\cal A, A, B})\lra \U, ~~~(\A_1,\A_2,\B_3)\lms u_{\B_{i+2},\B_{i+1}}^{\A_{i}}, ~~~i=1,2.
\ee
{Recall the map $\tau$ defined by \eqref{13.1.26.337h}}
\be \la{h13.1.16.12h}
\tau: {\rm Conf}({\cal A},{\cal A}, {\cal B})({\cal K})\lra {\rm Gr},~~~(\A_1, \A_2, \B_3)\lms [b_{\B_3}^{\A_1,\A_2}].
\ee
Note that $p_2^{t}$ induces a bijection from ${\bf P}_{\lambda}^{\mu}$ to ${\bf A}_{\lambda-\mu}$. The MV cycles of coweight $(\lambda-\mu,0)$ are 
$$
\overline{\kappa_{\rm Kam} \circ p_2({\cal C}_{l}^{\circ})}=\overline{\kappa_{\rm Kam}({\cal C}_{p_2^t(l)}^{\circ})}={\rm MV}_{p_2^t(l)},~~~~l\in {\bf P}_{\lambda}^{\mu}.
$$

Let $x=(\A_1,\A_2, \B_3)\in {\cal C}_{l}^{\circ}$. Note that
$$
\tau(x)=[b_{\B_3}^{\A_1,\A_2}]=[\mu_{\B_3}^{\A_1,\A_2}r_{\B_3}^{\B_2, \A_1}]=\mu(x)\cdot \kappa_{\rm Kam}(p_2(x)),~~~\mbox{where } [\mu(x)]=[\mu_{\B_3}^{\A_1,\A_2}]=t^{\mu}.
$$
We get $\overline{\tau({\cal C}_{l}^{\circ})}=t^{\mu}\cdot  {\rm MV}_{p_2^t(l)}$. They are precisely MV cycles of coweight $(\lambda,\mu)$. 
Recall the isomorphism $i$ in \eqref{10.9.12.2a}. Clearly ${\cal M}_l=i(\overline{\tau({\cal C}_l^\circ}))$. Thus 1) is proved.

\vskip 2mm

3). The set ${\bf B}_{\lambda}^{\mu}$ is a subset of ${\bf P}_{\lambda}^{\mu}$ such that $p_1^t({\bf B}_{\lambda}^{\mu})\subset \U_{\chi}^+(\Z^t)$.
By Lemma \ref{9.21.17.56h},  ${\bf B}_{\lambda}^{\mu}$ is empty unless $\lambda\in {\rm P}^+$. So we assume $\lambda\in {\rm P}^+$.
Let $l\in {\bf P}_{\lambda}^{\mu}$. 
Let $x=(\A_1,\A_2,\B_3)\in {\cal C}_{l}^\circ$.   
By Lemma \ref{3.23.12.1aa}, 
\be
\tau(x)=[b_{\B_3}^{\A_1,\A_2}]=[u_{\B_3,\B_2}^{\A_1}h_{\A_1,\A_2}\overline{w}_0u_{\B_1,\B_3}^{\A_2}]=p_1(x)\cdot t^{\lambda}.
\ee
The last identity is due to $p_2^t(l)\in \U_{\chi}^+(\Z^t)$ \big(hence $u_{\B_1,\B_3}^{\A_2} \in \U({\cal O})$\big).

By Lemma \ref{13.1.8.53h},
$ \tau(x)\in \overline{{\rm Gr}_{\lambda}}$ if and only if $p_1^t(l)\in \U^{+}_{\chi}(\Z^t)$.
Therefore 
 \be \la{13.1.22.168h}
\overline{\tau({\cal C}_{l}^\circ)}  \subset \overline{{\rm Gr}_{\lambda}}
 \Longleftrightarrow p_1^t(l)\in \U_{\chi}^+(\Z^t)
 \Longleftrightarrow l\in {\bf B}_{\lambda}^{\mu}.
\ee
The rest follows from  
Lemma \ref{12.15.12h35m}.
\end{proof}

\subsection{Proof of Theorems \ref{MVlm1a}, \ref{MVlm1sa}, \ref{MVn+1bbb}}
By  Theorem \ref{kth}, we have
\be \la{linlang}
{\rm S}_{w_0}^{\mu}\cap {\rm S}_e^{\lambda}=\bigcup_{l\in {\bf P}_{\lambda}^{\mu}}{\rm N}_l,~~~~~
{\rm S}_{w_0}^{\mu}\cap{\rm Gr}_{\lambda}=\bigcup_{l\in{\bf B}_{\lambda}^{\mu}} {\rm M}_{l},
\ee
Here ${\rm N}_l$ (resp. ${\rm M}_l$) are components containing $\tau({\cal C}_l^{\circ})$ as dense subsets.
They are all of dimension $\langle \rho, \lambda-\mu\rangle$. 
The closures $\overline{{\rm N}_l}=\overline{\tau({\cal C}_l^{\circ})}$ are MV cycles.

\paragraph{Proof of Theorem \ref{MVlm1a}.}
Scissoring the convex ($n$+2)-gon along diagonals emanating from the vertex labelled by $n$+2, see Fig \ref{cut}, we get a positive birational isomorphism between ${\rm Conf}({\cal A}^{n+1}, {\cal B})$ and $\big({\rm Conf}({\cal A}^2, {\cal B})\big)^n$. Its tropicalization provides a decomposition
\be \la{13.2.1.7.01h}
{\bf P}_{\lambda; \underline{\lambda}}^{\mu}=\bigsqcup_{\mu_1+\ldots+\mu_n=\mu} {\bf P}_{\lambda}^{\mu_1}\times{\bf B}_{\lambda_2}^{\mu_2} \ldots \times {\bf B}_{\lambda_n}^{\mu_n},~~~~\underline{\lambda}=(\lambda_2,\ldots, \lambda_n)\in ({\rm P}^{+})^{n-1}.
\ee
Let $l=(l_1,\ldots, l_n)\in {\bf P}_{\lambda; \underline{\lambda}}^{\mu}$. We construct an irreducible subset
$$
{\rm C}_l:=\{([b_1], [b_1b_2],\ldots, [b_1b_2\ldots b_n])\in {\rm Gr}^n~|~b_i\in {\rm B}^{-}({\cal K}), ~[b_1]\in {\rm N}_{l_1},~[b_i]\in {\rm M}_{{l}_i},~i\in [2,n]\}.
$$
By induction,  ${\rm C}_l$ is of dimension $\langle \rho, \lambda+\lambda_2+\ldots+\lambda_n-\mu\rangle.$
\bl \la{13.3.16.114h}
Recall the subvariety ${\rm Gr}_{\lambda, \underline{\lambda}}^{\mu}$ in \eqref{MVlm}.
We have
${\rm Gr}_{\lambda, \underline{\lambda}}^{\mu}=\cup {\rm C}_{l}$ where ${l\in {\bf P}_{\lambda; \underline{\lambda}}^{\mu}}.$
\el

\begin{proof}
Thanks to the isomorphism  ${\rm B}^{-}({\cal K})/{\B}^{-}({\cal O})\stackrel{\sim}{\ra}{\rm Gr}$, each  $x\in {\rm Gr}_{\lambda, \underline{\lambda}}^{\mu}$ can be presented as
$([b_1], [b_1b_2]\ldots, [b_1\ldots b_n])$, where $b_i\in {\rm B}^{-}({\cal K})$ for all $i\in[1,n].$ 
By the definition of ${\rm Gr}_{\lambda, \underline{\lambda}}^{\mu}$, we have
$$
[b_i]\in {\rm Gr}_{\lambda_i}, ~\forall i\in [2,n];~~~[b_1]\in {\rm S}_e^{\lambda},~[b_1\ldots b_n]\in {\rm S}_{w_0}^{\mu}.
$$

Let  ${\rm pr}: \B^{-}({\cal K})\ra{\rm H}({\cal K})\ra{\rm H}({\cal K})/{\rm H}({\cal O})={\rm P}$ be  the composite of standard projections. 
Set ${\rm pr}(b_i):=\mu_i$. Then $[b_i]\in {\rm S}_{w_0}^{\mu_i}$. 

When $i=1$, $[b_1]\in  {\rm S}_{w_0}^{\mu_1}\cap {\rm S}_e^{\lambda}$. Thus  $[b_1]\in {\rm N}_{l_1}$ for some $l_1\in {\bf P}_{\lambda}^{\mu_1}$.

When $i>1$,  $[b_i]\in {\rm S}_{w_0}^{\mu_i}\cap {\rm Gr}_{\lambda_i}$. Thus  $[b_i]\in {\rm M}_{l_i}$ for some $l_i\in {\bf B}_{\lambda_i}^{\mu_i}$.

Note that $\mu_1+\ldots+\mu_n={\rm pr}(b_1)+\ldots+{\rm pr}(b_n)={\rm pr}(b_1\ldots b_n)=\mu.$ Thus $l:=(l_1,\ldots, l_n)\in {\bf P}_{\lambda, \underline{\lambda}}^{\mu}$. By definition $x\in {\rm C}_l$. 
Therefore ${\rm Gr}_{\lambda, \underline{\lambda}}^{\mu}\subseteq\cup_{l\in 
{\bf P}_{\lambda; \underline{\lambda}}^{\mu}} {\rm C}_{l}$.
The other direction follows similarly.
 \end{proof}
 
 Let $l\in {\bf P}_{\lambda, \underline{\lambda}}^{\mu}$. Recall the map
$$
\tau: {\rm Conf}({\cal A}^{n+1}, {\cal B})\lra {\rm Gr}^n, ~~~~(\A_1,\ldots, \A_{n+1}, \B_{n+2})\lms ([b_{\B_{n+2}}^{\A_1,\A_2}],\ldots, [b_{\B_{n+2}}^{\A_1,\A_{n+1}}]).
$$
Clearly  $\tau({\cal C}_l^{\circ})$ is a dense subset of ${\rm C}_l$. 
Recall  the isomorphism  $i$ in \eqref{10.9.12.2a}. 
Following Lemma \ref{12.15.12h35m}, the isomorphism $i$ identifies $\tau({\cal C}_l^{\circ})$ with ${\cal M}_l^{\circ}$. By Theorem \ref{13.1.30.742h}, the cells ${\cal M}_l^{\circ}$  are disjoint.  Theorem \ref{MVlm1a} follows from Lemma \ref{13.3.16.114h}.



\paragraph{Proof of Theorem \ref{MVlm1sa}.}
The group ${\rm H}({\cal K})$ acts diagonally on ${\rm Gr}^n$. Let $h\in {\rm H}({\cal K})$ be such that $[h]=t^{\nu}$. Then
$
h\cdot {\rm Gr}_{\lambda; \underline{\lambda}}^{\mu}={\rm Gr}_{\lambda+\nu; \underline{\lambda}}^{\mu+\nu}.
$ 
One can choose $h$ such that $[h]=t^{-\mu}$. 
The rest  follows by the same argument in the proof of Theorem \ref{MVlm1a}.

\paragraph{Proof of Theorem \ref{MVn+1bbb}.}
By definition ${\bf B}_{\lambda_1,\lambda_2,\ldots, \lambda_n}^{\mu}\subset {\bf P}_{{\lambda_1}; \lambda_2, ..., \lambda_n}^{\mu}$. The Theorem follows by the same 
argument  in the proof of Theorem \ref{MVlm1a}.

%
%

  \subsection{Components of the fibers of convolution morphisms}\la{sec9.3new}
  
Let $\underline{\lambda}=(\lambda_1,\ldots, \lambda_n)\in ({\rm P}^+)^n$.  
Recall  the convolution variety ${\rm Gr}_{\underline{\lambda}}$ in \eqref{con.var.247}.
By the geometric Satake correspondence,
${\rm IH}(\overline{{\rm Gr}_{\underline{\lambda}}})
=V_{\underline{\lambda}}:=V_{\lambda_1}\otimes \ldots \otimes V_{\lambda_n}$.

Set $|\underline{\lambda}|:=\lambda_1+\ldots+\lambda_n$. Set 
 ${\rm ht}(\underline{\lambda};\mu):=\langle \rho, |\underline{\lambda}|-\mu\rangle.$
The {\it convolution morphism} $m_{\underline{\lambda}}: \overline{{\rm Gr}
_{\underline{\lambda}}}\to \overline{{\rm Gr}_{|\underline{\lambda}|}}$ 
projects $(\lr_1,\ldots, \lr_n)$ to $\lr_n$. 
It is semismall, i.e. for any $\mu\in {\rm P}^+$ such that $t^{\mu}\in 
\overline{{\rm Gr}_{|\underline{\lambda}|}}$, the fiber 
$m_{\underline{\lambda}}^{-1}(t^{\mu})$ over $t^{\mu}$ 
is of top dimension ${\rm ht}(\underline{\lambda};\mu)$. 
See \cite{MV} for proof.  

By the decomposition theorem \cite{BBD}, we have
$$
{\rm IH}(\overline{ {\rm Gr}_{\underline{\lambda}}})=\bigoplus_\mu F_{\mu}\otimes {\rm IH}(\overline{{\rm Gr}_{\mu}}).
$$
Here the sum is over $\mu\in {\rm P}^+$ such that 
$t^{\mu}\subseteq \overline{{\rm Gr}_{|\underline{\lambda}|}}$, and  
$F_{\mu}$ is the vector space spanned by the fundamental classes of  top dimensional components of $m_{\underline{\lambda}}^{-1}(t^{\mu})$. As a consequence, the number of top components of $m_{\underline{\lambda}}^{-1}(t^{\mu})$ equals the tensor product multiplicity $c_{\underline{\lambda}}^{\mu}$ of $V_\mu$ in $V_{\underline{\lambda}}$.

Recall the subsets ${\bf C}_{\underline{\lambda}}^{\mu}$ in \eqref{11.20.11.2}.
By Lemma \ref{9.21.17.56h}, the set ${\bf C}_{\underline{\lambda}}^{\mu}$ is 
empty unless $(\mu,\underline{\lambda})\in ({\rm P}^+)^{n+1}$.  
Recall the map $\omega$ in \eqref{13.1.23.23hh}.
In  this subsection we prove 

\bt \la{8.24.8.22h}
Let ${\bf T}_{\underline{\lambda}}^{\mu}$ be the set of top components of $m_{\underline{\lambda}}^{-1}(t^{\mu})$. 
For each $l\in {\bf C}_{\underline{\lambda}}^{\mu}$, the closure $\overline{{\omega({\cal C}_{l}^{\circ})}}\in {\bf T}_{\underline{\lambda}}^{\mu}$.  It gives a bijection between ${\bf C}_{\underline{\lambda}}^{\mu}$ and {${\bf T}_{\underline{\lambda}}^\mu$}.
\et

First we prove the case when $n=2$. In this case,
the fiber $m_{\lambda_1,\lambda_2}^{-1}(t^{\mu})$ is isomorphic to
$$
\{\lr\in {\rm Gr}~|~(\lr, t^{\mu})\in \overline{ {\rm Gr}_{\lambda_1,\lambda_2}}\}={\overline{{\rm Gr}_{\lambda_1}}}\cap t^{\mu}{\overline{{\rm Gr}_{\lambda_2^{\vee}}}}.
$$
Here $\lambda_2^{\vee}:=-w_0(\lambda_2)\in {\rm P}^+$. 
The following Theorem is due to Anderson.
\bt [\cite{A}] \la{Thm.A.}
 The  top 
 components of ${\overline{{\rm Gr}_{\lambda_1}}}\cap t^{\mu}{\overline{{\rm Gr}_{\lambda_2^{\vee}}}}$ are precisely  the MV cycles  of coweight $(\lambda_1,\mu-\lambda_2)$ contained in 
${\overline{{\rm Gr}_{\lambda_1}}}\cap t^{\mu}{\overline{{\rm Gr}_{\lambda_2^{\vee}}}}$.
\et
\begin{figure}[ht]
 \epsfxsize=3.5in 
\centerline{\epsfbox{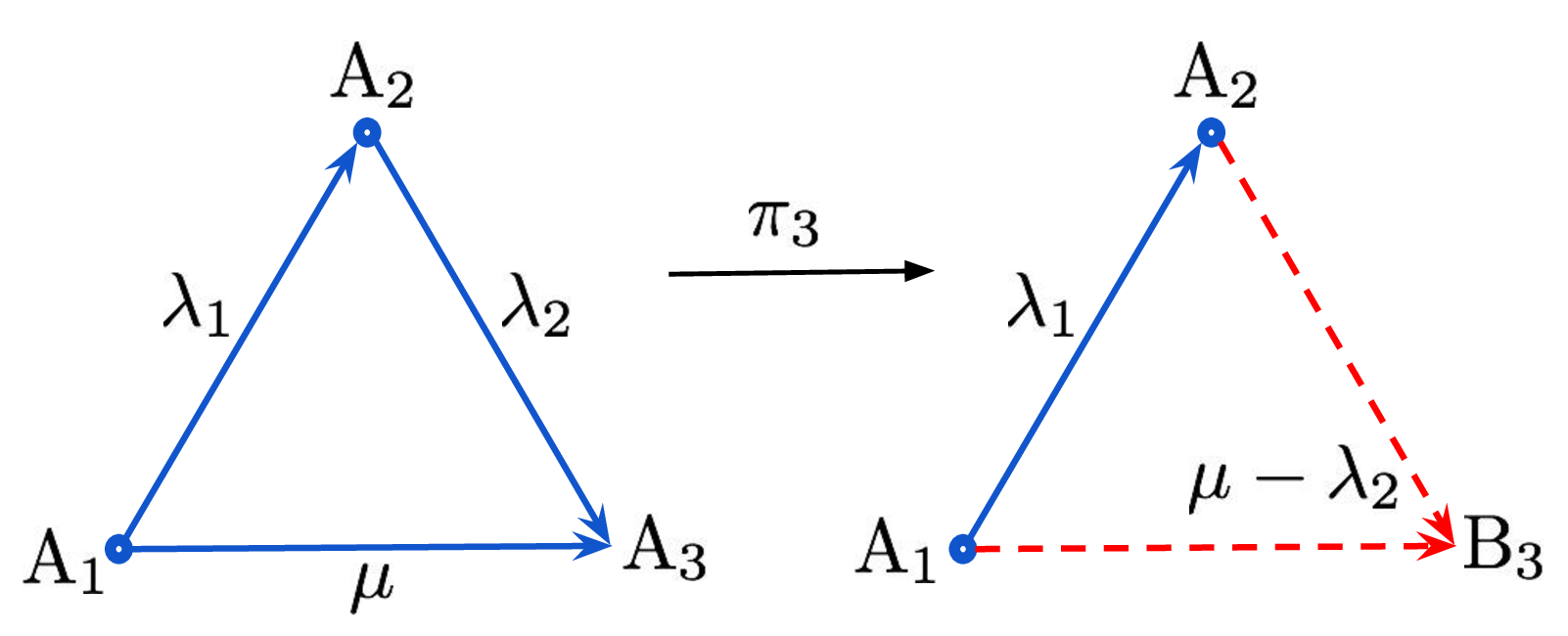}}
\caption{The projection $\pi_3$ induces a bijection 
$\pi_3^t:  \wt{\bf B}_{\lambda_1,\lambda_2}^{\mu}\ra {\bf B}_{\lambda_1}^{\mu-\lambda_2}$.}
\label{htpm}
\end{figure}
Recall the positive morphisms 
$$
p_i: {\rm Conf}_3({\cal A})\lra \U,~~~(\A_1, \A_2, \A_3)\lra u_{\B_{i-1}, \B_{i+1}}^{\A_i},~i\in \Z/3
$$ 
Let us put the potential condition on two vertices, see the left of Fig \ref{htpm}, getting
$$
 \wt{\bf B}_{\lambda_1, \lambda_2}^{\mu}:=\{l\in {\rm Conf}_3({\cal A})(\Z^t)~|~(\pi_{12}, \pi_{23},  \pi_{13})^t(l)=(\lambda_1,\lambda_2,\mu), {~p_1^t(l)\in \U^+_\chi(\Z^t),~p_2^t(l)\in \U^+_\chi(\Z^t)}\}.
$$
Consider the projection
$
\pi_3: {\rm Conf}_3({\cal A})\ra {\rm Conf}({\cal A}^2, {\cal B})$ which maps $(\A_1,\A_2,\A_3)$ to $(\A_1,\A_2,\B_3)$.
Its tropicalization $\pi_3^t$ induces  a bijection \footnote{There is a positive Cartan group action on ${\rm Conf}_3({\cal A})(\Z^t)$ defined via
$$
{\rm H}\times {\rm Conf}_3({\cal A})\lra {\rm Conf}_3({\cal A}), \quad \quad h\times (\A_1, \A_2, \A_3)\lms (\A_1, \A_2, \A_3\cdot h).
$$
Its tropicalization determines a free ${\rm H}(\Z^t)$-action on ${\rm Conf}_3({\cal A})(\Z^t)$.  By definition, one can thus identify the ${\rm H}(\Z^t)$-orbits of ${\rm Conf}_3({\cal A})(\Z^t)$ with points of ${\rm Conf}({\cal A},{\cal A}, {\cal B})(\Z^t)$. Note that each element in ${\bf B}_{\lambda_1}^{\mu-\lambda_2}$ has a unique representative in $\wt{\bf B}_{\lambda_1,\lambda_2}^{\mu}$. Hence the map $\pi_3^t$ is a bijection.} from 
$\wt{\bf B}_{\lambda_1,\lambda_2}^{\mu}$ to ${\bf B}_{\lambda_1}^{\mu-\lambda_2}$.
Recall $\omega_2$ in \eqref{3.23.12.3.lemh.i}. By \eqref{13.1.22.168h}, the cycles
$$
\overline{\omega_2({\cal C}_l^{\circ})}= \overline{\tau({\cal C}_{\pi_3^t(l)}^{\circ})},~~~~~l\in\wt{\bf B}_{\lambda_1,\lambda_2}^{\mu}
$$ 
are precisely MV cycles of coweight $(\lambda_1,\mu-\lambda_2)$ contained in $\overline{{\rm Gr}_{\lambda_1}}$.

Let $l\in\wt{\bf B}_{\lambda_1, \lambda_2}^{\mu}$.
Let $x=(\A_1,\A_2,\A_3)\in {\cal C}_{l}^{\circ}$. 
By identity 2 of Lemma \ref{3.23.12.1aa}, 
$$
\omega_2(x)=[\pi_{13}(x)\overline{w}_0 \cdot \big(p_3(x)\big)^{-1}{ \pi_{32}(x)}],~~~
\mbox{where }
 [\pi_{13}(x)]=t^{\mu},~~[\pi_{32}(x)]=t^{\lambda_2^{\vee}}.$$
Therefore 
{$$
\omega_2(x)\in t^{\mu}\overline{{\rm Gr}_{\lambda_2^{\vee}}} \Longleftrightarrow t^{-\mu}\omega_2(x)\in \overline{{\rm Gr}_{\lambda_2^{\vee}}}\Longleftrightarrow t^{-\mu}\pi_{13}(x)\overline{w}_0\cdot [\big(p_3(x)\big)^{-1}{\pi_{32}(x)}] \in \overline{{\rm Gr}_{\lambda_2^{\vee}}} \Longleftrightarrow \big(p_3(x)\big)^{-1}\cdot t^{\lambda_2^{\vee}}\in \overline{{\rm Gr}_{\lambda_2^{\vee}}}.
$$
Here the last equivalence is due to the fact that $t^{-\mu}\pi_{13}\overline{w}_0\in \G({\cal O})$. Therefore for any $l\in\wt{\bf B}_{\lambda_1,\lambda_2}^\mu$,  
$$ 
\omega_2({\cal C}_l^\circ)\subset t^\mu \overline{{\rm Gr}_{\lambda_2^{\vee}}}\Longleftrightarrow \big(p_3({\cal C}_l^\circ)\big)^{-1}\cdot t^{\lambda_2^{\vee}}\subset \overline{{\rm Gr}_{\lambda_2^{\vee}}} $$
By Lemma \ref{Lema10.1.1}, Lemma \ref{13.1.8.53h},  and the definition of ${\bf C}_{\lambda_1,\lambda_2}^\mu$, we get
$$\big(p_3({\cal C}_l^\circ)\big)^{-1}\cdot t^{\lambda_2^{\vee}}\in \overline{{\rm Gr}_{\lambda_2^{\vee}}}
\Longleftrightarrow \big(p_3({\cal C}_l^\circ)\big)^{-1}\in \U({\cal O}) \Longleftrightarrow  p_3({\cal C}_l^\circ) \in \U({\cal O})\Longleftrightarrow p_3^t(l) \in \U_{\chi }^{+}(\Z^t)\Longleftrightarrow l\in {\bf C}_{\lambda_1,\lambda_2}^{\mu}. 
$$
}
Let $l\in {\bf C}_{\lambda_1,\lambda_2}^{\mu}$. Let $x=(\A_1,\A_2,\A_3)\in {\cal C}_l^{\circ}$. Note that
$\omega_3(x)=[h_{\A_1,\A_3}]=t^{\mu}.$ 
Therefore $\omega(x)=(\omega_2(x), \omega_3(x))\in m_{\lambda_1,\lambda_2}^{-1}(t^{\mu})$.
The rest is due to  Theorem \ref{Thm.A.}.

\vskip 2mm

Now let us prove the general case. 
Consider the scissoring morphism
\begin{align} \la{13.1.18.1h}
 c=(c_1,c_2): {\rm Conf}_{n+1}({\cal A})&\lra {\rm Conf}_{n}({\cal A})\times {\rm Conf}_3({\cal A}), \nonumber\\
(\A_1,\ldots, \A_{n+1})&\lms (\A_1,\ldots, \A_{n-1}, \A_{n+1})\times (\A_{n-1},\A_n,\A_{n+1})
\end{align}
Due to the scissoring congruence invariance,  the map $c^t$ induces a decomposition
\be \la{13.1.17.101hh}
{\bf C}_{\lambda_1,\ldots, \lambda_{n}}^{\mu}=\bigsqcup_{\nu\in {\rm P}^+}   {\bf C}_{\lambda_1,\ldots, \lambda_{n-2}, \nu}^{\mu}\times {\bf C}_{\lambda_{n-1},\lambda_n}^{\nu}.
\ee

\bp \la{13.2.1.5.29h}
The cardinality of ${\bf C}_{\underline{\lambda}}^{\mu}$ is the  tensor product multiplicity $c_{\underline{\lambda}}^{\mu}$  of $V_{\mu}$ in $V_{\underline{\lambda}}$.
\ep

\begin{proof}
Decomposing  the last tensor products in $V_{\lambda_1}\otimes \ldots \otimes (V_{\lambda_{n-1}}\otimes V_{\lambda_n})$
 into a sum of irreducibles, and tensoring then each of them with $V_{\lambda_1}\otimes\ldots\otimes V_{\lambda_{n-2}}$, we get 
$$
c_{\lambda_1,\ldots, \lambda_n}^{\mu}=\sum_{\nu \in{\rm P}^+} c_{\lambda_1,\ldots, \lambda_{n-2},\nu}^{\mu}c_{\lambda_{n-1},\lambda_n}^{\nu}.
$$ 
As a consequence of $n=2$ case, 
$|{\bf C}_{\lambda, \mu}^{\nu}|=c_{\lambda, \mu}^{\nu}$. 
The Lemma follows by induction and \eqref{13.1.17.101hh}.
\end{proof}

\bl \la{13.1.18.24h}
For $l\in {\bf C}_{\underline{\lambda}}^{\mu}$, the cycles $\omega({\cal C}_{l}^{\circ})$ are disjoint.
\el
\begin{proof}
By Lemma \ref{13.1.24.1.hhh}, 
$\kappa({\cal C}_l^\circ)=i_1\circ \omega({\cal C}_l^\circ).$ 
The Lemma follows from Theorem \ref{5.8.10.45a}.  
\end{proof}

\bl \la{13.1.18.23h}
For any $l \in {\bf C}_{\underline{\lambda}}^{\mu}$, we have
$\omega({\cal C}_{l}^{\circ})\subset m_{\underline{\lambda}}^{-1}(t^{\mu}).$ 
\el
\begin{proof}
Let $x=(\A_1,\ldots, \A_{n+1})\in {\cal C}_{l}^{\circ}$. Recall the expression \eqref{3.23.12.3}. We have
$$
[g_i]:=[u_{\B_{i-1},\B_{i+1}}^{\A_i}h_{\A_i, \A_{i+1}}\overline{w}_0]=u_{\B_{i-1}, \B_{i+1}}^{\A_i}\cdot t^{\lambda_i}\in {\rm Gr}_{\lambda_i}, ~~~i\in [1,n].
$$
Thus $\omega(x)\in {{\rm Gr}_{\underline{\lambda}}}$.
Meanwhile
$m_{\underline{\lambda}}\circ \omega(x)=[h_{\A_1,\A_{n+1}}]=t^{\mu}.$ 
The Lemma is proved.
\end{proof}

\bl \la{13.1.18.25h}
Let $l\in {\bf C}_{\underline{\lambda}}^{\mu}$. The closure $\overline{\omega({\cal C}_l^{\circ})}$ is an irreducible variety of dimension 
${\rm ht}(\underline{\lambda};\mu)$.
\el

\begin{proof}
By construction, $\overline{\omega({\cal C}_{l}^{\circ})}$ is irreducible. 
 Note that $m_{\underline{\lambda}}^{-1}(t^{\mu})$ is of top dimension ${\rm ht}(\underline{\lambda};\mu)$. 
 By Lemma \ref{13.1.18.23h}, ${\rm dim}~\overline{\omega({\cal C}_{l}^{\circ})}\leq {\rm ht}(\underline{\lambda};\mu).$
To show that ${\rm dim}~\overline{\omega({\cal C}_{l}^{\circ})}\geq {\rm ht}(\underline{\lambda};\mu)$, we use induction.
\vskip 2mm

Set $\pi_{n-1,n+1}^t(l):=\nu.$
Recall  $c=(c_1,c_2)$ in \eqref{13.1.18.1h}.  Then
$c_1^t(l)\in {\bf C}_{\lambda_1,\ldots, \lambda_{n-2},\nu}^{\mu}$, $c_2^t(l)\in {\bf C}_{\lambda_{n-1},\lambda_n}^{\nu}.$ 
Consider the projection
$$
{\rm pr}: \omega({\cal C}_{l}^{\circ})\lra {\rm Gr}^{n-1},~~~(\lr_1,\ldots,\lr_{n-1}, \lr_n)\lra (\lr_1,\ldots,\lr_{n-2}, \lr_n)
$$
Its image
${\rm pr}(\omega({\cal C}_{l}^{\circ}))=\omega({\cal C}_{c_1^t({l})}^\circ).$ 
Let ${\bf b}=(\lr_1,\ldots, \lr_{n-2},\lr_n)\in \omega({\cal C}_{c_1^t(l)}^\circ)$.
The fiber over ${\bf b}$ is 
$$
{\rm pr}^{-1}({\bf b}):=\{\lr \in {\rm Gr}~|~(\lr_1,\ldots, \lr_{n-2},\lr,\lr_{n})\in \omega({\cal C}_{l}^{\circ})\}.
$$
Let $y=(\A_1,\ldots,\A_{n-1}, \A_{n+1})\in {\cal C}_{c_1^t(l)}^\circ$ such that $\omega(y)={\bf b}$. Set $b_y:=b_{\B_{n+1}}^{\A_1,\A_{n-1}}.$ 
For any $x\in {\cal C}_{l}^{\circ}$ such that $c_1(x)=y$, we have ${\rm pr}(\omega(x))=\omega(y)={\bf b}$. 
By \eqref{3.23.12.3.lemh.i}, we have
$$
\omega_{n-1}(x)=[b_{\B_{n+1}}^{\A_1,\A_n}]=b_y\cdot \omega_2(c_2(x)) \in {\rm pr}^{-1}(b).
$$
Then it is easy to see that
$
b_y\cdot \overline{\omega_2({\cal C}_{c_2^t(l)}^{\circ})}\subset \overline{{\rm pr}^{-1}(b)}.
$ 
Therefore  
$${\rm dim}~\overline{\omega({\cal C}_{l}^{\circ})}\geq {\rm dim}~\overline{\omega({\cal C}_{c_1^t(l)}^{\circ})}+{\rm dim}~\overline{\omega({\cal C}_{c_2^t(l)}^{\circ})}.$$
The case when $n=2$ is proved above. The Lemma follows by induction.
\end{proof}

\begin{proof} [Proof of Theorem \ref{8.24.8.22h}]
By Lemmas \ref{13.1.18.23h}, \ref{13.1.18.25h}, the  map 
${\bf C}_{\underline{\lambda}}^{\mu} \lra {\bf T}_{\underline{\lambda}}^{\mu}$, $l\lms\overline{\omega({\cal C}_{l}^\circ)}$ is well-defined.  
By Lemma \ref{13.1.18.24h}  and the very construction of the cell ${\cal C}_{l}^{\circ}$, it 
is injective.
Since  $|{\bf C}_{\underline{\lambda}}^{\mu}|=|{\bf T}_{\underline{\lambda}}^{\mu}|=c_{\underline{\lambda}}^{\mu}$,  the map is  a bijection.
\end{proof}

\subsection{Proof of Theorem \ref{5.8.10.45b}}
\la{sec9.4}
We focus on the case when $\mu=0$ for ${\bf C}_{\underline{\lambda}}^{\mu}$.
Consider the scissoring morphism 
\begin{align}
c=(c_1,c_2): {\rm Conf}_{n+1}({\cal A})&\lra {\rm Conf}_n({\cal A})\times {\rm Conf}_3({\cal A}),\nonumber\\
(\A_1,\ldots, \A_n, \A_{n+1})&\lms (\A_1,\ldots, \A_n)\times (\A_1,\A_n, \A_{n+1}).\nonumber
\end{align}
Due to the scissoring congruence invariance, the morphism $(c_1^t, c_2^t)$ induces a
decomposition
$$
{\bf C}_{\underline{\lambda}}^{0}=\bigsqcup _{\nu} {\bf C}_{\lambda_1,\ldots, \lambda_{n-1},\nu}\times {\bf C}_{\nu^{\vee}, \lambda_n}^{0}.
$$
Note that ${\bf C}_{\nu^{\vee}, \lambda_n}^{0}$ is empty if $\nu\neq\lambda_n$. Moreover $|{\bf C}_{\lambda_n^{\vee}, \lambda_n}^{0}|=1$. Thus  $c_1^t: {\bf C}_{\underline{\lambda}}^{0}\ra {\bf C}_{\underline{\lambda}}$ is a bijection.

\vskip 2mm

Consider the shifted projection
$$
p_s : {\rm Gr}^n\lra {\rm Conf}_n({\rm Gr}),~~~~~~\{\lr_1,\ldots, \lr_n\}\lra (\lr_n, \lr_1,\ldots, \lr_{n-1}).
$$
\bl
\la{13.1.24.47hhh}
Let $l\in {\bf C}_{\underline{\lambda}}^{0}$. Then
$p_s\circ \omega({\cal C}_{l}^\circ)=\kappa({\cal C}_{c_1^t(l)}^{\circ}).$
\el
\begin{proof}
Let $x=(\A_1,\ldots, \A_{n+1})\in {\cal C}_{l}^{\circ}$. Then $u:=u_{\B_{n+1},\B_n}^{\A_1}\in \U({\cal O})$.
Let $y:=c_1(x)\in {\cal C}_{c_1^t(l)}^\circ$.

Recall $\omega_i$ in \eqref{13.1.24.41h}.  Then $\omega_{n+1}(x)=[1]$. 
For $i\in [2,n]$, we have
$$
\omega_i(x)=[g_{\{\U,\B^{-}\}}(\{\A_1,\B_{n+1}\}, \{\A_i,\B_1\})]=u\cdot [g_{\{\U,\B^{-}\}}(\{\A_1,\B_{n}\}, \{\A_i,\B_1\})]= u\cdot \omega_i(y).
$$
Therefore
$$
p_s\circ \omega(x)=(\omega_{n+1}(x), u\cdot \omega_2(y),\ldots, u\cdot \omega_n(y))=([1], \omega_2(y),\ldots, \omega_n(y))=\kappa(y).
$$
Here the last step is due to Lemma  \ref{13.1.24.1.hhh}.
Since $c_1({\cal C}_{l}^{\circ})={\cal C}_{c_1^t(l)}^{\circ}$, the Lemma is proved.
\end{proof}

Recall  ${\rm Gr}_{c(\underline{\lambda})}$ and the set  ${\bf T}_{\underline{\lambda}}$ of its top components in Theorem \ref{5.8.10.45b}. The connected group ${\rm G}({\cal O})$ acts on ${\rm Gr}_{c(\underline{\lambda})}$. It preserves each component of ${\rm Gr}_{c(\underline{\lambda})}$. So these components live naturally on the stack
${\rm Conf}_n({\rm Gr})={\rm G}({\cal O})\backslash \big([1]\times {\rm Gr}^{n-1}\big)$.

Recall the fiber $m_{\underline{\lambda}}^{-1}([1])$ and the set ${\bf T}_{\underline{\lambda}}^0$ in Theorem \ref{8.24.8.22h}. 
Note that
$p_s\big(m_{\underline{\lambda}}^{-1}([1])\big)={\G}({\cal O})\backslash\overline{{\rm Gr}_{c(\underline{\lambda})}}\subset {\rm Conf}_{n}({\rm Gr}).$ 
It induces a bijection ${\bf T}_{\underline{\lambda}}^0\stackrel{\sim}{\lra} {\bf T}_{\underline{\lambda}}$.

\paragraph{Proof of Theorem \ref{5.8.10.45b}.}
By Theorem \ref{8.24.8.22h} and above discussions, there is a chain of bijections:
$
{\bf C}_{\underline{\lambda}}\stackrel{\sim}{\lra} {\bf C}_{\underline{\lambda}}^0\stackrel{\sim}{\lra} {\rm T}_{\underline{\lambda}}^0\stackrel{\sim}{\lra} {\rm T}_{\underline{\lambda}}.
$ 
By Lemma \ref{13.1.24.47hhh}, this chain is achieved by the map $\kappa$. The Theorem  is proved.

\section{Positive $\G$-laminations and surface affine Grassmannians}
\la{sec11}

A decorated surface $S$ comes with an unordered collection $\{s_1, ..., s_n\}$ of special points, defined up to isotopy. 
Denote by $\partial S$ the boundary of $S$. We assume that $\partial S$ is not empty. We define  {\it punctured boundary} 
\be \la{punctb}
\widehat \partial S:= \partial S - \{s_1, ..., s_n\}.
\ee
Its components are called  {\it boundary circles} and {\it boundary intervals}. 

Let us shrink all holes without special points on $S$ into {\it punctures}, 
getting a homotopy equivalent surface. Abusing notation, we denote it again by $S$. 
We say that the punctures and special points on $S$ form the set of {\it marked points} on $S$:
$$
\{\mbox{marked points}\}:= \{\mbox{special points $s_1, ..., s_n$}\} \cup \{\mbox{punctures}\}.
$$

Pick a point $\ast s_i$ in each of the boundary intervals. 
The dual decorated surface $\ast S$ is given by the same surface $S$ with 
the set of special points $\{\ast s_1, ..., \ast s_n\}$.
We have a duality: 
$\ast \ast S =S$. 

Observe that the marked points are in bijection with the components of the punctured boudary 
$\widehat \partial (\ast S)$. 

\subsection{The space ${\cal A}_{\G, S}$ with the potential ${\cal W}$}
\paragraph{Twisted local systems and decorations.} 
Let ${\rm T}'S $ be the complement to the zero section of the tangent bundle on a surface $S$. 
Its fiber ${\rm T}'_y$  at $y \in S$  is homotopy equivalent to a circle. 
Let $x\in {\rm T}'_yS$. 
The fundamental group $\pi_1({\rm T}'S, x)$ is a central extension:
\be \la{cen.ext}
0\lra \pi_1({\rm T}_y'S,x)\lra \pi_1({\rm T}'S, x)\lra \pi_1(S,y)\lra 0, ~~~~\pi_1({\rm T}_y'S,x)=\Z.
\ee

Let ${\cal L}$ be a $\G$-local system on ${\rm T}'S$ with the monodromy $s_{\G}$ around a generator of $\pi_1({\rm T}'_yS,x)$. 
Let us assume that $\G$ acts on ${\cal L}$ on the right. 
We call ${\cal L}$ a {\it twisted} $\G$-local system on $S$.  
It gives rise to the {\it associated decorated flag bundle} ${\cal L}_{\cal A}:={\cal L}\times_{\G}{\cal A}$.

Let ${\rm C}$ be a component of  $\widehat\partial  (\ast S)$. 
There is a canonical up to isotopy section $\sigma: {\rm C}\ra {\rm T}'{\rm C}$ 
given by the tangent vectors to ${\rm C}$ directed according to the orientation of ${\rm C}$.
A {\it decoration on ${\cal L}$ over ${\rm C}$} is a flat section of the 
restriction of ${\cal L}_{\cal A}$ to $\sigma({\rm C})$.

\bd[\cite{FG1}]  \la{moduliMGS}
A twisted decorated $\G$-local system on $S$ is a pair $({\cal L,\alpha})$, where ${\cal L}$ is a twisted $\G$-local system on $S$, and $\alpha$ is given by a decoration on ${\cal L}$ over each component of $\widehat\partial  (\ast S)$.

The moduli space ${\cal A}_{\G, S}$ parametrizes twisted decorated $\G$-local systems on $S$. 
\ed

Abusing terminology, a decoration is given by decorated flags at the marked points.

\paragraph{Remark.} Since the boundary $\partial S$ of $S$ is not empty,  the extension \eqref{cen.ext} splits:
$$
\pi_1({\rm T}'S, x)\stackrel{\sim}{=}\pi_1({\rm T}_y'S,x)\times \pi_1(S,y). 
$$
However the splitting is not unique. 
As a space, ${\cal A}_{\G, S}$ is isomorphic, although non canonically if $s_{\G}\neq 1$,  
to its counterpart of usual unipotent 
$\G$-local systems on $S$ with decorations. The mapping class group $\Gamma_S$ acts differently on the
 two spaces.  For example, when $S$ is a disk $D_n$ with $n$ special points on the boundary, then $\Gamma_{D_n}=\Z/n\Z$. Both moduli spaces are isomorphic to the configuration space ${\rm Conf}_n({\cal A})$. The mapping class group $\Z/n\Z$ acts on the untwisted moduli space is by the cyclic rotation $(\A_1,\ldots, \A_n)\mapsto (\A_n, \A_1,\ldots, \A_{n-1})$, while its action on ${\cal A}_{\G, D^n}$ is given by the  ``twisted" rotation

$$
(\A_1,\A_2, \ldots, \A_n)\lms (\A_n\cdot s_{\rm G}, \A_1,\ldots, \A_{n-1}).
$$ 

\bt [${\it loc.cit.}$]
The space ${\cal A}_{\G, S}$ admits a natural positive structure such that the mapping class group $\Gamma_S$ acts on ${\cal A}_{\G, S}$
by positive birational isomorphisms.
\et
Below we give two equivalent definitions of the  potential  ${\cal W}$ on ${\cal A}_{\G, S}$.

\paragraph{Potential via generalized monodromy.} A decorated flag $\A$ provides an isomorphism
\be \la{isoia}
i_{\A}: {\rm U_{\A}/[U_{\A},U_{\A}]} \stackrel{\sim}{\lra} \oplus_{\alpha \in \Pi}{\Bbb A}^1.
\ee
Let $\Sigma: \oplus_{\alpha \in \Pi}{\Bbb A}^1\to {\Bbb A}^1$ be the sum map. Then $\chi_\A = \Sigma\circ i_\A$. 
 This characterizes the map $i_{\A}$.

\vskip 2mm
Let us assign to each component  ${\rm C}$ of  $\widehat \partial (\ast S)$ a canonical rational map, called {\it generalized monodromy  at ${\rm C}$}: ${\mu}_{\rm C}: {\cal A}_{{\rm G},   S} \lra \oplus_{\alpha \in \Pi}{\Bbb A}^1.$
There are two possible cases.
\vskip 2mm
(i) The component C is a boundary circle. The decoration over C is 
a decorated flag $\A_{\rm C}$ in the fiber of  ${\cal L}_{\cal A}$  
on C, invariant under the monodromy around C. It defines a conjugacy class in the 
 unipotent subgroup ${\rm U}_{\A_{\rm C}}$ preserving $\A_{\rm C}$. So we get a regular map 
$$
{\mu}_{\rm C}: {\cal A}_{{\G},   S} \lra {\rm U_{A_C}/[U_{\A_C},U_{\A_C}]}\stackrel{i_{{\rm A}_{\rm C}}}{=}\oplus_{\alpha \in \Pi}{\Bbb A}^1.
$$

(ii) The component C is a boundary interval on a hole $h$. The universal cover of $h$ is a line. We get an infinite sequence of intervals on this line projecting to the boundary interval(s) on $h$. 
There are decorated flags assigned to these intervals. 
Take an interval ${\rm C}'$ on the cover projecting to 
C. Let ${\rm C}_{-}'$ and ${\rm C}_{+}'$ be the intervals
just before and after ${\rm C}'$.  We get a  triple of decorated flags $({\rm A}_{-}, {\rm A}, {\rm A}_{+})$ 
sitting over these intervals. There is a unique 
 $u\in {\rm U}_{{\rm A}}$  such that  
$
 {\rm B}_{+} = u \cdot {\rm B}_{-},
$ 
where $\B_{\pm}=\pi(\A_{\pm})\in {\cal B}$.
Projecting $u$ to ${\rm U}_{{\rm A}}/[{\rm U}_{{\rm A}}, {\rm U}_{{\rm A}}]$, we get a map
$\mu_{\rm C}: {\cal A}_{{\rm G},   S} \ra \oplus_{\alpha \in \Pi}{\Bbb A}^1.$
It is clear that $\mu_{\rm C}$
does not depend on the choice of ${\rm C}'$.
\vskip 2mm

Composing the generalized monodromy $\mu_{\rm C}$ with the sum map $\oplus_{\alpha \in \Pi}{\Bbb A}^1\to {\Bbb A}^1$,  
we get  
\be \la{6}
{\cal W}_{\rm C}:= \Sigma \circ \mu_{\rm C}:  {\cal A}_{{\rm G},   S} \lra {\Bbb A}^1,
\ee
called {\it the potential associated with ${\rm C}$}. 
\bd
The potential ${\cal W}$ on the space ${\cal A}_{{\rm G},   S}$ is defined as
\be \la{defpot1}
{\cal W}:= \sum_{\mbox{\rm components C of $\widehat\partial (\ast S)$}} {\cal W}_{\rm C}. 
\ee
\ed

\paragraph{Potential via ideal triangulations.}

\bd An ideal triangulation of a decorated surface $S$ is 
a triangulation of the surface whose 
 vertices are the marked points of $S$. 
\ed

Let $T$ be an ideal triangulation of $S$.
Pick a triangle $t$  of  $T$. 
The restriction to $t$ provides a projection\footnote{If the vertices of $t$ coincide, one can first pull back to a sufficient big cover $\widetilde{S}$ of $S$, and then consider the restriction to a triangle $\wt{t}\subset\wt{S}$ which projects onto $t$. Clearly the result is independent of the pair $\wt{t}\subset \wt{S}$ chosen.}
$\pi_t$ from ${\cal A}_{{\rm G}, S}$ to ${\rm Conf}_3({\cal A})$. 
Recall the potential ${\cal W}_3$ on the latter space. 
\bd The potential on the space  ${\cal A}_{{\rm G}, S}$ is defined as
\be \la{defpot2}
{\cal W}:= \sum_{\mbox{\rm triangles $t$ of $T$}} {\cal W}_3\circ \pi_t.
\ee
\ed

Changing $T$ by a flip we do not change the sum (\ref{defpot2}) since the potential on a quadrilateral is 
invariant under a flip (Section \ref{sec2}). 
Since any two ideal triangulations are related by a sequence of flips, 
the potential  (\ref{defpot2}) is independent of  the ideal triangulation $T$ chosen.

\paragraph{The above definitions are equivalent.}
There is a natural bijection between the marked points, that is the vertices of $T$,
 and the components of $\widehat \partial (\ast S)$.
Working with definition (\ref{defpot2}), 
the sum over all angles of the triangles shared by a puncture is the potential ${\cal W}_{\rm C}$
 assigned to the corresponding boundary circle. 
 A similar sum over all angles shared by a special point is the potential ${\cal W}_{\rm C}$ assigned to the corresponding boundary interval. 
Thus the potentials  (\ref{defpot1}) and (\ref{defpot2}) coincide. 

\paragraph{Positivity of the potential ${\cal W}$.}
In the positive structure of ${\cal A}_{\G,S}$ introduced in \cite{FG1}, the projection $\pi_t: {\cal A}_{\G, S}\ra {\rm Conf}_3({\cal A})$ is a positive morphism. By Theorem \ref{mth1} and \eqref{defpot2}, we get 
\bt
The potential ${\cal W}$ is a positive  function on the space ${\cal A}_{\G, S}$. 
\et

\paragraph{Positive integral $\G$-laminations.}
We define the set of {\it positive integral $\G$-laminations on $S$}:
\be
{\cal A}^{+}_{\G, S}(\Z^t)=\{l\in {\cal A}_{\G, S}(\Z^t)~|~ {\cal W}^t(l)\geq 0\}.
\ee
By tropicalization, the mapping class group $\Gamma_S$ acts on ${\cal A}_{\G, S}(\Z^t)$. 
The potential ${\cal W}$ is $\Gamma_S$-invariant. Thus $\Gamma_S$ acts on the subset ${\cal A}_{\G, S}^+(\Z^t)$. 

\paragraph{Partial potentials.} Given any simple positive root $\alpha$, 
there is a component $\chi_{A, \alpha}$ of the character $\chi_{A}$ so that 
 $\chi_{A} = \sum_{\alpha \in \Pi}\chi_{A, \alpha}$. Let $S$ be a decorated surface. 
Then to each boundary component $C \in \partial (\ast S)$ one associates 
a function $W_{C, \alpha}$. It is evidently invariant under the action of the mapping class group $\Gamma_S$  of $S$. 

\bt
Let $S$ be a surface with $n$ holes and no special points. Then 
the algebra of regular $\Gamma_S$-invariant functions on the space ${\cal A}_{G,S}$
 is a polynomial algebra in $n {\rm rk}(G)$ variables 
freely generated by the partial potentials 
$W_{C, \alpha}$, where $C$ run through all boundary circles on $S$, and $\alpha$ are
 simple positive roots. 
\et

\begin{proof}
It is well known that the action of the mapping class group $\Gamma_S$ on the moduli space 
${\rm Loc}^{\rm un}_{G, S}$ of unipotent $G$-local systems on a surface $S$ with holes 
is ergodic. So there are no non-constant $\Gamma_S$-invariant regular functions 
on this space. On the other hand, there is a canonical $\Gamma_S$-invariant 
projection given by the generalised monodromy 
around the holes: 
$$
{\cal A}_{G, S} \lra \prod_{\mbox{\rm holes of $S$}}({\Bbb A}^1)^{\prod}.
$$
Its fiber over zero is the space ${\rm Loc}^{\rm un}_{G, S}$. 
\end{proof}

\subsection{Duality Conjectures for decorated surfaces}
\la{secdualitycon}
\bd \la{loc.gl.s}
The moduli space ${\rm Loc}_{\G , S}$ parametrizes pairs $({\cal L},\gamma)$, where ${\cal L}$ is a twisted $\G $-local system on $S$, and $\gamma$ assigns a decoration on ${\cal L}$ to each boundary interval of 
$\widehat\partial (\ast S)$. 
\ed

{It is important to consider several different types of twisted $\G$-local system on $S$ 
which differ by the data assigned to the boundary. 
Recall that components of the punctured boundary $\widehat\partial (\ast S)$ are in bijection with the 
marked points of $S$. 
There are three options for the data at a given marked point, 
which could be either a special point, or a puncture:
 
1) No data. 

2) A decoration, 
that is a flat section of the associated decorated flag bundle ${\cal L}_{\cal A}$ near $m$. 

3) A framing, 
 that is a flat section of the associated flag bundle ${\cal L}_{\cal B}$ near $m$. 

In accordance to this, there are five different moduli spaces:

\begin{itemize}

\item ${\cal A}_{\G , S}$: decorations at both special points and punctures. 

\item ${\cal L}oc_{\G , S}$: no extra data. 

\item ${\rm Loc}_{\G , S}$: decorations at the special points only. No extra data at the punctures.

\item ${\cal P}_{\G , S}$: decorations at the special points, framings at the  punctures. 

\item ${\cal X}_{\G , S}$: framings at the special points and punctures.   
\end{itemize}

If $S$ does have special points, 
it is silly to consider ${\cal L}oc_{\G , S}$ since it ignores  
them. 

If $S$ has no punctures, then (besides ${\cal L}oc_{\G , S}$) there are three different moduli spaces:
$$
{\cal A}_{\G, S} = {\rm Loc}_{\G, S}, ~~~~{\cal P}_{\G , S}, ~~~~{\cal X}_{\G , S}. 
$$

If $S$ has no special points, i.e. it is a punctured surface, 
there are three different moduli spaces: 
$$
{\cal A}_{\G , S}, ~~~~{\cal L}oc_{\G , S}  = {\rm Loc}_{\G , S}, ~~~~{\cal P}_{\G , S} = {\cal X}_{\G , S}. 
$$

Duality Conjectures 
interchange a group $\G$ with the Langlands dual group $\G^L$, and a decorated surface $S$ with the 
dual decorated surface $\ast S$.\footnote{Although the decorated surface $\ast S$ is isomorphic to $S$, 
the isomorphism is not quite canonical.}  
Here are some examples.

If $S$ has no special points, the dual pairs look as follows:
$$
{\cal A}_{\G, S}~~~\mbox{\rm is dual to} ~~~{\cal P}_{\G^L, \ast S} = {\cal X}_{\G^L, \ast S}, ~~~~~~~~
({\cal A}_{\G, S}, {\cal W}) ~~~\mbox{\rm is dual to} ~~~{\cal L}oc_{\G^L, \ast S}  = {\rm Loc}_{\G^L, \ast S}.
$$

If $S$ does have special points, the moduli space ${\cal X}_{\G , S}$ plays a secondary role. 
The key dual pair is this: 
$$
({\cal A}_{\G, S}, {\cal W}) ~~~\mbox{\rm is dual to} ~~~{\rm Loc}_{\G^L, \ast S}.
$$
There are plenty of other dual pairs, obtained from this one by degenerating 
the potential, and simulateneously altering 
the dual space. Let us discuss some of them.

\paragraph{Generalisations.} Let us assign to each marked point $m$ of $S$ 
a subset ${I_m}\subset I$, possibly empty. 

First, let us define a new potential on the space ${\cal A}_{\G, S}$. 
Observe that any non-degenerate additive character  $\chi$ of ${\rm U}$ is naturally decomposed into a sum 
of characters parametrised by the set of positive simple roots: $\chi = \sum_{i\in I}\chi_i$. 
Then, replacing in the definition of the potential at a given marked point $m$ the nondegenerate 
character $\chi$ by the character $\sum_{i\in I_m}\chi_i$, we get a new function 
${\cal W}_{m, {I_m}}$ at $m$, and set  
\be \la{modpot}
{\cal W}_{\{{I_m}\}}:= \sum_{\mbox{\rm marked points $m$ on $S$}} {\cal W}_{m, {I_m}}. 
\ee

Next, let us define a modified moduli space 
${\cal P}^{\{{\rm I_m}\}}_{\G^L, \ast S}$.

Recall that for each simple positive root $\alpha_i$ there is a $\G$-invariant divisor in 
${\cal B} \times {\cal B}$. Let $D_i$ be its preimage in ${\cal A} \times {\cal A}$. 
We say that a pair $(A_1, A_2) \in {\cal A} \times {\cal A}$ is in position $I-I_m$ 
if $(A_1, A_2) \in {\cal A} \times {\cal A} - \cup_{i\in I-I_m}D_i$. 

Recall that $C_m$ is the boundary component of $\ast S$ matching a marked point $m$ on $S$. 
\bd \la{bulletdef} The moduli space 
${\cal P}^{\{{I_m}\}}_{\G^L, \ast S}$ paramatrizes twisted $\G^L$-local systems on $S$ plus 

a) A reduction of the structure group $\G^L$ near each puncture $m$ 
to the parabolic subgroup of type $I-I_m$.

b) A decoration at every boundary interval $C_m$ of $\ast S$ such that

\begin{itemize}

\item
The decorated flags 
at the ends of the boundary interval $C_m$ are in the position 
$I-I_m$. 
\end{itemize}
\ed
So if $I = I_m$, the data a) is empty, and  the condition b) is vacuous. 

\vskip 2mm
Finally, we consider  the largest subspace 
$$
{\cal A}^{\{{I_m}\}}_{\G, S} \subset {\cal A}_{\G, S}
$$
 on which the potential 
${\cal W}_{\{{I_m}\}}$ is regular. This condition is vacuous at punctures, and boils down to 
the $\bullet$-condition from Definition \ref{bulletdef} at boundary intervals of $\ast S$. 
So if $I_m = \emptyset$ at every special point $m$, then ${\cal A}^{\{{I_m}\}}_{\G, S} 
= {\cal A}_{\G, S}$. 

\bcon 
$
({\cal A}^{\{{I_m}\}}_{\G, S}, {\cal W}_{\{{I_m}\}}) ~~~\mbox{\rm is dual to} ~~~{\cal P}^{\{{I_m}\}}_{\G^L, \ast S}.
$
\econ
Let us now formulate what the Duality Conjecture tells about canonical bases 
for the most interesting moduli space ${\rm Loc}_{{\rm G}^L, S}$, leaving similar formulations 
in other cases as a straightforward exercise. }

\paragraph{Duality Conjecture for the space ${\rm Loc}_{{\rm G}^L, \ast S}$.} 
The group $\Gamma_S$ acts on the set ${\cal A}_{{\rm G}, S}^+(\Z^t)$, and on the 
space ${\cal O}({\rm Loc}_{{\rm G^L} , \ast S})$ of regular functions on ${\rm Loc}_{\G^L , \ast S}$.

\bcon \la{cbcon}
There is a canonical basis in the space ${\cal O}({\rm Loc}_{{\rm G}^L, \ast S})$ parametrized by 
the set ${\cal A}_{{\rm G}, S}^+(\Z^t)$. This parametrization is $\Gamma_S$-equivariant. 
\econ
\vskip 2mm

{\bf Example.}
 If $S$ is a disc $D_n$ with $n$ special points  
on the boundary, then $\Gamma_{D_n}=\Z/n\Z$. Theorem \ref{11.18.11.1} provides 
a $\Gamma_{D_n}$-equivariant canonical basis.  Thus Conjecture \ref{cbcon} is proved.
\vskip 2mm

If $\G={\rm SL}_2$ (or $\G={\rm PGL}_2$), then \cite{FG1} provides a concrete construction of the $\Gamma_S$-equivariant parametrization, using laminations.

 The following Theorem tells that the set ${\cal A}_{\G, S}^+(\Z^t)$ is of the right size.

\bt \la{cbcont}
Given an ideal triangulation $T$ of a decorated surface $S$, there is a linear basis  in  
${\cal O}({\rm Loc}_{{\rm G}^L, \ast S})$ parametrized by 
the set ${\cal A}_{{\rm G}, S}^+(\Z^t)$. 
\et

{\bf Remark.} The parametrization depends on the choice of the ideal triangulations. In particular, it is not $\Gamma_{S}$-equivariant. 

\begin{proof} 
The graph $\Gamma$ dual to the triangulation $T$ is a ribbon trivalent graph homotopy equivalent to $S$. An {\it end vertex} of $\Gamma$ is a univalent vertex of the graph. It corresponds to a boundary interval of $\widehat\partial 
S$. Let ${\rm Loc}_{{\rm G}^L, \Gamma}$ be the moduli space of pairs 
$({\cal L}, \gamma)$, where ${\cal L}$ is a ${\rm G}^L$-local system on $\Gamma$, and $\gamma$ is 
a flat section of the restriction of the local system  ${\cal L}_{\cal A}$ to the end vertices of $ \Gamma$. 

Choose an orientation of the edges of $\Gamma$. Let $V(\Gamma)$ and $E(\Gamma)$ be the sets of vertices and edges of $\Gamma$. 
Pick an edge $E = (v_1, v_2)$ of $\Gamma$, oriented from $v_1$ to $v_2$.  Given  a function $\lambda: E(\Gamma)\lra {\rm P}^+, $ 
we assign irreducible ${\rm G}^L$-modules to the two flags of $E$, denoted $V_{v, E}$:
$$
V_{(v_1, E)}:= V_{\lambda(E)}, ~~~V_{(v_2, E)}:= V_{-w_0(\lambda(E))}.
$$ 
According to \cite[Section 12.5, (12.30)]{FG1}, there is a canonical isomorphism 
\be \la{21}
{\cal O}({\rm Loc}_{{\rm G}^L, \Gamma}) = \bigoplus_{\{\lambda: E(\Gamma)\lra P^+\}}
\bigotimes_{v\in V(\Gamma)}\Bigl(\bigotimes_{(v, E)}V_{\lambda(v,E)}\Bigr)^{{\rm G}^L}
\ee
The second tensor product is over all flags incident to a given vertex $v$ of $\Gamma$. 
By Applying Theorem \ref{11.18.11.1} parametrizing a basis in the ${\rm G}^L$-invariants of the tensor product 
for each vertex of $\Gamma$, it follows that ${\cal O}({\rm Loc}_{{\rm G}^L, \Gamma}) $ admits a linear basis parametrized by ${\cal A}_{\G, S}^+(\Z^t)$. Note that the central extension  \eqref{cen.ext} is split. Following the remark after Definition \ref{moduliMGS}, the moduli space ${\rm Loc}_{{\rm G}^L, S}$ is isomorphic  to  ${\rm Loc}_{{\rm G}^L, \Gamma}.$ The Theorem is proved.
\end{proof}

\subsection{Canonical basis in the space of functions on ${\rm Loc}_{{\rm SL}_2, S}$}\la{sec10.3n}

Given any decorated surface $S$, there is a generalisation 
of integral laminations on $S$.
\bd \la{9.19.13.3} Let $S$ be a decorated  surface. An integral lamination $l$ on $S$ 
is a formal sum
\be \la{laml}
l = \sum_i n_i[\alpha_i] + \sum_j m_j [\beta_j], ~~~~n_i, m_j \in \Z_{>0}.  
\ee
where $\{\alpha_i\}$ is a collection  of simple  
nonisotopic 
loops, $\{\beta_j\}$ is a collection  of simple nonisotopic  
intervals ending inside of boundary intervals on $\partial S - \{s_1, ..., s_n\}$, 
such that the curves do not intersect, 
considered modulo isotopy. The set of integral laminations on $S$ is denoted by ${\cal L}_\Z({S})$.
\ed
\begin{figure}[ht]
\epsfxsize130pt
\centerline{\epsfbox{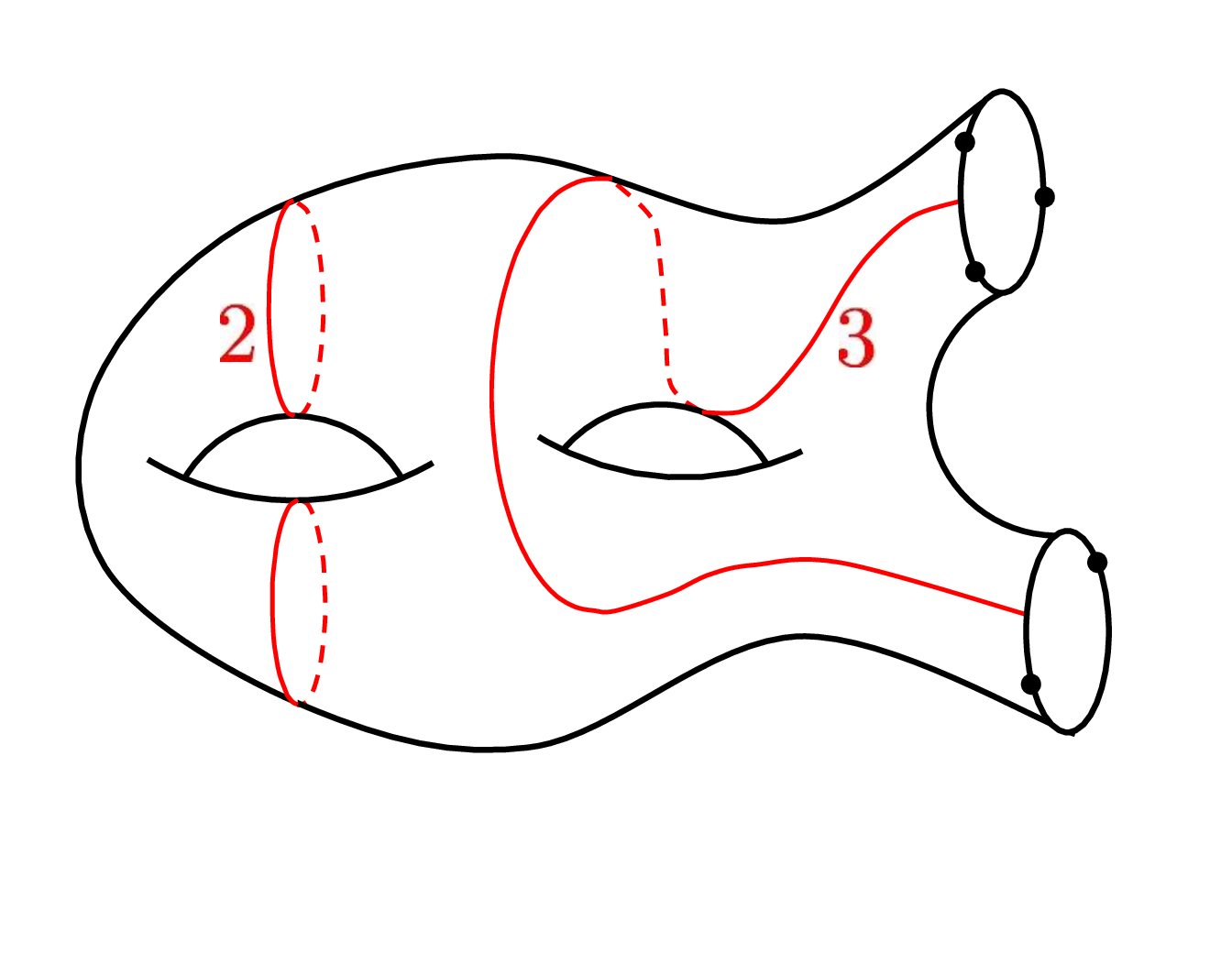}}
\caption{An integral lamination on a surface with two holes, with $2+3$ special points.}
\label{fish1}
\end{figure}
Let 
${\rm Mon}_{\alpha}({\cal L}, \alpha)$ be the monodromy of a twisted  ${\rm SL}_2$-local system 
$({\cal L}, \alpha)$ over a loop $\alpha$ on $S$.

Let us show that a simple path $\beta$ on $S$ connecting two points $x$ and $y$ on 
$\widehat \partial S$ gives rise to a regular function $\Delta_\beta$ on 
${\rm Loc}_{{\rm SL}_2, S}$. 

Let $({\cal L}, \alpha)$ be a decorated ${\rm SL}_2$-local system on $S$. The associated 
flat bundle ${\cal L}_{\cal A}$ is a two dimensional flat vector bundle without zero section. 
Let $v_x$ and $v_y$ be the tangent vectors to $\partial S$ 
at the points $x,y$.  
The decoration $\alpha$ at $x$ and $y$ provides 
vectors $l_x$ and $l_y$ in the fibers of ${\cal L}_{\cal A}$ over $v_x$ and $v_y$. The 
set $S_\beta$ of non-zero tangent vectors to $\beta$ is homotopy equivalent to a circle.  
Let us connect $v_x$ and $v_y$ by a path $p$ in $S_\beta$, and transform 
the vector $l_x$ at $v_x$ to the fiber of ${\cal L}_{\cal A}$ over $v_y$, getting there a vector 
$l_x'$. We claim that $\Delta(l_x', l_y)$ is independent of the choice of $p$. 
This uses crucially the fact that ${\cal L}$ is a twisted local system. 
So we arrive at a well defined number $\Delta(l_x', l_y)$ assigned to $({\cal L}, \alpha)$. 
We denote by $\Delta_\beta$ the obtained function on ${\rm Loc}_{{\rm SL}_2, S}$.

Given an integral lamination $l$ on $S$ as in (\ref{laml}), we a regular function $M_l$ on 
${\rm Loc}_{{\rm SL}_2, S}$ by 
$$
M_l({\cal L}, \alpha):= \prod_i {\rm Tr} ({\rm Mon}^{n_i}_{\alpha_i}({\cal L}, \alpha)) \prod_j\Delta_{\beta_j}^{m_j}({\cal L}, \alpha).  
$$

\bt \la{9.19.13.20} The functions $M_l$, $l\in {\cal L}_\Z({S})$, form a linear basis 
in the space ${\cal O}({\rm Loc}_{{\rm SL}_2, S})$. 
\et

\bt \la{9.19.13.21} For any decorated surface $S$, there is 
a canonical isomorphism
$$
{\cal A}^+_{{\rm PGL}_2, S}(\Z^t) = {\cal L}_\Z(S).
$$ 
\et

Theorem \ref{9.19.13.21} is proved similarly to Theorem 12.1 in \cite{FG1}. 
Notice that ${\cal A}_{{\rm PGL}_2, S}$ is a positive space for the adjoint group ${\rm PGL}_2$,  
the potential ${\cal W}$ lives on this space and is a positive function there. 
Theorem \ref{9.19.13.20} is proved by using arguments similar to the proof of 
Theorem \ref{cbcont} and \cite[Proposition 12.2]{FG1}. 

Combining Theorem \ref{9.19.13.20} and Theorem \ref{9.19.13.21} 
we arrive at a construction of the canonical basis predicted by Conjecture \ref{cbcon} for $\G={\rm PGL}_2$. 

\subsection{Surface
affine Grassmannian and amalgamation.} 
\la{sec11.2}

\paragraph{The surface affine Grassmannian ${\rm Gr}_{{\rm G},   S}$.}

Given a twisted 
right ${\rm G}({\cal K})$-local system ${\cal L}$ on $S$, there is the associated flat  affine Grassmannian bundle 
$
{\cal L}_{\rm Gr}:= {\cal L}\times_{{\rm G}({\cal K})}{\rm Gr}. 
$ 
Similarly to Definition \ref{moduliMGS}, we define 

\bd \la{srfag}
Let $S$ be a decorated surface. 
The moduli space ${\rm Gr}_{{\rm G},   S}$  parametrizes 
pairs $({\cal L}, \nu)$ where ${\cal L}$ is a twisted 
right ${\rm G}({\cal K})$-local system on $S$, and 
$\nu$ a flat section of the restriction of ${\cal L}_{\rm Gr}$ 
 to  the punctured boundary  $\widehat \partial (\ast S)$. 
\ed
 Abusing terminology, the data $\nu$ is given by the lattices $\lr_m$ at the marked points $m$ on $S$.

The moduli space $\widetilde {\rm Gr}_{{\rm G},   S}$  
parametrizes similar data 
$(\widetilde {\cal L}, \nu)$, where $\widetilde {\cal L}$ is a twisted 
${\rm G}({\cal K})$-local system on $S$ 
trivialized at a given point of $S$. So one has 
$
{\rm Gr}_{{\rm G},   S} = {\rm G}\backslash \widetilde {\rm Gr}_{{\rm G},   S}.
$
\vskip 2mm

{\bf Example}. Let $D_n$ be a disc with $n$ special points on the boundary. 
Then a choice of a special point provides isomorphisms
$$
{\rm Gr}_{{\rm G}, D_n}  = {\rm Conf}_n({\rm Gr}), ~~~ 
\widetilde {\rm Gr}_{{\rm G}, D_n}  = {\rm Gr}^n.
$$


\paragraph{Cutting and amalgamating decorated surfaces.} 

Let ${\rm I}$ be an ideal edge on a decorated surface $S$, i.e. a path connecting two marked points. 
Cutting $  S$ along the edge ${\rm I}$ we get a decorated surface $  S^*$. 
Denote by ${\rm I}'$ and ${\rm I}''$ the 
boundary intervals on $  S^*$ obtained by cutting along ${\rm I}$. 

Conversely, gluing boundary intervals ${\rm I}'$ and ${\rm I}''$ 
on a decorated surface $  S^*$,  we get a new decorated surface $  S$. 
We assume that the intervals ${\rm I}'$ and ${\rm I}''$ 
on $  S^*$ are oriented by the orientation of the surface, 
and the gluing preserves the orientations. 

More generally, let $  S$ be a decorated surface obtained from  decorated surfaces 
$  S_1, ...,   S_n$ by gluing pairs $\{{\rm I}'_1, {\rm I}''_1\}$, ..., $\{{\rm I}'_m, {\rm I}''_m\}$ 
of oriented boundary intervals. We say that 
$  S$ is the {\it amalgamation} of decorated surfaces $  S_1, ...,   S_n$, and use the notation 
$
  S =   S_1 \ast ... \ast   S_n. 
$ 
 Abusing notation, we do not specify 
the pairs  
$\{{\rm I}'_1, {\rm I}''_1\}, ..., \{{\rm I}'_m, {\rm I}''_m\}$.

\paragraph{Amalgamating surface affine Grassmannians.} 
There is a moduli space ${\rm Gr}_{{\rm G}, {\rm I}}$ related to an oriented closed interval ${\rm I}$, so that 
there is a canonical isomorphism of stacks
$$
{\rm Gr}_{{\rm G}, {\rm I}} = {\rm Conf}_2({\rm Gr}). 
$$
\bd \la{6.9.12.10}
Let ${\rm I}'$, ${\rm I}''$ be   boundary intervals on a decorated 
surface $  S^*$, perhaps disconnected. The 
{\rm amalgamation stack} ${\rm Gr}_{{\rm G},   S^*}({\rm I}'\ast {\rm I}'')$ parametrises triples 
$({\cal L}, \gamma, g)$, where $({\cal L}, \gamma)$ is the data parametrised by 
${\rm Gr}_{{\rm G},   S^*}$, and $g$ is a {\rm gluing data}, given by 
an equivalence of stacks
\be \la{6.9.12.101a}
g: {\rm Gr}_{{\rm G}, {\rm I}'} \stackrel{\sim}{\lra} {\rm Gr}_{{\rm G}, {\rm I}''}.
\ee
\ed

This immediately implies that there is a canonical equivalence of stacks:
\be \la{isostacks}
{\rm Gr}_{{\rm G},   S} \stackrel{\sim}{\lra} {\rm Gr}_{{\rm G},   S^*}({\rm I}'\ast {\rm I}'').
\ee

Given decorated surfaces $  S_1, ...,   S_n$ and 
a collection 
$\{{\rm I}'_1, {\rm I}''_1\}$, ..., $\{{\rm I}'_m, {\rm I}''_m\}$ of pairs of boundary intervals, 
generalising the construction from Definition \ref{6.9.12.10}, 
we get the amalgamation stack 
$$
{\rm Gr}_{{\rm G},   S_1\ast ...\ast   S_n} = {\rm Gr}_{{\rm G},   S_1\ast ...\ast   S_n}({\rm I}'_1\ast {\rm I}''_1, \ldots , {\rm I}'_m\ast {\rm I}''_m).
$$
Applying equivalences (\ref{isostacks}) we get 

\bl \la{6.9.12.100} There is a canonical equivalence of stacks:
\be \la{isostacks2}
{\rm Gr}_{{\rm G},   S} \stackrel{\sim}{\lra}
{\rm Gr}_{{\rm G},   S_1\ast ...\ast   S_n}({\rm I}'_1\ast {\rm I}''_1, \ldots , {\rm I}'_m\ast {\rm I}''_m).
\ee
\el

Let $T$ be an ideal triangulation of a decorated surface $  S$. 
Let $t_1, ..., t_n$ be the triangles of the triangulation. 
Abusing notation, denote by $t_i$ the decorated surface given by the triangle $t_i$, 
with the special points given by the vertices. 
Denote by ${\rm I}'_i$ and ${\rm I}''_i$ the pair of edges obtained by cutting an edge 
${\rm I}_i$ of the triangulation $t$, 
$i=1, ..., m$. 
Then one has an isomorphism of 
stacks
\be \la{isostacks3}
{\rm Gr}_{{\rm G},   S} = {\rm Gr}_{{\rm G}, t_1\ast ...\ast t_n}({\rm I}'_1\ast {\rm I}''_1, \ldots , 
{\rm I}'_m\ast {\rm I}''_m).
\ee

\subsection{Top components of the surface affine Grassmannian}
\la{sec11.3}
\subsubsection{Regularised dimensions} 
\la{sec11.3.1}
Recall that if a finite dimensional group ${\rm A}$ acts on a finite dimensional variety $X$, 
we define the dimension of the stack $X/{\rm A}$ by 
$$
{\rm dim}~X/{\rm A}:= {\rm dim}~X - {\rm dim}~{\rm A}. 
$$
Our goal is to generalise this definition to the case when $X$ and ${\rm A}$ could be infinite dimensional. 

\paragraph{Dimension torsors ${\bf t}^n$.}
Let us first define a rank one $\Z$-torsor ${\bf t}$. The kernel ${\Bbb N}$ of the evaluation map 
$\G({\cal O}) \to \G(\C)$ is a prounipotent algebraic group 
over $\C$. 
Let $N$ be its finite codimension normal subgroup. 
We assign to each such an $N$ a copy $\Z_{(N)}$ of $\Z$, and for each pair $N_1 \subset N_2$ 
such that $N_2/N_1$ is a finite dimensional, an isomorphism of $\Z$-torsors 
\be \la{in}
i_{N_1, N_2}: \Z_{(N_1)} \lra \Z_{(N_2)}, ~~~ x \lms x + {\rm dim}~N_2/N_1.
\ee
\bd
A $\Z$-torsor ${\bf t}$ is given by the collection of $\Z$-torsors $\Z_{(N)}$ and isomorphisms 
$i_{N_1, N_2}$. We set ${\bf t}^n:= {\bf t}^{\otimes n}$ for any $n \in \Z$. 
\ed
In particular, ${\bf t}^0=\Z$. To define an element of ${\bf t}^n$ means to exhibit a collection 
of integers $d_N$ assigned to the finite codimension  subgroups $N$ of ${\Bbb N}$ 
related by isomorphisms (\ref{in}).

\vskip 2mm

{\bf Example}. There is an element ${\bf dim}~{\rm G}({\cal O}) \in {\bf t}$, given by an assignment 
$$
{\bf dim}~{\rm G}({\cal O}):= \{N \lms {\rm dim}~{\rm G}({\cal O})/N\in \Z_{(N)}\}\in {\bf t}.
$$
More generally, there is an element 
$$
n~{\bf dim}~{\rm G}({\cal O}) := \{N \lms {\rm dim}~({\rm G}({\cal O})/N)^n\in \Z_{(N)}\} \in {\bf t}^n.
$$ 
For example, the stack $\ast/{\rm G}({\cal O})^n$, where $\ast={\rm Spec}(\C)$ is the point,  has dimension
$$
{\bf dim}~\ast/{\rm G}({\cal O})^n = - n~{\bf dim}~{\rm G}({\cal O}) \in {\bf t}^{-n}.
$$

\vskip 2mm
If $X$ and $Y$ have dimensions ${\bf dim}~X \in {\bf t}^n$ and ${\bf dim}~Y \in {\bf t}^m$, then 
${\bf dim}~X\times Y \in {\bf t}^{n+m}$.

\paragraph{Dimension torsors ${\bf t}_{\rm A}^n$.} We generalise this construction by replacing the group ${\rm G}({\cal O})$ by a 
pro-algebraic group ${\rm A}$, which 
has a finite codimension prounipotent normal subgroup.\footnote{Taking the quotient by a unipotent group does not affect the category of equivariant sheaves. This is why we require the prounipotence condition here.} Then there are 
the dimension torsor ${\bf t}_{\rm A}$, its tensor powers ${\bf t}_{\rm A}^n$, $n\in \Z$, and  
an element ${\bf dim}~{\rm A}\in {\bf t}_{\rm A}$. One has ${\bf t}_{{\rm A}^n}= {\bf t}^n_{{\rm A}}$. Moreover, 
$$
n~{\bf dim}~{\rm A} \in {\bf t}_{\rm A}^n, ~~~~{\bf t}_{\rm A}^n= \{m+ n~{\bf dim}~{\rm A}\}, m\in \Z.
$$

\paragraph{Regularised dimension.} Given such a group ${\rm A}$, we can define the dimension of a stack 
${\cal X}$   
under  the following assumptions.  

\begin{enumerate}


\item There is a finite codimension prounipotent subgroup ${\rm N} \subset {\rm A}$ such that $$
{\rm N}^n  ~~\mbox{acts freely on ${\cal X}$.}
$$ 

\item There is a finite dimensional stack ${\cal Y}$ and an action of the group 
${\rm A}^m$ on ${\cal Y}$ such that 
\be \la{canYX}
{\cal Y}/{\rm A}^m = {\cal X}/{\rm N}^n.
\ee

\item There exists a finite codimension normal prounipotent subgroup 
${\rm M} \subset {\rm A}$ such that the action of ${\rm A}^m$ on ${\cal Y}$ 
restricts to the trivial 
action of the subgroup ${\rm M}^m$  on ${\cal Y}$. 
\end{enumerate}
The last condition implies that we have a finite dimensional stack ${\cal Y}/({\rm A}/{\rm M}^m)$. 
The stack ${\cal Y}/{\rm A}^m$ is the quotient of the stack ${\cal Y}/({\rm A}/{\rm M}^m)$ 
by the trivial action of the group ${\rm M}^m$. 

In this case we define an element of the torsor ${\bf t}_{\rm A}^{n-m}$ by the assignment 
\be \la{9.30.12.1000}
({\rm N}, {\rm M})\lms {\rm dim}({\cal Y}/{\rm A}^m) + {\bf dim}~({\rm N}^n):= 
(n- m)~ {\bf dim}~{\rm A} +{\rm dim}~{\cal Y}- n ~
{\rm dim}({\rm A}/{\rm N})\in {\bf t}_{\rm A}^{n-m}.
\ee

\bd \la{regddef}
Assuming 1) -- 2), the assignment (\ref{9.30.12.1000}) defines the regularised dimension
$$
{\bf dim}~{\cal X} \in {\bf t}_{\rm A}^{n-m}.
$$
\ed

{\bf Remark}. Often an  infinite dimensional stack ${\cal X}$ does not have a canonical 
presentation (\ref{canYX}), but rather a collection of such presentations. 
For instance such a presentation of the stack ${\cal M}_l^\circ$ defined 
below depends on a choice of an 
ideal triangulation $T$ of $S$. 
Then we need to prove that the regularised dimension is independent of the choices.

\subsubsection{Top components of the stack ${\rm Gr}_{{\rm G},   S}$} \la{sec11.3.2}

Suppose that a decorated surface $  S$ is an amalgamation of decorated surfaces:
\be \la{6.9.12.1}
  S =   S_1 \ast \ldots \ast   S_n.
\ee
\bd
Given an amalgamation pattern 
(\ref{6.9.12.1}), define the amalgamation
$$
{\cal A}_{{\rm G},   S_1}(\Z^t)\ast \ldots \ast {\cal A}_{{\rm G},   S_n}(\Z^t):= 
\{(l_1, ..., l_n) \in {\cal A}_{{\rm G},   S_1}(\Z^t)\times \ldots \times {\cal A}_{{\rm G},   S_n}(\Z^t) ~~|~~ \mbox{(\ref{7.30.12.1}) holds} \}:
$$
\be \la{7.30.12.1}
\mbox{$\pi^t_{{\rm I}'_k}(l_i) = \pi^t_{{\rm I}''_k}(l_j)$ for any boundary intervals ${\rm I}'_k\subset   S_i$ and ${\rm I}''_k\subset   S_j$ 
glued in $  S$}. 
\ee

\ed
\bl \la{9.18.13.1}
Given an amalgamation pattern 
(\ref{6.9.12.1}), there are canonical isomorphism of sets 
$$
{\cal A}_{{\rm G},   S}(\Z^t) = {\cal A}_{{\rm G},   S_1}(\Z^t)\ast \ldots \ast {\cal A}_{{\rm G},   S_n}(\Z^t).
$$
$$
{\cal A}^+_{{\rm G},   S}(\Z^t) = {\cal A}^+_{{\rm G},   S_1}(\Z^t)\ast \ldots \ast {\cal A}^+_{{\rm G},   S_n}(\Z^t).
$$
\el

In this case we say that $l$ is presented as an amalgamation, and write $l=l_1 \ast \ldots \ast l_n$. 

Let us pick an ideal triangulation $T$ of $S$, and 
present $S$  as an amalgamation of the triangles:
\be \la{6.9.12.1a}
  S =   t_1 \ast \ldots \ast   t_n.
\ee
By Lemma \ref{9.18.13.1},  any $l\in {\cal A}_{{\rm G},   S}^+(\Z^t)$ is uniquely presented 
 as an amalgamation 
\be \la{6.9.12.1b}
l =   l_1 \ast \ldots \ast   l_n, ~~~~l_i\in {\cal A}_{{\rm G},   t_i}^+(\Z^t). 
\ee
Recall that given a polygon $D_n$, there are cycles 
$$
{\cal M}^\circ_l:= \kappa({\cal C}_l^\circ)\subset {\rm Gr}_{G, D_n}, ~~~~ 
l\in {\cal A}^+_{{\rm G},   D_n}(\Z^t). 
$$

\bd Given an ideal triangulation $T$ of $S$
and an $l\in {\cal A}_{{\rm G},   S}^+(\Z^t)$ we set, using amalgamations (\ref{6.9.12.1a}) and  
(\ref{6.9.12.1b}), 
$$
{\cal M}^\circ_{T, l} = {\cal M}^\circ_{t_1, l_1}\ast \ldots \ast {\cal M}^\circ_{t_n, l_n}, ~~~~
{\cal M}_{T, l}:= ~~\mbox{Zariski closure of} ~~{\cal M}^\circ_{T, l}.
$$
\ed

Thanks to Lemma \ref{9.21.17.56h}, the restriction to the boundary intervals of $S$ leads to a map of sets
$$
{\cal A}_{{\rm G},   S}^+(\Z^t) \lra {{\rm P}^+}^{\{\mbox{boundary intervals of $S$}\}}.
$$
It assigns to a point $l\in {\cal A}_{{\rm G},   S}^+(\Z^t)$ 
a collection of dominant coweights $\lambda_{{\rm I}_1}, ..., \lambda_{{\rm I}_n}\in {\rm P}^+$ 
at the boundary intervals ${\rm I}_1, ..., {\rm I}_n$ of $  S$. 

For any decorated subsurface $i: S' \subset   S$ there is a projection given by 
the restriction map for the surface affine Grassmannian: 
$
r_{\rm Gr}: {\rm Gr}_{{\rm G},   S} \lra {\rm Gr}_{{\rm G},   S'} .
$ 
There are two canonical projections:
\be \la{9.17.13.3}
\begin{array}{ccccccc}
{\cal A}^+_{{\rm G},   S}(\Z^t)& &&& & {\rm Gr}_{{\rm G},   S} \\
&&&&&&\\
r^t_{\cal A}\downarrow &&&& &\downarrow r_{\rm Gr}\\
&&&&&&\\
{\rm Conf}^+_{{\rm G},   S'}({\cal A})(\Z^t)&&&&& {\rm Gr}_{{\rm G},   S'} )
\end{array}
\ee

\bt \la{6.8.12.1}
 Let $  S$ be a decorated surface.

i) The stack ${\cal M}_{T, l}$ does not depend on the triangulation $T$. We denote it by ${\cal M}_l$. 

ii) Let $l\in {\cal A}_{\G,   S}^+(\Z^t)$. Let  
$\{{\rm I}_1, ..., {\rm I}_n\}$ be  the set of 
boundary intervals of  $  S$, 
and $\lambda_{{\rm I}_1}, ..., \lambda_{{\rm I}_n}$ are the dominant coweights 
assigned to them by $l$. Then 
\be \la{6.8.12.3}
{\bf dim}~{\cal M}_l = \langle \rho, \lambda_{{\rm I}_1} + \ldots + \lambda_{{\rm I}_n}\rangle -\chi(S) ~{\bf dim}~{\rm G}({\cal O})\in
{\bf t}^{-\chi(S)}. 
\ee

iii)  The stacks ${\cal M}_l$, $l\in {\cal A}^+_{{\rm G},   S}(\Z^t)$, 
 are top dimensional 
components of 
${\rm Gr}_{{\rm G},   S}$. 

iv) The map $l \lms {\cal M}_l$ provides a bijection 
$$
{\cal A}^+_{{\rm G},   S}(\Z^t) \stackrel{\sim}{\lra} \{\mbox{top dimensional components of the stack 
${\rm Gr}_{{\rm G},   S}$}\}.
$$
This isomorphism commutes with the restriction to decorated subsurfaces of $S$. 
\et

\begin{proof} 
Let us calculate first dimensions of the stacks 
${\cal M}^\circ_{T, l}$, and show that they are given by formula (\ref{6.8.12.3}). 
We present first a heuristic dimension count, 
and then fill the necessary details. 

\paragraph{Heuristic dimension count.} Let us present a decorated surface $  S$ 
as an amalgamation of a (possible disconnected) decorated surface 
along a pair of boundary intervals ${\rm I}', {\rm I}''$, 
as in Definition \ref{6.9.12.10}. 
The space of isomorphisms $g$ from (\ref{6.9.12.101a}) is a disjoint union $\G({\cal K})$-torsors parametrised by 
dominant coweights $\lambda$, since the latter parametrise $\G({\cal K})$-orbits on 
${\rm Gr}\times {\rm Gr}$. Pick one of them. 

Let $\lr'_0 \stackrel{{\lambda}}{\lra} \lr'_1$ (respectively 
$\lr_0'' \stackrel{\lambda}{\lra} \lr_1''$) be a pair of lattices assigned to 
the vertices of the interval 
${\rm I}'$ (respectively 
${\rm I}''$).
Then the gluing data is a map 
$
g: (\lr'_0, \lr'_1) \lra (\lr''_0, \lr''_1).
$ 
Let ${\rm G}_{\lambda}$ be the subgroup stabilising the pair 
$\lr'_0 \stackrel{\lambda}{\lra} \lr'_1$.  The space of gluings is a ${\rm G}_{\lambda}$-torsor. 
The group ${\rm G}_{\lambda}$
 is a subgroup of codimension $2 \langle \rho, \lambda\rangle$ in ${\rm Aut}~\lr_0 \stackrel{\sim}{=} {\rm G}({\cal O})$. 
So
$$
{\bf dim}~{\rm G}_{\lambda} = {\bf dim}~{\rm G}({\cal O}) -  2 \langle \rho, \lambda\rangle = 
 {\bf dim}~{\rm G}({\cal O}) - {\bf dim}~{\rm Gr}_{{\lambda, \lambda^{\vee}}}.
$$

Take the stack ${\cal M}^\circ_{t, l}$  
assigned to a triangle $t$ and a point $l\in {\rm Conf}_3^+({\cal A})(\Z^t)$. 
Let $\lambda_1, \lambda_2, \lambda_3$ be the 
dominant coweights assigned to 
the sides of the triangle by $l$. Then ${\cal M}^\circ_{t, l}$ is an open part of a component of the stack 
${\rm Gr}_{\lambda_1, \lambda_2, \lambda_3}/{\rm G}({\cal O})$. Thus 
\be \la{9.30.12.105}
{\bf dim}~{\cal M}^\circ_{t, l} = \langle \rho, \lambda_1+ \lambda_2+ \lambda_3 \rangle - 
{\bf dim}~{\rm G}({\cal O}) \in {\bf t}^{-1}.
\ee
  
Let us calculate now the dimension of the stack ${\cal M}^\circ_{T, l}$. 
Let $|{\cal T}|$ be the number of triangles, and ${\cal E}_{\rm int}$ 
(respectively ${\cal E}_{\rm ext}$) the set of the internal 
(respectively external) edges of the triangulation $T$.  
Then the  dimension of the product of stacks assigned to the triangles is 
$$
\sum_{E\in {\cal E}_{\rm ext}} \langle \rho, \lambda_{E} \rangle  + 2 \sum_{E\in {\cal E}_{\rm int}} \langle \rho, \lambda_{E} \rangle - 
|{\cal T}| ~{\bf dim}~{\rm G}({\cal O})\in {\bf t}^{-|{\cal T}|}.
$$

Gluing two boundary intervals into an internal edge $E$, with the 
dominant weights $\lambda_E$ associated to it,   
we have to add the dimension of the corresponding gluing data torsor, that is 
$$
{\bf dim}~{\rm G}({\cal O})- 2\langle \rho, \lambda_E\rangle\in {\bf t}.
$$ 
So, gluing all the intervals, we get 
$$
\sum_{E\in {\cal E}_{\rm ext}} \langle \rho, \lambda_{E} \rangle  + (|{\cal E}_{\rm int}|- |{\cal T}|) ~{\bf dim}~{\rm G}({\cal O}) = 
\sum_{E\in {\cal E}_{\rm ext}} \langle \rho, \lambda_{E} \rangle  -\chi(S)~{\bf dim}~{\rm G}({\cal O}) = (\ref{6.8.12.3}).
$$
Notice that 
$|{\cal E}_{\rm int}|- |{\cal T}| =-\chi(S)$. Indeed, the triangles $t$ 
with external sides removed cover the surface $S$ minus the boundary, 
which has the same Euler characteristic as $S$.

\paragraph{Rigorous dimension count.} For each of the triangles $t$ of the triangulation $T$ 
there are three dominant coweights $\underline {\lambda}(t):= \lambda_1(t), \lambda_2(t), \lambda_3(t)$ assigned  by $l$ to the sides 
of $t$. 
Pick a vertex $v(t)$ of the triangle $t$. We present the stack 
${\rm Gr}_{\G, t}$ as a quotient of the convolution variety 
\be \la{9.30.12.1a}
{\rm Gr}_{\G, t}= {\rm Gr}_{\underline {\lambda}(t)}/{\rm G}({\cal O}).
\ee
Namely, choose the lattice $\lr_{v(t)}$ at the vertex $v(t)$ to be the standard lattice 
$\lr_{v(t)} = {\rm G}({\cal O})$. 

There exists a finite codimension normal prounipotent subgroup 
$
N_{t, l} \subset {\rm G}({\cal O})
$   
acting trivially on ${\rm Gr}_{\underline{\lambda}(t)}$. 
It depends on the choice of coweights $\underline{\lambda}(t)$, 
and, via them,  on the choice of the $t$ and $l$. 
We assign to each finite codimension normal subgroup $N'_{t, l} \subset N_{t, l}$ a finite dimensional 
stack
$$
\frac{{\rm Gr}_{\underline {\lambda}(t)}}{{\rm G}({\cal O})/N'_{t, l}}.
$$
Its dimension is 
$
\langle \rho, \lambda_1+ \lambda_2+ \lambda_3 \rangle - 
{\rm dim}~{\rm G}({\cal O})/N'_{t, l}.
$ 
This just means that we have formula (\ref{9.30.12.105}). 
\vskip 3mm

There is a canonical surjective map of stacks
\be \la{9.30.12.2000}
{\rm Gr}_{G, S} \lra \prod_{t\in T}{\rm Gr}_{G, t} = 
\prod_{t\in T}{\rm Gr}_{\underline {\lambda}(t)}/{\rm G}({\cal O}).
\ee
Its fibers are torsors over the product 
over the set ${\cal E}_{\rm int}$ of internal edges $E$ of $T$ of certain  groups $G_{\lambda(E)}$ defined as follows. 
Let $\lambda(E)$ be the dominant coweight 
assigned to $E$  by $l$. Consider the pair $E', E''$  of edges of triangles 
glued into the edge $E$.  
For each of them, there is a pair of the lattices assigned to its vertices. We get two pairs of lattices: 
$$
(\lr_{E'}^- \stackrel{\lambda(E)}{\lra} \lr_{E'}^+) ~~~\mbox{and}~~~  (\lr_{E''}^-\stackrel{\lambda(E)}{\lra}  \lr_{E''}^+). 
$$
Choose one of the edges, say $E'$. 
Set 
$
\G_{\lambda(E)}:= {\rm Aut}~ (\lr_{E'}^- \stackrel{\lambda(E)}{\lra} \lr_{E'}^+).
$ 
Therefore we conclude that
$$
\mbox{The fibers of the map (\ref{9.30.12.2000}) are torsors over the group} ~~
\prod_{E \in {\cal E}_{\rm int}}\G_{\lambda(E)}.
$$

 For each $E$, choose a finite codimension subgroup 
$
N_{\lambda(E)} \subset \G_{\lambda(E)}.
$
Then we are in the situation discussed right before Definition \ref{regddef}, where
$$
{\cal X} = {\cal M}^\circ_l, ~~~~{\rm A} = {\G}({\cal O}), ~~~~
N:= \cap_{E \in {\cal E}_{\rm int}}N_{\lambda(E)}, ~~~~M=\cap_{t}N'_{t, l}, ~~~~n=|{\cal E}_{\rm int}|, ~~~~ 
m=|{\cal T}|.
$$
So we get the expected formula for the regularised dimension of ${\cal M}_{T, l}^\circ$. 

The resulting regularised dimension does not depend on the choice of ideal triangulation $T$ --
 the triangulation does not enter to the answer. 

Alternatively, one can see this  as follows.  
Any two ideal triangulations of $S$ are related by a sequence of flips. 
Let $T \lra T'$ be a flip at an edge $E$. Let $R_E$ be the unique rectangle of the triangulation $T$ 
with the diagonal $E$. 
Consider the restriction map $\pi: {\rm Gr}_{G, S} \lra {\rm Gr}_{R_E, S}$. So 
one can fiber   ${\cal M}^\circ_{l}$ over the component 
${\cal M}^\circ_{\pi^t(l)}$. 
The dimension of the latter does not depend on the choice of the triangulation of the rectangle.

A similar argument with a flip of triangulation proves i). Combining with the formula for 
 the regularised dimension of ${\cal M}_{T, l}^\circ$ we get ii). 

iii), iv). 
Present $  S$ as an amalgamation of the triangles of an ideal triangulation. It is known that 
the cycles ${\cal M}_l$ are the top dimensional components of the  convolution 
variety, and thus the stack ${\rm Gr}_{G, t}$, assigned to the triangle. 
It remains to use Lemma \ref{6.9.12.100}.

\end{proof}


\section{Cluster varieties, frozen variables and potentials} \la{seccluster}
\subsection{Basics of cluster varieties} 


\bd A quiver ${\bf q}$ is described by a data 
$
(\Lambda, \Lambda_0, \{e_i\}, (\ast, \ast)),
$ 
 where

\begin{enumerate}

\item $\Lambda$ is a lattice, $\Lambda_0$ is a sublattice of $\Lambda$, and 
$\{e_i\}$ is a basis  of $\Lambda$  such that $\Lambda_0$ is generated by a subset of 
{\it frozen basis vectors}; 

\item  $(\ast, \ast)$ is a skewsymmetric $\frac{1}{2}\Z$-valued bilinear form 
on $\Lambda$ with
$(e_i, e_j) \in \Z$ unless 
$e_i,e_j \in {\Lambda}_0$.
\end{enumerate}
\ed

Any non-frozen 
basis element $e_k$ provides  a {\it mutated in the 
direction $e_k$} quiver ${\mathbf q'}$. 
The  quiver 
${\bf q}'$ is defined by changing the 
basis $\{e_i\}$ only. 
The new basis $\{e'_i\}$ is defined via halfreflection of the $\{e_i\}$ along the hyperplane 
$(e_k, \cdot)=0$:
\begin{equation} \label{12.12.04.2a}
e'_i := 
\left\{ \begin{array}{lll} e_i + [\varepsilon_{ik}]_+e_k
& \mbox{ if } &  i\not = k\\
-e_k& \mbox{ if } &  i = k.\end{array}\right.
\end{equation}
Here $[\alpha]_+:= \alpha$ if $\alpha\geq 0$ 
and $[\alpha]_+:=0$ otherwise. The frozen/non-frozen basis vectors 
of the mutated quiver are the images of the  ones of the original quiver. 
The composition of two mutations in the same direction $k$ is  an isomorphism of quivers.

Set ${\varepsilon}_{ij} := (e_i, e_j)$. A {quiver}  can be described by 
  a data ${\bf q}=({\rm I}, {\rm I}_0,\varepsilon)$, where   
${\rm I}$ (respectively ${\rm I}_0$) is the set parametrising the basis vectors (respectively 
frozen vectors). Formula (\ref{12.12.04.2a}) 
amounts then to the Fomin-Zelevinsky formula telling how the $\varepsilon$-matrix changes under mutations. 
\begin{equation} \label{5.11.03.6}
 \varepsilon'_{ij} := \left\{ \begin{array}{lll} 
- \varepsilon_{ij} & \mbox{ if $k \in \{i,j\}$} \\ 
\varepsilon_{ij} & \mbox{ if $\varepsilon_{ik}
\varepsilon_{kj} \leq 0, \quad k \not \in \{i,j\}$} \\
\varepsilon_{ij} + |\varepsilon_{ik}| \cdot \varepsilon_{ kj}& 
\mbox{ if $\varepsilon_{ik}
\varepsilon_{kj} > 0, \quad k \not \in \{i,j\}.$}\end{array}\right.
\end{equation}

We assign to every quiver ${\bf q}$ two sets of coordinates, each
 parametrised by the set ${\rm I}$: the  
${\cal X}$-coordinates $\{X_i\}$, and the  ${\cal A}$-coordinates $\{A_i\}$. 
Given a mutation of quivers $\mu_k: {\bf q} \lms {\bf q}'$, 
the cluster coordinates assigned to these quivers are related as follows. 
Denote the cluster coordinates related to 
the quiver ${\bf q'}$ by $\{X'_i\}$ and $\{A'_i\}$. Then 
\begin{equation} \label{5.11.03.1a}
  A_{k}A'_{k} := \quad \prod_{j| \varepsilon_{kj} >0} 
A_{j}^{\varepsilon_{kj}} + \prod_{j| \varepsilon_{kj} <0} 
A_{j}^{-\varepsilon_{kj}}; \qquad A'_{i} =  
A_{i}, \quad i \not = k.
\end{equation}
If any of the sets $\{j| \varepsilon_{kj} >0\}$ or
$\{j| \varepsilon_{kj} < 0\}$ is empty, the corresponding monomial 
is $1$. 
\begin{equation} \label{5.11.03.1x}
X'_{i} := \left\{\begin{array}{ll} X_k^{-1}& \mbox{ if }  i=k \\
 X_i(1+X_k^{-{\rm sgn} (\varepsilon_{ik})})^{-\varepsilon_{ik}} & \mbox{ if }   i\neq k,
\end{array} \right.
\end{equation}
The tropicalizations of these transformations are
\begin{equation} \label{5.11.03.1atr}
  a'_{k} := - a_{k}+  {\rm min}\left\{\sum_{j| \varepsilon_{kj} >0} 
{\varepsilon_{kj}}a_{j}, \sum_{j| \varepsilon_{kj} <0} 
-\varepsilon_{kj}a_{j}\right\}; \qquad a'_{i} =  
a_{i}, \quad i \not = k.
\end{equation}
\begin{equation} \label{5.11.03.1xtr}
x'_{i} := \left\{\begin{array}{ll} -x_k& \mbox{ if }  i=k \\
 x_i-\varepsilon_{ik}{\rm min}\{0, -{\rm sgn} (\varepsilon_{ik})x_k\} & \mbox{ if }   i\neq k,
\end{array} \right.
\end{equation}

Cluster transformations are transformations of cluster 
coordinates obtained by composing mutations. 
Cluster ${\cal A}$-coordinates and mutation formulas (\ref{12.12.04.2a})
 and (\ref{5.11.03.1a}) 
are main ingredients of the definition of cluster algebras \cite{FZI}. 
Cluster ${\cal X}$-coordinates and mutation formulas (\ref{5.11.03.1x}) describe a dual object, 
introduced in \cite{FG2} under the name {\it cluster ${\cal X}$-variety}. 

\paragraph{The cluster volume forms \cite{FG5}.}  
Given a quiver ${\bf q}$, consider the volume forms
$$
{\rm Vol}^{\bf q}_{\cal A}:= d\log A_1 \wedge \ldots \wedge d\log A_n, ~~~~
{\rm Vol}^{\bf q}_{\cal X}:= d\log X_1 \wedge \ldots \wedge d\log X_n.
$$
Cluster transformations 
preserve them up to a sign: 
given a mutation ${\bf q} \lms {\bf q}'$, we have 
$$
{\rm Vol}^{\bf q'}_{\cal A} = - 
{\rm Vol}^{\bf q}_{\cal A}, \qquad {\rm Vol}^{\bf q'}_{\cal X} = - {\rm Vol}^{\bf q}_{\cal X}.
$$
Denote by ${\rm Or}_\Lambda$  the two element set of orientations  of a rank $n$ 
lattice $\Lambda$, given by expressions 
$l_1\wedge ...\wedge l_n$ where $\{l_i\}$ form a basis of $\Lambda$. {\it An orientation ${\rm or}_\Lambda$ of $\Lambda$} is 
a choice of one of its elements. 
Given a basis $\{e_i\}$ of $\Lambda$, we define its sign
${\rm sign}(e_1, ..., e_n)$ by
$
e_1\wedge  ...\wedge e_n = {\rm sign}(e_1, ..., e_n){\rm or}_\Lambda.
$ 
A quiver mutation changes the sign 
of the basis, and the sign of each of the cluster volume forms. 
So there is a definition of the cluster volume forms invariant under cluster transformations.

 \bd Choose an orientation ${\rm or}_\Lambda$ for a quiver ${\bf q}$. Then in 
 any quiver obtained by from ${\bf q}$ by mutations, 
the cluster volume forms 
 are  given by
$$
{\rm Vol}_{\mathcal A}= {\rm sign}(e_1, ..., e_n) d\log A_1 \wedge \ldots \wedge d\log A_n, 
~~~~{\rm Vol}_{\mathcal X}= {\rm sign}(e_1, ..., e_n) d\log X_1 \wedge \ldots \wedge d\log X_n.  
$$
\ed

\paragraph{Residues of the cluster volume form ${\rm Vol}_{\mathcal A}$ and frozen variables.} Take a space $M$ 
equipped with a cluster ${\cal A}$-coordinate system $\{A_i\}$. 

\bl \la{nonfr} Let us assume 
that $k\in {\rm I}-{\rm I}_0$ is nonfrozen,  and 
$\varepsilon_{kj}\not =0$ for some $j$. Then 
\be \la{reszero}
{\rm Res}_{A_k=0}({\rm Vol}_{\cal A})=0.
\ee
\el

\begin{proof} We have 
$
{\rm Res}_{A_k=0}({\rm Vol}_{\cal A}) = \pm \bigwedge_{i\not = k}d\log A_i.
$ 
Since $k$ is nonfrozen, there is 
an exchange relation 
(\ref{5.11.03.1a}). It  
implies a monomial relation on the locus $A_k=0$:
$\prod_{j} A_{j}^{\varepsilon_{kj}}= -1.$ 
Since $\varepsilon_{kj}$ is not identically zero, this monomial is nontrivial. 
Thus 
$\bigwedge_{i\not = k}d\log A_i =0$ at the $A_k=0$ locus. 
\end{proof}

 \bc \la{9.12.14.1} 
A coordinate $A_k$, with $\varepsilon_{kj}\not =0$ for some $j$, can be nonfrozen 
only if we have (\ref{reszero}), i.e. the functions 
$A_1, ..., \widehat A_k, ... , A_n$ become dependent on every component of the $A_k=0$ locus. 
\ec

If we define a cluster algebra axiomatically, without referring to a particular space 
on which it is realised, then any subset of an initial quiver can be declared to be 
the frozen subset. However if a cluster algebra is realised geometrically, we do not have 
much freedom  in the definition of frozen variables, as Corollary \ref{9.12.14.1} shows. 
This leads to the following geometric definition of the frozen coordinates. 

\bd \la{9.12.14.2} 
Let  $M$ be a space
equipped with a cluster ${\cal A}$-coordinate system. Then a cluster variable $\A$ is a frozen variable if and only if the residue form ${\rm Res}_{\A}({\rm Vol}_{\cal A})$ is not zero. 
\ed

\paragraph{Non-negative real points for a cluster algebra.}

The space of positive real points of any positive space is well defined. 
Let us define the space of non-negative real points 
for a cluster algebra.  

 Let 
$\{A_i^{\bf q}\}$, $i \in {\rm I}$, be the set of all cluster coordinates in a given quiver ${\bf q}$. 
The cluster algebra ${\cal O}_{\rm aff}({\cal A})$ is the algebra 
generated by the formal variables $\{A_i^{\bf q}\}$, 
for all quivers ${\bf q}$ related by mutations to a given one, 
modulo the ideal generated by exchange relations (\ref{5.11.03.1a}):
\be \la{affcv}
{\cal O}_{\rm aff}({\cal A}):= \frac{\Z[A_i^{\bf q}]}{(\mbox{\rm exchange relations})}.
\ee
This ring is not necessarily finitely generated. 
 Let ${\cal A}_{\rm aff}$ be its spectrum. 
Then the points of 
${\cal A}_{\rm aff}(\R_{\geq 0})$ are just the collections of positive real 
numbers $\{a_i^{\bf q}\in \R_{\geq 0}\}$ satisfying the exchange relations. The {\it positive boundary} 
is defined as the complement to the set of positive real points:
$$
 \partial{\cal A}_{\rm aff}(\R_{\geq 0}):= {\cal A}_{\rm aff}(\R_{\geq 0}) -  {\cal A}_{\rm aff}(\R_{> 0}).
$$

Let $A_{f}$ be a frozen variable. Then 
$\{A_{f} =0\} \cap \partial{\cal A}_{\rm aff}(\R_{\geq 0})$  is of real 
codimension one 
in ${\cal A}_{\rm aff}(\R_{\geq 0})$. 
Indeed, the frozen ${\cal A}$-cluster coordinates do not mutate, and so 
the codimension one domain given by the points with the coordinates 
$A_{f_t} =0, 
A^{\bf q}_{j} >0$ where $j$ is different then $f_t$ is a part of the intersection.

Let $A_k^{\bf q}$ be a non-frozen variable. 
It is likely, although we did not  prove this,   that in many cases
\be \la{claimposi}
\{A_k^{\bf q} =0\} \cap \partial{\cal A}_{\rm aff}(\R_{\geq 0}) ~~
\mbox{is of real codimension $\geq 2$ in} ~~{\cal A}_{\rm aff}(\R_{\geq 0}).
\ee
Indeed, 
the exchange relation for the $A_k^{\bf q}$, 
restricted to the $A_k^{\bf q} =0$ hyperplane, reads
$$
0 \cdot A_k^{\bf q'} = \prod_{j| \varepsilon_{kj} >0} A_{j}^{\varepsilon_{kj}} + \prod_{j| \varepsilon_{kj} <0} 
A_{j}^{-\varepsilon_{kj}}.
$$
So both monomials on the right, being non-negative, are zero, and 
each of them is non-empty: the empty one contributes $1$, violating 
$0$ on the left.  So we get at least two different cluster coordinates equal to zero. It is easy to see that then in any cluster coordinate system at 
least two of cluster coordinates are zero.


\subsection{Frozen variables, partial compactification
 $\widehat {\cal A}$,  and potential on the ${\cal X}$-space} \la{potfrv}

\paragraph{Potential on the ${\cal X}$-space}

\bl Any frozen $f\in {\rm I}_0$ gives rise to a tropical point $l_f\in {\cal A}(\Z^t)$ 
such that in any cluster ${\cal A}$-coordinate 
system  all tropical ${\cal A}$-coordinates except $a_f$ are zero, and $a_f=1$. 
\el

\begin{proof} Pick a cluster ${\cal A}$-coordinate system $\alpha=\{{\rm A}_f, \ldots\}$ starting from a coordinate $A_f$. Consider 
a tropical point in ${\cal A}(\Z^t)$ with the coordinates 
$(1,0,\ldots, 0)$.  
It is clear from (\ref{5.11.03.1atr}) that  the coordinates of this point are invariant under mutations 
at non-frozen vertices. 
Indeed, at least one of the two quantities we minimize in (\ref{5.11.03.1atr}) is zero, and the other 
must be non-negative. 
\end{proof}

\paragraph{The potential.} Let us assume that there there are canonical maps, implied by the cluster 
Duality Conjectures for the dual pair $({\cal A}, {\cal X}^\vee)$ of cluster varieties: 
$$
\mathbb{I}_{\cal A}:  ~~~{\cal A}(\Z^t)\stackrel{}{\lra}{\Bbb L}_+({\cal X}^\vee), ~~~~
\mathbb{I}_{\cal X}:  ~~~{\cal X}^\vee(\Z^t)\stackrel{}{\lra}{\Bbb L}_+({\cal A}).
$$
Here ${\Bbb L}_+({\cal X}^\vee)$ and ${\Bbb L}_+({\cal A})$ are the sets of 
universally Laurent functions. 

\bd Let us assume that for each frozen $f\in {\rm I}_0$  there is a function 
 $$
{\cal W}_{{\cal X}^\vee,f}:=\mathbb{I}_{{\cal A}}(l_f)\in {\Bbb L}_+({\cal X}^\vee)
$$
 predicted by the Duality Conjectures. Then the potential on the space  ${\cal X}$ is given by the sum 
$$
{\cal W}_{{\cal X}^\vee}:=\sum_{f\in {\rm I}_0}{\cal W}_{{\cal X}^\vee,f}. 
$$
\ed

\paragraph{Partial compactifications of the ${\cal A}$-space.} 
 Given any subset  ${\rm I}_0'\in {\rm I}_0$, we can define a partial completion 
${\cal A}\bigsqcup_{f\in {\rm I}_0'} {\rm D}_f$ of ${\cal A}$ 
by attaching to ${\cal A}$  
the divisor ${\rm D}_f$ corresponding to the equation $\A_f=0$ for each $f \in {\rm I}_0'$. 
The duality should look like
$$
({\cal A}\bigsqcup_{f\in {\rm I}_0'} {\rm D}_f) <=> ({\cal X}^\vee, \sum_{f\in {\rm I}_0'}{\cal W}_f).
$$
The order of pole of $\mathbb{I}_{\cal X}(l)$ at the divisor ${\rm D}_f$ should be equal to 
${\cal W}_f^t(l)$. In particular, $\mathbb{I}_{\cal X}(l)$ extends to ${\cal A}\bigsqcup {\rm D}_f$ 
if and only if it is in the  subset 
$\{l\in {\cal X}^\vee(\Z^t)~|~ {\cal W}_f^t(l)\geq 0\} \subset {\cal X}^\vee(\Z^t)$.

\paragraph{Canonical tropical points of the ${\cal X}$-space.}
Let $i \in {\rm I}$. 
Given a cluster ${\cal X}$-coordinate system,  consider 
a point $t_i \in {\cal X}(\Z^t)$ with the coordinates $\varepsilon_{ji}$, $j\in {\rm I}$.

\bl The point $t_i$  is invariant under mutations of cluster ${\cal X}$-coordinate systems.  
So there is a point $t_i \in {\cal X}(\Z^t)$ 
which in any cluster ${\cal X}$-coordinate system has 
coordinates   $\varepsilon_{ji}$, $j\in {\rm I}$.
\el

\begin{proof} Given a mutation in the direction of $k$, 
let us compare, using 
(\ref{5.11.03.1xtr}),  the rule how the ${\cal X}$-coordinates 
$\{\varepsilon_{ji}\}$, $j\in {\rm I}$ change with the mutation formulas (\ref{5.11.03.6}) for the matrix $\varepsilon_{ij}$.

Let us assume that  $k \not \in \{i,j\}$. Then, 
due to formula (\ref{5.11.03.1xtr}) for mutation of tropical ${\cal X}$-points, we have to prove that
\be \la{ef}
\varepsilon'_{ji} \stackrel{?}{=}\varepsilon_{ji}  - \varepsilon_{jk}{\rm min}\{0, -{\rm sgn} (\varepsilon_{jk}) \varepsilon_{ki}\}.
\ee

Let us assume now that  $\varepsilon_{jk}\varepsilon_{ki}<0$. Then ${\rm sgn} (-\varepsilon_{jk}) \varepsilon_{ki}>0$. So
 ${\rm min}\{0, {\rm sgn} (-\varepsilon_{jk}) \varepsilon_{ki}\}=0$, and the right hand side is 
$\varepsilon_{ji}$. This agrees with $\varepsilon'_{ij}= \varepsilon_{ij}$, see (\ref{5.11.03.6}), in this case. 

If $\varepsilon_{jk}\varepsilon_{ki}>0$, then  
${\rm sgn} (-\varepsilon_{jk}) \varepsilon_{ki}<0$. So the right hand side is 
$$
\varepsilon_{ji}  - \varepsilon_{jk}{\rm min}\{0, {\rm sgn} (-\varepsilon_{jk}) \varepsilon_{ki}\} = 
\varepsilon_{ji}  - \varepsilon_{jk}{\rm sgn} (-\varepsilon_{jk}) \varepsilon_{ki} = 
\varepsilon_{ji} + |\varepsilon_{jk}| \varepsilon_{ki}. 
$$
Comparing with (\ref{5.11.03.6}), we see that in both cases we get the expected formula (\ref{ef}). 

Finally, if $k \in \{i,j\}$, then $\varepsilon'_{ij} = -\varepsilon_{ij}$, and  
by formula (\ref{5.11.03.1xtr}), we also get $-\varepsilon_{ij}$. 
\end{proof}

Let us assume that, for each frozen $f\in {\rm I}_0$,  there is a function 
 $\mathbb{I}_{{\cal X}}(t_f)\in {\Bbb L}_+({\cal A}^\vee)$. 
 predicted by the duality conjectures. Then we conjecture that 
in many situations there exist monomials $M_f$  of frozen ${\cal A}$-coordinates  
such that the potential on the space  ${\cal A}$ is given by 
$$
{\cal W}_{{\cal A}^\vee}:= \sum_{f\in {\rm I}_0}M_f\cdot \mathbb{I}_{{\cal X}}(t_f). 
$$

\begin{appendices}
\section{Geometric crystal structure on ${\rm Conf}({\cal A}^{n}, {\cal B})$} \la{sec9}
We construct a  geometric crystal structure on ${\rm Conf}({\cal A}^{n}, {\cal B})$. See  \cite{BK3} for definition of geometric crystals. 
 By the positive birational isomorphism $p: {\rm Conf}({\cal A}^2, {\cal B})\ra \B^{-}$, it recovers the crystal structure on $\B^{-}$ defined in Example 1.10 of {\it loc.cit}.

 Using machinery of tropicalization, the subset ${\bf B}_{\lambda}^{\mu}\subset {\rm Conf}({\cal A}^2,{\cal B})(\Z^t)$ becomes a crystal basis. We refer the reader to Section 2 of {\it loc.cit.}  for details concerning the tropicalization of geometric crystals. 
Braverman and Gaitsgory \cite{BG} construct a crystal structure for the MV basis. As a direct consequence of this paper, the MV basis are parametrized by the set ${\bf B}_{\lambda}^{\mu}$. Therefore we give a direct isomorphism between the crystals of \cite[Theorem 6.15]{BK2}  and \cite{BG} without using the uniqueness theorem from \cite{J}.  

The tensor product of crystals can be interpreted as tropicalizations of the convolution product $*$ from Section \ref{sec7.1h}.
Given the geometric background of the configuration spaces, definitions/proofs become simple.

\subsection{Geometric crystal structure on ${\rm Conf}({\cal A}^n, {\cal B})$}
\la{sec7.01h}

Let $x=({\rm A}_1,\ldots, {\rm A}_{n}, {\rm B}_{n+1})\in {\rm Conf}({\cal A}^{n}, {\cal B})$. Let $i\in I$. Recall the following positive maps:
\begin{itemize}
\item
$p: {\rm Conf}({\cal A}^n,{\cal B})\ra {\B^-}$ such that $p(x)=b^{\A_1,\A_n}_{\B_{n+1}}$.
\item
$\mu: {\rm Conf}({\cal A}^{n},{\cal B})\ra {\rm H}$ such that $\mu(x)=\mu_{{\rm B}_{n+1}}^{{\rm A}_{1}, {\rm A}_{n}}.$
\item
$\varphi_i, \varepsilon_i, {\cal W}: {\rm Conf}({\cal A}^{n},{\cal B})\ra {\Bbb A}^1$ such that 
$$\varphi_i(x)={\cal L}_i(u_{{\rm B}_{n+1}, {\rm B}_{n}}^{{\rm A}_1}),  ~~\varepsilon_i(x)={\cal R}_i(u_{{\rm B}_{1}, {\B}_{n+1}}^{{\rm A}_{n}}), ~~{\cal W}(x)=\sum_{k=1}^{n}\chi(u_{{\rm B}_{k-1}, {\rm B}_{k+1}}^{{\rm A}_{k}}).$$
Here $k$ is modulo {\it n}+1.
\item
$e_i^{\cdot}: {\Bbb G}_m\times {\rm Conf}({\cal A}^{n},{\cal B})\ra {\rm Conf}({\cal A}^{n},{\cal B})$, where $e_i^c(\A_1,\ldots, \A_n,\B_{n+1})=(\A_1,\ldots, \A_n, \wt{\B}_{n+1})$ is such that 
$$
e_i^c(\A_1,\A_n, \B_{n+1})=(\A_1,\A_n,\wt{\B}_{n+1})
$$
The action $e_i$ on ${\rm Conf}({\cal A}^2,{\cal B})$ is defined by \eqref{13.4.22.2108h}.
\end{itemize}

\bt
The 6-tuple $\big({\rm Conf}({\cal A}^{n}, {\cal B}), \mu, {\cal W}, \varphi_i, \varepsilon_i, e_i^{\cdot}|i\in I\big)$ is a positive decorated geometric crystal.
\et

{\bf Warning.} The maps $\varphi_i$, $\varepsilon_i$ are the inverse of those $\varphi_i$, $\varepsilon_i$ in \cite[Definition 1.3]{BK3}.

\begin{proof}
By \cite[Definitions 1.3 $\&$ 2.7]{BK3}, it remains to show the following Lemmas.
\bl \la{8.13.4.7h}
Let $\alpha_i^\vee$ and $\alpha_i$ be the simple coroot and simple root corresponding to $i\in I$.  Let $x=(\A_1,\ldots, \A_n, \B_{n+1})$. Let $c\in {\Bbb G}_m$. Then
\begin{itemize}
\item[1.]$\mu(e_i^c(x))=\alpha_i^\vee (c)\mu(x)$.
\item[2.] $\varepsilon_i(x)\alpha_i(\mu(x))=\varphi_i(x)$.
\item[3.] $\varphi_i(e_i^c(x))=c\varphi_i(x)$, $\varepsilon_i(e_i^c(x))=c^{-1}\varepsilon_i(x)$.
\item[4.] ${\cal W}(e_i^c(x))={\cal W}(x)+(c-1)\varphi_i(x)+(c^{-1}-1)\varepsilon_i(x)$.
\end{itemize}
\el

\begin{proof} 
We  pick ${\A_{n+1}}$ such that $\pi(\A_{n+1})=\B_{n+1}.$ By \eqref{8.13.4a},
$$
\chi_{i^*}(u_{\B_n,\B_1}^{\A_{n+1}})=\frac{\alpha_{i}(h_{\A_1,\A_{n+1}})}{{\cal L}_i(u_{\B_{n+1},\B_n}^{\A_1})}=\frac{\alpha_i(h_{\A_n,\A_{n+1}})}{{\cal R}_i(u_{\B_1,\B_{n+1}}^{\A_n})}.
$$
Therefore,
\be \la{13.6.9.2059h}
\frac{\varphi_i(x)}{\varepsilon_i(x)}=\frac{{\cal L}_i(u_{\B_{n+1},\B_n}^{\A_1})}{{\cal R}_i(u_{\B_1,\B_{n+1}}^{\A_n})}=\alpha_{i}(h_{\A_1,\A_{n+1}}h_{\A_n,\A_{n+1}}^{-1})=\alpha_i(\mu(x)).
\ee
The last identity is due to Property 4 of Lemma \ref{12.12.11.h}.
Thus 2 follows.

Recall the proof of Lemma \ref{13.6.9.1548h}. Let ${\bf i}=(i_1,\ldots, i_m)$ be a reduced word for $w_0$ such that $i_1=i$. Assume $p(x)$ is expressed by \eqref{8.15.3.0h}.
Recall the positive coroots $\beta_k^{\bf i}$ in Lemma \ref{8.13.1.20h}. In particular $\beta_1^{\bf i}=\alpha_{i}^\vee$.
 By  Lemma \ref{8.13.1.20h} and property 4 of Lemma \ref{12.12.11.h}, we have
$$
\mu(x)=h_{\A_1, \A_n}\beta(u_{\B_1, \B_{n+1}}^{\A_n})=h\prod_{k=1}^m\beta_{k}^{\bf i}(b_k^{-1}).
$$
Similarly, by \eqref{13.4.22.1550h}, we have 
$$\mu(e_i^c(x))=h\beta_1^{\bf i}(cb_1^{-1})\prod_{k=2}^m\beta_{k}^{\bf i}(b_k^{-1})=\alpha_i^\vee(c)\mu(x).
$$
Thus 1 follows. By \eqref{13.4.22.1550h} and definitions of the functions $\varepsilon_i$, $\varphi_i$, ${\cal W}$,  we get 3 and 4.
\end{proof}

\bl \la{13.4.22.2051h}
For two different $i,j \in I$, set $a_{ij}=\langle\alpha_i, \alpha_j^{\vee}\rangle$. We have the following relation
\begin{align}
&e_i^{c_1}e_j^{c_2}=e_j^{c_2}e_i^{c_1}~~~\mbox{if }a_{ij}=0;\\
&e_i^{c_1}e_j^{c_1c_2}e_i^{c_2}=e_j^{c_2}e_{i}^{c_1c_2}e_{j}^{c_1}~~~\mbox{if }a_{ij}=a_{ji}=-1;\\
&e_i^{c_1}e_j^{c_1^2c_2}e_i^{c_1c_2}e_j^{c_2}=e_j^{c_2}e_i^{c_1c_2}e_{j}^{c_1^2c_2}e_{i}^{c_1}~~~\mbox{if } a_{ij}=-1, ~a_{ji}=-2;\\
&e_i^{c_1}e_j^{c_1^3c_2}e_i^{c_1^2c_2}e_j^{c_1^3c_2^2}e_i^{c_1c_2}e_j^{c_2}=e_j^{c_2}e_i^{c_1c_2}e_j^{c_1^3c_2^3}e_i^{c_1^2c_2}e_j^{c_1^3c_2}e_i^{c_1}~~~\mbox{if }a_{ij}=-1, ~a_{ji}=-3.
\end{align} 
\el
\begin{proof}
By the definition of the action $e_i^{\cdot}$, it is enough to prove the case when $n=2$, i.e. ${\rm Conf}({\cal A}, {\cal A}, {\cal B})$.
By \eqref{13.4.22.1550h}, we reduce the Lemma to the case when $\G$  is of rank 2.
The first identity is clear. 
For the second identity, one can check for $\G={\rm PGL}_3$ case directly.
The third and the fourth identities can be reduced to simple-laced case by ``folding".  
See \cite[Section 5.2]{BK1} for details.
\end{proof}

\end{proof}

\bt
\la{8.18h.con.pro}Let $a\in {\rm Conf}^*({\cal A}^{m+1},{\cal B})$ and let $b\in {\rm Conf}^*({\cal A}^{n+1},{\cal B})$.
Recall the convolution product $*$ from Section \ref{sec7.1h}.
The following identities hold
\begin{itemize}
\item[1.] $p(a*b)=p(a)p(b)$, $\mu(a*b)=\mu(a)\mu(b)$.
\item[2.] ${\cal W}(a*b)={\cal W}(a)+{\cal W}(b)$.
\item[3.] $\varphi_i(a*b)=\frac{\varphi_i(a)\varphi_i(b)}{\varepsilon_i(a)+\varphi_i(b)}$, $\varepsilon_i(a*b)=\frac{\varepsilon_i(a)\varepsilon_i(b)}{\varepsilon_i(a)+\varphi_i(b)}$.
\item[4.] $e_i^c(a*b)=e_i^{c_1}(a)*e_i^{c_2}(b)$, where $c_1=\frac{\varepsilon_i(a)+c\varphi_i(b)}{\varepsilon_i(a)+\varphi_i(b)}$, $c_2=\frac{\varepsilon_i(a)+\varphi_i(b)}{c^{-1}\varepsilon_i(a)+\varphi_i(b)}$.
\end{itemize}
\et

\begin{proof}
1-2.  Follow from Lemma \ref{LEMMA.4.2}.
\vskip 2mm
3. We prove the second formula. The first one follows similarly.
By Figure \ref{convmap}, it suffices to prove the case when  $a=(\A_1, \A_2, \B_4), b=(\A_2, \A_3, \B_4)\in {\rm Conf}({\cal A}^2,{\cal B})$.  Then $a*b=(\A_1,\A_2,\A_3,\B_4)$. 
Pick $\A_4\in{\cal A}$ such that $\pi(\A_4)=\B_4$. By \eqref{13.6.9.2059h}, $\alpha_i(h_{\A_2,\A_4})={\alpha_i(h_{\A_3,\A_4})\varphi_i(b)}/{\varepsilon_i(b)}.$ By \eqref{8.13.4a}, we have
$$
\chi_{i^*}(u_{\B_3,\B_2}^{\A_4})=\frac{\alpha_i(h_{\A_3,\A_4})}{\varepsilon_i(b)},~~
\chi_{i^*}(u_{\B_2,\B_1}^{\A_4})=\frac{\alpha_i(h_{\A_2,\A_4})}{\varepsilon_i(a)}=\frac{\alpha_i(h_{\A_3,\A_4})\varphi_i(b)}{\varepsilon_i(a)\varepsilon_i(b)},~~
\chi_{i^*}(u_{\B_3,\B_1}^{\A_4})=\frac{\alpha(h_{\A_3,\A_4})}{\varepsilon_i(a*b)}.
$$
Since $\chi_{i^*}(u_{\B_3,\B_1}^{\A_4})=\chi_{i^*}(u_{\B_3,\B_2}^{\A_4})+\chi_{i^*}(u_{\B_2,\B_1}^{\A_4})$, the formula follows.
\vskip 2mm

4. It suffices to prove the case when  $a=(\A_1, \A_2, \B_4), b=(\A_2, \A_3, \B_4)$. Let
$e_i^c(a*b)=(\A_1,\A_2,\A_3,{\B}_4')$.
Recall the definition of $e_i$ by cross ratio in Section \ref{sec7.01h}.
Let ${\rm P}\in {\cal P}_i $ be the parabolic subgroup containing $\B_4$ and $\B_4'$. 
Let $\B_1',\B_2', \B_3'\in {\cal B}_{\rm P}$ such that
$$
\B_4\stackrel{s_i}{\lra}{\B_k'}\stackrel{s_iw_0}{\lra}\pi(\A_k),~~~~k=1,2,3.
$$
 We have
\be \la{13.4.22.2139h}
c=r(\B_4, \B_4'; \B_1', \B_3' )=r(\B_4, \B_4'; \B_1', \B_2' )r(\B_4, \B_4'; \B_2', \B_3' )=c_1c_2.
\ee
Note that 
\be \la{13.4.22.2140h}
\varepsilon_i(a)+\varphi_i(b)=\chi(u_{\B_1', \B_4}^{\A_2})+\chi(u_{\B_4, \B_3'}^{\A_2})=\chi(u_{\B_1', \B_3'}^{\A_2})=
\chi(u_{\B_1', \B_4'}^{\A_2})+\chi(u_{\B_4', \B_3'}^{\A_2})=c_1^{-1}\varepsilon_i(a)+c_2\varphi_i(b).
\ee
Combining \eqref{13.4.22.2139h}\eqref{13.4.22.2140h}, we get 4. 
\end{proof}
{\bf Remark.} This Theorem recovers the properties in \cite[Lemma 3.9]{BK1}. It is analogous to the tensor product of Kashiwara's crystals.

\begin{figure}[ht]
\epsfxsize350pt
\centerline{\epsfbox{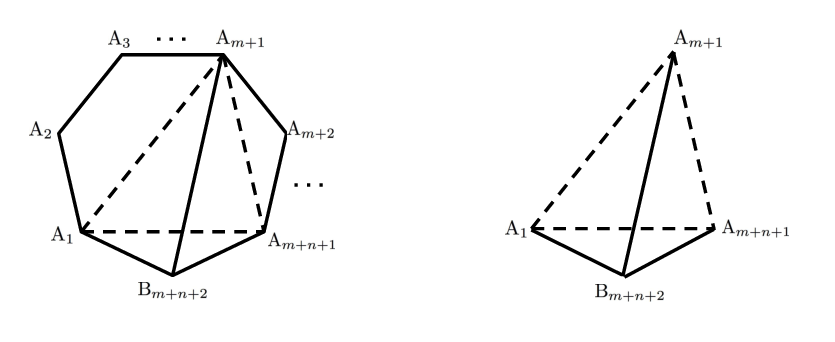}}
\caption{Convolution products of configurations.}
\label{convmap}
\end{figure}


\subsection{A simple illustration of tensor product  of crystals.} 
Recall the subsets ${\bf B}_{\underline{\lambda}}^{\mu}$ and ${\bf C}_{\underline{\lambda}}^{\nu}$. We tropicalize the map
$$
c_{1,n+1}: {\rm Conf}({\cal A}^{n+1}, {\cal B})\lra {\rm Conf}_{n+1}({\cal A}) \times {\rm Conf}({\cal A}^2,{\cal B}),$$
$$(\A_1,\ldots,\A_{n+1},\B_{n+2})\lms (\A_1,\ldots, \A_{n+1})\times (\A_1,\A_{n+1},\B_{n+2}).
$$
It provides a canonical decomposition
\be \la{12.12.18.3h18}
{\bf B}_{\underline{\lambda}}^{\mu}=\bigsqcup_{\nu} {\bf C}_{\underline{\lambda}}^{\nu}\times {\bf B}_{\nu}^{\mu}.
\ee
Recall the decomposition \eqref{13.2.1.7.01h}. As illustrated by Figure \ref{pd}, we get
\bt \la{12.18.thm8.9}
There is a canonical bijection
$$
\bigsqcup_{\mu_1+\ldots+\mu_{n}=\mu}{\bf B}_{\lambda_1}^{\mu_1} \times \ldots \times {\bf B}_{\lambda_{n}}^{\mu_{n}}=\bigsqcup_{\nu} {\bf C}_{\lambda_1,\ldots, \lambda_n}^{\nu}\times {\bf B}_{\nu}^{\mu}
$$
\et
 
 We show that ${\bf B}_{\lambda}^{\mu}$ parametrizes a crystal basis. Figure \ref{pd} illustrates the tensor product of crystals. When $n=2$, it  recovers \cite[Theorem 3.2]{BG}.
\begin{figure}[ht]
\epsfxsize400pt
\centerline{\epsfbox{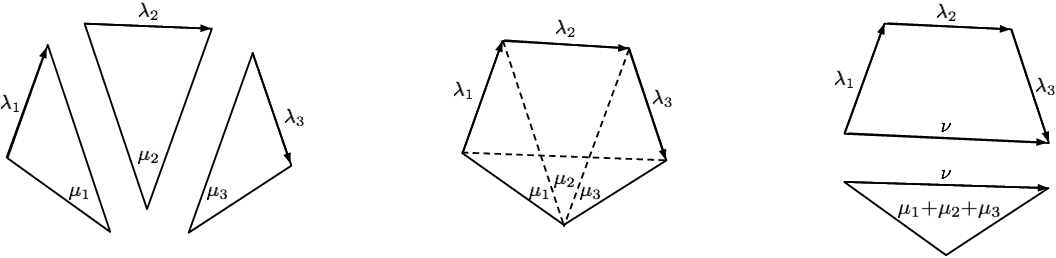}}
\caption{Tensor product structure of the crystal basis}
\label{pd}
\end{figure}


\section{Weight multiplicities and tensor product multiplicities}
\la{sec10}

 The following Theorem is due to \cite[Section 8]{L}. We provide a simple proof by using the set ${\bf B}_{\lambda}^\mu$.
\bt [\cite{L}] \la{8.21.wtm.h}
The weight multiplicity  $b_{\lambda}^{\mu}:={\rm dim}~V_{\lambda}^{\mu}$ is equal to the cardinality of the subset
\be \la{13.2.1.519h}
\{l\in {\bf A}_{\lambda-\mu}~|~ {\cal L}_i^t(l)\leq \langle \lambda, \alpha_{i^*}\rangle,~\forall i\in I\}\subset \U^+_{\chi}(\Z^t).
\ee
\et
\begin{proof}
Recall the positive birational isomorphism 
$$
\alpha_2=(\pi_{12}, p_2): {\rm Conf}({\cal A}^2;{\cal B})\lra {\rm H}\times \U,~~~(\A_1,\A_2,\B_3)\lms (h_{\A_1,\A_2}, u_{\B_1,\B_3}^{\A_2}).
$$
The potential ${\cal W}$ on ${\rm Conf}({\cal A}^2;{\cal B})$ induces a positive function
${\cal W}_{\alpha_2}={\cal W}\circ \alpha_2^{-1}: {\rm H}\times \U \ra {\Bbb A}^1.$
By Lemma \ref{lem1}, we get
$$
{\cal W}_{\alpha_2}(h,u)=\sum_{i\in I}\frac{\alpha_{i^*}(h)}{{\cal L}_i(u)}+\chi(u).
$$
Its tropicalization becomes
$$
{\cal W}_{\alpha_2}^{t}(\lambda, l)=\min_{i\in I}\{\langle \lambda, \alpha_{i^*}\rangle- {\cal L}_{i}^{t}(l), \chi^t(l)\}
$$
Therefore, under the map $p_2^t$, the set ${\bf B}_{\lambda}^{\mu}$ is identified with the set \eqref{13.2.1.519h}.
Note that the set ${\bf B}_{\lambda}^\mu$ parametrized the MV basis of the weight space $V_{\lambda}^{\mu}$ of $V_{\lambda}$. 
Therefore the weight multiplicity  $b_{\lambda}^{\mu}$ is equal to the cardinality of the set ${\bf B}_{\lambda}^\mu$.
\end{proof}

Below we give a (rather simple) proof of \cite[Corollary 3.4]{BZ} based on Proposition \ref{13.2.1.5.29h}. 
\bt[\cite{BZ}] \la{COMBZ} The tensor product multiplicity $c_{\lambda, \nu}^{\mu}$ is equal to the cardinality of the set
\be \la{13.2.1.6.13h}
\{l\in {\bf A}_{\lambda+\nu-\mu}~|~ {\cal L}_i^t(l)\leq \langle \lambda, \alpha_{i^*}\rangle \text{ and }  {\cal R}_i^t(l)\leq \langle \nu, \alpha_{i}\rangle\text{ for all } i\in I\}\subset \U^+_{\chi}(\Z^t)
\ee
\et

\begin{proof}
Recall the positive birational isomorphism
$$
\alpha_2=(\pi_{12}, p_2, \pi_{23}): {\rm Conf}_3({\cal A})\stackrel{\sim}{\lra} {\rm H}\times\U \times {\rm H},~~~(\A_1,\A_2,\A_3)\lra(h_{\A_1,\A_2},  u_{\B_1,\B_3}^{\A_2}, h_{\A_2,\A_3}). 
$$
By Theorem \ref{13.2.1.6.05h}, the function ${\cal W}_{\alpha_2}={\cal W}\circ \alpha_2^{-1}$ becomes
\be \la{25huh}
{\cal W}_{\alpha_2}(h_1,u,h_2)=\sum_{i\in I}\frac{\alpha_i(h_1)}{{\cal L}_{i^*}(u)}+\chi(u)+\sum_{i\in I}\frac{\alpha_i(h_2)}{{\cal R}_i(u)}.
\ee
We tropicalize (\ref{25huh}):
$$
{\cal W}_{\alpha_2}^t(\lambda, l, \nu)=\min \big\{\min_{i\in I}\{\langle \lambda, \alpha_{i^*}\rangle-{\cal L}_{i}^t(l)\},~\chi^t(l),~\min_{i\in I}\{\langle \nu, \alpha_i\rangle-{\cal R}_{i}^t(l)\}\big\}.
$$
Therefore, under $p_2^t$, the set ${\bf C}_{\lambda, \nu}^{\mu}$ is identified with \eqref{13.2.1.6.13h}. The rest is due  to Proposition \ref{13.2.1.5.29h}.
\end{proof}

\section{Cycles assigned to ${\rm Conf}^+_w({\cal A}^n, {\cal B}, {\cal B})(\Z^t)$} 
\la{sec2.3.3}

Given an element $w\in W$, there are moduli spaces
$$
{\rm Conf}_{w}({\cal A}^n, {\cal B}, {\cal B}) \subset 
{\rm Conf}({\cal A}^n, {\cal B}, {\cal B}), ~~~{\rm Conf}^{\cal O}_{w}({\cal A}^n, {\cal B}, {\cal B})\subset 
{\rm Conf}^{\cal O}({\cal A}^n, {\cal B}, {\cal B})
$$
determined by the condition that the last two flags belong to the 
${\rm G}$-orbit $({\cal B}\times {\cal B})_w \subset {\cal B}\times {\cal B}$ 
parametrised by $w$. So there are decompositions into disjoint unions
$$
{\rm Conf}({\cal A}^n, {\cal B}, {\cal B}) = \coprod_{w\in W}{\rm Conf}_{w}({\cal A}^n, {\cal B}, {\cal B}), ~~~~
{\rm Conf}^{\cal O}({\cal A}^n, {\cal B}, {\cal B}) = \coprod_{w\in W}
{\rm Conf}^{\cal O}_{w}({\cal A}^n, {\cal B}, {\cal B}).
$$
Similarly there are moduli space ${\rm Conf}_{w}({\rm Gr}^n, {\cal B}, {\cal B})$, and  a canonical map 
$$
\kappa_w: {\rm Conf}^{\cal O}_{w}({\cal A}^n, {\cal B}, {\cal B}) \lra 
{\rm Conf}_{w}({\rm Gr}^n, {\cal B}, {\cal B}). 
$$
The subgroup ${\rm B}_w:= {\rm B}\cap w{\rm B}w^{-1}$ is a 
stabiliser of the group ${\rm G}$ acting  on $({\cal B}\times {\cal B})_w$. So one has 
\be \la{stack10}
{\rm Conf}_{w}({\rm Gr}^n, {\cal B}, {\cal B})=
{\rm B}_w({\cal K})\backslash 
{\rm Gr}^{n}, ~~~~{\rm B}_w:= {\rm B}\cap w{\rm B}w^{-1}.
\ee

For $w=e$ we get 
\be \la{stack1}
{\rm Conf}_{e}({\rm Gr}^n, {\cal B}, {\cal B})= {\rm Conf}({\rm Gr}^n, {\cal B}).
\ee

The space ${\rm Conf}_{w}({\cal A}, {\cal B}, {\cal B})$ is birational isomorphic to a subgroup 
${\rm U}_{(w)}:= {\rm U} \cap w^{-1}B^-w \subset {\rm U}$. The latter has 
a positive structure defined by Lusztig \cite{L} using reduced decompositions of $w$. 
Combined with the standard construction, we arrive at a positive atlas on 
${\rm Conf}_{w}({\cal A}^n, {\cal B}, {\cal B})$.

There is a potential 
${\cal W}$ on ${\rm Conf}_{w}({\cal A}^n, {\cal B}, {\cal B})$, 
defined by restriction of the usual one. 
So the set ${\rm Conf}^+_{w}({\cal A}^n, {\cal B}, {\cal B})(\Z^t)$ is defined. 
The canonical map $\kappa_w$ 
provides a collection of cycles
\be \la{MVGCyy}
{\cal M}_l^\circ\subset {\rm Conf}_{w}({\rm Gr}^n, {\cal B}, {\cal B}), ~~~~l\in 
{\rm Conf}^+_{w}({\cal A}^n, {\cal B}, {\cal B})(\Z^t).
\ee 

\vskip 3mm
Notice that our approach  makes the  map $\kappa$ transparent, and allows to avoid any kind of explicit parametrisations in its definition.  It makes obvious a parametrisation of generalized MV cycles, defined as components of 
$\overline{{\rm S}_{e}^{\lambda}\cap {\rm S}_{w}^{\mu}}$ for arbitrary $w\in W$ --  one needs to use 
the whole configuration space 
${\rm Conf}({\cal A}, {\cal B}, {\cal B})$, not only its generic part. 

\paragraph{Constructible functions $D_F$.} Let 
$F$ be a 
rational function on ${\cal A}^{n}\times ({\cal B}\times {\cal B})_w$, invariant under the left diagonal 
action of ${\rm G}$. Using the isomorphism
$
\Q({\cal A}^{n}\times ({\cal B}\times {\cal B})_w)^{\rm G} = \Q({\cal A}^{n})^{{\rm B}_w}, 
$ 
we realize $F$ as an ${\rm B}_w$-invariant rational function on ${\cal A}^{n}$. 
Define a function $D_F$ on 
${\rm G}({\cal K})^{n}$ by 
\be \la{dff}
D_F(g_1(t), ..., g_{n}(t)):= {\rm val}~F(g_1(t)A_1, ..., g_{n}(t)A_n) ~~~\mbox{\rm for some $A_1, ..., A_{n}
\in {\cal A}(\C)$}. 
\ee
It is left ${\rm B}_w({\cal K})$-equivariant, and right ${\rm G}({\cal O})^{n}$-equivariant, 
and so descends to a function 
$$
D_F:  {\rm B}_w({\cal K})\backslash {\rm Gr}^{n}\lra \Z.
$$

{\bf Remark}. The function $D_F$ assigned to a positive rational function $F$ on 
${\rm Conf}({\cal A}^n, {\cal B}, {\cal B})$ is not a function on the whole space 
${\rm Conf}({\rm Gr}^n, {\cal B}, {\cal B})$, only on its generic part. 
One has 
$$
{\rm Conf}({\rm Gr}^n, {\cal B}, {\cal B}) =\coprod_{w\in W}{\rm Conf}_{w}({\rm Gr}^n, {\cal B}, {\cal B}),
$$
and one needs to use the positive rational functions on the strata ${\rm Conf}_{w}({\cal A}^n, {\cal B}, {\cal B})$ 
to define constructable functions on the strata ${\rm Conf}_{w}({\rm Gr}^n, {\cal B}, {\cal B})$.

\bt \la{5.8.10.45xw}
Let $l\in{\rm Conf}_w^+({\cal A}^{n}, {\cal B}, {\cal B})(\Z^t)$, and $F\in {\Q}_{+}({\rm Conf}_w({\cal A}^{n}, 
{\cal B}, {\cal B}))$. 
Then we have
$$
D_F\big({\cal M}_l^\circ\big)\equiv F^t(l).
$$
\et

\end{appendices}

\end{document}